\pdfoutput=1

\documentclass[12pt,a4paper]{article}
\usepackage{modnatbib}
\usepackage{amssymb,amsmath,amsthm,epsfig,verbatim,pstricks,url}
\usepackage{ifpdf}  
\newtheorem{thm}{\sc Theorem.}[section]
\newtheorem{lem}[thm]{\sc Lemma.}
\newtheorem{rem}[thm]{\sc Remark.}
\renewcommand{\theequation}{\arabic{section}.\arabic{equation}}
{{\upshape\bfseries AMS subject classifications. }\ignorespaces}{}
\newenvironment{keywords}{{\upshape\bfseries Key words. }\ignorespaces}{}

\newcommand{\Rplus}{{\mathbb R}_{>0}}
\newcommand{\Rgeq}{{\mathbb R}_{\geq 0}}

\newcommand{\R}{{\mathbb R}}

\newcommand{\bZ}{\mathbb{Z}}

\newcommand{\tr}{\operatorname{tr}}

\newcommand{\vol}{\mathcal{L}^d} 

\newcommand{\Gauss}{{\mathcal{K}}}
\newcommand{\GT}{{\Gamma_T}}
\newcommand{\GhT}{{\Gamma^h_T}}

\newcommand{\dH}[1]{\;{\rm d}{\mathcal{H}}^{#1}} 
\newcommand{\dL}[1]{\;{\rm d}{\mathcal{L}}^{#1}} 

\newcommand{\spont}{{\overline\varkappa}}
\newcommand{\bigchi}{\ensuremath{\mathrm{\mathcal{X}}}}
\newcommand{\charfcn}[1]{\bigchi_{#1}} 

\newcommand{\matVhz}{\mat V(\Gamma^0)}
\newcommand{\Vhz}{\underline{V}(\Gamma^0)}
\newcommand{\Whz}{W(\Gamma^0)}
\newcommand{\Whp}{W(\Gamma^{m+1})}
\newcommand{\Whpint}{W_{\int}(\Gamma^{m+1})}
\newcommand{\WhpVI}{W_{\leq1}(\Gamma^{m+1})}
\newcommand{\WhVI}{W_{\leq1}(\Gamma^m)}

\newcommand{\matVh}{\mat V(\Gamma^m)}
\newcommand{\Vh}{\underline{V}(\Gamma^m)}

\newcommand{\Wh}{W(\Gamma^m)}

\newcommand{\Vt}{[H^1(\Gamma(t))]^d}
\newcommand{\Wt}{H^1(\Gamma(t))}

\newcommand{\Vdeltat}{[H^1(\Gamma_\epsilon(t))]^d}
\newcommand{\Wdeltat}{H^1(\Gamma_\epsilon(t))}
\newcommand{\Vht}{\underline{V}(\Gamma^h(t))}
\newcommand{\Vhtz}{\underline{V}(\Gamma^h(0))}
\newcommand{\matVht}{\mat V(\Gamma^h(t))}

\newcommand{\Wht}{W(\Gamma^h(t))}
\newcommand{\Whtz}{W(\Gamma^h(0))}
\newcommand{\uspace}{\mathbb{U}}
\newcommand{\utimespace}{\mathbb{V}}
\newcommand{\pspace}{\mathbb{P}}
\newcommand{\Psing}{P_{\rm sing}}
\newcommand{\sigmaO}{o}
\newcommand{\kappastar}{\vec\varkappa\rule{0pt}{0pt}^\star}
\newcommand{\nabs}{\nabla_{\!s}}
\newcommand{\Id}{\rm Id}
\newcommand{\id}{\rm id}
\newcommand{\deldel}[1]{\frac{\delta}{{\delta}#1}}
\newcommand{\dd}[1]{\frac{\rm d}{{\rm d}#1}}
\newcommand{\ddt}{\dd{t}}

\newcommand{\matpartu}{\partial_t^\bullet}

\newcommand{\matpartx}{\partial_t^\circ}
\newcommand{\matpartxh}{\partial_t^{\circ,h}}
\newcommand{\matparteps}{\partial_\epsilon^0}
\newcommand{\SbulkT}{\vec{S}_{\Gamma,\Omega}^T}
\newcommand{\Sbulk}{\vec{S}_{\Gamma,\Omega}}

\newcommand{\Mbulk}{\vec{M}_{\Gamma,\Omega}}
\newcommand{\MbulkT}{(\vec{M}_{\Gamma,\Omega})^T}

\newcommand{\phasec}{\mathfrak{c}}
\newcommand{\phaseC}{\mathfrak{C}}
\newcommand{\chempot}{\mathfrak{m}}
\newcommand{\Chempot}{\mathfrak{M}}
\newcommand{\unitn}{\vec{\rm n}}

\newcommand{\ek}{e}

\def\epsilon{\varepsilon} 
\newcommand{\mat}[1]{\underline{\underline{#1}}\rule{0pt}{0pt}}

\def\vL{L\kern-0.08cm\char39}
\begin{document}
\title{
Finite Element Approximation for the \\
Dynamics of Fluidic Two-Phase Biomembranes 
}
\author{John W. Barrett\footnotemark[2] \and 
        Harald Garcke\footnotemark[3]\ \and 
        Robert N\"urnberg\footnotemark[2]}

\renewcommand{\thefootnote}{\fnsymbol{footnote}}
\footnotetext[2]{Department of Mathematics, 
Imperial College London, London, SW7 2AZ, UK}
\footnotetext[3]{Fakult{\"a}t f{\"u}r Mathematik, Universit{\"a}t Regensburg, 
93040 Regensburg, Germany}

\date{}

\maketitle

\begin{abstract}
  Biomembranes and vesicles consisting of multiple phases can attain a
  multitude of shapes, undergoing complex shape transitions. We study a
  Cahn--Hilliard model on an evolving hypersurface coupled to
  Navier--Stokes equations on the surface and in the surrounding
  medium to model these phenomena. The evolution is driven by a
  curvature energy, modelling the elasticity of the membrane, and by a
  Cahn--Hilliard type energy, modelling line energy effects. A stable
  semidiscrete finite element approximation is introduced and, with
  the help of a fully discrete method, several phenomena occurring for
  two-phase membranes are computed.
\end{abstract} 


\renewcommand{\thefootnote}{\arabic{footnote}}

\setcounter{equation}{0}
\section{Introduction} \label{sec:1}

In lipid bilayer membranes a large variety of different shapes and
complex shape transition behaviour can be observed. Biological
membranes are composed of several components, and lateral separation
into different phases or domains have been studied in experiments.
Mathematical models for biological membranes treat them as a
deformable inextensible fluidic surface governed by bending energies,
which involve the curvature of the membrane. If different phases occur,
these bending energies will depend on the individual phases, and the
local shape of the membrane will depend on the phase present locally.
It has also been observed that the interfacial
energy of the phase boundaries on the membrane can have a pronounced
effect on the membrane shape, and might lead to effects like budding
and fission. We refer to \cite{BaumgartHW03} for experimental studies
and to \cite{DobereinerKNSS93,Lipowsky92,JulicherL96,VeatchK03,BaumgartDWJ05} 
for further information on membranes with different fluid phases.

There has been a huge interest in the modelling
of (two-phase) biomembranes. Both equilibrium shapes, as well as the
evolution of membranes, have been studied intensively. 
However, a model taking the fluidic
behaviour of the membrane, the curvature elasticity, the 
interfacial line energy and the phase separation in a time dependent
model into account is missing so far. 
It is the goal of this paper to present such a model and --which will
be the main contribution of this paper-- to come up with a stable numerical
approximation scheme for the resulting equations. The model will be
based on an elastic bending energy of Canham--Evans--Helfrich type and
a Ginzburg--Landau energy modelling the interfacial energy. Through
their first variation these energy contributions lead to driving
forces for the evolution, which is given by a surface Navier--Stokes
system, coupled to bulk dissipation of an ambient fluid, and a
convective Cahn--Hilliard type equation, which is formulated on the
evolving membrane. The fluid part of the model goes back to
\cite{ArroyoS09}, whereas an evolution based on a
Canham--Evans--Helfrich energy coupled to a Ginzburg--Landau energy
on the surface has been studied in the context of gradient flows by
\cite{ElliottS10,ElliottS10a,ElliottS13,Helmers13,MerckerM-CRH13,Helmers15,
MerckerM-C15}. However, a coupling, which will
give the natural dynamics on the interface, is stated here for the
first time, and we will show that physically reasonable energy
dissipation inequalities hold. Here the dissipation has
contributions stemming from 
viscous friction in the bulk and on the surface, and from dissipation
due to diffusion on the membrane. 

For the elastic energy we consider the classical
Canham--Evans--Helfrich energy
\begin{equation} \label{eq:CEH}
\int_\Gamma\tfrac{1}{2}\,\alpha\,(\varkappa-\spont)^2+\alpha^G\,
\Gauss \dH{d-1}\,,
\end{equation}
where $\Gamma\subset\mathbb{R}^d$, $d=2,3$, is a hypersurface without boundary,
$\alpha>0$ and $\alpha^G$
are the bending and Gaussian bending rigidities, $\varkappa$ is the
mean curvature, $\spont$ is the spontaneous curvature,
which can be caused by local inhomogeneities within the membrane, $\Gauss$
is the Gaussian curvature and $\mathcal{H}^{d-1}$ is the
$(d-1)$-dimensional surface Hausdorff measure.
As discussed in \cite{Nitsche93}, the most general form of a curvature energy
density that is at most quadratic in the principal curvatures and is 
also symmetric in the principal curvatures has the form 
$\tfrac12\,\alpha\,\varkappa^2 + \alpha^G\,\Gauss + \alpha_1\,\varkappa +
\alpha_2$, which leads to (\ref{eq:CEH}) by choosing
$\alpha_1 = - \alpha\,\spont$ and $\alpha_2 = \tfrac12\,\alpha\,\spont^2$.
In the case $d=2$ the most general form which is at most quadratic in the 
curvature is $\tfrac12\,\alpha\,\varkappa^2 + \alpha_1\,\varkappa +
\alpha_2$. Hence throughout this paper we set $\alpha^G=0$ in the case $d=2$.

We also introduce an order parameter $\phasec$, 
which takes the values $\pm 1$ in the two different
phases, and this parameter is related to the composition of the chemical
species within the membrane. On the surface we then use a phase field model
to approximate the interfacial energy by the Ginzburg--Landau functional 
\begin{equation*}
\beta\,\int_\Gamma \tfrac{1}{2}\,\gamma\,|\nabs\,\phasec|^2+\gamma^{-1}\,
\Psi(\phasec) \dH{d-1}\,,
\end{equation*}
where $\beta>0$ is related to the line tension coefficient and $\gamma$
is a multiple of the interfacial thickness of the diffusional layer
separating the two phases. Furthermore, $\nabs$ is the surface
gradient and $\Psi$ is a double well potential.

In the different phases $\alpha$, $\spont$ and $\alpha^G$
will take different values, and we will interpolate these values
obtaining functions $\alpha(\phasec)>0$, $\spont(\phasec)$ and
$\alpha^G(\phasec)$. The total energy will hence have the form 
\begin{subequations}
\begin{align} \label{eq:E}
E(\Gamma,\phasec) & = \int_{\Gamma} b(\varkappa,\phasec) + 
\alpha^G (\phasec)\, \Gauss + \beta\,b_{CH}(\phasec) \dH{d-1}\,, 
\end{align}
where
\begin{equation} \label{eq:bendingb}
b(\varkappa,\phasec) = \tfrac12\,\alpha(\phasec) \,(\varkappa - \spont(\phasec))^2
\quad\text{and}\quad
b_{CH}(\phasec) = \tfrac12\,\gamma\,|\nabs\,\phasec|^2 + 
\gamma^{-1}\,\Psi(\phasec)\,.
\end{equation}
\end{subequations}
We recall that we assume $\alpha^G=0$ in the case $d=2$.
In the case $d=3$, and if $\alpha^G$ is constant, then the contribution 
$\int_\Gamma \alpha^G(\phasec)\,\Gauss\dH{2}$ is constant for a 
fixed topological type, which is a consequence of the Gauss--Bonnet theorem 
for closed surfaces,
\begin{equation} \label{eq:GB}
\int_{\Gamma} \Gauss \dH{2} = 2\,\pi\,m(\Gamma) \,,
\end{equation}
where $m(\Gamma) \in \bZ$ denotes the Euler characteristic of $\Gamma$.
However, if $\alpha^G$ is inhomogeneous, this term plays a role, 
which was discussed for example in
\cite{JulicherL96} in the context of two-phase membranes. 
Here we also mention that the contributions
$\frac12\,\int_\Gamma \alpha(\phasec)\,\varkappa^2 \dH{2} + 
\int_\Gamma \alpha^G(\phasec)\,\Gauss \dH{2}$ to the energy $E(\Gamma,\phasec)$ 
are positive semidefinite with respect to the principal curvatures if 
$\alpha^G(s) \in [-2\,\alpha(s),0]$ for all $s\in\R$. 
On account of the Gauss--Bonnet theorem, (\ref{eq:GB}), we hence obtain that
the energy $E(\Gamma,\phasec)$ can be bounded from below if
$\alpha^G(s) \geq \alpha^G_{\max}-2\,\alpha(s)$ for
all $s\in\R$, which will hold whenever
\begin{equation} \label{eq:alphaGbound}
\alpha_{\min} \geq \tfrac12\,(\alpha^G_{\max} - \alpha^G_{\min})\,,
\end{equation}
where $\alpha_{\min} = \min_{s\in\R} \alpha(s)$, and similarly for
$\alpha^G_{\min}$, $\alpha^G_{\max}$.
We note that this constraint is likely to have 
implications for the existence and regularity theory of 
gradient and related flows for $E(\Gamma,\phasec)$ in the case $d=3$.

The energy (\ref{eq:E}) represents a phase field approximation of a two-phase
membrane curvature energy with line tension. In the limit $\gamma\to0$ the
diffusive interface disappears and a sharp interface limit is obtained.
Sharp interface limits of phase field approaches to two-phase membranes have
been studied with the help of formal asymptotics by \cite{ElliottS10a} 
in the case of a $C^1$--limiting surface, and rigorously by
\cite{Helmers13} for axisymmetric two-phase membranes allowing for
tangent discontinuities at interfaces. Later \cite{Helmers15} also showed a
rigorous convergence result for the axisymmetric situation in the 
$C^1$--case.

We will now consider a closed membrane, which evolves in time in a
domain $\Omega$, separating the domain into regions $\Omega_+(t)$ and
$\Omega_-(t):=\Omega\setminus\overline{\Omega}_+(t)$. We hence
consider an evolving hypersurface $(\Gamma(t))_{t\in[0,T]}$, where
$T>0$ is a fixed time. We will assume that the classical Navier--Stokes
equations hold in $\Omega_-(t)$ and $\Omega_+(t)$ and on the membrane
we require the conditions
\begin{subequations}
\begin{alignat}{2}
[\vec u]_-^+ & = \vec 0 \qquad &&\mbox{on } \Gamma(t)\,, \label{eq:1a} \\ 
\rho_\Gamma \, \matpartu\,\vec u - \nabs\,.\,\mat\sigma_\Gamma & =
[\mat\sigma\,\vec\nu]_-^+ + \vec f_\Gamma 
\qquad &&\mbox{on } \Gamma(t)\,, \label{eq:1b} \\ 
\nabs\,.\,\vec u & = 0
\qquad &&\mbox{on } \Gamma(t)\,, \label{eq:1c} \\ 
\vec{\mathcal{V}}\,.\,\vec\nu &= 
\vec u\,.\,\vec \nu \qquad &&\mbox{on } \Gamma(t)\,, 
\label{eq:1d} 
\end{alignat}
\end{subequations}
where $\vec u$ is the fluid velocity, $\vec{\mathcal{V}}$ is the interface
velocity, $\vec\nu$ is a unit normal to $\Gamma(t)$, $\rho_\Gamma \in \Rgeq$
denotes the surface material density
and 
the source term $\vec f_\Gamma = -\delta E/\delta\Gamma$ is the first
variation of the total energy of $\Gamma(t)$ with respect to $\Gamma$, 
and will be stated in (\ref{eq:fGamma}) below in detail.
In addition, 
$\nabs\,.\,$ denotes the surface divergence on $\Gamma(t)$, and the surface
stress tensor is given by
\begin{equation} \label{eq:sigmaG}
\mat\sigma_\Gamma = 2\,\mu_\Gamma\,\mat D_s (\vec u) 
 - p_\Gamma\,\mat{\mathcal{P}}_\Gamma
\quad\text{on}\ \Gamma(t)\,,
\end{equation}
where $\mu_\Gamma \in \Rgeq$ is the interfacial shear viscosity and
$p_\Gamma$ denotes the surface pressure, which acts as a Lagrange
multiplier for the incompressibility condition (\ref{eq:1c}). 
Here
\begin{subequations}
\begin{equation} \label{eq:Ps}
\mat{\mathcal{P}}_\Gamma = \mat\Id - \vec \nu \otimes \vec \nu
\quad\text{on}\ \Gamma(t)\,,
\end{equation}
with $\mat\Id \in \R^{d\times d}$ denoting the identity matrix, and
\begin{equation} \label{eq:Ds} 
\mat D_s(\vec u) = \tfrac12\,\mat{\mathcal{P}}_\Gamma\,(\nabs\,\vec u + (\nabs\,\vec u)^T)\,
\mat{\mathcal{P}}_\Gamma
\quad\text{on}\ \Gamma(t)\,,
\end{equation}
\end{subequations}
where the surface gradient 
$\nabs = \mat{\mathcal{P}}_\Gamma \,\nabla = 
(\partial_{s_1}, \ldots, \partial_{s_d})$ on $\Gamma(t)$, and
$\nabs\,\vec u = \left( \partial_{s_j}\, u_i \right)_{i,j=1}^d$.
Similarly, the bulk stress tensor in (\ref{eq:1b}) is defined by
\begin{equation} \label{eq:sigma}
\mat\sigma = \mu \,(\nabla\,\vec u + (\nabla\,\vec u)^T) - p\,\mat\Id
= 2\,\mu\, \mat D(\vec u)-p\,\mat\Id\,,
\end{equation}
where 
$\mat D(\vec u):=\frac12\, (\nabla\,\vec u+(\nabla\,\vec u)^T)$ is the bulk
rate-of-deformation tensor, with $\nabla\,\vec u = \left( \partial_{x_j}\,u_i
\right)_{i,j=1}^d$. Moreover, $p$ is the bulk pressure and
$\mu(t) = \mu_+\,\charfcn{\Omega_+(t)} + \mu_-\,\charfcn{\Omega_-(t)}$,
with $\mu_\pm \in \R_{>0}$, denotes the dynamic viscosities in the two 
phases, where here and
throughout $\charfcn{\mathcal{A}}$ defines the characteristic function for a
set $\mathcal{A}$. 
Moreover, as usual, $[\vec u]_-^+ := \vec u_+ - \vec u_-$ and
$[\mat\sigma\,\vec\nu]_-^+ := \mat\sigma_+\,\vec\nu - \mat\sigma_-\,\vec\nu$
denote the jumps in velocity and normal stress across the interface
$\Gamma(t)$. Here and throughout, we employ the shorthand notation
$\vec a_\pm := \vec a\!\mid_{\Omega_\pm(t)}$ for a function 
$\vec a : \Omega \times [0,T] \to \R^d$; and similarly for scalar and
matrix-valued functions. 
In addition,
\begin{equation} \label{eq:matpartu}
\matpartu\, \zeta = \zeta_t + \vec u \,.\,\nabla\,\zeta
\end{equation}
denotes the material time derivative of $\zeta$ on $\Gamma(t)$, see e.g.\ 
\citet[p.\ 324]{DziukE13}.

The overall model is completed by the following Cahn--Hilliard dynamics on
$\Gamma(t)$
\begin{subequations}
\begin{align} \label{eq:strongCHa}
& \vartheta\, \matpartu\,\phasec = \Delta_s\,\chempot 
\,,\\
& \chempot = - \beta\,\gamma\,\Delta_s\,\phasec + \beta\,\gamma^{-1}\,\Psi'(\phasec)
+ b_{,\phasec}(\varkappa,\phasec) + (\alpha^G)'(\phasec)\,\Gauss
\,, \label{eq:strongCHb}
\end{align}
\end{subequations}
where $\chempot$ denotes the chemical potential, 
$\Delta_s = \nabs\,.\,\nabs$ is the
Laplace--Beltrami operator and $\vartheta \in \Rplus$
is a kinetic coefficient. We note here that $\chempot = \delta
E/\delta \phasec$ is the first variation of the total energy with respect to $\phasec$,
see Sections~\ref{sec:2}, \ref{sec:3} and the Appendix for more
details. 
Equation (\ref{eq:strongCHa}) is a convection-diffusion equation for
the species concentration on an evolving surface driven by the
chemical potential $\chempot$. For more information on the
Cahn--Hilliard equation we refer to \cite{Elliott89,NovickCohen08}. 
We note that the Cahn--Hilliard equation on an evolving
surface was studied by \cite{ElliottR15}, including its finite element
approximation.

It turns out that the overall model with suitable boundary conditions, e.g.\
$\vec u =0$ on $\partial\Omega$, fulfils, 
in the case where the outer forces are zero, the following dissipation identity
\begin{align}\nonumber
&\ddt \left( \tfrac12\, \int_\Omega \rho\,|\vec u|^2 \dL{d} +
\tfrac12\,\int_{\Gamma(t)} \rho_\Gamma\,|\vec u|^2 \dH{d-1}
+E(\Gamma(t),\phasec(t))\right)\\ & \quad 
+2\,\int_\Omega \mu\,|\mat D(\vec u)|^2 \dL{d} 
+ 2\,\mu_\Gamma\,\int_{\Gamma(t)} |\mat D_s(\vec u)|^2 \dH{d-1}
+\theta^{-1}\,\int_{\Gamma(t)}|\nabs\,\chempot|^2\dH{d-1} = 0\,,
\label{eq:disseq}
\end{align}
which is consistent with the second law of thermodynamics in its
isothermal formulation. 
The fourth and fifth term in (\ref{eq:disseq}) describe dissipation by
viscous friction in the bulk and on the surface, and the last term
models dissipation due to diffusion of molecules on the surface. We
also note that the introduced model conserves the volume of the bulk
phases, the surface area and the total species concentration on the
surface, i.e.\ 
\begin{equation}\label{eq:consprop}
\ddt\,|\Omega^-(t)|= \ddt\, \mathcal{H}^{d-1}(\Gamma(t)) =
\ddt\, \int_{\Gamma(t)} \phasec(t) \dH{d-1} =0\,.
\end{equation}
In particular, in contrast to other works,  no artificial Lagrange
multipliers are needed to conserve the enclosed volume, the total
surface area and the total species concentration. 

It is one of the main goals of this contribution to introduce and
analyze a numerical method that fulfils discrete variants of the
dissipation inequality and of the conservation properties
(\ref{eq:consprop}), see the results in Theorem \ref{thm:stab} and
Theorem \ref{thm:stab2}, below.

Let us now discuss related works on two-phase membranes. The
interest in two-phase membranes increased due to the fascinating work
of \cite{BaumgartHW03,BaumgartDWJ05}, as experiments seem to validate
earlier theories by \cite{Lipowsky92,JulicherL96} on
two-phase membranes, and showed an amazing multitude of complex shapes
and patterns. There have been many studies on two-phase axisymmetric
two-phase membranes, both from an analytical and from a numerical point of
view, see \cite{JulicherL96,Helmers11,Helmers13,Helmers15,ChoksiMV13,CoxL15}, 
and the references therein. However, only very
few works study general shapes of two-phase membranes from a theoretical or
computational point of view. In this context we refer to
\cite{WangD08,LowengrubRV09,DasJB09,ElliottS10,ElliottS10a,MerckerPKHWJ12,%
ElliottS13,Tu13,MerckerM-CRH13,MerckerM-C15}. But we note that none of the
above mentioned contributions considered a stability analysis for their 
numerical approximations.
We combine aspects of some of these approaches with the dynamics
studied by \cite{ArroyoS09}, and we generalize computational approaches of
the present authors for one-phase membranes, 
see e.g.\ \cite{nsns,nsnsade}, to numerically
compute evolving two-phase membranes.

The outline of the paper is as follows. In the following section we
introduce the model with all its details. In Section~\ref{sec:3} we
introduce a weak formulation, which is then discretized in space in
Section~\ref{sec:4}. We then also show that this scheme decreases the
total energy and obeys the relevant global conservation properties. In
Section~\ref{sec:5} we introduce a fully discrete scheme and show
existence and uniqueness of a fully discrete solution assuming an LBB
condition. In Section~\ref{sec:6} we comment on the methods used to
solve the fully discrete systems. In Section~\ref{sec:7} we present
several numerical computations in two and three spatial dimensions,
illustrating the properties of the numerical approach and showing the
complex interplay between the curvature functional, the
Ginzburg--Landau energy and the Navier--Stokes dynamics. In the
Appendix we finally state the details of the derivation of the model,
and we show that the weak formulation we introduce is consistent with 
the strong formulation for smooth solutions. 

\setcounter{equation}{0}
\section{Notation and governing equations} \label{sec:2}

In this section we formulate the model, which was sketched in the
Introduction, with all its details. 
Let $\Omega\subset\mathbb{R}^d$ be a given domain, where $d=2$ or $d=3$. 
We seek a time dependent interface $(\Gamma(t))_{t\in[0,T]}$,
$\Gamma(t)\subset\Omega$, 
which for all $t\in[0,T]$ separates
$\Omega$ into a domain $\Omega_+(t)$, occupied by the outer phase,
and a domain $\Omega_{-}(t):=\Omega\setminus\overline\Omega_+(t)$, 
which is occupied by the inner phase, 
see Figure~\ref{fig:sketch} for an illustration.
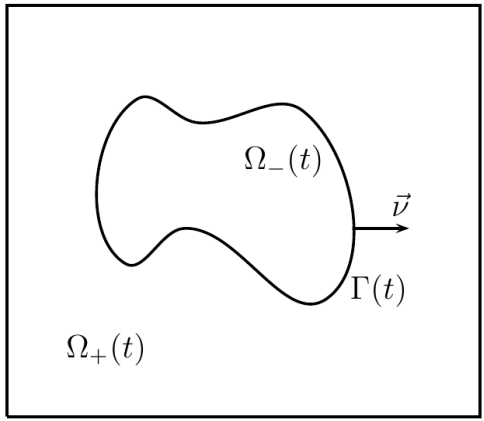
\begin{figure}
\begin{center}
\ifpdf
\includegraphics[angle=0,width=0.4\textwidth]{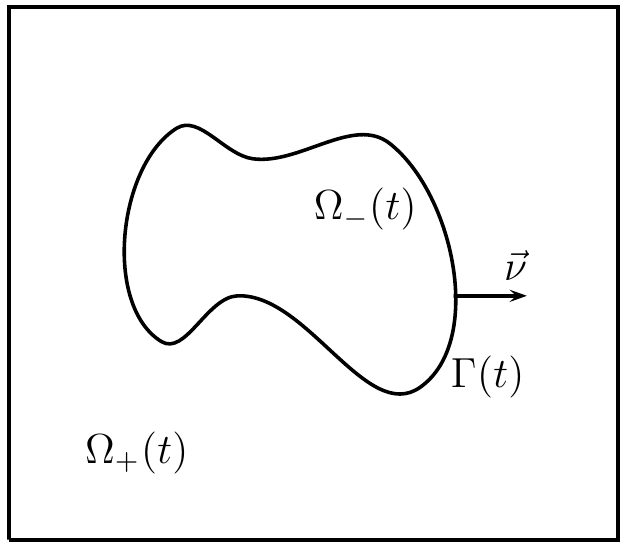}
\else
\unitlength15mm
\psset{unit=\unitlength,linewidth=1pt}
\begin{picture}(4,3.5)(0,0)
\psline[linestyle=solid]{-}(0,0)(4,0)(4,3.5)(0,3.5)(0,0)
\psccurve[showpoints=false,linestyle=solid] 
 (1,1.3)(1.5,1.6)(2.7,1.0)(2.5,2.6)(1.6,2.5)(1.1,2.7)
\psline[linestyle=solid]{->}(2.92,1.6)(3.4,1.6)
\put(3.25,1.7){{\black $\vec\nu$}}
\put(2.9,1.0){{$\Gamma(t)$}}
\put(2,2.1){{$\Omega_-(t)$}}
\put(0.5,0.5){{$\Omega_+(t)$}}
\end{picture}
\fi
\end{center}
\caption{The domain $\Omega$ in the case $d=2$.}
\label{fig:sketch}
\end{figure}%
For later use, we assume that $(\Gamma(t))_{t\in [0,T]}$ 
is an 
evolving hypersurface without boundary that is
parameterized by $\vec x(\cdot,t):\Upsilon\to\R^d$,
where $\Upsilon\subset \R^d$ is a given reference manifold, i.e.\
$\Gamma(t) = \vec x(\Upsilon,t)$. Then
\begin{equation} \label{eq:V}
\vec{\mathcal{V}}(\vec z, t) := \vec x_t(\vec q, t)
\qquad \forall\ \vec z = \vec x(\vec q,t) \in \Gamma(t)
\end{equation}
defines the velocity of $\Gamma(t)$, and
$\mathcal{V} := \vec{\mathcal{V}} \,.\,\vec{\nu}$ is
the normal velocity of the evolving hypersurface $\Gamma(t)$,
where $\vec\nu(t)$ is the unit normal on $\Gamma(t)$ pointing into 
$\Omega_+(t)$.
Moreover, we define the space-time surface
$\GT := \bigcup_{t \in [0,T]} \Gamma(t) \times \{t\}$.

Let $\rho(t) = \rho_+\,\charfcn{\Omega_+(t)} + \rho_-\,\charfcn{\Omega_-(t)}$,
with $\rho_\pm \in \Rgeq$, 
denote the fluid densities.
Denoting by $\vec u : \Omega \times [0, T] \to \R^d$ the fluid velocity,
by $p : \Omega \times [0, T] \to \R$ the pressure,
by $\mat\sigma : \Omega \times [0,T] \to \R^{d\times d}$ the stress tensor,
and by $\vec f : \Omega \times [0, T] \to \R^d$ a possible volume force,
the incompressible Navier--Stokes equations in the two phases are given by
(\ref{eq:sigma}) and
\begin{subequations}
\begin{alignat}{2}
\rho\,(\vec u_t + (\vec u \,.\,\nabla)\,\vec u)
- \nabla\,.\,\mat\sigma & = \rho\,\vec f
\qquad &&\mbox{in } 
\Omega_\pm(t)\,, \label{eq:NSa} \\
\nabla\,.\,\vec u & = 0 \qquad &&\mbox{in } \Omega_\pm(t)\,, \label{eq:NSb} \\
\vec u & = \vec g \qquad &&\mbox{on } \partial_1\Omega\,, \label{eq:NSc} \\
\mat\sigma\,\unitn & = \vec 0 
\qquad &&\mbox{on } \partial_2\Omega\,, \label{eq:NSd} 
\end{alignat}
\end{subequations}
where $\partial\Omega =\partial_1\Omega \cup
\partial_2\Omega$, with $\partial_1\Omega \cap \partial_2\Omega
=\emptyset$, denotes the boundary of $\Omega$ with outer unit normal $\unitn$.
Hence (\ref{eq:NSc}) prescribes a possibly
inhomogeneous Dirichlet condition for the velocity 
on $\partial_1\Omega$, which collapses to the standard no-slip condition
when $\vec g = \vec 0$,
while (\ref{eq:NSd}) prescribes a stress-free condition on 
$\partial_2\Omega$. Throughout this paper we assume that 
$\mathcal{H}^{d-1}(\partial_1\Omega) > 0$.
We will also assume w.l.o.g.\ that $\vec g$ is extended so that 
$\vec g : \Omega \to \R^d$. 
On the free surface $\Gamma(t)$ the conditions (\ref{eq:1a}--d) 
need to hold, recall the Introduction. 
The system (\ref{eq:NSa}--d), (\ref{eq:sigma}), (\ref{eq:1a}--d),
(\ref{eq:sigmaG}) is closed with the initial conditions
\begin{equation} \label{eq:1init}
\Gamma(0) = \Gamma_0 \,, \quad 
\rho\,\vec u(\cdot,0) = \rho\,\vec u_0 \quad \mbox{in } \Omega\,, \quad
\rho_\Gamma\,\vec u(\cdot,0) = \rho_\Gamma\,\vec u_0 \quad \mbox{on } \Gamma_0
\,,
\end{equation}
where $\Gamma_0 \subset \Omega$ and
$\vec u_0 : \Omega \to \R^d$ are given initial data satisfying
$\rho\,\nabla\,.\,\vec u_0 = 0$ in $\Omega$, 
$\rho_\Gamma\,\nabs\,.\,\vec u_0 = 0$ on $\Gamma_0$ and 
$\rho_+\,\vec u_0 = \rho_+\,\vec g$ on $\partial_1\Omega$.
Of course, in the case $\rho_- = \rho_+ = \rho_\Gamma = 0$ the 
initial data $\vec u_0$ is not needed. Similarly, in the case
$\rho_- = \rho_+ = 0$ and $\rho_\Gamma > 0$ the initial data $\vec u_0$ is only
needed on $\Gamma_0$. However, for ease of exposition, and in view of the
unfitted nature of our numerical method, we will always assume that $\vec u_0$,
if required, is given on all of $\Omega$.

It is not difficult to show that the conditions (\ref{eq:NSb}) enforce volume
preservation for the phases, while (\ref{eq:1c}) leads to the conservation of
the total surface area $\mathcal{H}^{d-1}(\Gamma(t))$, see 
(\ref{eq:areacons}) and (\ref{eq:conserved}) in Section~\ref{sec:3}
below for the relevant proofs. 
As an immediate consequence we obtain that
a sphere $\Gamma(t)$ remain a sphere, and that a sphere $\Gamma(t)$ with a 
zero bulk velocity is a stationary solution.

In addition, the source term $\vec f_\Gamma$ in (\ref{eq:1b}) is 
given by minus the first variation of the energy (\ref{eq:E}) with respect to
$\Gamma$, i.e.\
\begin{align}\label{eq:fGamma}
\vec 
f_\Gamma & = 
- \deldel{\Gamma}\,E(\Gamma,\phasec) = 
\left[-\Delta_s\,[\alpha(\phasec)\,(\varkappa - \spont(\phasec))]
-\alpha(\phasec)\,(\varkappa-\spont(\phasec)) \,|\nabs\,\vec \nu|^2
+b(\varkappa,\phasec)\,\varkappa 
\right. \nonumber \\ & \qquad \left.
- \nabs\,.\,([\varkappa\,\mat\Id + \nabs\,\vec\nu]\,\nabs\,\alpha^G(\phasec)) 
\right]\vec\nu
+ (b_{,\phasec}(\varkappa,\phasec) +
(\alpha^G)'(\phasec)\,\Gauss)\,\nabs\,\phasec 
\nonumber \\ & \qquad
+ \beta\left[ b_{CH}(\phasec)\,
\varkappa\,\vec\nu 
+ \nabs\, b_{CH}(\phasec)
- \gamma\,\nabs\,.\,( (\nabs\,\phasec)\otimes(\nabs\,\phasec)) \right]
\,,
\end{align}
where we have defined
\begin{equation} \label{eq:bc}
b_{,\phasec}(\varkappa,\phasec) = \tfrac\partial{\partial \phasec}\,b(\varkappa,\phasec)
= \tfrac12\,\alpha'(\phasec) \,(\varkappa - \spont(\phasec))^2 
- \alpha(\phasec)\,(\varkappa - \spont(\phasec))\,\spont'(\phasec)\,.
\end{equation}
Throughout we assume that $\alpha,\alpha^G \in C^1(\R)$, with 
$\alpha(s) > 0$ for all $s\in\R$. 
We refer to the appendix for a detailed derivation of
(\ref{eq:fGamma}). 
In contrast to situations where the energy density does not depend on a
species concentration, we now have tangential contributions to 
$\vec f_\Gamma$. In particular, the terms 
$(b_{,\phasec}(\varkappa,\phasec) +
(\alpha^G)'(\phasec)\,\Gauss)\,\nabs\,\phasec 
+ \beta\,\nabs\, b_{CH}(\phasec)
-\beta\,\gamma\,\nabs\,.\,( (\nabs\,\phasec)\otimes(\nabs\,\phasec))$ 
give rise to a tangential flow and
hence can induce a Marangoni-type effect.

The overall model we are going to study in this work is 
the coupled bulk and surface Navier--Stokes equations
(\ref{eq:NSa}--d), (\ref{eq:sigma}), (\ref{eq:1a}--d), (\ref{eq:sigmaG}),
(\ref{eq:1init}) together with the convective Cahn--Hilliard system
(\ref{eq:strongCHa},b) on the evolving interface, suitably supplemented with
initial conditions for $\phasec$. 
Here the double well potential $\Psi$ in (\ref{eq:bendingb}) and
(\ref{eq:strongCHb}) may be chosen,
for example, as a quartic potential
\begin{subequations}
\begin{equation} \label{eq:quartic}
\Psi(s) = \tfrac14\,(s^2 - 1 )^2\,,
\end{equation}
or as the obstacle potential
\begin{equation} 
\Psi(s):=
\begin{cases}
\textstyle \frac 12 \left(1-s^2\right)  & \mbox{if }\ |s| \leq 1,\\
\infty & \mbox{if }\ |s| > 1,\\
\end{cases} \qquad \label{eq:obener}
\end{equation}
\end{subequations}
which restricts $\phasec \in [-1,1]$.
For the analysis we will always assume that
$\Psi \in C^1(\R)$ for ease of exposition, 
but we will use (\ref{eq:obener}) for our fully discrete approximations.

As stated previously, $\varkappa$ 
in (\ref{eq:fGamma}) denotes the so-called mean curvature of $\Gamma(t)$, 
i.e.\ the sum of
the principal curvatures $\varkappa_i$, $i=1,\ldots,d-1$, of $\Gamma(t)$, 
where we have adopted the sign
convention that $\varkappa$ is negative where $\Omega_-(t)$ is locally convex.
In particular, it holds that
\begin{equation} \label{eq:LBop}
\Delta_s\, \vec\id = \varkappa\, \vec\nu =: \vec\varkappa
\qquad \mbox{on $\Gamma(t)$}\,,
\end{equation}
where $\vec\id$ is the identity function on $\R^d$.
For later use, we recall that the second
fundamental tensor for $\Gamma(t)$ is given by $\nabs\,\vec\nu$. Moreover, we
note that $-\nabs\,\vec\nu(\vec{z})$, for any $\vec{z}\in\Gamma(t)$, is a
symmetric linear map that has a zero eigenvalue 
with eigenvector $\vec\nu$, i.e.\
\begin{equation} \label{eq:Weingarten}
(\nabs\,\vec\nu)^T = \nabs\,\vec\nu \quad\text{and}\quad
(\nabs\,\vec\nu)\,\vec\nu = \vec 0\,,
\end{equation}
and the remaining $(d-1)$ eigenvalues,
$\varkappa_1,\dots,\varkappa_{d-1}$, are the principal curvatures of
$\Gamma$ at $\vec{z}$; see e.g.\ \cite[p.~152]{DeckelnickDE05}.
The mean curvature $\varkappa$ and the Gauss
curvature $\Gauss$ can now be stated as
\begin{equation}
\varkappa = - \tr{(\nabs\,\vec\nu)}
= - \nabs\,.\,\vec\nu = \sum_{i=1}^{d-1} \varkappa_i
\qquad \mbox{and} \qquad
\Gauss = \prod_{i=1}^{d-1} \varkappa_i
\,, \label{eq:secondform}
\end{equation}
which in the case $d=3$ immediately yields that
\begin{equation} \label{eq:Gauss3d}
\Gauss= \tfrac12\,(\varkappa^2 - |\nabs\,\vec\nu|^2)\,.
\end{equation}
We recall that in the case $d=2$, we always assume that 
$\alpha^G = 0$.
In the case $d=3$, on the other hand, we have from (\ref{eq:Gauss3d}) that
\begin{align}
\int_{\Gamma(t)} \alpha^G (\phasec)\, \Gauss \dH{d-1}
= \tfrac12 \int_{\Gamma(t)} \alpha^G (\phasec)\, (|\vec\varkappa|^2 - |\mat w|^2)
\dH{d-1}\,,
\label{eq:matw}
\end{align}
where $\mat w \in [\Wt]^{d\times d}$ is such that for all
$\mat\zeta \in [\Wt]^{d\times d}$
\begin{align}
\int_{\Gamma(t)} \mat w : \mat\zeta \dH{d-1}
= \int_{\Gamma(t)} \nabs\,\vec\nu : \mat\zeta \dH{d-1}
= - \int_{\Gamma(t)} \vec\nu\,.\,(\nabs\,.\,\mat\zeta) + \vec\nu\,.\,
(\mat\zeta \,\vec\varkappa) \dH{d-1} \,.
\label{eq:matwmateta}
\end{align}
Here we have recalled from \citet[Theorem~2.10]{DziukE13} that
\begin{align} \label{eq:nabszeta}
& \left\langle \nabs\,\zeta, \vec\eta \right\rangle_{\Gamma(t)}
+ \left\langle \zeta, \nabs\,.\,\vec\eta \right\rangle_{\Gamma(t)}
= \left\langle \nabs\,.\,(\zeta\,\vec\eta), 1 \right\rangle_{\Gamma(t)}
= - \left\langle \zeta\,\varkappa\,\vec\nu, \vec\eta \right\rangle_{\Gamma(t)}
\nonumber \\ & \hspace{8cm}
\quad \forall\ \zeta\in\Wt,\,\vec\eta\in\Vt\,.
\end{align}
Hence the total energy $E(\Gamma(t), \phasec(t))$, on recalling (\ref{eq:E},b), 
can be rewritten as
\begin{align} \label{eq:Enew}
E(\Gamma(t), \phasec(t)) & = \int_{\Gamma} \tfrac12\,\alpha(\phasec)\,
|\vec\varkappa - \spont(\phasec)\,\vec\nu|^2 + 
\tfrac12\,\alpha^G (\phasec)\, (|\vec\varkappa|^2 - |\mat w|^2) 
+ \beta\,b_{CH}(\phasec) \dH{d-1}\,, 
\end{align}
where $\mat w$ is given by (\ref{eq:matwmateta}), and where
$\vec\varkappa$, on recalling (\ref{eq:nabszeta}), can
be defined by
\begin{align}
& \left\langle \vec\varkappa, \vec\eta \right\rangle_{\Gamma(t)}
+ \left\langle \nabs \,\vec\id,\nabs\, \vec{\eta}\right\rangle_{\Gamma(t)}= 0 
\qquad \forall\ \vec\eta \in \Vt\,. \label{eq:side3first}
\end{align}

\setcounter{equation}{0}
\section{Weak formulation} \label{sec:3}
We begin by recalling the weak formulation of 
(\ref{eq:NSa}--d), (\ref{eq:sigma}), (\ref{eq:1a}--d), (\ref{eq:sigmaG}) 
from \cite{nsns}. To this end,
we introduce the following function spaces for a given
$\vec a \in [H^1(\Omega)]^d$:
\begin{subequations}
\begin{align}
\uspace(\vec a) & := \{ \vec\varphi \in [H^1(\Omega)]^d 
: \vec\varphi = \vec a \ \mbox{ on } \partial_1\Omega 
 \} \,,\quad
\utimespace(\vec a) :=  L^2(0,T; \uspace(\vec a)) 
\cap H^1(0,T;[L^2(\Omega)]^d)\,, \nonumber \\
\utimespace_\Gamma(\vec a) & := \{ \vec\varphi \in \utimespace(\vec a) : 
\vec\varphi \!\mid_\GT
\in [H^1(\GT)]^d \}\,. \label{eq:UVV}
\end{align}
In addition, we let $\pspace := L^2(\Omega)$ and define
\begin{equation} \label{eq:hatP}
\widehat\pspace := 
\begin{cases} 
\{\eta \in \pspace : \int_\Omega\eta \dL{d}=0 \} & \text{if\ }
\mathcal{H}^{d-1}(\partial_2\Omega) = 0\,, \\
\pspace & \text{if\ } \mathcal{H}^{d-1}(\partial_2\Omega) > 0 \,.
\end{cases}
\end{equation}
\end{subequations}
Here and throughout, $\mathcal{H}^{d-1}$ denotes the $(d-1)$-dimensional 
Hausdorff measure in $\R^d$, while $\mathcal{L}^d$ denotes the Lebesgue 
measure in $\R^d$.
Moreover, we let $(\cdot,\cdot)$ and 
$\langle \cdot,\cdot \rangle_{\partial_2\Omega}$
denote the $L^2$--inner products on $\Omega$ and $\partial_2\Omega$,
and similarly for $\langle \cdot,\cdot \rangle_{\Gamma(t)}$.

Similarly to (\ref{eq:matpartu}) we define the following time derivative that
follows the parameterization $\vec x(\cdot, t)$ of $\Gamma(t)$, rather than
$\vec u$. In particular, we let
\begin{equation} \label{eq:matpartx}
\matpartx\, \zeta = \zeta_t + \vec{\mathcal{V}} \,.\,\nabla\,\zeta
\qquad \forall\ \zeta \in H^1(\GT)\,;
\end{equation}
where we stress that this definition is well-defined, even though
$\zeta_t$ and $\nabla\,\zeta$ do not make sense separately for a
function $\zeta \in H^1(\GT)$.
On recalling (\ref{eq:matpartu}) we obtain that 
$\matpartx = \matpartu$ if $\vec{\mathcal{V}} = \vec u$ on $\Gamma(t)$.
Moreover, for later use we note that
\begin{equation} \label{eq:DElem5.2}
\ddt \left\langle \chi, \zeta \right\rangle_{\Gamma(t)}
 = \left\langle \matpartx\,\chi, \zeta \right\rangle_{\Gamma(t)}
 + \left\langle \chi, \matpartx\,\zeta \right\rangle_{\Gamma(t)}
+ \left\langle \chi\,\zeta, \nabs\,.\,\vec{\mathcal{V}} 
 \right\rangle_{\Gamma(t)}
\qquad \forall\ \chi,\zeta \in H^1(\GT)\,,
\end{equation}
see \citet[Lem.~5.2]{DziukE13}.

The weak formulation of 
(\ref{eq:NSa}--d), (\ref{eq:sigma}), (\ref{eq:1a}--d), (\ref{eq:sigmaG}),
with $E(\Gamma(t),\phasec(t))$ replaced by \linebreak
$\frac12\,\alpha\,\langle\vec\varkappa,\vec\varkappa\rangle_{\Gamma(t)}$, 
from \cite{nsns} is then given as follows.
Find $\Gamma(t) = \vec x(\Upsilon, t)$ for $t\in[0,T]$ 
with $\vec{\mathcal{V}} \in [L^2(\GT)]^d$ 
and $\vec{\mathcal{V}}(\cdot, t) \in [H^1(\Gamma(t)]^d$ for almost all 
$t \in (0,T)$, and functions 
$\vec u \in \utimespace_\Gamma(\vec g)$, $p \in L^2(0,T; \widehat\pspace)$, 
$p_\Gamma \in L^2(\GT)$, $\vec\varkappa \in [H^1(\GT)]^d$ and
$\vec f_\Gamma \in [L^2(\GT)]^d$ such that 
the initial conditions (\ref{eq:1init}) hold and such that
for almost all $t \in (0,T)$ it holds that
\begin{subequations}
\begin{align}
& \tfrac12\left[ \ddt (\rho\,\vec u, \vec \xi) + (\rho\,\vec u_t, \vec \xi)
- (\rho\,\vec u, \vec \xi_t)
+ (\rho, [(\vec u\,.\,\nabla)\,\vec u]\,.\,\vec \xi
- [(\vec u\,.\,\nabla)\,\vec \xi]\,.\,\vec u) 
+ \rho_+\left\langle \vec u \,.\,\unitn , \vec u\,.\,\vec\xi 
 \right\rangle_{\partial_2\Omega}
\right] 
\nonumber \\ & \qquad
+ 2\,(\mu\,\mat D(\vec u), \mat D(\vec \xi)) 
- (p, \nabla\,.\,\vec \xi)
+ \rho_\Gamma \left\langle \matpartx\,\vec u , \vec \xi 
 \right\rangle_{\Gamma(t)}
+ 2\,\mu_\Gamma \left\langle \mat D_s (\vec u) , \mat D_s ( \vec \xi )
\right\rangle_{\Gamma(t)}
\nonumber \\ & \qquad
- \left\langle p_\Gamma, \nabs\,.\,\vec \xi \right\rangle_{\Gamma(t)}
 = (\rho\,\vec f, \vec \xi) 
+ \left\langle \vec f_\Gamma, \vec \xi \right\rangle_{\Gamma(t)}
\qquad \forall\ \vec \xi \in \utimespace_\Gamma(\vec 0)\,,
\label{eq:weakGDa} \\
& (\nabla\,.\,\vec u, \varphi) = 0  \qquad \forall\ \varphi \in 
\widehat\pspace\,, \label{eq:weakGDb} \\
& \left\langle \nabs\,.\,\vec u, \eta \right\rangle_{\Gamma(t)}
 = 0  \qquad \forall\ \eta \in L^2(\Gamma(t))\,, \label{eq:weakGDc} \\
&\left\langle\vec{\mathcal{V}} - \vec u, \vec\chi \right\rangle_{\Gamma(t)} = 0
 \qquad\forall\ \vec\chi \in [L^2(\Gamma(t))]^d\,,
\label{eq:weakGDd}
\end{align}
\end{subequations}
as well as
\begin{subequations}
\begin{align}
& \left\langle \vec\varkappa, \vec\eta \right\rangle_{\Gamma(t)}
+ \left\langle \nabs\,\vec\id, \nabs\,\vec \eta \right\rangle_{\Gamma(t)}
 = 0  \qquad\forall\ \vec\eta \in [H^1(\Gamma(t))]^d\,, \label{eq:weakGDe} \\
& \left\langle \vec f_\Gamma, \vec \chi \right\rangle_{\Gamma(t)} =
\alpha
\left\langle \nabs\,\vec \varkappa , \nabs\,\vec \chi \right\rangle_{\Gamma(t)}
+\alpha\left\langle \nabs\,.\, \vec \varkappa, \nabs\,.\, \vec \chi 
\right\rangle_{\Gamma(t)}
+ \tfrac12\,\alpha\left\langle |\vec \varkappa|^2,\nabs\,.\,\vec \chi 
\right\rangle_{\Gamma(t)} \nonumber \\ & \hspace{4cm}
- 2\,\alpha \left\langle (\nabs\,\vec\varkappa)^T, \mat D_s(\vec{\chi})\,
 (\nabs\,\vec\id)^T \right\rangle_{\Gamma(t)}
\qquad\forall\ \vec \chi \in [H^1(\Gamma(t))]^d \,, \label{eq:weakGDf}
\end{align}
\end{subequations}
where in (\ref{eq:weakGDd}) we have recalled (\ref{eq:V}). 

For the case $\vec g = \vec 0$, it was shown in \cite{nsns} that
choosing $\vec\xi = \vec u \in \utimespace_\Gamma(\vec 0)$ in 
(\ref{eq:weakGDa}), $\varphi = p(\cdot, t) \in \widehat\pspace$ 
in (\ref{eq:weakGDb}), $\eta = p_\Gamma(\cdot, t) \in L^2(\Gamma(t))$,
$\vec\chi = \vec f_\Gamma$ in (\ref{eq:weakGDd}) and
$\vec\chi = \vec{\mathcal{V}}$ in (\ref{eq:weakGDf}) yields that
\begin{align*}
& \tfrac12\,\ddt \left( \|\rho^\frac12\,\vec u \|_0^2 
+ \rho_\Gamma \left\langle \vec u, \vec u \right\rangle_{\Gamma(t)} 
+ \alpha \left\langle \vec\varkappa, \vec\varkappa \right\rangle_{\Gamma(t)} 
 \right)
+ 2\,\| \mu^\frac12\,\mat D(\vec u)\|_0^2 
+ 2\,\mu_\Gamma\left\langle \mat D_s(\vec u), \mat D_s(\vec u) 
\right\rangle_{\Gamma(t)} 
\nonumber \\ & \hspace{3cm}
+ \tfrac12\,\rho_+ \left\langle \vec u \,.\,\unitn , |\vec u|^2 
 \right\rangle_{\partial_2\Omega}
= (\rho\,\vec f, \vec u) \,.
\end{align*}
Moreover, we recall from \cite{nsns} that it follows from 
(\ref{eq:DElem5.2}) and (\ref{eq:weakGDc},d) that
\begin{equation} \label{eq:areacons}
\ddt\, \mathcal{H}^{d-1} (\Gamma(t)) = \ddt\,\langle 1, 1 \rangle_{\Gamma(t)}
= \left\langle 1, \nabs\,.\,\vec{\mathcal{V}} \right\rangle_{\Gamma(t)}
= \left\langle 1, \nabs\,.\,\vec u \right\rangle_{\Gamma(t)}
= 0 \,,
\end{equation}
while \citet[Lemma~2.1]{DeckelnickDE05}, (\ref{eq:weakGDb},d) 
and (\ref{eq:hatP}) imply that
\begin{equation}
\frac{\rm d}{{\rm d}t} \vol(\Omega_-(t)) = 
\left\langle \vec{\mathcal{V}}, \vec\nu \right\rangle_{\Gamma(t)}
= \left\langle \vec u, \vec\nu \right\rangle_{\Gamma(t)}
= \int_{\Omega_-(t)} \nabla\,.\,\vec u \dL{d} 
=0\,. \label{eq:conserved}
\end{equation}

\subsection{The first variation of $E(\Gamma(t),\phasec(t))$}
\label{sec:32}
In this section we would like to derive a weak formulation for the 
first variation of $E(\Gamma(t), \phasec(t))$ 
with respect to $\Gamma(t) = \vec x(\Upsilon,t)$.
To this end, for a given $\vec\chi \in \Vt$ 
and for $\epsilon \in (0,\epsilon_0)$, where $\epsilon_0\in\Rplus$, let
$\vec\Phi(\cdot,\epsilon)$ be a family of transformations such that
\begin{equation} \label{eq:Gammadelta}
\Gamma_\epsilon(t) := \{ \vec\Phi(\vec z,\epsilon) : \vec z \in \Gamma(t) \} \,,
\quad\text{where}\quad
\vec\Phi(\vec z, 0) = \vec z \text{~and~}
\tfrac{\partial\vec\Phi}{\partial\epsilon}(\vec z, 0) = \vec\chi(\vec z)
\quad \forall\ \vec z \in \Gamma(t)\,.
\end{equation}
Then the first variation of $\mathcal{H}^{d-1}(\Gamma(t))$
with respect to $\Gamma(t)$ in the direction $\vec{\chi} \in \Vt$ is given by
\begin{align}
\left[ \deldel{\Gamma} \, \mathcal{H}^{d-1}(\Gamma(t))\right]
(\vec{\chi}) & = 
\dd\epsilon\,\mathcal{H}^{d-1}(\Gamma_\epsilon(t)) \mid_{\epsilon=0} 
= \lim_{\epsilon\to0} \tfrac1\epsilon\left[\mathcal{H}^{d-1}(\Gamma_\epsilon(t))
- \mathcal{H}^{d-1}(\Gamma(t)) \right] \nonumber \\ & 
= 
\left\langle \nabs\,\vec\id , \nabs \,\vec\chi \right\rangle_{\Gamma(t)} 
= \left\langle 1, \nabs \,.\,\vec\chi \right\rangle_{\Gamma(t)} 
,
\label{eq:vardet}
\end{align}
see e.g.\ the proof of Lemma~1 in \cite{Dziuk08}. 
For any quantity $w$, that is naturally defined on $\Gamma_\epsilon(t)$, we
define
\begin{align}
\matparteps\,w(\vec{z})=
\dd\epsilon w_\epsilon(\vec\Phi(\vec{z},\epsilon)) \mid_{\epsilon=0}
\qquad \forall\ \vec z \in \Gamma(t)\,,
\label{eq:deltaderiv}
\end{align}
and similarly for $\matparteps\,\vec w$ and $\matparteps\,\mat w$. 
A common example is
$\vec\nu_\epsilon$, the outer normal on $\Gamma_\epsilon(t)$.
In cases where $w \in L^\infty(\Gamma(t))$ is meaningful only on $\Gamma(t)$, 
we let $w_\epsilon \in L^\infty(\Gamma_\epsilon(t))$ be such that
\begin{equation} \label{eq:wdelta}
w_\epsilon(\vec\Phi(\vec z, \epsilon)) = w(\vec z) \qquad \forall\
\vec z \in \Gamma(t)\,,
\end{equation}
which immediately implies that for such $w$ it holds that
$\matparteps\, w=0$. Once again, we extend (\ref{eq:wdelta}) also to
vector- and tensor-valued functions.
For later use we note that
generalized variants of (\ref{eq:vardet}) also hold. 
Similarly to (\ref{eq:DElem5.2}) it holds that
\begin{equation} \label{eq:DElem5.2eps}
\left[ \deldel\Gamma \left\langle w, v \right\rangle_{\Gamma(t)} \right]
(\vec\chi)
 = \left\langle \matparteps\,w, v \right\rangle_{\Gamma(t)}
 + \left\langle w, \matparteps\,v \right\rangle_{\Gamma(t)}
+ \left\langle w\,v, \nabs\,.\,\vec\chi \right\rangle_{\Gamma(t)}
\quad \forall\ w,v \in L^\infty(\Gamma(t))\,.
\end{equation}
Similarly, it holds that
\begin{align}
\left[ \deldel{\Gamma} \left\langle \vec w, \vec\nu \right\rangle_{\Gamma(t)}
\right] (\vec{\chi}) &= 
\dd\epsilon \left\langle \vec w_\epsilon, \vec\nu_\epsilon 
\right\rangle_{\Gamma_\epsilon(t)} \mid_{\epsilon=0}
= \left\langle \matparteps\,\vec w, \vec\nu \right\rangle_{\Gamma(t)}
+ \left\langle \vec w, \matparteps\,\vec\nu \right\rangle_{\Gamma(t)}
\nonumber \\ & \hspace{4cm}
+ \left\langle \vec w\,.\,\vec\nu, \nabs\,.\,\vec\chi 
\right\rangle_{\Gamma(t)}
 \quad \forall\ \vec w \in [L^\infty(\Gamma(t))]^d\,,
\label{eq:vardetwnu}
\end{align}
where 
$\vec\nu_\epsilon(t)$ denotes the unit normal on $\Gamma_\epsilon(t)$.
In this regard, we note the following result
concerning the variation of $\vec\nu$, with respect to $\Gamma(t)$, 
in the direction $\vec\chi \in \Vt$:
\begin{align}
\matparteps \,\vec\nu
= - [\nabs\, \vec \chi]^T \,\vec\nu \quad \text{on}\quad \Gamma(t)
\quad \Rightarrow \quad
\matpartx\,\vec\nu = - [\nabs\, \vec{\mathcal{V}}]^T \,\vec\nu
\quad \text{on}\quad \Gamma(t)\,, 
\label{eq:normvar}
\end{align}
see \citet[Lemma~9]{SchmidtS10}.
Next we note that for $\vec{\eta} \in \Vt$ it holds that
\begin{align}
&
\left[ \deldel{\Gamma} \left\langle \nabs\,\vec\id,
\nabs\,\vec{\eta}\right\rangle_{\Gamma(t)} \right](\vec{\chi})
= 
\dd\epsilon \left\langle \nabs\,\vec\id,
\nabs\,\vec\eta_\epsilon\right\rangle_{\Gamma_\epsilon(t)} \mid_{\epsilon=0} 
= \left\langle \nabs\,.\,\vec \eta, \nabs \,.\,\vec \chi 
\right\rangle_{\Gamma(t)}
\nonumber \\ & \qquad \qquad
+ \sum_{l,\,m=1}^d \left[
\left\langle (\vec\nu)_l\,(\vec\nu)_m\,
 \nabs\,(\vec \eta)_m, \nabs \,(\vec \chi)_l \right\rangle_{\Gamma(t)}
- \left\langle (\nabs)_m\,(\vec \eta)_l, (\nabs)_l \,(\vec \chi)_m 
\right\rangle_{\Gamma(t)} \right]  \nonumber \\ & \qquad
 = 
 \left\langle \nabs\,\vec \eta, \nabs \,\vec \chi \right\rangle_{\Gamma(t)}
+ \left\langle \nabs\,.\,\vec \eta, \nabs \,.\,\vec \chi 
 \right\rangle_{\Gamma(t)}
- 2 \left\langle (\nabs \,\vec\eta)^T, \mat D_s(\vec \chi)\,(\nabs\,\vec\id)^T 
 \right\rangle_{\Gamma(t)},
\label{eq:secvar}
\end{align}
where $\matparteps\, \vec\eta =\vec 0$.
We refer to Lemma~2 and the proof of Lemma~3 in \cite{Dziuk08} for a proof
of (\ref{eq:secvar}). Here we observe that our notation is such that 
$\nabs\,\vec\chi = (\nabla_{\!\Gamma}\,\vec\chi)^T$, with 
$\nabla_{\!\Gamma}\,\vec\chi=\left( \partial_{s_i}\, \chi_j \right)_{i,j=1}^d$ 
defined as in \cite{Dziuk08}. Moreover, it holds, on noting (\ref{eq:Ps}), 
that
\begin{subequations}
\begin{equation} \label{eq:PGamma}
\nabs\,\vec\chi\,\,\mat{\mathcal{P}}_\Gamma = \nabs\,\vec\chi
\quad\Rightarrow\quad \mat{\mathcal{P}}_\Gamma \, (\nabs\,\vec\chi)^T
= (\nabs\,\vec\chi)^T
\end{equation}
and
\begin{equation} \label{eq:DD}
2\,(\nabs\,\vec\eta)^T : \mat D_s(\vec \chi)\,(\nabs\,\vec\phi)^T
= (\nabs\,\vec\eta)^T : [ \nabs\,\vec\chi + (\nabs\,\vec\chi)^T]\,
(\nabs\,\vec\phi)^T\,,
\end{equation}
\end{subequations}
which yields that the last term on the right hand side in (\ref{eq:secvar}) 
can be rewritten as in \cite{Dziuk08}.

As $\nabs\,\vec\id = \mat{\mathcal{P}}_\Gamma$, one can deduce from
(\ref{eq:Ps}), (\ref{eq:secvar}) and (\ref{eq:DElem5.2eps}) that for
sufficiently smooth $\vec\eta$ 
\begin{align}
\matparteps \,(\nabs\,.\,\vec \eta) & = 
\matparteps \,(\nabs\,\vec\id :\nabs\,\vec \eta)
= \nabs\,\vec\eta : \nabs\,\vec\chi -
2\,(\nabs \,\vec\eta)^T : \mat D_s(\vec \chi)\,(\nabs\,\vec\id)^T
\nonumber \\ &
= [ \nabs\,\vec\chi - 2\,\mat D_s(\vec \chi)] : \nabs\,\vec\eta
\quad \text{a.e.\ on}\quad \Gamma(t)\,, 
\label{eq:secvarlocal}
\end{align}
where $\matparteps\, \vec\eta =\vec 0$. From (\ref{eq:secvarlocal}) we can also
derive that for sufficiently smooth $w$
\begin{align}
\matparteps \,(\nabs\, w)
= [ \nabs\,\vec\chi - 2\,\mat D_s(\vec \chi)]\,\nabs\,w
\quad \text{a.e.\ on}\quad \Gamma(t)\,, 
\label{eq:secvarwlocal}
\end{align}
where $\matparteps\, w = 0$.
In addition, it follows from (\ref{eq:secvarwlocal}) 
that
\begin{align}
\matparteps\,|\nabs\,w|^2 & = 2\,\nabs\,w\,.\,\matparteps \,(\nabs\, w) 
= - 2\,\nabs\,w\,.\,(\nabs\,\vec\chi\, \nabs\, w) \nonumber \\ & 
= - 2\,(\nabs\,w \otimes \nabs\,w) : \nabs\,\vec\chi 
\quad \text{a.e.\ on}\quad \Gamma(t)\,, 
\label{eq:nabsw2}
\end{align}
where $\matparteps\, w = 0$.

\begin{rem} \label{rem:simplifications1}
We note from {\rm (\ref{eq:secvarlocal})} 
that the last term in {\rm (\ref{eq:secvar})} can be 
simplified to 
\begin{equation} \label{eq:LM3ashort}
- 2 \left\langle \nabs\,\vec\eta, \mat D_s(\vec \chi) \right\rangle_{\Gamma(t)}.
\end{equation}
However, to be consistent with our approximations in \cite{pwfade}, we prefer
the form used in {\rm (\ref{eq:secvar})}.
\end{rem}

It is straightforward to derive results for the time derivative of the
considered quantities from the collected first variations above. For example,
it follows from (\ref{eq:secvar}) that
\begin{align}
& \ddt \left\langle \nabs\,\vec\id,
\nabs\,\vec{\eta}\right\rangle_{\Gamma(t)} = 
\left\langle \nabs\,.\,\vec \eta, \nabs \,.\,\vec{\mathcal V}
 \right\rangle_{\Gamma(t)}
+ \left\langle \nabs\,\vec \eta, \nabs \,\vec{\mathcal V} 
 \right\rangle_{\Gamma(t)}
\nonumber \\ & \hspace{2cm}
- 2\left\langle (\nabs\,\vec\eta)^T, \mat D_s(\vec{\mathcal V})\,
 (\nabs\,\vec\id)^T \right\rangle_{\Gamma(t)} 
\quad \forall\ \vec\eta \in \{ \vec\xi \in [H^1(\GT)]^d 
: \matpartx\,\vec\xi = \vec0 \}\,.
\label{eq:secvar2}
\end{align}

On recalling (\ref{eq:side3first}), (\ref{eq:matwmateta}) and 
(\ref{eq:Weingarten}), 
we now consider the first variation of (\ref{eq:Enew}) subject to the 
side constraints
\begin{subequations}
\begin{align}
& \left\langle \kappastar, \vec\eta \right\rangle_{\Gamma(t)}
+ \left\langle \nabs \,\vec\id,\nabs\, \vec{\eta}\right\rangle_{\Gamma(t)}= 0 
\qquad \forall\ \vec\eta \in \Vt\,, \label{eq:side3} \\
& \left\langle \mat w^\star , \mat\zeta \right\rangle_{\Gamma(t)}
+ \tfrac12 \left\langle \vec\nu,
[\mat\zeta + \mat\zeta^T]\,\kappastar + 
\nabs\,.\,[\mat\zeta + \mat\zeta^T] \right\rangle_{\Gamma(t)} =0 
\qquad \forall\ \mat\zeta \in [\Wt]^{d\times d} \,. \label{eq:sidez}
\end{align}
\end{subequations}
Here we use the symmetric formulation in (\ref{eq:sidez}), because its
discretized form will then ensure that the discrete approximations to
$\mat w^\star$ are also symmetric, since
\begin{equation} \label{eq:wsymm}
\left\langle (\mat w^\star)^T , \mat\zeta \right\rangle_{\Gamma(t)}
= \left\langle \mat w^\star , \mat\zeta^T \right\rangle_{\Gamma(t)}
= \left\langle \mat w^\star , \mat\zeta \right\rangle_{\Gamma(t)}
\qquad \forall\ \mat\zeta \in [\Wt]^{d\times d} \,.
\end{equation}

On recalling (\ref{eq:Enew}), we define the Lagrangian 
\begin{align}
L(\Gamma, \kappastar, \vec y, \mat w^\star, \mat z,\phasec) & =
\tfrac12 \left\langle \alpha(\phasec)\,|\kappastar - \spont(\phasec)\,\vec\nu|^2,1 
\right\rangle_{\Gamma(t)} 
+ \tfrac12 \left\langle \alpha^G(\phasec) , |\kappastar|^2 - |\mat w^\star|^2
\right\rangle_{\Gamma(t)} \nonumber \\ & \quad
+ \beta \left\langle b_{CH}(\phasec), 1 \right\rangle_{\Gamma(t)}
- \left\langle \kappastar, \vec y \right\rangle_{\Gamma(t)} - 
\left\langle \nabs \,\vec\id, \nabs\,\vec y \right\rangle_{\Gamma(t)}
\nonumber \\ & \quad
-\left\langle \mat w^\star , \mat z \right\rangle_{\Gamma(t)}
- \tfrac12 \left\langle \vec\nu, 
[\mat z + \mat z^T]\,\kappastar + \nabs\,.\,[\mat z + \mat z^T]  
\right\rangle_{\Gamma(t)}  , 
\label{eq:Lag3}
\end{align}
where $\vec y\in \Vt$ and $\mat z \in [\Wt]^{d\times d}$ are
Lagrange multipliers for (\ref{eq:side3},b).
In order to compute the 
direction of steepest descent, $\vec f_\Gamma$, of 
$E(\Gamma(t), \phasec(t))$, with respect to $\Gamma(t)$ and subject to the
constraints (\ref{eq:side3},b), we set the variations of 
$L(\Gamma, \kappastar, \vec y, \mat w^\star, \mat z,\phasec)$ 
with respect to $\kappastar$, $\vec y$, $\mat w^\star$ and $\mat z$ to zero, 
and we use the variation with respect to $\phasec$ to define the
Cahn--Hilliard dynamics. 
Moreover, 
we obtain on using the formal calculus of PDE constrained optimization, see
e.g.\ \cite{Troltzsch10}, that
\begin{subequations}
\begin{align}
\left[\deldel{\Gamma}\,L\right](\vec\chi) 
& = \lim_{\epsilon\to0} \tfrac1\epsilon\left[ 
L(\Gamma_\epsilon, \kappastar_\epsilon, \vec y_\epsilon,
\mat w^\star_\epsilon, \mat z_\epsilon, \phasec_\epsilon) -
L(\Gamma, \kappastar, \vec y, \mat w^\star, \mat z,\phasec)\right] 
= - \left\langle \vec f_\Gamma, \vec\chi \right\rangle_{\Gamma(t)} , 
\label{eq:ddGLa}\\
\left[\deldel{\kappastar}\,L\right](\vec\xi) & =
\lim_{\epsilon\to0} \tfrac1\epsilon\left[ 
L(\Gamma, \kappastar + \epsilon\,\vec\xi, \vec y, \mat w^\star, \mat z,\phasec) 
- L(\Gamma, \kappastar, \vec y, \mat w^\star, \mat z,\phasec)
\right] = 0\,, 
\label{eq:ddGLb}\\
\left[\deldel{\vec y}\,L\right](\vec\eta) & =
\lim_{\epsilon\to0} \tfrac1\epsilon\left[ 
L(\Gamma, \kappastar , \vec y + \epsilon\,\vec\eta, \mat w^\star, \mat z,\phasec) 
- L(\Gamma, \kappastar, \vec y, \mat w^\star, \mat z,\phasec) \right] = 0 \,,
\label{eq:ddGLc}\\
\left[\deldel{\mat w^\star}\,L\right](\mat\phi) & =
\lim_{\epsilon\to0} \tfrac1\epsilon\left[ 
L(\Gamma, \kappastar , \vec y, \mat w^\star + \epsilon\,\mat\phi, \mat z,\phasec) 
- L(\Gamma, \kappastar, \vec y, \mat w^\star, \mat z,\phasec) \right] = 0 \,,
\label{eq:ddGLd}\\
\left[\deldel{\mat z}\,L\right](\mat\zeta) & =
\lim_{\epsilon\to0} \tfrac1\epsilon\left[ 
L(\Gamma, \kappastar , \vec y, \mat w^\star , \mat z+ \epsilon\,\mat\zeta,\phasec) 
- L(\Gamma, \kappastar, \vec y, \mat w^\star, \mat z,\phasec) \right] = 0 \,,
\label{eq:ddGLe}\\
\left[\deldel{\phasec}\,L\right](\xi) & =
\lim_{\epsilon\to0} \tfrac1\epsilon\left[ 
L(\Gamma, \kappastar , \vec y , \mat w^\star, \mat z, \phasec + \epsilon\,\xi) 
- L(\Gamma, \kappastar, \vec y, \mat w^\star, \mat z,\phasec) \right] = 
\left\langle \chempot, \xi \right\rangle_{\Gamma(t)} \,,
\label{eq:ddGLf}
\end{align}
\end{subequations}
where $\kappastar_\epsilon,\,\vec y_\epsilon \in \Vdeltat$,
$\mat w^\star_\epsilon,\,\mat z_\epsilon \in [\Wdeltat]^{d\times d}$,
$\phasec_\epsilon \in \Wdeltat$ are defined as 
in (\ref{eq:wdelta}), and where $\chempot$ defines the chemical potential.
We note that (\ref{eq:ddGLc},e) immediately yield (\ref{eq:side3},b), which
means that we can recover $\kappastar$ and $\mat w^\star$ in terms of
$\Gamma(t)$ again. In particular, combining (\ref{eq:nabszeta}) and 
(\ref{eq:side3}) yields, on recalling (\ref{eq:LBop}) that 
$\kappastar = \vec\varkappa$. In
addition, it then follows from (\ref{eq:sidez}) and (\ref{eq:matwmateta}) that
$\mat w^\star = \mat w = \nabs\,\vec\nu$.
On recalling (\ref{eq:bendingb}), (\ref{eq:DElem5.2eps})--(\ref{eq:secvar}), 
(\ref{eq:secvarlocal}) and (\ref{eq:nabsw2}) this yields that
\begin{subequations}
\begin{align}
& \left\langle \vec f_\Gamma, \vec\chi \right\rangle_{\Gamma(t)}-   
\left\langle \nabs\,\vec y, \nabs\,\vec\chi \right\rangle_{\Gamma(t)}-
\left\langle \nabs\,.\,\vec y, \nabs\,.\,\vec\chi \right\rangle_{\Gamma(t)}
+2\left\langle (\nabs\,\vec y)^T , \mat D_s(\vec \chi)\,(\nabs\,\vec\id)^T
\right\rangle_{\Gamma(t)}
\nonumber \\ & \quad 
+\tfrac12\left\langle \alpha(\phasec)\,|\vec\varkappa - \spont(\phasec)\,\vec\nu|^2
- 2\,\vec y\,.\,\vec\varkappa,\nabs\,.\,\vec\chi \right\rangle_{\Gamma(t)}
 + \left\langle \alpha(\phasec)\,\spont(\phasec)\,(\vec\varkappa -\spont(\phasec)\,\vec\nu),
[\nabs\,\vec\chi]^T\,\vec \nu \right\rangle_{\Gamma(t)}
\nonumber \\ & \quad 
+ \beta 
\left\langle 
b_{CH}(\phasec), \nabs\,.\,\vec\chi \right\rangle_{\Gamma(t)}
- \beta\,\gamma \left\langle (\nabs\,\phasec)\otimes(\nabs\,\phasec) , \nabs\,\vec\chi
\right\rangle_{\Gamma(t)} 
\nonumber \\ & \quad 
+ \tfrac12 \left\langle \alpha^G(\phasec) \,( |\vec\varkappa|^2 - |\mat w|^2 ),
\nabs\,.\,\vec\chi \right\rangle_{\Gamma(t)} 
- \left\langle \mat w: \mat z,\nabs\,.\,\vec\chi
\right\rangle_{\Gamma(t)} 
\nonumber \\ & \quad 
- \tfrac12 \left\langle \vec\nu\,.\,
([\mat z + \mat z^T]\,\vec\varkappa + \nabs\,.\,[\mat z + \mat z^T] ),
\nabs\,.\,\vec\chi \right\rangle_{\Gamma(t)} 
- \sum_{i=1}^d
\left\langle \nu_i\,\nabs\,\vec z_i, \nabs\,\vec\chi - 2\,\mat D_s(\vec \chi)
\right\rangle_{\Gamma(t)} 
\nonumber \\ & \quad 
+ \tfrac12 \left\langle [\mat z + \mat z^T]\,\vec\varkappa 
+ \nabs\,.\,[\mat z + \mat z^T] , [\nabs\,\vec\chi]^T\,\vec\nu
\right\rangle_{\Gamma(t)} 
 = 0 
\qquad \forall\ \vec\chi\in \Vt\,, \label{eq:LM3a} \\
& \left\langle \alpha(\phasec)\,(\vec\varkappa - \spont(\phasec)\,\vec\nu) 
+ \alpha^G(\phasec)\,\vec\varkappa - \tfrac12\,[\mat z + \mat z^T]\,\vec\nu
- \vec y, \vec\xi \right\rangle_{\Gamma(t)}
 = 0 \qquad
\forall\ \vec\xi \in \Vt \,, \label{eq:LM3b} \\
& \mat z = -\alpha^G(\phasec)\,\mat w \,,\label{eq:LM3c} \\
& \left\langle \vec\varkappa, \vec\eta \right\rangle_{\Gamma(t)}
+ \left\langle \nabs \,\vec\id,\nabs\, \vec{\eta}\right\rangle_{\Gamma(t)}= 0 
\qquad \forall\ \vec\eta \in \Vt\,, \label{eq:LM3d} \\
& \left\langle \mat w , \mat\zeta \right\rangle_{\Gamma(t)}
+ \tfrac12 \left\langle \vec\nu,
[\mat\zeta + \mat\zeta^T]\,\vec\varkappa + 
\nabs\,.\,[\mat\zeta + \mat\zeta^T] \right\rangle_{\Gamma(t)} =0 
\qquad \forall\ \mat\zeta \in [\Wt]^{d\times d} \,. \label{eq:LM3e}
\end{align}
\end{subequations}
The above is coupled to (\ref{eq:weakGDa}--d) subject to the initial
conditions (\ref{eq:1init}).
Here we have introduced $\vec z_i = \tfrac12\,[\mat z + \mat z^T]\,\vec\ek_i$, 
$i=1\to d$, as well as $\nu_i = \vec\nu\,.\,\vec\ek_i$, $i=1\to d$.
Finally, on recalling (\ref{eq:strongCHa}), and on using 
(\ref{eq:nabszeta}),
(\ref{eq:DElem5.2}),
(\ref{eq:matpartx}), (\ref{eq:matpartu}) and
(\ref{eq:weakGDc},d), a weak form of the Cahn--Hilliard dynamics is given by
\begin{subequations}
\begin{align}
& \vartheta\,\ddt \left\langle \phasec, \eta \right\rangle_{\Gamma(t)}
+ \left\langle \nabs\,\chempot, \nabs\,\eta \right\rangle_{\Gamma(t)} = 0
\quad \forall\ \eta \in \{ \xi \in H^1(\GT) : \matpartx\,\xi = 0 \}\,,
\label{eq:CHa} \\
& \left\langle \chempot, \xi \right\rangle_{\Gamma(t)} =
 \beta\,\gamma \left\langle \nabs\,\phasec, \nabs\,\xi \right\rangle_{\Gamma(t)}
+ \beta\,\gamma^{-1} \left\langle \Psi'(\phasec), \xi \right\rangle_{\Gamma(t)}
 \nonumber \\ & \qquad\qquad\quad
 + \tfrac12 \left\langle
\alpha'(\phasec)\,|\vec\varkappa - \spont(\phasec)\,\vec\nu|^2 
- 2\,\spont'(\phasec)\,\alpha(\phasec)\,(\vec\varkappa - \spont(\phasec)\,\vec\nu)\,.\,\vec\nu,\xi
\right\rangle_{\Gamma(t)} 
\nonumber \\ & \qquad\qquad\quad
+ \tfrac12 \left\langle (\alpha^G)'(\phasec) \,( |\vec\varkappa|^2 - |\mat w|^2 ),
\xi \right\rangle_{\Gamma(t)} 
\qquad \forall\ \xi \in \Wt\,, \label{eq:CHb} \\
& \phasec(\cdot,0) = \phasec_0 \qquad \text{on }\ \Gamma_0\,, \label{eq:CHc}
\end{align}
\end{subequations}
with $\phasec_0:\Gamma_0 \to \R$ given initial data, recall (\ref{eq:1init}). 
Here we note that (\ref{eq:CHb}) is well-posed for nonconstant $\alpha$,
$\alpha^G$ and $\spont$ only in the case $\beta > 0$, 
which is why we assume that $\beta$ is
positive throughout the manuscript.
In addition, we observe that choosing
$\eta=1$ in (\ref{eq:CHa}) yields that
\begin{equation} \label{eq:ccons}
\ddt \left\langle \phasec, 1 \right\rangle_{\Gamma(t)} = 0\,.
\end{equation}

\begin{rem} \label{rem:simplifications2}
With regards to {\rm (\ref{eq:LM3b})} we note from 
{\rm (\ref{eq:LM3c})} 
and {\rm (\ref{eq:Weingarten})},
as $\mat w = \nabs\,\vec\nu = (\nabs\,\vec\nu)^T$,
it holds that $\mat z = - \alpha^G(\phasec)\,\mat w = -\alpha^G(\phasec)\,\nabs\,\vec\nu$, 
and so $\mat z\,\vec\nu = \mat z^T\,\vec\nu = \vec 0$. 
For further simplifications we refer to the appendix.
\end{rem}

We note the following LBB-type condition:
\begin{equation} \label{eq:LBBG}
\inf_{(\varphi,\eta) \in \widehat\pspace \times L^2(\Gamma(t))} 
\sup_{\vec \xi \in \uspace_{\Gamma(t)} (\vec 0)}
\frac{( \varphi, \nabla \,.\,\vec \xi) 
+ \left\langle \eta, \nabs \,.\,\vec \xi \right\rangle_{\Gamma(t)}}
{(\|\varphi\|_0 + \| \eta \|_{0,\Gamma(t)})
\,(\|\vec \xi\|_1 + \| \mat{\mathcal{P}}_\Gamma\,\vec\xi\!\mid_{\Gamma(t)} 
\|_{1,\Gamma(t)})} 
\geq C > 0\,,
\end{equation}
which we also refer to as the LBB$_\Gamma$ condition.
Here we have defined the space $\uspace_{\Gamma(t)}(\vec 0) := \{
\vec\xi \in \uspace(\vec 0) : 
\mat{\mathcal{P}}_\Gamma\,\vec\xi\!\mid_{\Gamma(t)}
\in [H^1(\Gamma(t))]^d\}$, and let 
$\| \vec\eta \|_{1,\Gamma(t)}^2 :=
\left\langle \vec\eta,\vec\eta \right\rangle_{\Gamma(t)}
+ \left\langle \nabs\,\vec\eta,\nabs\,\vec\eta \right\rangle_{\Gamma(t)}$.
In the case that the smooth hypersurface $\Gamma(t)$ is not a sphere, 
then (\ref{eq:LBBG}) is shown to hold
if $\partial_1\Omega=\partial\Omega$ is a smooth boundary in 
\citet[p.~15]{Lengeler15preprint}. See also the discussion around
(2.11a,b) in \cite{nsns}.

Overall the weak formulation for the free boundary problem
(\ref{eq:NSa}--d), (\ref{eq:sigma}), (\ref{eq:1a}--d), (\ref{eq:sigmaG}), 
(\ref{eq:strongCHa},b), (\ref{eq:1init}), (\ref{eq:CHc}) 
that we consider in this paper is given by
\begin{equation} \label{eq:P}
\text{(P) \qquad\qquad
(\ref{eq:weakGDa}--d), 
(\ref{eq:LM3a}--e),
(\ref{eq:CHa}--c), (\ref{eq:1init}).}
\end{equation}

\begin{rem} \label{rem:2d}
We note that in the case $d=2$ we do not consider Gaussian curvature terms,
i.e.\ we assume that $\alpha^G(\phasec) = 0$. Then {\rm (\ref{eq:LM3a})} simplifies
to
\begin{align}
& \left\langle \vec f_\Gamma, \vec\chi \right\rangle_{\Gamma(t)}-   
\left\langle \nabs\,\vec y, \nabs\,\vec\chi \right\rangle_{\Gamma(t)}-
\left\langle \nabs\,.\,\vec y, \nabs\,.\,\vec\chi \right\rangle_{\Gamma(t)}
+2\left\langle (\nabs\,\vec y)^T , \mat D_s(\vec \chi)\,(\nabs\,\vec\id)^T
\right\rangle_{\Gamma(t)}
\nonumber \\ & \quad 
+\tfrac12\left\langle \alpha(\phasec)\,|\vec\varkappa - \spont(\phasec)\,\vec\nu|^2
- 2\,\vec y\,.\,\vec\varkappa,\nabs\,.\,\vec\chi \right\rangle_{\Gamma(t)}
 + \left\langle \alpha(\phasec)\,\spont(\phasec)\,(\vec\varkappa -\spont(\phasec)\,\vec\nu),
[\nabs\,\vec\chi]^T\,\vec \nu \right\rangle_{\Gamma(t)}
\nonumber \\ & \quad 
+ \beta 
\left\langle 
b_{CH}(\phasec), \nabs\,.\,\vec\chi \right\rangle_{\Gamma(t)}
- \beta\,\gamma \left\langle (\partial_s\,\phasec)^2 , \nabs\,.\,\vec\chi
\right\rangle_{\Gamma(t)} 
 = 0 
\qquad \forall\ \vec\chi\in \Vt\,. \label{eq:2dLM3a} 
\end{align}
Clearly, the last two terms in {\rm (\ref{eq:2dLM3a})} can be absorbed by the 
surface pressure $p_\Gamma$ in {\rm (\ref{eq:weakGDa})}. Hence, for constant
$\alpha$ and constant $\spont$, the evolution of the interface is totally
independent of the Cahn--Hilliard system.
Of course, for $d=3$ even for constant $\alpha$, $\spont$ and $\alpha^G$, the
line tension term 
$\beta\,\gamma\,\langle (\nabs\,\phasec)\otimes(\nabs\,\phasec),
\nabs\,\vec\chi\rangle_{\Gamma(t)}$
in {\rm (\ref{eq:LM3a})} means that nonconstant values of $\phasec$ do
have an influence on the membrane evolution.
\end{rem}

\setcounter{equation}{0}
\section{Semidiscrete finite element approximation} \label{sec:4}
For simplicity we consider $\Omega$ to be a polyhedral domain. Then 
let  ${\mathcal{T}}^h$ 
be a regular partitioning of $\Omega$ into disjoint open simplices
$\sigmaO^h_j$, $j = 1,\ldots, J_\Omega$.
Associated with ${\mathcal{T}}^h$ are the finite element spaces
\begin{equation*} 
 S^h_k := \{\chi \in C(\overline{\Omega}) : \chi\!\mid_{\sigmaO} \in
\mathcal{P}_k(\sigmaO) \quad \forall\ \sigmaO \in {\mathcal{T}}^h\} 
\subset H^1(\Omega)\,, \qquad k \in \mathbb{N}\,,
\end{equation*}
where $\mathcal{P}_k(\sigmaO)$ denotes the space of polynomials of degree $k$
on $\sigmaO$. We also introduce $S^h_0$, the space of  
piecewise constant functions on ${\mathcal{T}}^h$.
Let $\{\varphi_{k,j}^h\}_{j=1}^{K_k^h}$ be the standard basis functions 
for $S^h_k$, $k\geq 0$.
We introduce $\vec I^h_k:[C(\overline{\Omega})]^d\to [S^h_k]^d$, $k\geq 1$, 
the standard interpolation
operators, such that $(\vec I^h_k\,\vec\eta)(\vec{p}_{k,j}^h)= 
\vec\eta(\vec{p}_{k,j}^h)$ for $j=1,\ldots, K_k^h$; where
$\{\vec p_{k,j}^h\}_{j=1}^{K_k^h}$ denotes the coordinates of the degrees of
freedom of $S^h_k$, $k\geq 1$. In addition we define the standard projection
operator $I^h_0:L^1(\Omega)\to S^h_0$, such that
\begin{equation*}
(I^h_0 \eta)\!\mid_{o} = \frac1{\mathcal{L}^d(o)}\,\int_{o}
\eta \dL{d} \qquad \forall\ o \in \mathcal{T}^h\,.
\end{equation*}
Our approximation to the velocity and pressure on ${\mathcal{T}}^h$
will be based on standard finite element spaces
$\uspace^h(\vec g)\subset\uspace(\vec I^h_k\,\vec g)$, for some $k \geq 2$,
and $\pspace^h(t)\subset\pspace$, recall (\ref{eq:UVV},b). 
Here, for the former we assume from now on that 
$\vec g \in [C(\overline\Omega)]^d$.
We require also the space $\widehat\pspace^h(t):= 
\pspace^h(t) \cap \widehat\pspace$. 
Here, in general, we will choose pairs of 
velocity/pressure finite element spaces 
that satisfy the LBB inf-sup condition,
see e.g.\ \citet[p.~114]{GiraultR86}.
For example, we may choose the lowest order Taylor-Hood element 
P2--P1 for $d=2$ and $d=3$, the P2--P0 element or the 
P2--(P1+P0) element for $d=2$ on setting 
$\uspace^h(\vec g)=[S^h_2]^d\cap\uspace(\vec I^h_2\,\vec g)$, 
and $\pspace^h = S^h_1,\,S^h_0$ or $S^h_1+S^h_0$, respectively.

The parametric finite element spaces in order to approximate e.g.\ 
$\vec\varkappa$ and $\phasec$ are defined as follows.
Similarly to \cite{gflows3d}, we introduce the following discrete
spaces, based on the work of 
\cite{Dziuk91}. 
Let $\Gamma^h(t)\subset\R^d$ be a $(d-1)$-dimensional {\em polyhedral surface}, 
i.e.\ a union of non-degenerate $(d-1)$-simplices with no
hanging vertices (see \citet[p.~164]{DeckelnickDE05} for $d=3$),
approximating the closed surface $\Gamma(t)$. In
particular, let $\Gamma^h(t)=\bigcup_{j=1}^{J_\Gamma}
\overline{\sigma^h_j(t)}$, where $\{\sigma^h_j(t)\}_{j=1}^{J_\Gamma}$ is a
family of mutually disjoint open $(d-1)$-simplices with vertices
$\{\vec{q}^h_k(t)\}_{k=1}^{K_\Gamma}$.
Then let
\begin{align*}
\Wht & := \{\chi \in C(\Gamma^h(t)):\chi\!\mid_{\sigma^h_j}
\mbox{ is linear}\ \forall\ j=1,\ldots, J_\Gamma\} \,, \\ 
\Vht & := \{\vec\chi \in [C(\Gamma^h(t))]^d:\vec\chi\!\mid_{\sigma^h_j}
\mbox{ is linear}\ \forall\ j=1,\ldots, J_\Gamma\} \,,\\ 
\matVht & := \{\mat\chi \in [C(\Gamma^h(t))]^{d\times d}:
\mat\chi\!\mid_{\sigma^h_j} \mbox{ is linear}\ \forall\ j=1,\ldots, J_\Gamma\}
 \,.
\end{align*}
Hence $\Wht$ is the space of scalar continuous
piecewise linear functions on $\Gamma^h(t)$, with 
$\{\chi^h_k(\cdot,t)\}_{k=1}^{K_\Gamma}$ 
denoting the standard basis of $\Wht$, i.e.\
\begin{equation} \label{eq:bf}
\chi^h_k(\vec q^h_l(t),t) = \delta_{kl}\qquad
\forall\ k,l \in \{1,\ldots,K_\Gamma\}\,, t \in [0,T]\,.
\end{equation}
We require that $\Gamma^h(t) = \vec X^h(\Gamma^h(0), t)$ with
$\vec X^h \in \Vhtz$, and that $\vec q^h_k \in [H^1(0,T)]^d$,
$k = 1,\ldots,K_\Gamma$.
For later purposes, we also introduce 
$\pi^h(t): C(\Gamma^h(t))\to \Wht$, the standard interpolation operator
at the nodes $\{\vec{q}_k^h(t)\}_{k=1}^{K_\Gamma}$, and similarly
$\vec\pi^h(t): [C(\Gamma^h(t))]^d\to \Vht$.

For scalar and vector 
functions $\eta,\zeta$ on $\Gamma^h(t)$ 
we introduce the $L^2$--inner 
product $\langle\cdot,\cdot\rangle_{\Gamma^h(t)}$ over
the polyhedral surface $\Gamma^h(t)$ 
as follows
\begin{equation*} 
\left\langle \eta, \zeta \right\rangle_{\Gamma^h(t)} := 
\int_{\Gamma^h(t)} \eta\,.\,\zeta \dH{d-1}\,.
\end{equation*}
In order to derive a stable numerical method, it is crucial to consider
numerical integration in the discrete energy, see (\ref{eq:Eh}) below. Hence,
for piecewise continuous functions $v,w$, with possible jumps
across the edges of $\{\sigma_j^h(t)\}_{j=1}^{J_\Gamma}$,
we introduce the mass lumped inner product
$\langle\cdot,\cdot\rangle_{\Gamma^h(t)}^h$ as
\begin{equation} \label{eq:defNI}
\left\langle \eta, \phi \right\rangle^h_{\Gamma^h(t)} 
= \sum_{j=1}^J \left\langle \eta, \phi \right\rangle^h_{\sigma^h_j(t)} 
 :=
\sum_{j=1}^{J}\tfrac1d\,\mathcal{H}^{d-1}(\sigma^h_j(t))\,\sum_{k=1}^{d} 
(\eta\,\phi)((\vec{q}^h_{j_k}(t))^-)\,,
\end{equation}
where $\{\vec{q}^h_{j_k}(t)\}_{k=1}^{d}$ are the vertices of $\sigma^h_j(t)$,
and where
we define $\eta((\vec{q}^h_{j_k}(t))^-):= $ \linebreak $
\underset{\sigma^h_j(t)\ni \vec{p}\to \vec{q}^h_{j_k}(t)}{\lim}\, \eta(\vec{p})$.
We
naturally extend this definition to vector and tensor functions.

Following \citet[(5.23)]{DziukE13}, we define the discrete material velocity
for $\vec z \in \Gamma^h(t)$ by
\begin{equation} \label{eq:Xht}
\vec{\mathcal{V}}^h(\vec z, t) := \sum_{k=1}^{K_\Gamma}
\left[\ddt\,\vec q^h_k(t)\right] \chi^h_k(\vec z, t) \,.
\end{equation}
For later use, we also introduce the finite element spaces
\begin{align*}
W_T(\GhT) := \{ \phi \in C(\GhT) & : 
\phi(\cdot, t) \in \Wht \quad \forall\ t \in [0,T]\,,\nonumber \\ & \quad
\phi(\vec q^h_k(t), t) \in H^1(0,T) \quad \forall\ k \in\{1,\ldots,K\}
 \}\,,
\end{align*}
where $\GhT := \bigcup_{t \in [0,T]} \Gamma^h(t) \times \{t\}$,
as well as the vector- and tensor-valued analogues $\underline V_T(\GhT)$ and
$\mat V_T(\GhT)$.
In a similar fashion, we introduce $W_T(\sigma^h_{j,T})$ via
\begin{align*}
W_T(\sigma^h_{j,T}) := \{ \phi \in C(\overline{\sigma^h_{j,T}}) & : 
\phi(\cdot, t) \text{~ is linear~} \quad \forall\ t \in [0,T]\,,
\nonumber \\ & \quad
\phi(\vec q^h_{j_k}(t), t) \in H^1(0,T) \quad k = 1,\ldots,d
 \}\,,
\end{align*}
where 
$\{\vec{q}^h_{j_k}(t)\}_{k=1}^{d}$ are the vertices of $\sigma^h_j(t)$,
and where
$\sigma^h_{j,T} := \bigcup_{t \in [0,T]} \sigma^h_j(t) \times \{t\}$,
for $j\in\{1,\ldots,J\}$.

Then, similarly to (\ref{eq:matpartx}), we define the discrete material
derivatives on $\Gamma^h(t)$ element-by-element via the equations
\begin{equation} \label{eq:matpartxh}
(\matpartxh\, \phi)\!\mid_{\sigma^h_j(t)} = 
(\phi_t + \vec{\mathcal{V}}^h\,.\,\nabla\,\phi) \!\mid_{\sigma^h_j(t)}
\qquad\forall\ \phi\in W_T(\sigma^h_{j,T})\,,\ j \in \{1,\ldots,J\}\,.
\end{equation}
Moreover, similarly to (\ref{eq:Gammadelta}), 
for any given $\vec \chi \in \Vht$ we introduce 
\begin{align} \label{eq:Gammahdelta}
& 
\Gamma^h_\epsilon(t) := \{ \vec\Phi^h(\vec z,\epsilon) : \vec z \in \Gamma^h(t)
 \} \,, \quad\text{where}\quad
\vec\Phi^h(\vec z, 0) = \vec z \text{~~and~~}
\nonumber \\ & \hspace{8cm}
\tfrac{\partial\vec\Phi^h}{\partial\epsilon}(\vec z, 0) = \vec\chi(\vec z)
\quad \forall\ \vec z \in \Gamma^h(t)\,,
\end{align}
as well as $\partial^{0,h}_\epsilon$ defined by
{\rm (\ref{eq:deltaderiv})} with $\Gamma(t)$ and $\vec\Phi$ replaced by 
$\Gamma^h(t)$ and $\vec\Phi^h$, respectively.
We also introduce
\begin{equation} \label{eq:VhG}
\utimespace^h_{\Gamma^h}(\vec g) := \{
\vec \phi \in H^1(0,T; \uspace^h(\vec g)) : \exists\,
\vec\chi \in \underline{V}_T(\GhT) \text{, s.t. }  \vec\chi(\cdot,t) = 
\vec\pi^h\,[\vec\phi\!\mid_{\Gamma^h(t)}]\ \forall\ t\in[0,T]\}\,.
\end{equation}

On differentiating (\ref{eq:bf}) with respect to $t$, it immediately follows
that
\begin{equation} \label{eq:mpbf}
\matpartxh\, \chi^h_k = 0
\quad\forall\ k \in \{1,\ldots,K_\Gamma\}\,,
\end{equation}
see \citet[Lem.\ 5.5]{DziukE13}. 
It follows directly from (\ref{eq:mpbf}) that
\begin{equation*} 
\matpartxh\,\zeta(\cdot,t) = \sum_{k=1}^{K_\Gamma} \chi^h_k(\cdot,t)\,
\ddt\,\zeta_k(t) \quad \text{on}\ \Gamma^h(t)
\end{equation*}
for $\zeta(\cdot,t) = \sum_{k=1}^{K_\Gamma} \zeta_k(t)\,\chi^h_k(\cdot,t)
\in \Wht$, and hence
$\matpartxh\,\vec\id = \vec{\mathcal{V}}^h$ on $\Gamma^h(t)$.

We recall from \citet[Lem.~5.6]{DziukE13} that
\begin{equation} \label{eq:DEeq5.28}
\ddt\, \int_{\sigma^h_j(t)} \zeta \dH{d-1} 
= \int_{\sigma^h_j(t)} \matpartxh\,\zeta + \zeta\,\nabs\,.\,\vec{\mathcal{V}}^h 
\dH{d-1} \quad \forall\ \zeta \in W_T(\sigma^h_{j,T})\,, j \in
\{1,\ldots,J_\Gamma\}\,.
\end{equation}
Moreover, on recalling (\ref{eq:defNI}), we have that
\begin{align}
& \ddt \left\langle \eta, \zeta \right\rangle_{\sigma^h_j(t)}^h
 = \left\langle \matpartxh\,\eta, \zeta \right\rangle_{\sigma^h_j(t)}^h
+ \left\langle \eta, \matpartxh\,\zeta \right\rangle_{\sigma^h_j(t)}^h
+ \left\langle \eta\,\zeta,\nabs\,.\,\vec{\mathcal{V}}^h 
\right\rangle_{\sigma^h_j(t)}^h \nonumber \\ & \hspace{8cm}
\quad \forall\ \eta, \zeta \in W_T(\sigma^h_{j,T})\,, j \in
\{1,\ldots,J_\Gamma\}\,.
\label{eq:dtsigma}
\end{align}

Given $\Gamma^h(t)$, we 
let $\Omega^h_+(t)$ denote the exterior of $\Gamma^h(t)$ and let
$\Omega^h_-(t)$ denote the interior of $\Gamma^h(t)$, so that
$\Gamma^h(t) = \partial \Omega^h_-(t) = \overline{\Omega^h_-(t)} \cap 
\overline{\Omega^h_+(t)}$. 
We then partition the elements of the bulk mesh 
$\mathcal{T}^h$ into interior, exterior and interfacial elements as follows.
Let
\begin{align*}
\mathcal{T}^h_-(t) & := \{ o \in \mathcal{T}^h : o \subset
\Omega^h_-(t) \} \,, \nonumber \\
\mathcal{T}^h_+(t) & := \{ o \in \mathcal{T}^h : o \subset
\Omega^h_+(t) \} \,, \nonumber \\
\mathcal{T}^h_{\Gamma^h}(t) & := \{ o \in \mathcal{T}^h : o \cap
\Gamma^h(t) \not = \emptyset \} \,. 
\end{align*}
Clearly $\mathcal{T}^h = \mathcal{T}^h_-(t) \cup \mathcal{T}^h_+(t) \cup
\mathcal{T}^h_{\Gamma}(t)$ is a disjoint partition.
In addition, we define the piecewise constant unit normal 
$\vec{\nu}^h(t)$ to $\Gamma^h(t)$ such that $\vec\nu^h(t)$ points into
$\Omega^h_+(t)$.
Moreover, we introduce the discrete density 
$\rho^h(t) \in S^h_0$ and the discrete viscosity $\mu^h(t) \in S^h_0$ as
\begin{equation*} 
\rho^h(t)\!\mid_{o} = \begin{cases}
\rho_- & o \in \mathcal{T}^h_-(t)\,, \\
\rho_+ & o \in \mathcal{T}^h_+(t)\,, \\
\tfrac12\,(\rho_- + \rho_+) & o \in \mathcal{T}^h_{\Gamma^h}(t)\,,
\end{cases}
\quad\text{and}\quad
\mu^h(t)\!\mid_{o} = \begin{cases}
\mu_- & o \in \mathcal{T}^h_-(t)\,, \\
\mu_+ & o \in \mathcal{T}^h_+(t)\,, \\
\tfrac12\,(\mu_- + \mu_+) & o \in \mathcal{T}^h_{\Gamma^h}(t)\,.
\end{cases}
\end{equation*}

Similarly to (\ref{eq:Ps},b), we introduce
\begin{subequations}
\begin{equation} \label{eq:Psh}
\mat{\mathcal{P}}_{\Gamma^h} = \mat\Id - \vec\nu^h \otimes \vec\nu^h
\quad\text{on}\ \Gamma^h(t)\,,
\end{equation}
and
\begin{equation} \label{eq:Dsh}
\mat D_s^h(\vec \eta) = \tfrac12\,\mat{\mathcal{P}}_{\Gamma^h}\,
(\nabs\,\vec\eta + (\nabs\,\vec \eta)^T)\,\mat{\mathcal{P}}_{\Gamma^h}
\quad\text{on}\ \Gamma^h(t)\,,
\end{equation}
\end{subequations}
where here $\nabs = \mat{\mathcal{P}}_{\Gamma^h} \,\nabla$ 
denotes the surface gradient on $\Gamma^h(t)$.
Moreover, we introduce the vertex normal function
$\vec\omega^h(\cdot, t) \in\Vht$ with
\begin{equation} \label{eq:omegah}
\vec\omega^h(\vec{q}^h_k(t),t ) := 
\frac{1}{\mathcal{H}^{d-1}(\Lambda^h_k(t))}
\sum_{j\in\Theta_k^h} \mathcal{H}^{d-1}(\sigma^h_j(t))\,
\vec{\nu}^h\!\mid_{\sigma^h_j(t)}\,,
\end{equation}
where for $k= 1 ,\ldots, K_\Gamma^h$ we define
$\Theta_k^h:= \{j : \vec{q}^h_k(t) \in \overline{\sigma^h_j(t)}\}$
and set
\begin{equation*}
\Lambda_k^h(t) := \cup_{j \in \Theta_k^h} \overline{\sigma^h_j(t)}
\,. 
\end{equation*}
For later use we note that 
\begin{equation} 
\left\langle \vec z, w\,\vec\nu^h\right\rangle_{\Gamma^h(t)}^h =
\left\langle \vec z, w\,\vec\omega^h\right\rangle_{\Gamma^h(t)}^h 
\qquad \forall\ \vec z \in \Vht\,,\ w \in \Wht \,,
\label{eq:NIh}
\end{equation}
and so, in particular, 
$\left\langle \vec z, \vec\nu^h\right\rangle_{\Gamma^h(t)} =
\left\langle \vec z, \vec\nu^h\right\rangle_{\Gamma^h(t)}^h =
\left\langle \vec z, \vec\omega^h\right\rangle_{\Gamma^h(t)}^h$ 
for all $\vec z \in \Vht$.

In what follows we will introduce a finite element approximation
for the weak formulation (P), recall (\ref{eq:P}).
By repeating on the discrete level the steps in \S\ref{sec:32},
we will now derive a discrete analogue of (\ref{eq:LM3a}--e).

Similarly to the continuous setting in (\ref{eq:Enew}) and (\ref{eq:side3},b),
we consider the first variation of the discrete energy
\begin{align} \label{eq:Eh}
E^h(\Gamma^h(t), \phaseC^h(t)) & 
= \tfrac12 \left\langle \alpha(\phaseC^h), |\vec\kappa^h - \spont(\phaseC^h)\,\vec\nu^h|^2
\right\rangle_{\Gamma^h(t)}^h
+ \tfrac12 \left\langle \alpha^G(\phaseC^h), |\vec\kappa^h|^2 - |\mat W^h|^2
\right\rangle_{\Gamma^h(t)}^h \nonumber \\ & \qquad
+ \beta\,\left\langle b_{CH}(\phaseC^h), 1 \right\rangle_{\Gamma^h(t)}^h, 
\end{align}
where $\vec\kappa^h\in\Vht$ and $\mat W^h\in \matVht$ have to 
satisfy side constraints
\begin{subequations}
\begin{align}
& \left\langle \vec\kappa^h, \vec\eta \right\rangle_{\Gamma^h(t)}^h
+ \left\langle
\nabs \,\vec\id,\nabs\, \vec{\eta}\right\rangle_{\Gamma^h(t)}= 0 \qquad 
\forall\ \vec\eta \in \Vht\,,\label{eq:side3h}  \\
& \left\langle \mat W^h , \mat\zeta \right\rangle_{\Gamma^h(t)}^h
+ \tfrac12\left\langle \vec\nu^h, [\mat\zeta + \mat\zeta^T]\,\vec\kappa^h
+ \nabs\,.\,[\mat\zeta + \mat\zeta^T] \right\rangle_{\Gamma^h(t)}^h =0 
\qquad \forall\ \mat\zeta \in \matVht \,. \label{eq:sidezh}
\end{align}
\end{subequations}
Similarly to (\ref{eq:Lag3}), we define the Lagrangian 
\begin{align*} &
L^h(\Gamma^h, \vec\kappa^h, \vec Y^h, \mat W^h, \mat Z^h, \phaseC^h) 
\nonumber \\ & \quad
=
\tfrac12 \left\langle \alpha(\phaseC^h), |\vec\kappa^h - \spont(\phaseC^h)\,\vec\nu^h|^2 
\right\rangle_{\Gamma^h(t)}^h
+ \tfrac12 \left\langle \alpha^G(\phaseC^h), |\vec\kappa^h|^2 - |\mat W^h|^2
\right\rangle_{\Gamma^h(t)}^h
\nonumber \\ & \qquad \qquad
+ \beta \left\langle b_{CH}(\phaseC^h), 1 \right\rangle_{\Gamma^h(t)}^h
- \left\langle \vec\kappa^h, \vec Y^h \right\rangle_{\Gamma^h(t)}^h 
- \left\langle \nabs \,\vec\id, \nabs\,\vec Y^h \right\rangle_{\Gamma^h(t)}
\nonumber \\ & \qquad \qquad
- \left\langle \mat W^h , \mat Z^h \right\rangle_{\Gamma^h(t)}^h
- \tfrac12 \left\langle \vec\nu^h, [\mat Z^h + (\mat Z^h)^T]\,\vec\kappa^h
+ \nabs\,.\,[\mat Z^h + (\mat Z^h)^T] \right\rangle_{\Gamma^h(t)}^h,
\end{align*}
where $\vec\kappa^h\in \Vht$, $\mat W^h \in \matVht$,
$\phaseC^h \in \Wht$, 
with $\vec Y^h\in \Vht$ and $\mat Z^h \in \matVht$
being Lagrange multipliers for (\ref{eq:side3h},b), respectively.
Similarly to (\ref{eq:LM3a}--c), on recalling the formal calculus of PDE
constrained optimization, we obtain the gradient 
$\vec F_\Gamma^h \in \Vht$ of 
$E^h(\Gamma^h(t),\phaseC^h(t))$ with respect to $\Gamma^h(t)$ subject to the
side constraints (\ref{eq:side3h},b) by setting
$[\deldel{\Gamma^h}\,L^h](\vec\chi)=
-\left\langle \vec F_\Gamma^h, \vec\chi \right\rangle_{\Gamma^h(t)}^h$
for $\vec\chi\in\Vht$, 
where we have recalled the definition (\ref{eq:Gammahdelta}),
and by setting the remaining variations with respect to
$\vec\kappa^h$, $\vec Y^h$, $\mat W^h$ and $\mat Z^h$ to zero.
On noting (\ref{eq:bendingb}), (\ref{eq:NIh}) and the variation
analogue of (\ref{eq:dtsigma}), 
as well as the obvious discrete variants of 
(\ref{eq:DElem5.2eps})--(\ref{eq:secvar}), 
(\ref{eq:secvarlocal}) and (\ref{eq:nabsw2}), we then obtain that
\begin{subequations}
\begin{align}
& \left\langle \vec F_\Gamma^h, \vec\chi \right\rangle_{\Gamma^h(t)}^h
- \left\langle \nabs\,\vec Y^h, \nabs\,\vec\chi \right\rangle_{\Gamma^h(t)}-
\left\langle \nabs\,.\,\vec Y^h, \nabs\,.\,\vec\chi \right\rangle_{\Gamma^h(t)}
\nonumber \\ & \qquad 
+\tfrac12\left\langle \alpha(\phaseC^h)\,|\vec\kappa^h - \spont(\phaseC^h)\,\vec\nu^h|^2
- 2\,\vec Y^h\,.\,\vec\kappa^h,\nabs\,.\,\vec\chi\right\rangle_{\Gamma^h(t)}^h
\nonumber \\ & \qquad 
+2 \left\langle (\nabs\,\vec Y^h)^T , \mat D_s^h(\vec \chi)\,(\nabs\,\vec\id)^T
\right\rangle_{\Gamma^h(t)} 
+ \left\langle \alpha(\phaseC^h)\,\spont(\phaseC^h)\,(\vec\kappa^h -
\spont(\phaseC^h)\,\vec\nu^h), 
[\nabs\,\vec\chi]^T\,\vec\nu^h \right\rangle_{\Gamma^h(t)}^h
 \nonumber \\ & \qquad
+ \beta 
\left\langle 
b_{CH}(\phaseC^h), \nabs\,.\,\vec\chi \right\rangle_{\Gamma^h(t)}^h
- \beta\,\gamma \left\langle (\nabs\,\phaseC^h)\otimes(\nabs\,\phaseC^h), \nabs\,\vec\chi
\right\rangle_{\Gamma^h(t)} 
\nonumber \\ & \qquad 
+ \tfrac12 \left\langle \alpha^G(\phaseC^h) \,( |\vec\kappa^h|^2 - |\mat W^h|^2 ),
\nabs\,.\,\vec\chi \right\rangle_{\Gamma^h(t)}^h
- \left\langle \mat W^h : \mat Z^h,\nabs\,.\,\vec\chi
\right\rangle_{\Gamma^h(t)}^h 
\nonumber \\ & \qquad 
- \tfrac12 \left\langle \vec\nu^h\,.\,(
[\mat Z^h + (\mat Z^h)^T]\,\vec\kappa^h + \nabs\,.\,[\mat Z^h + (\mat Z^h)^T]),
\nabs\,.\,\vec\chi \right\rangle_{\Gamma^h(t)}^h
\nonumber \\ & \qquad 
- \sum_{i=1}^d 
\left\langle \nu^h_i\,\nabs\,\vec Z^h_i,\nabs\,\vec\chi
- 2\,\mat D_s^h(\vec\chi) \right\rangle_{\Gamma^h(t)} 
\nonumber \\ & \qquad 
+ \tfrac12 \left\langle [\mat Z^h + (\mat Z^h)^T]\,\vec\kappa^h
+\nabs\,.\,[\mat Z^h + (\mat Z^h)^T] , [\nabs\,\vec\chi]^T\,\vec\nu^h
\right\rangle_{\Gamma^h(t)}^h 
 = 0 \qquad
 \forall\ \vec\chi\in \Vht\,, \label{eq:LMh3a} \\
& \left\langle \alpha(\phaseC^h)\,(\vec\kappa^h - \spont(\phaseC^h)\,\vec\nu^h) 
+ \alpha^G(\phaseC^h)\,\vec\kappa^h 
- \tfrac12\,[\mat Z^h + (\mat Z^h)^T]\,\vec\nu^h\,
- \vec Y^h, \vec\xi \right\rangle_{\Gamma^h(t)}^h  = 0 
\nonumber \\ & \hspace{10cm}
\qquad \forall\ \vec\xi \in \Vht \,, \label{eq:LMh3b} \\
& \mat Z^h = \mat\pi^h[-\alpha^G(\phaseC^h)\,\mat W^h]\,,\label{eq:LMh3c}
\end{align} 
\end{subequations}
as well as (\ref{eq:side3h},b) from the variations with respect to $\vec Y^h$
and $\mat Z^h$.
Here we have introduced $\vec Z^h_i = \tfrac12\,[\mat Z^h + (\mat Z^h)^T]\,
\vec\ek_i$, $i=1\to d$, as well as $\nu^h_i = \vec\nu^h\,.\,\vec\ek_i$, 
$i=1\to d$.
Similarly to (\ref{eq:wsymm}) it clearly follows from (\ref{eq:sidezh}) that
\begin{equation} \label{eq:WZsymm}
(\mat W^h)^T = \mat W^h \quad \Rightarrow \quad 
(\mat Z^h)^T = \mat Z^h \,,
\end{equation}
and so many terms in (\ref{eq:LMh3a},b) can be simplified. We will perform
these simplifications when we introduce the semidiscrete finite element
approximation, see (\ref{eq:sd3a}--d), (\ref{eq:sd3e}--d) below.
The Cahn--Hilliard dynamics are defined by
\begin{subequations}
\begin{align}
& \vartheta\,\ddt \left\langle \phaseC^h, \chi^h_k \right\rangle_{\Gamma^h(t)}^h
+ \left\langle \nabs\,\Chempot^h, \nabs\,\chi^h_k\right\rangle_{\Gamma^h(t)} =0
\qquad \forall\ k \in \{1,\ldots,K_\Gamma\}\,,\label{eq:CHha} \\
& \left\langle \Chempot^h, \xi \right\rangle_{\Gamma^h(t)}^h 
= \beta\,\gamma \left\langle \nabs\,\phaseC^h, \nabs\,\xi
\right\rangle_{\Gamma^h(t)}
+ \beta\,\gamma^{-1} \left\langle \Psi'(\phaseC^h), \xi \right\rangle_{\Gamma^h(t)}^h
\nonumber \\ & \qquad\qquad\qquad
 + \tfrac12 \left\langle
\alpha'(\phaseC^h)\,|\vec\kappa^h - \spont(\phaseC^h)\,\vec\nu^h|^2 
- 2\,\spont'(\phaseC^h)\,\alpha(\phaseC^h)\,(\vec\kappa^h - \spont(\phaseC^h)\,\vec\nu^h)
\,.\,\vec\nu^h,\xi \right\rangle_{\Gamma^h(t)}^h
\nonumber \\ & \qquad\qquad\qquad
+ \tfrac12 \left\langle (\alpha^G)'(\phaseC^h) \,( |\vec\kappa^h|^2 - |\mat W^h|^2 ),
\xi \right\rangle_{\Gamma^h(t)}^h 
\quad \forall\ \xi \in \Wht\,,
\label{eq:CHhb}
\end{align}
\end{subequations}
where, similarly to the continuous setting (\ref{eq:CHa},b), we have defined 
$\Chempot^h\in\Wht$ by
$\left\langle \Chempot^h, \xi \right\rangle_{\Gamma^h(t)}^h
=[\deldel{\phaseC^h}\,L^h](\xi)$ for all $\xi \in \Wht$.

Overall, we then obtain the following semidiscrete
continuous-in-time finite element approximation, which is the semidiscrete
analogue of the weak formulation (P), recall (\ref{eq:P}). 
Given $\Gamma^h(0)$, $\vec U^h(\cdot,0) \in \uspace^h(\vec g)$ 
and $\phaseC^h(\cdot, 0) \in \Whtz$,
find $(\Gamma^h(t))_{t\in(0,T]}$ such that
$\vec\id\!\mid_{\Gamma^h(\cdot)} \in \underline{V}_T(\GhT)$, with
$\vec{\mathcal{V}}^h = \matpartxh\,\vec\id\!\mid_{\Gamma^h(t)} \in\Vht$ 
for all $t\in(0,T]$, and
$\vec U^h \in \utimespace^h_{\Gamma^h}(\vec g)$,
$\phaseC^h \in W_T(\GhT)$, and,
for all $t\in(0,T]$,
$P^h(t) \in 
\widehat\pspace^h(t)$, 
$P_\Gamma^h(T) \in \Wht$, $\mat W^h(t) \in \matVht$ and
$\vec\kappa^h(t),\,\vec Y^h(t),\,\vec F_\Gamma^h(t) \in \Vht$,
$\Chempot^h \in \Wht$ such that 
(\ref{eq:CHha},b) holds, as well as
\begin{subequations}
\begin{align}
&
\tfrac12 \left[ \ddt \left( \rho^h\,\vec U^h , \vec \xi \right)
+ \left( \rho^h\,\vec U^h_t , \vec \xi \right)
- (\rho^h\,\vec U^h, \vec \xi_t) 
+ \rho_+\left\langle \vec U^h \,.\,\unitn , \vec U^h\,.\,\vec\xi 
 \right\rangle_{\partial_2\Omega}
\right]
 \nonumber \\ & \qquad
+ 2\left(\mu^h\,\mat D(\vec U^h), \mat D(\vec \xi) \right)
+ \tfrac12\left(\rho^h, 
 [(\vec U^h\,.\,\nabla)\,\vec U^h]\,.\,\vec \xi
- [(\vec U^h\,.\,\nabla)\,\vec \xi]\,.\,\vec U^h \right)
- \left(P^h, \nabla\,.\,\vec \xi\right)
\nonumber \\ & \qquad
+ \rho_\Gamma \left\langle \matpartxh\,\vec\pi^h\,\vec U^h, \vec \xi 
\right\rangle_{\Gamma^h(t)}^h
+ 2\,\mu_\Gamma \left\langle \mat D_s^h (\vec\pi^h\,\vec U^h) , 
\mat D_s^h (\vec\pi^h\, \vec \xi ) \right\rangle_{\Gamma^h(t)}
\nonumber \\ & \qquad
- \left\langle P_\Gamma^h , 
\nabs\,.\,(\vec\pi^h\,\vec \xi) \right\rangle_{\Gamma^h(t)}
= \left(\rho^h\,\vec f^h, \vec \xi\right)
+ \left\langle \vec F_\Gamma^h , \vec \xi \right\rangle_{\Gamma^h(t)}^h 
\qquad \forall\ \vec\xi \in H^1(0,T;\uspace^h(\vec 0)) \,, \label{eq:sd3a}\\
& \left(\nabla\,.\,\vec U^h, \varphi\right) = 0 
\qquad \forall\ \varphi \in \widehat\pspace^h(t)\,,
\label{eq:sd3b} \\
& \left\langle \nabs\,.\,(\vec\pi^h\,\vec U^h), \eta 
\right\rangle_{\Gamma^h(t)}  = 0 
\qquad \forall\ \eta \in \Wht\,,
\label{eq:sd3c} \\
& \left\langle \vec{\mathcal{V}}^h ,
\vec\chi \right\rangle_{\Gamma^h(t)}^h
= \left\langle \vec U^h, \vec\chi \right\rangle_{\Gamma^h(t)}^h
 \qquad\forall\ \vec\chi \in \Vht\,,
\label{eq:sd3d}
\end{align}
\end{subequations}
where we recall (\ref{eq:Xht}), and
\begin{subequations}
\begin{align}
& \left\langle \vec\kappa^h, \vec\eta \right\rangle_{\Gamma^h(t)}^h
+ \left\langle \nabs\,\vec\id, \nabs\,\vec \eta \right\rangle_{\Gamma^h(t)}
= 0 \quad\forall\ \vec\eta \in \Vht\,, \label{eq:sd3e} \\
& \left\langle \mat W^h , \mat\zeta \right\rangle_{\Gamma^h(t)}^h
+\tfrac12 \left\langle \vec\nu^h, 
[\mat\zeta + \mat\zeta^T] \,\vec\kappa^h 
+ \nabs\,.\,[\mat\zeta + \mat\zeta^T] \right\rangle_{\Gamma^h(t)}^h =0 
\qquad \forall\ \mat\zeta \in \matVht \,, \label{eq:sd3ee} \\
& \left\langle \alpha(\phaseC^h)\,(\vec\kappa^h - \spont(\phaseC^h)\,\vec\nu^h) 
+ \alpha^G(\phaseC^h)\,(\vec\varkappa + \mat W^h\,\vec\nu^h)\,
- \vec Y^h, \vec\xi \right\rangle_{\Gamma^h(t)}^h  = 0 
\qquad \forall\ \vec\xi \in \Vht \,, \label{eq:sd3f} \\
& \left\langle \vec F_\Gamma^h, \vec\chi \right\rangle_{\Gamma^h(t)}^h
= \left\langle \nabs\,\vec Y^h, \nabs\,\vec\chi \right\rangle_{\Gamma^h(t)}
+ \left\langle\nabs\,.\,\vec Y^h,\nabs\,.\,\vec\chi \right\rangle_{\Gamma^h(t)}
\nonumber \\ &\qquad
-\tfrac12\left\langle [ \alpha(\phaseC^h)\,|\vec\kappa^h - \spont(\phaseC^h)\,\vec\nu^h|^2
- 2\,\vec Y^h\,.\,\vec\kappa^h],\nabs\,.\,\vec\chi\right\rangle_{\Gamma^h(t)}^h
\nonumber \\ & \qquad 
-2 \left\langle (\nabs\,\vec Y^h)^T , \mat D_s^h(\vec \chi)\,(\nabs\,\vec\id)^T
\right\rangle_{\Gamma^h(t)} 
- \left\langle \alpha(\phaseC^h)\,\spont(\phaseC^h)\,\vec\kappa^h , 
[\nabs\,\vec\chi]^T\,\vec\nu^h \right\rangle_{\Gamma^h(t)}^h
 \nonumber \\ & \qquad
- \beta 
\left\langle 
b_{CH}(\phaseC^h), \nabs\,.\,\vec\chi \right\rangle_{\Gamma^h(t)}^h
+ \beta\,\gamma \left\langle (\nabs\,\phaseC^h)\otimes(\nabs\,\phaseC^h), \nabs\,\vec\chi
\right\rangle_{\Gamma^h(t)} 
\nonumber \\ & \qquad 
- \tfrac12 \left\langle \alpha^G(\phaseC^h) \,( |\vec\kappa^h|^2 + |\mat W^h|^2 ),
\nabs\,.\,\vec\chi \right\rangle_{\Gamma^h(t)}^h
+ \left\langle \vec\nu^h\,.\,(\mat Z^h\,\vec\kappa^h + \nabs\,.\,\mat Z^h),
\nabs\,.\,\vec\chi \right\rangle_{\Gamma^h(t)}^h
\nonumber \\ & \qquad 
+ \sum_{i=1}^d 
\left\langle \nu^h_i\,\nabs\,\vec Z^h_i,\nabs\,\vec\chi
- 2\,\mat D_s^h(\vec\chi) \right\rangle_{\Gamma^h(t)} 
- \left\langle \mat Z^h\,\vec\kappa^h
+ \nabs\,.\,\mat Z^h , [\nabs\,\vec\chi]^T\,\vec\nu^h
\right\rangle_{\Gamma^h(t)}^h 
\nonumber \\ & \hspace{10cm}
\qquad \forall\ \vec\chi\in \Vht\,, \label{eq:sd3g}
\end{align}
\end{subequations}
where $\mat Z^h = \mat\pi^h[-\alpha^G(\phaseC^h)\,\mat W^h]$
and $\vec Z^h_i = \mat Z^h\,\vec\ek_i$, $i=1\to d$.
In addition, we have noted (\ref{eq:WZsymm}) and
that
$\alpha(\phaseC^h)\,\spont^2(\phaseC^h)\,\vec\nu^h\,.\, [\nabs\,\vec\chi]^T\,\vec\nu^h = 0$ 
on $\Gamma^h(t)$.
Here we have defined 
$\vec f^h(\cdot,t) := \vec I^h_2\,\vec f(\cdot,t)$, 
where here and throughout we assume that $\vec f \in
L^2(0,T;[C(\overline\Omega)]^d)$.
We note that in the special case of uniform $\alpha$ and $\spont$,
and if $\alpha^G = \beta=0$, 
the scheme (\ref{eq:sd3a}--d), \mbox{(\ref{eq:sd3e}--d)} 
collapses to the semidiscrete
approximation (4.15a--g), with $\beta=0$, from \cite{nsnsade}. 

The following lemma is crucial in establishing a 
direct discrete analogue of (\ref{eq:disseq}).

\begin{lem} \label{lem:sd3stab}
Let $\{(\Gamma^h, \vec U^h, P^h, P^h_\Gamma, \vec\kappa^h, \vec Y^h, 
\vec F_\Gamma^h, \mat W^h, \mat Z^h, \phaseC^h, \Chempot^h)(t)\}_{t\in[0,T]}$ 
be a solution to {\rm (\ref{eq:CHha},b)}, {\rm (\ref{eq:sd3a}--d)}, 
{\rm (\ref{eq:sd3e}--d)}.
In addition, we assume that 
$\vec\kappa^h \in \underline{V}_T(\GhT)$ and
$\mat W^h \in \mat V_T(\GhT)$.
Then 
\begin{equation}
\ddt \, E^h(\Gamma^h(t), \phaseC^h(t))
= - \left\langle \vec F_\Gamma^h, \vec{\mathcal{V}}^h 
\right\rangle_{\Gamma^h(t)}^h 
+ \left\langle \Chempot^h, \matpartxh\,\phaseC^h \right\rangle_{\Gamma^h(t)}^h .
\label{eq:sd3stab}
\end{equation}
\end{lem}
\begin{proof}
Taking the time derivatives of (\ref{eq:side3h},b), where we choose
discrete test functions $ \vec\eta$ and $\mat\zeta$ such that
$\matpartxh\,\vec\eta = \vec0$ and $\matpartxh\,\mat\zeta = \mat 0$,
respectively, yields that
\begin{subequations}
\begin{align} \label{eq:s31h}
& \left\langle \matpartxh\,\vec\kappa^h, \vec\eta\right\rangle_{\Gamma^h(t)}^h
+ \left\langle \vec\kappa^h\,.\,\vec \eta,
 \nabs\,.\,\vec{\mathcal{V}}^h \right\rangle_{\Gamma^h(t)}^h
+ \left\langle \nabs\,.\, \vec{\mathcal{V}}^h, \nabs\,.\,\vec\eta\right\rangle_{\Gamma^h(t)} 
+ \left\langle \nabs\,\vec{\mathcal{V}}^h, \nabs\,\vec\eta\right\rangle_{\Gamma^h(t)} 
\nonumber \\ & \qquad
- 2 \left\langle \mat D_s^h(\vec{\mathcal{V}}^h)\, (\nabs\,\vec\id)^T, 
 (\nabs\,\vec\eta)^T \right\rangle_{\Gamma^h(t)} 
 = 0\,, \\
& \left\langle \matpartxh\,\mat W^h , \mat\zeta \right\rangle_{\Gamma^h(t)}^h
+ \left\langle \mat W^h : \mat\zeta , \nabs\,.\,\vec{\mathcal{V}}^h 
\right\rangle_{\Gamma^h(t)}^h
+ \tfrac12 \left\langle \matpartxh\,\vec\nu^h,
[\mat\zeta + \mat\zeta^T] \,\vec\kappa^h + 
\nabs\,.\,[\mat\zeta + \mat\zeta^T] \right\rangle_{\Gamma^h(t)}^h
\nonumber \\ & \qquad
+ \tfrac12 \left\langle \vec\nu^h\,.\,
([\mat\zeta + \mat\zeta^T]\,\vec\kappa^h +
\nabs\,.\,[\mat\zeta + \mat\zeta^T]), 
\nabs\,.\,\vec{\mathcal{V}}^h \right\rangle_{\Gamma^h(t)}^h
\nonumber \\ & \qquad
+ \sum_{i=1}^d 
\left\langle \nu^h_i\,\nabs\,\vec \zeta_i,\nabs\,\vec{\mathcal{V}}^h
-2\,\mat D_s^h(\vec{\mathcal{V}}^h) \right\rangle_{\Gamma^h(t)} 
+ \tfrac12 \left\langle \vec\nu^h , 
[\mat\zeta + \mat\zeta^T]\,\matpartxh\,\vec\kappa^h 
\right\rangle_{\Gamma^h(t)}^h  =0  \,, \label{eq:s31hb}
\end{align}
\end{subequations}
where $\vec\zeta_i = \tfrac12\,[\mat\zeta + \mat\zeta^T]\,\vec\ek_i$, 
$i=1,\ldots,d$. 
Here we have noted $\vec\kappa^h \in \underline{V}_T(\GhT)$,
$\mat W^h \in \mat V_T(\GhT)$, 
(\ref{eq:dtsigma}) and the discrete versions of (\ref{eq:secvar2}) 
and (\ref{eq:secvarlocal}). 
Choosing $\vec\chi = \vec{\mathcal{V}}^h$ in (\ref{eq:LMh3a}),
$\vec\eta = \vec Y^h$ in (\ref{eq:s31h}) 
$\mat\zeta = \mat Z^h$ in (\ref{eq:s31hb}) and combining yields,
on recalling (\ref{eq:dtsigma})
and the discrete variants of (\ref{eq:normvar}) and (\ref{eq:nabsw2}), that
\begin{align} \label{eq:s32h}
& - \left\langle \vec F_\Gamma^h, \vec{\mathcal{V}}^h 
\right\rangle_{\Gamma^h(t)}^h =
\tfrac12\left\langle \alpha(\phaseC^h)\,|\vec\kappa^h - \spont(\phaseC^h)\,\vec\nu^h|^2,
\nabs\,.\,\vec{\mathcal{V}}^h\right\rangle_{\Gamma^h(t)}^h
\nonumber \\ &\qquad
-  \left\langle \alpha(\phaseC^h)\,\spont(\phaseC^h)\, (\vec\kappa^h
-\spont(\phaseC^h)\,\vec\nu^h),\matpartxh\,\vec\nu^h
\right\rangle_{\Gamma^h(t)}^h 
+ \left\langle \matpartxh\,\vec\kappa^h,\vec Y^h \right\rangle_{\Gamma^h(t)}^h 
\nonumber \\ & \qquad
+ \beta 
\left\langle 
b_{CH}(\phaseC^h), \nabs\,.\,\vec{\mathcal{V}}^h \right\rangle_{\Gamma^h(t)}^h
- \beta\,\gamma \left\langle (\nabs\,\phaseC^h)\otimes(\nabs\,\phaseC^h), 
\nabs\,\vec{\mathcal{V}}^h\right\rangle_{\Gamma^h(t)} 
\nonumber \\ & \qquad 
+ \tfrac12 \left\langle \alpha^G(\phaseC^h) \,( |\vec\kappa^h|^2 - |\mat W^h|^2 ),
\nabs\,.\,\vec{\mathcal{V}}^h \right\rangle_{\Gamma^h(t)}^h 
+ \left\langle \matpartxh\,\mat W^h, \mat Z^h\right\rangle_{\Gamma^h(t)}^h 
\nonumber \\ & \qquad 
+ \tfrac12
 \left\langle \vec\nu^h, [\mat Z^h + (\mat Z^h)^T]\,\matpartxh\,\vec\kappa^h
\right\rangle_{\Gamma^h(t)}^h \nonumber \\ &
= \tfrac12\left\langle \alpha(\phaseC^h)\,|\vec\kappa^h - \spont(\phaseC^h)\,\vec\nu^h|^2,
\nabs\,.\,\vec{\mathcal{V}}^h\right\rangle_{\Gamma^h(t)}^h
\nonumber \\ &\qquad
-  \left\langle \alpha(\phaseC^h)\,\spont(\phaseC^h)\, (\vec\kappa^h
-\spont(\phaseC^h)\,\vec\nu^h),\matpartxh\,\vec\nu^h
\right\rangle_{\Gamma^h(t)}^h 
\nonumber \\ &\qquad
+ \left\langle \alpha(\phaseC^h)\,(\vec\kappa^h - \spont(\phaseC^h)\,\vec\nu^h) 
+ \alpha^G(\phaseC^h)\,\vec\kappa^h, \matpartxh\,\vec\kappa^h
 \right\rangle_{\Gamma^h(t)}^h 
\nonumber \\ & \qquad
+ \beta 
\left\langle 
b_{CH}(\phaseC^h), \nabs\,.\,\vec{\mathcal{V}}^h \right\rangle_{\Gamma^h(t)}^h
- \beta\,\gamma \left\langle (\nabs\,\phaseC^h)\otimes(\nabs\,\phaseC^h), 
\nabs\,\vec{\mathcal{V}}^h\right\rangle_{\Gamma^h(t)} 
\nonumber \\ & \qquad 
+ \tfrac12 \left\langle \alpha^G(\phaseC^h) \,( |\vec\kappa^h|^2 - |\mat W^h|^2 ),
\nabs\,.\,\vec{\mathcal{V}}^h \right\rangle_{\Gamma^h(t)}^h 
- \left\langle \alpha^G(\phaseC^h)\, \matpartxh\,\mat W^h, \mat W^h
\right\rangle_{\Gamma^h(t)}^h \nonumber \\ & 
= \ddt \left[
\tfrac12 \left\langle \alpha(\phaseC^h), |\vec\kappa^h - \spont(\phaseC^h)\,\vec\nu^h|^2
\right\rangle_{\Gamma^h(t)}^h
+ \tfrac12 \left\langle \alpha^G(\phaseC^h), |\vec\kappa^h|^2 - |\mat W^h|^2
\right\rangle_{\Gamma^h(t)}^h \right. \nonumber \\ & \qquad \left.
+ \beta\,\left\langle b_{CH}(\phaseC^h), 1 \right\rangle_{\Gamma^h(t)}^h
\right] \nonumber \\ & \qquad
- \tfrac12 \left\langle
\alpha'(\phaseC^h)\,|\vec\kappa^h - \spont(\phaseC^h)\,\vec\nu^h|^2 
- 2\,\spont'(\phaseC^h)\,\alpha(\phaseC^h)\,(\vec\kappa^h - \spont(\phaseC^h)\,\vec\nu^h)
\,.\,\vec\nu^h, \matpartxh\,\phaseC^h \right\rangle_{\Gamma^h(t)}^h
\nonumber \\ & \qquad
- \tfrac12 \left\langle (\alpha^G)'(\phaseC^h) \,( |\vec\kappa^h|^2 - |\mat W^h|^2 ),
\matpartxh\,\phaseC^h \right\rangle_{\Gamma^h(t)}^h 
-\beta\,\gamma \left\langle \nabs\,\phaseC^h, \nabs\,\matpartxh\,\phaseC^h
\right\rangle_{\Gamma^h(t)}
\nonumber \\ & \qquad
- \beta\,\gamma^{-1} \left\langle \Psi'(\phaseC^h), \matpartxh\,\phaseC^h 
\right\rangle_{\Gamma^h(t)}^h \nonumber \\ &
= \ddt E^h(\Gamma^h(t), \phaseC^h(t)) - \left\langle \Chempot^h,
\matpartxh\,\phaseC^h \right\rangle_{\Gamma^h(t)}^h
\end{align}
where we have noted (\ref{eq:LMh3b},c) and (\ref{eq:CHhb}), as well as
$\phaseC^h \in W_T(\GhT)$.
This yields the desired result (\ref{eq:sd3stab}). 
\qquad\end{proof}

In the following theorem we derive discrete analogues of (\ref{eq:disseq}), 
(\ref{eq:areacons}) and (\ref{eq:ccons}) 
for the scheme (\ref{eq:CHha},b), (\ref{eq:sd3a}--d), (\ref{eq:sd3e}--d). 
\begin{thm} \label{thm:stab}
Let the assumptions of {\rm Lemma~\ref{lem:sd3stab}} hold.
Then, in the case $\vec g = \vec 0$, it holds that
\begin{align} \label{eq:thm}
& \ddt \left( \tfrac12\,\|[\rho^h]^\frac12\,\vec U^h \|_0^2 
+ \tfrac12\,
  \rho_\Gamma \left\langle \vec U^h, \vec U^h \right\rangle_{\Gamma^h(t)}^h
+ E^h(\Gamma^h(t), \phaseC^h(t)) \right)
+ 2\,\| [\mu^h]^\frac12\,\mat D(\vec U^h)\|_0^2 
\nonumber \\ & \quad
+ 2\,\mu_\Gamma \left\langle \mat D_s^h (\vec\pi^h\,\vec U^h) , 
\mat D_s^h (\vec\pi^h\, \vec U^h ) \right\rangle_{\Gamma^h(t)}
+ \tfrac12\,\rho_+\left\langle \vec U^h \,.\,\unitn , |\vec U^h|^2 
 \right\rangle_{\partial_2\Omega}
\nonumber \\ & \qquad
+ \vartheta^{-1}
\left\langle \nabs\,\Chempot^h,\nabs\,\Chempot^h \right\rangle_{\Gamma^h(t)}
= (\rho^h\,\vec f^h, \vec U^h) \,.
\end{align}
Moreover, it holds that
\begin{subequations}
\begin{equation} \label{eq:supportconsh}
\ddt\left\langle \chi^h_k, 1 \right\rangle_{\Gamma^h(t)} = 0
\quad\forall\ k \in \{1,\ldots,K_\Gamma\}
\end{equation}
and hence that
\begin{equation} \label{eq:areaconsh}
\ddt\, \mathcal{H}^{d-1} (\Gamma^h(t)) = 0 \,.
\end{equation}
Finally, we have that
\begin{equation} \label{eq:cconsh}
\ddt\left\langle \phaseC^h, 1 \right\rangle_{\Gamma^h(t)} = 0\,.
\end{equation}
\end{subequations}
\end{thm} 
\begin{proof}
Choosing
$\vec \xi = \vec U^h$ in (\ref{eq:sd3a}), 
recall that $\vec g = \vec 0$, $\varphi = P^h$ in (\ref{eq:sd3b})
and $\eta = P^h_\Gamma$ in (\ref{eq:sd3c}) yields that
\begin{align} \label{eq:lemGD}
& \tfrac12\,\ddt \,\|[\rho^h]^\frac12\,\vec U^h \|_0^2 +
2\,\| [\mu^h]^\frac12\,\mat D(\vec U^h)\|_0^2 
+ \rho_\Gamma \left\langle \matpartxh\,\vec\pi^h\,\vec U^h, \vec U^h
\right\rangle_{\Gamma^h(t)}^h
+ \tfrac12\,\rho_+\left\langle \vec U^h \,.\,\unitn , |\vec U^h|^2 
 \right\rangle_{\partial_2\Omega}
\nonumber \\ & \hspace{1cm}
+ 2\,\mu_\Gamma \left\langle \mat D_s^h (\vec\pi^h\,\vec U^h) , 
\mat D_s^h (\vec\pi^h\, \vec U^h ) \right\rangle_{\Gamma^h(t)}
= (\rho^h\,\vec f^h, \vec U^h)
+ \left\langle\vec F_\Gamma^h,\vec U^h \right\rangle_{\Gamma^h(t)}^h.
\end{align}
Moreover, we note that (\ref{eq:dtsigma}), (\ref{eq:sd3d}) and 
(\ref{eq:sd3c}) with $\eta = 
\pi^h\,[|\vec U^h \! \mid_{\Gamma^h(t)}|^2]$ 
imply that
\begin{align} \label{eq:dtUU}
\tfrac12\,\rho_\Gamma \,\ddt
\left\langle \vec U^h, \vec U^h \right\rangle_{\Gamma^h(t)}^h
& = \tfrac12\,\rho_\Gamma \left\langle \matpartxh\,\vec\pi^h\,[|\vec U^h|^2],1
\right\rangle_{\Gamma^h(t)}^h
+ \tfrac12\,\rho_\Gamma \left\langle \nabs\,.\,\vec{\mathcal{V}}^h, |\vec U^h|^2
\right\rangle_{\Gamma^h(t)}^h
\nonumber \\ &
 = \rho_\Gamma \left\langle \matpartxh\,\vec\pi^h\,\vec U^h, \vec U^h
\right\rangle_{\Gamma^h(t)}^h
+ \tfrac12\,\rho_\Gamma \left\langle \nabs\,.\,(\vec\pi^h\, \vec U^h), 
|\vec\pi^h\,\vec U^h|^2 \right\rangle_{\Gamma^h(t)}
\nonumber \\ &
= \rho_\Gamma \left\langle \matpartxh\,\vec\pi^h\,\vec U^h, \vec U^h
\right\rangle_{\Gamma^h(t)}^h \,,
\end{align}
where we have recalled $\vec U^h \in \utimespace^h_{\Gamma^h}(\vec g)$,
see (\ref{eq:VhG}). 
Choosing $\vec\chi = \vec F^h_\Gamma$ in
(\ref{eq:sd3d}), 
and combining with (\ref{eq:sd3stab}), yields that
\begin{align} \label{eq:pwfsurff}
 \left\langle \vec F_\Gamma^h,\vec U^h \right\rangle_{\Gamma^h(t)}^h & =
 \left\langle \vec F_\Gamma^h,\vec{\mathcal{V}}^h \right\rangle_{\Gamma^h(t)}^h
= -  \ddt E^h(\Gamma^h(t), \phaseC^h(t))
+ \left\langle \Chempot^h, \matpartxh\,\phaseC^h \right\rangle_{\Gamma^h(t)}^h .
\end{align}
Moreover, similarly to (\ref{eq:dtUU}), 
it follows from (\ref{eq:dtsigma}), (\ref{eq:mpbf}) and
(\ref{eq:sd3c},d), on recalling $\phaseC^h \in W_T(\GhT)$, that
\begin{align}
& \ddt \left\langle \phaseC^h, \chi^h_k \right\rangle_{\Gamma^h(t)}^h
= \left\langle \matpartxh\,\phaseC^h, \chi^h_k \right\rangle_{\Gamma^h(t)}^h
+ \left\langle \phaseC^h\, \chi^h_k, \nabs\,.\,\vec{\mathcal{V}}^h 
\right\rangle_{\Gamma^h(t)}^h \nonumber \\ & \qquad
= \left\langle \matpartxh\,\phaseC^h, \chi^h_k \right\rangle_{\Gamma^h(t)}^h
+ \left\langle \pi^h\,[\phaseC^h\, \chi^h_k], \nabs\,.\,\vec{\mathcal{V}}^h 
\right\rangle_{\Gamma^h(t)} \nonumber \\ & \qquad
= \left\langle \matpartxh\,\phaseC^h, \chi^h_k \right\rangle_{\Gamma^h(t)}^h
+ \left\langle \pi^h\,[\phaseC^h\, \chi^h_k], \nabs\,.\,(\vec\pi^h\,\vec U^h)
\right\rangle_{\Gamma^h(t)}
= \left\langle \matpartxh\,\phaseC^h, \chi^h_k \right\rangle_{\Gamma^h(t)}^h\,,
\label{eq:dtCh}
\end{align}
for $k = 1,\ldots,K_\Gamma$. Hence we obtain from (\ref{eq:CHha}) that
\begin{equation} \label{eq:CM}
- \left\langle \Chempot^h,
\matpartxh\,\phaseC^h \right\rangle_{\Gamma^h(t)}^h = 
\vartheta^{-1}
\left\langle \nabs\,\Chempot^h,
\nabs\,\Chempot^h \right\rangle_{\Gamma^h(t)} .
\end{equation}
The desired result (\ref{eq:thm}) now directly follows from combining
(\ref{eq:lemGD}), (\ref{eq:dtUU}), (\ref{eq:pwfsurff}) and (\ref{eq:CM}).

Similarly to (\ref{eq:areacons}), 
it immediately follows from (\ref{eq:DEeq5.28}) and (\ref{eq:mpbf}), 
on choosing $\eta = \chi^h_k$ in (\ref{eq:sd3c}), 
and on recalling from (\ref{eq:sd3d}) 
that $\vec{\mathcal{V}}^h = \vec\pi^h\,[\vec U^h\!\mid_{\Gamma^h(t)}]$, that
\begin{equation} \label{eq:supportproof}
\ddt\left\langle \chi^h_k, 1 \right\rangle_{\Gamma^h(t)} = 
\left\langle \chi^h_k, \nabs\,.\,\vec{\mathcal{V}}^h \right\rangle_{\Gamma^h(t)}
= 0 \,,
\end{equation}
which proves the desired result (\ref{eq:supportconsh}). Summing
(\ref{eq:supportconsh}) for all $k = 1,\ldots,K_\Gamma$ then yields
the desired result (\ref{eq:areaconsh}). 
Similarly, summing (\ref{eq:CHha}) for $k = 1,\ldots,K_\Gamma$ yields
the desired result (\ref{eq:cconsh}). 
\end{proof}

We observe that it does not appear possible to prove a discrete analogue of
(\ref{eq:conserved}) for the scheme 
(\ref{eq:CHha},b), (\ref{eq:sd3a}--d), (\ref{eq:sd3e}--d),
even if the pressure
space $\pspace^h(t)$ is enriched by the characteristic function of the inner
phase, $\charfcn{\Omega_-^h(t)}$. 
Following the approach introduced in \cite{nsns,nsnsade}, we enforce 
\begin{equation} \label{eq:Udotomega}
\left\langle \vec U^h, \vec \omega^h \right\rangle_{\Gamma^h(t)}^h = 
0\,,
\end{equation}
which will lead to volume conservation for the two phases on the discrete
level.
As (\ref{eq:Udotomega}) cannot be interpreted in terms of enriching
$\pspace^h(t)$, we enforce it separately with the help of a Lagrange
multiplier, which we denote by $\Psing^h$.
We are now in a position to propose the following adaptation of
(\ref{eq:CHha},b), \mbox{(\ref{eq:sd3a}--d)}, (\ref{eq:sd3e}--d)

Given $\Gamma^h(0)$, $\vec U^h(\cdot,0) \in \uspace^h(\vec g)$ 
and $\phaseC^h(\cdot, 0) \in \Whtz$,
find $(\Gamma^h(t))_{t\in(0,T]}$ such that
$\vec\id\!\mid_{\Gamma^h(\cdot)} \in \underline{V}_T(\GhT)$, with
$\vec{\mathcal{V}}^h = \matpartxh\,\vec\id\!\mid_{\Gamma^h(t)} \in\Vht$ 
for all $t\in(0,T]$, and
$\vec U^h \in \utimespace^h_{\Gamma^h}(\vec g)$,
$\phaseC^h \in W_T(\GhT)$, and,
for all $t\in(0,T]$,
$P^h(t) \in 
\widehat\pspace^h(t)$, $\Psing^h(t) \in \R$,
$P_\Gamma^h(T) \in \Wht$, $\mat W^h(t) \in \matVht$ and
$\vec\kappa^h(t),\,\vec Y^h(t),\,\vec F_\Gamma^h(t) \in \Vht$,
$\Chempot^h \in \Wht$ such that 
(\ref{eq:CHha},b) holds, as well as
\begin{subequations}
\begin{align}
&
\tfrac12 \left[ \ddt \left( \rho^h\,\vec U^h , \vec \xi \right)
+ \left( \rho^h\,\vec U^h_t , \vec \xi \right)
- (\rho^h\,\vec U^h, \vec \xi_t) 
+ \rho_+\left\langle \vec U^h \,.\,\unitn , \vec U^h\,.\,\vec\xi 
 \right\rangle_{\partial_2\Omega}
\right]
 \nonumber \\ & \qquad
+ 2\left(\mu^h\,\mat D(\vec U^h), \mat D(\vec \xi) \right)
+ \tfrac12\left(\rho^h, 
 [(\vec U^h\,.\,\nabla)\,\vec U^h]\,.\,\vec \xi
- [(\vec U^h\,.\,\nabla)\,\vec \xi]\,.\,\vec U^h \right)
- \left(P^h, \nabla\,.\,\vec \xi\right)
\nonumber \\ & \qquad
- \Psing^h \left\langle\vec\omega^h, \vec \xi \right\rangle_{\Gamma^h(t)}^h
+ \rho_\Gamma \left\langle \matpartxh\,\vec\pi^h\,\vec U^h, \vec \xi 
\right\rangle_{\Gamma^h(t)}^h
+ 2\,\mu_\Gamma \left\langle \mat D_s^h (\vec\pi^h\,\vec U^h) , 
\mat D_s^h (\vec\pi^h\, \vec \xi ) \right\rangle_{\Gamma^h(t)}
\nonumber \\ & \qquad
- \left\langle P_\Gamma^h , 
\nabs\,.\,(\vec\pi^h\,\vec \xi) \right\rangle_{\Gamma^h(t)}
= \left(\rho^h\,\vec f^h, \vec \xi\right)
+ \left\langle \vec F_\Gamma^h , \vec \xi \right\rangle_{\Gamma^h(t)}^h 
\qquad \forall\ \vec\xi \in H^1(0,T;\uspace^h(\vec 0)) \,, \label{eq:sd2a}\\
& \left(\nabla\,.\,\vec U^h, \varphi\right) = 0 
\qquad \forall\ \varphi \in \widehat\pspace^h(t)
\qquad\text{and}\qquad
\left\langle\vec U^h, \vec \omega^h \right\rangle_{\Gamma^h(t)}^h = 0
\label{eq:sd2b} 
\end{align}
\end{subequations}
and (\ref{eq:sd3c},d), (\ref{eq:sd3e}--d) hold. 
We now have the following result.

\begin{thm} \label{thm:stab2}
Let $\{(\Gamma^h,\vec U^h, P^h, \Psing^h, P^h_\Gamma, \vec\kappa^h,\vec Y^h, 
\vec F_\Gamma^h, \mat W^h, \mat Z^h, \phaseC^h, \Chempot^h)(t)
\}_{t\in[0,T]}$ 
be a solution to {\rm (\ref{eq:CHha},b)}, 
{\rm (\ref{eq:sd2a},b)}, {\rm (\ref{eq:sd3c},d)}, {\rm (\ref{eq:sd3e}--d)}.
In addition, we assume that 
$\vec\kappa^h \in \underline{V}_T(\GhT)$ and
$\mat W^h \in \mat V_T(\GhT)$.
Then {\rm (\ref{eq:thm})} holds if $\vec g = \vec 0$. In addition, 
{\rm (\ref{eq:supportconsh}--c)} and
\begin{equation}
\ddt\, \vol(\Omega_-^h(t)) = 0\, \label{eq:cons2}
\end{equation}
hold.
\end{thm} 
\begin{proof}
The proofs for (\ref{eq:thm}) and {\rm (\ref{eq:supportconsh}--c)} 
are analogous to the proofs in Theorem~\ref{thm:stab}. In order to prove
(\ref{eq:cons2}) we choose $\vec\chi = \vec\omega^h \in \Vht$ in 
(\ref{eq:sd3d}) to yield that
\begin{equation*}
\frac{\rm d}{{\rm d}t} \vol(\Omega_-^h(t)) = 
\left\langle \vec{\mathcal{V}}^h , \vec\nu^h \right\rangle_{\Gamma^h(t)}
= \left\langle \vec{\mathcal{V}}^h , \vec\nu^h \right\rangle^h_{\Gamma^h(t)}
= \left\langle \vec{\mathcal{V}}^h , \vec\omega^h \right\rangle^h_{\Gamma^h(t)}
= \left\langle \vec U^h, \vec\omega^h \right\rangle_{\Gamma^h(t)}^h
=0\,,
\end{equation*}
where we have used \citet[Lemma~2.1]{DeckelnickDE05}, 
(\ref{eq:NIh}) and (\ref{eq:sd2b}). 
\end{proof}

\setcounter{equation}{0}
\section{Fully discrete finite element approximation} \label{sec:5}

We consider the partitioning $t_m = m\,\tau$, $m = 0,\ldots, M$, 
of $[0,T]$ into uniform time steps $\tau = T / M$.
The time discrete spatial discretizations then directly follow from the finite
element spaces introduced in \S\ref{sec:4}, where in order to allow for
adaptivity in space we consider bulk finite element spaces that change in time.
For all $m\ge 0$, let $\mathcal{T}^m$ 
be a regular partitioning of $\Omega$ into disjoint open simplices
$\sigmaO^m_j$, $j = 1 ,\ldots, J^m_\Omega$. 
Associated with ${\mathcal{T}}^m$ are the finite element spaces
$S^m_k(\Omega)$ for $k\geq 0$.
We introduce also $\vec I^m_k:[C(\overline{\Omega})]^d\to [S^m_k(\Omega)]^d$, 
$k\geq 1$, the standard interpolation operators, and the standard projection
operator $I^m_0:L^1(\Omega)\to S^m_0(\Omega)$.
The parametric finite element spaces are given by
\begin{equation*} 
\Vh := \{\vec\chi \in [C(\Gamma^m)]^d:\vec\chi\!\mid_{\sigma^m_j}
\mbox{ is linear}\ \forall\ j=1,\ldots, J_\Gamma\} 
=: [\Wh]^d \,,
\end{equation*}
for $m=0 ,\ldots, M-1$, and similarly for $\matVh$. Here
$\Gamma^m=\bigcup_{j=1}^{J_\Gamma} 
\overline{\sigma^m_j}$,
where $\{\sigma^m_j\}_{j=1}^{J_\Gamma}$ is a family of mutually disjoint open 
$(d-1)$-simplices 
with vertices $\{\vec{q}^m_k\}_{k=1}^{K_\Gamma}$. 
We denote the standard basis of $\Wh$ by
$\{\chi^m_k(\cdot,t)\}_{k=1}^{K_\Gamma}$.
We also introduce 
$\pi^m: C(\Gamma^m)\to \Wh$, the standard interpolation operator
at the nodes $\{\vec{q}_k^m\}_{k=1}^{K_\Gamma}$,
and similarly $\vec\pi^m: [C(\Gamma^m)]^d\to \Vh$.
Throughout this paper, we will parameterize the new closed surface 
$\Gamma^{m+1}$ over $\Gamma^m$, with the help of a parameterization
$\vec X^{m+1} \in \Vh$, i.e.\ $\Gamma^{m+1} = \vec X^{m+1}(\Gamma^m)$.
Moreover, let
\begin{equation} \label{eq:WhpVI}
\WhVI := \{\chi \in \Wh : |\chi| \leq 1 \}\,.
\end{equation}

Given $\Gamma^m$, we 
let $\Omega^m_+$ denote the exterior of $\Gamma^m$ and let
$\Omega^m_-$ denote the interior of $\Gamma^m$, so that
$\Gamma^m = \partial \Omega^m_- = \overline{\Omega^m_-} \cap 
\overline{\Omega^m_+}$. 
In addition, we define the piecewise constant unit normal 
$\vec\nu^m$ to $\Gamma^m$ such that $\vec\nu^m$ points into
$\Omega^m_+$.
We then partition the elements of the bulk mesh 
$\mathcal{T}^m$ into interior, exterior and interfacial elements as before, and
we introduce  
$\rho^m,\,\mu^m \in S^m_0(\Omega)$, for $m\geq 0$, as 
\begin{equation*} 
\rho^m\!\mid_{o^m} = \begin{cases}
\rho_- & o^m \in \mathcal{T}^m_-\,, \\
\rho_+ & o^m \in \mathcal{T}^m_+\,, \\
\tfrac12\,(\rho_- + \rho_+) & o^m \in \mathcal{T}^m_{\Gamma^m}\,,
\end{cases}
\quad\text{and}\quad
\mu^m\!\mid_{o^m} = \begin{cases}
\mu_- & o^m \in \mathcal{T}^m_-\,, \\
\mu_+ & o^m \in \mathcal{T}^m_+\,, \\
\tfrac12\,(\mu_- + \mu_+) & o^m \in \mathcal{T}^m_{\Gamma^m}\,.
\end{cases}
\end{equation*}

We also introduce the $L^2$--inner 
product $\langle\cdot,\cdot\rangle_{\Gamma^m}$ over
the current polyhedral surface $\Gamma^m$, as well as the 
the mass lumped inner product
$\langle\cdot,\cdot\rangle_{\Gamma^m}^h$.
We introduce, similarly to (\ref{eq:Psh},b), 
\begin{equation*} 
\mat{\mathcal{P}}_{\Gamma^m} = \mat\Id - \vec \nu^m \otimes \vec \nu^m
\quad\text{on}\ \Gamma^m\,,
\end{equation*}
and
\begin{equation*} 
\mat D_s^m(\vec \eta) = \tfrac12\,\mat{\mathcal{P}}_{\Gamma^m}\,
(\nabs\,\vec\eta + (\nabs\,\vec \eta)^T)\,\mat{\mathcal{P}}_{\Gamma^m}
\quad\text{on}\ \Gamma^m\,,
\end{equation*}
where here $\nabs = \mat{\mathcal{P}}_{\Gamma^m} \,\nabla$ 
denotes the surface gradient on $\Gamma^m$.

We introduce the following pushforward operator for the discrete interfaces
$\Gamma^m$ and $\Gamma^{m-1}$, for $m=0,\ldots,M$. Here we set
$\Gamma^{-1}:=\Gamma^0$. Let $\vec\Pi_{m-1}^m : [C(\Gamma^{m-1})]^d \to
\Vh$ such that
\begin{equation} \label{eq:Pi}
(\vec\Pi_{m-1}^m\,\vec z)(\vec q^m_k) = \vec z(\vec q^{m-1}_k)\,,
\qquad k = 1,\ldots,K_\Gamma\,,\qquad
\forall\ \vec z \in [C(\Gamma^{m-1})]^d\,,
\end{equation}
for $m=1,\ldots,M$, and set $\vec\Pi_{-1}^0 := \vec \pi^0$.
Analogously to (\ref{eq:Pi}) we also introduce 
$\Pi_{m-1}^m : C(\Gamma^{m-1}) \to \Wh$ and
$\mat \Pi_{m-1}^m : [C(\Gamma^{m-1})]^{d\times d} \to \matVh$.
We also introduce the short hand notations 
\begin{align}
\alpha^m = \pi^m\,[\alpha(\phaseC^m)]\,,\quad
\spont^m = \pi^m\,[\spont(\phaseC^m)]\,,\quad
\alpha^{G,m} = \pi^m\,[\alpha^G(\phaseC^m)]\,,
\label{eq:alpham}
\end{align}
for $m = 0,\ldots,M-1$.
We note, similarly to (\ref{eq:NIh}), that 
\begin{equation*} 
\left\langle \vec z, w\,\vec\nu^m\right\rangle_{\Gamma^m}^h =
\left\langle \vec z, w\,\vec\omega^m\right\rangle_{\Gamma^m}^h 
\qquad \forall\ \vec z \in \Vh\,,\ w \in \Wh \,, 
\end{equation*}
where
$\vec\omega^m := \sum_{k=1}^{K_\Gamma} \chi^m_k\,\vec\omega^m_k \in \Vh$,
and where for $k= 1 ,\ldots, K_\Gamma$ we let
$\Theta_k^m:= \{j : \vec{q}^m_k \in \overline{\sigma^m_j}\}$
and set
$\Lambda_k^m := \cup_{j \in \Theta_k^m} \overline{\sigma^m_j}$
and
$\vec\omega^m_k := \frac{1}{\mathcal{H}^{d-1}(\Lambda^m_k)}
\sum_{j\in \Theta_k^m} \mathcal{H}^{d-1}(\sigma^m_j)
\;\vec{\nu}^m_j$.

For the approximation to the velocity and pressure on ${\mathcal{T}}^m$
we use the finite element spaces
$\uspace^m(\vec g)$ and $\pspace^m$, which are the direct
time discrete analogues of $\uspace^h(\vec g)$ and $\pspace^h(t_m)$,
as well as $\widehat\pspace^m \subset \widehat\pspace$.

Analogously to (\ref{eq:LBBG}), we recall the following 
discrete LBB$_\Gamma$ inf-sup assumption from \cite{nsnsade}. 
Let there exist a $C_0 \in \R_{>0}$, independent of $\mathcal{T}^m$ and
$\{\sigma^m_j\}_{j=1}^{J_\Gamma}$, such that
\begin{equation} \label{eq:LBBGamma0}
\inf_{(\varphi,\lambda,\eta) \in \widehat\pspace^m \times \R \times \Wh} 
\sup_{\vec \xi \in \uspace^m(\vec 0)}
\frac{( \varphi, \nabla \,.\,\vec \xi) 
+ \lambda \left\langle \vec\omega^m, \vec \xi \right\rangle_{\Gamma^m}^h
+ \left\langle \eta, \nabs \,.\,(\vec \pi^m\,\vec \xi\!\mid_{\Gamma^m}) 
\right\rangle_{\Gamma^m}}
{(\|\varphi\|_0 + |\lambda| + \| \eta \|_{0,\Gamma^m})
\,(\|\vec \xi\|_1 + 
 \| \mat{\mathcal{P}}_{\Gamma^m}\,(\vec\pi^m\,\vec\xi\!\mid_{\Gamma^m}) 
 \|_{1,\Gamma^m,h})} \geq C_0\,,
\end{equation}
where 
$\| \eta \|_{0,\Gamma^m}^2 :=
\left\langle \eta,\eta \right\rangle_{\Gamma^m}$
and $\| \vec\eta \|_{1,\Gamma^m,h}^2 :=
\left\langle \vec\eta,\vec\eta \right\rangle_{\Gamma^m}
+ \sum_{j = 1}^{J_\Gamma} \int_{\sigma_j^m} |\nabs\,\vec\eta|^2 \dH{d-1}$.
See \citet[(5.2)]{nsnsade} for more details.

Our proposed fully discrete approximation is given as follows.
Let $\Gamma^0$, an approximation to $\Gamma(0)$, as well as
$\vec\kappa^0 \in \Vhz$, $\phaseC^0 \in \Whz$ and 
$\vec U^0\in \uspace^0(\vec g)$ be given.
For $m=0,\ldots, M-1$, find $\vec U^{m+1} \in \uspace^m(\vec g)$, 
$P^{m+1} \in \widehat\pspace^m$, $\Psing^{m+1} \in \R$, 
$P_\Gamma^{m+1} \in \Wh$,
$\vec X^{m+1}\in\Vh$,
$\vec\kappa^{m+1}\in\Vh$, $\mat W^{m+1} \in \matVh$ and
$\vec Y^{m+1},\,\vec F_\Gamma^{m+1} \in \Vh$ such that
\begin{subequations}
\begin{align}
&
\tfrac12 \left( \frac{\rho^m\,\vec U^{m+1} - (I^m_0\,\rho^{m-1})
\,\vec I^m_2\,\vec U^m}{\tau}
+(I^m_0\,\rho^{m-1}) \,\frac{\vec U^{m+1}- \vec I^m_2\,\vec{U}^m}{\tau}, \vec \xi 
\right)
 \nonumber \\ & \qquad
+ 2\left(\mu^m\,\mat D(\vec U^{m+1}), \mat D(\vec \xi) \right)
+ \tfrac12\left(\rho^m, 
 [(\vec I^m_2\,\vec U^m\,.\,\nabla)\,\vec U^{m+1}]\,.\,\vec \xi
- [(\vec I^m_2\,\vec U^m\,.\,\nabla)\,\vec \xi]\,.\,\vec U^{m+1} \right)
\nonumber \\ & \qquad
- \left(P^{m+1}, \nabla\,.\,\vec \xi\right)
- \Psing^{m+1} \left\langle\vec\omega^m, \vec \xi \right\rangle_{\Gamma^m}^h
+ \rho_\Gamma \left\langle \frac{\vec U^{m+1} - 
 \vec\Pi_{m-1}^m\,(\vec I^m_2\,\vec U^m)\!\mid_{\Gamma^{m-1}}}{\tau}, 
\vec \xi \right\rangle_{\Gamma^m}^h \nonumber \\ & \qquad
+ 2\,\mu_\Gamma \left\langle \mat D_s^m (\vec\pi^m\,\vec U^{m+1}) , 
\mat D_s^m ( \vec\pi^m\,\vec \xi ) \right\rangle_{\Gamma^m}
- \left\langle P_\Gamma^{m+1} , 
\nabs\,.\,(\vec\pi^m\,\vec \xi) \right\rangle_{\Gamma^m}
\nonumber \\ & \qquad\qquad
= \left(\rho^m\,\vec f^{m+1}, \vec \xi\right)
+ \left\langle \vec F_\Gamma^{m+1}, \vec \xi \right\rangle_{\Gamma^m}^h 
- \tfrac12\,\rho_+\left\langle \vec U^m \,.\,\unitn , \vec U^m\,.\,\vec\xi 
 \right\rangle_{\partial_2\Omega}
\qquad \forall\ \vec\xi \in \uspace^m(\vec 0) \,, \label{eq:GDa}\\
& \left(\nabla\,.\,\vec U^{m+1}, \varphi\right) = 0 
\qquad \forall\ \varphi \in \widehat\pspace^m
\qquad\text{and}\qquad
\left\langle\vec U^{m+1}, \vec \omega^m \right\rangle_{\Gamma^m}^h = 0\,,
\label{eq:GDb} \\
& \left\langle \nabs\,.\,(\vec\pi^m\,\vec U^{m+1}), \eta 
\right\rangle_{\Gamma^m}  = 0 
\qquad \forall\ \eta \in \Wh\,,
\label{eq:GDc} \\
& \left\langle \frac{\vec X^{m+1} - \vec\id}{\tau} ,
\vec\chi \right\rangle_{\Gamma^m}^h
= \left\langle \vec U^{m+1}, \vec\chi \right\rangle_{\Gamma^m}^h
 \qquad\forall\ \vec\chi \in \Vh\,,
\label{eq:GDd}
\end{align}
\end{subequations}
\begin{subequations}
\begin{align}
& \left\langle \vec\kappa^{m+1} , \vec\eta \right\rangle_{\Gamma^m}^h
+ \left\langle \nabs\,\vec X^{m+1}, \nabs\,\vec \eta \right\rangle_{\Gamma^m}
 = 0  \qquad\forall\ \vec\eta \in \Vh\,,\label{eq:GDe} \\
& \left\langle \mat W^{m+1} , \mat\zeta \right\rangle_{\Gamma^m}^h
+\tfrac12 \left\langle \vec\nu^m, 
[\mat\zeta + \mat\zeta^T] \,\vec\kappa^{m+1} 
+ \nabs\,.\,[\mat\zeta + \mat\zeta^T] \right\rangle_{\Gamma^m}^h =0 
\qquad \forall\ \mat\zeta \in \matVh \,, \label{eq:GDee} \\
& \left\langle \vec Y^{m+1} , \vec\xi \right\rangle_{\Gamma^m}^h
- \left\langle \alpha^m\,(\vec\kappa^{m+1} - \spont^m\,\vec\nu^m)
, \vec\xi \right\rangle_{\Gamma^m}^h
\nonumber \\ & \hspace{2cm}
- \left\langle \alpha^{G,m}\,(\vec\Pi_{m-1}^m\,\vec\kappa^m + 
\mat\Pi_{m-1}^m\,\mat W^m\,\vec\nu^m), \vec\xi \right\rangle_{\Gamma^m}^h
 = 0  
\qquad\forall\ \vec\xi \in \Vh\,,\label{eq:GDf} \\
& \left\langle \vec F_\Gamma^{m+1}, \vec\chi \right\rangle_{\Gamma^m}^h =
\left\langle \nabs\,\vec Y^{m+1} , \nabs\,\vec\chi \right\rangle_{\Gamma^m}
+ \left\langle \nabs\,.\,(\vec\Pi_{m-1}^m\,\vec Y^m), \nabs\,.\,\vec\chi 
\right\rangle_{\Gamma^m}
 \nonumber \\ & \quad
-\tfrac12 \left\langle 
\alpha^m\,|\vec\Pi_{m-1}^m\,\vec\kappa^m - \spont^m\,\vec\nu^m|^2
- 2\,\vec\Pi_{m-1}^m\,\vec Y^m\,.\,\vec\Pi_{m-1}^m\,\vec\kappa^m,
\nabs\,.\,\vec\chi \right\rangle_{\Gamma^m}^h
 \nonumber \\ & \quad
-2 \left\langle [\nabs\,(\vec\Pi_{m-1}^m\,\vec Y^m)]^T, 
 \mat D_s^m(\vec\chi)\,(\nabs\,\vec\id)^T \right\rangle_{\Gamma^m}
- \left\langle \alpha^m\,\spont^m\,\vec\Pi_{m-1}^m\,\vec\kappa^m , 
[\nabs\,\vec\chi]^T\,\vec\nu^m \right\rangle_{\Gamma^m}^h
 \nonumber \\ & \quad
- \beta 
\left\langle 
b_{CH}(\phaseC^m), \nabs\,.\,\vec\chi \right\rangle_{\Gamma^m}^h
+ \beta\,\gamma \left\langle (\nabs\,\phaseC^m)
\otimes(\nabs\,\phaseC^m), \nabs\,\vec\chi \right\rangle_{\Gamma^m} 
\nonumber \\ & \quad 
- \tfrac12 \left\langle \alpha^{G,m} \,( |\vec\Pi_{m-1}^m\,\vec\kappa^m|^2 + 
|\mat\Pi_{m-1}^m\,\mat W^m|^2 ),
\nabs\,.\,\vec\chi \right\rangle_{\Gamma^m}^h
\nonumber \\ & \quad 
+ \left\langle \vec\nu^m\,.\,(\mat Z^m\,\vec\Pi_{m-1}^m\,\vec\kappa^m 
+ \nabs\,.\,\mat Z^m),\nabs\,.\,\vec\chi \right\rangle_{\Gamma^m}^h
\nonumber \\ & \quad 
+ \sum_{i=1}^d 
\left\langle \nu^m_i\,\nabs\,\vec Z^m_i,\nabs\,\vec\chi
- 2\,\mat D^m_s(\vec\chi) \right\rangle_{\Gamma^m} 
- \left\langle \mat Z^m\,\vec\kappa^m + \nabs\,.\,\mat Z^m, 
[\nabs\,\vec\chi]^T\,\vec\nu^m \right\rangle_{\Gamma^m}^h 
\nonumber \\ & \hspace{11cm}
\quad \forall\ \vec\chi \in \Vh\,, \label{eq:GDg} 
\end{align}
\end{subequations}
and set $\Gamma^{m+1} = \vec X^{m+1}(\Gamma^m)$.
Here we have defined $\vec f^{m+1} := \vec I^m_2\,\vec f(\cdot,t_{m+1})$,
\linebreak
$\mat Z^m = \mat\pi^m[-\alpha^G(\phaseC^m)\,\mat\Pi_{m-1}^m\,\mat W^m]$ and
$\vec Z^m_i = \mat Z^m\,\vec\ek_i$, $i=1\to d$.
Having computed $\Gamma^{m+1}$,
find $\phaseC^{m+1} \in \WhpVI$ and $\Chempot^{m+1} \in \Whp$ such that
\begin{subequations}
\begin{align}
& \frac\vartheta{\tau}
\left\langle \phaseC^{m+1}, \chi^{m+1}_k \right\rangle_{\Gamma^{m+1}}^h
+ \left\langle \nabs\, \Chempot^{m+1}, \nabs\, \chi^{m+1}_k
\right\rangle_{\Gamma^{m+1}}
= \frac\vartheta{\tau}
\left\langle \phaseC^m, \chi^m_k \right\rangle_{\Gamma^m}^h 
\quad\forall\ k \in \{1,\ldots,K_\Gamma\}\,.
\label{eq:FDCHa} \\
& \beta\,\gamma \left\langle \nabs\, \phaseC^{m+1}, \nabs\, [\chi - \phaseC^{m+1}]
\right\rangle_{\Gamma^{m+1}}
 \geq \left\langle \Chempot^{m+1} + \beta\,\gamma^{-1}\,
\Pi_m^{m+1}\,\phaseC^m,\chi - \phaseC^{m+1}
\right\rangle_{\Gamma^{m+1}}^h
\nonumber \\ & \
- \tfrac12 \left\langle
\alpha'(\Pi_m^{m+1}\,\phaseC^m)\,|\vec\Pi_m^{m+1}\,\vec\kappa^{m+1} 
- \spont(\Pi_m^{m+1}\,\phaseC^m)\,\vec\nu^{m+1}|^2 
,\chi - \phaseC^{m+1}\right\rangle_{\Gamma^{m+1}}^h
\nonumber \\ & \
+\left\langle
\spont'(\Pi_m^{m+1}\,\phaseC^m)\,\alpha(\Pi_m^{m+1}\,\phaseC^m)\,(
\vec\Pi_m^{m+1}\,\vec\kappa^{m+1} 
- \spont(\Pi_m^{m+1}\,\phaseC^m)\,\vec\nu^{m+1})\,.\,
\vec\nu^{m+1},\chi - \phaseC^{m+1} \right\rangle_{\Gamma^{m+1}}^h
\nonumber \\ & \
- \tfrac12 \left\langle (\alpha^G)'(\Pi_m^{m+1}\,\phaseC^m) 
\,( |\vec\Pi_m^{m+1}\vec\kappa^{m+1}|^2 - |\mat\Pi_m^{m+1}\,\mat W^{m+1}|^2 ),
\chi - \phaseC^{m+1} \right\rangle_{\Gamma^{m+1}}^h 
\nonumber \\ & \hspace{10cm}
\quad\forall\ \chi \in \WhpVI\,. \label{eq:FDCHb} 
\end{align}
\end{subequations}
Here we note that (\ref{eq:FDCHa},b) is a fully discrete approximation of
(\ref{eq:CHha},b) for the obstacle potential (\ref{eq:obener}). 

In the absence of the LBB$_\Gamma$ condition 
(\ref{eq:LBBGamma0}) 
we need to consider the
reduced system (\ref{eq:GDa},d), (\ref{eq:GDe}--d), 
where $\uspace^m(\vec 0)$ in 
(\ref{eq:GDa}) is replaced by $\uspace^m_0(\vec 0)$. Here we define
\begin{align} \label{eq:uspace00}
\uspace^m_0(\vec a) := 
\left\{ \vec U \in \uspace^m(\vec a) : \right. & \left. 
(\nabla\,.\,\vec U, \varphi) = 0 \ \
\forall\ \varphi \in \widehat\pspace^m\,,
\left\langle \nabs\,.\,(\vec\pi^m\,\vec U), \eta 
\right\rangle_{\Gamma^m}  = 0 \ \ \forall\ \eta \in \Wh
\right. \nonumber \\ & \left. \text{ and } 
\left\langle\vec U, \vec \omega^m \right\rangle_{\Gamma^m}^h = 0
 \right\} \,,
\end{align}
for given data $\vec a \in [C(\overline\Omega)]^d$.

In order to prove the existence of a unique solution to (\ref{eq:GDa}--d),
(\ref{eq:GDe}--d) we make the following very mild well-posedness assumption.

\begin{itemize}
\item[$(\mathcal{A})$]
We assume for $m=0,\ldots, M-1$ that $\mathcal{H}^{d-1}(\sigma^m_j) > 0$ 
for all $j=1,\ldots, J_\Gamma$,
and that $\Gamma^m \subset \Omega$.
\end{itemize}

\begin{thm} \label{thm:GD}
Let the assumption $(\mathcal{A})$ hold.
If the LBB$_\Gamma$ condition {\rm (\ref{eq:LBBGamma0})} 
holds, then there exists a unique solution 
$(\vec U^{m+1}, P^{m+1}, \Psing^{m+1}, P_\Gamma^{m+1}, \vec X^{m+1}, 
\vec\kappa^{m+1}, \vec Y^{m+1}, \vec F_\Gamma^{m+1}, \mat W^{m+1})$ $
\in \uspace^m(\vec g)\times\widehat\pspace^m \times \R \times \Wh 
\times [\Vh]^4 \times \matVh$ 
to {\rm (\ref{eq:GDa}--d)}, {\rm (\ref{eq:GDe}--d)}. In all other cases,
on assuming that $\uspace^m_0(\vec g)$ is nonempty, 
there exists a unique solution 
$(\vec U^{m+1}, \vec X^{m+1},$ $\vec\kappa^{m+1}, \vec Y^{m+1},$ $ 
\vec F_\Gamma^{m+1}, \mat W^{m+1}) \in \uspace^m_0(\vec g) 
\times [\Vh]^4 \times \matVh$ to the
reduced system {\rm (\ref{eq:GDa},d)}, {\rm (\ref{eq:GDe}--d)} 
with $\uspace^m(\vec 0)$ replaced by $\uspace^m_0(\vec 0)$.
\end{thm}
\begin{proof}
As the system (\ref{eq:GDa}--d), (\ref{eq:GDe}--d) 
is linear, existence follows from uniqueness.
In order to establish the latter, we consider the homogeneous 
system.
Find $(\vec U, P, \Psing, P_\Gamma, \vec X, \vec\kappa, \vec Y, \vec F_\Gamma,$
$\mat W) \in \uspace^m(\vec 0)\times\widehat\pspace^m \times \R \times \Wh
\times [\Vh]^4 \times \matVh$ such that
\begin{subequations}
\begin{align}
&
\tfrac1{2\,\tau} \left( (\rho^m+I^m_0\,\rho^{m-1})\,\vec U, \vec \xi \right)
+ 2\left(\mu^m\,\mat D(\vec U), \mat D(\vec \xi) \right)
- \left(P, \nabla\,.\,\vec \xi\right)
- \Psing \left\langle\vec\omega^m, \vec \xi \right\rangle_{\Gamma^m}^h
\nonumber \\ & \qquad
+ \tfrac12\left(\rho^m, [(\vec I^m_2\,\vec U^m\,.\,\nabla)\,\vec U]\,.\,\vec \xi
- [(\vec I^m_2\,\vec U^m\,.\,\nabla)\,\vec \xi]\,.\,\vec U \right)
\nonumber \\ & \quad
+ \tfrac1{\tau}\, \rho_\Gamma \left\langle \vec U , \vec\xi
 \right\rangle_{\Gamma^m}^h
+ 2\,\mu_\Gamma \left\langle \mat D_s^m (\vec\pi^m\,\vec U) , 
\mat D_s^m ( \vec\pi^m\,\vec \xi ) \right\rangle_{\Gamma^m}
\nonumber \\ & \quad
- \left\langle P_\Gamma, 
\nabs\,.\, (\vec\pi^m\,\vec \xi) \right\rangle_{\Gamma^m}
- \left\langle \vec F_\Gamma, \vec \xi \right\rangle_{\Gamma^m}^h 
= 0 
 \qquad \forall\ \vec\xi \in \uspace^m(\vec 0) \,, \label{eq:proofa}\\
& \left(\nabla\,.\,\vec U, \varphi\right)  = 0 
\qquad \forall\ \varphi \in \widehat\pspace^m
\qquad\text{and}\qquad
\left\langle\vec U, \vec \omega^m \right\rangle_{\Gamma^m}^h = 0\,,
\label{eq:proofb} \\
& \left\langle \nabs\,.\,(\vec\pi^m\,\vec U), \eta 
\right\rangle_{\Gamma^m}  = 0 
\qquad \forall\ \eta \in \Wh\,,
\label{eq:proofc} \\
& \tfrac1{\tau} \left\langle \vec X,\vec\chi \right\rangle_{\Gamma^m}^h
= \left\langle \vec U, \vec\chi \right\rangle_{\Gamma^m}^h
 \qquad\forall\ \vec\chi \in \Vh\,,
\label{eq:proofd} \\
& \left\langle \vec\kappa , \vec\eta \right\rangle_{\Gamma^m}^{h}
+ \left\langle \nabs\,\vec X, \nabs\,\vec \eta \right\rangle_{\Gamma^m}
 = 0  \qquad\forall\ \vec\eta \in \Vh \,,\label{eq:proofe}  \\
& \left\langle \mat W , \mat\zeta \right\rangle_{\Gamma^m}^h
+ \tfrac12 \left\langle \vec\nu^m , 
[\mat\zeta + \mat\zeta^T]\,\vec\kappa \right\rangle_{\Gamma^m}^h =0 
\qquad \forall\ \mat\zeta \in \matVh \,, \label{eq:proofee} \\
& \left\langle \vec Y , \vec\eta \right\rangle_{\Gamma^m}^{h}
- \left\langle \alpha^m\,\vec\kappa , \vec \eta \right\rangle_{\Gamma^m}^h
 = 0  \qquad\forall\ \vec\eta \in \Vh \,,\label{eq:prooff}  \\
& \left\langle \vec F_\Gamma, \vec\chi \right\rangle_{\Gamma^m}^h -
\left\langle \nabs\,\vec Y , \nabs\,\vec\chi \right\rangle_{\Gamma^m}
= 0 \qquad \forall\ \vec\chi \in \Vh\,. \label{eq:proofg} 
\end{align}
\end{subequations}
Choosing $\vec\xi=\vec U$ in (\ref{eq:proofa}),
$\varphi = P$ in (\ref{eq:proofb}), $\eta = P_\Gamma$ in (\ref{eq:proofc}), 
$\vec\chi = \vec F_\Gamma$ in (\ref{eq:proofd}), 
$\vec\chi = \vec X$ in (\ref{eq:proofg}), 
$\vec\eta=\vec Y$ in (\ref{eq:proofe}) and
$\vec\eta=\vec \kappa$ in (\ref{eq:prooff})
yields that
\begin{align}
& \tfrac12\left((\rho^m + I^m_0\,\rho^{m-1})\,\vec U, \vec U \right) + 
2\,\tau\left(\mu^m\,\mat D(\vec U), \mat D(\vec U) \right)
+ \rho_\Gamma \left\langle \vec U , \vec U \right\rangle_{\Gamma^m}^h
\nonumber \\ & \qquad\quad
+ 2\,\tau\,\mu_\Gamma \left\langle \mat D_s^m (\vec\pi^m\,\vec U) , 
\mat D_s^m ( \vec\pi^m\,\vec U) \right\rangle_{\Gamma^m}
\nonumber \\ &\quad
= \tau \left\langle \vec F_\Gamma, \vec U \right\rangle_{\Gamma^m}^h
= \left\langle \vec F_\Gamma, \vec X \right\rangle_{\Gamma^m}^h
= \left\langle \nabs\,\vec Y, \nabs\,\vec X \right\rangle_{\Gamma^m}
= - \left\langle \vec\kappa, \vec Y \right\rangle_{\Gamma^m}^h
\nonumber \\ &\quad
= - \left\langle \alpha^m\,\vec\kappa, \vec\kappa \right\rangle_{\Gamma^m}^h
. \label{eq:proof2GD}
\end{align}
It immediately follows from (\ref{eq:proof2GD}), 
Korn's inequality and $\alpha^m > 0$,
that $\vec U = \vec 0 \in \uspace^m(\vec 0)$ and $\vec \kappa = \vec 0$.
(For the application of Korn's inequality we recall that
$\mathcal{H}^{d-1}(\partial_1\Omega) > 0$.)
Hence (\ref{eq:proofd},f,g,h) yield that $\vec X = \vec 0$, 
$\mat W = \mat 0$, $\vec Y = \vec 0$ and
$\vec F_\Gamma = \vec0$, respectively.
Finally, if (\ref{eq:LBBGamma0}) holds then
(\ref{eq:proofa}) with $\vec U = \vec 0$ 
and $\vec F_\Gamma = \vec0$ implies
that $P = 0 \in \widehat\pspace^m$, $\Psing = 0$ 
and $P_\Gamma = 0 \in \Wh$. 
This shows existence and uniqueness of 
$(\vec U^{m+1}, P^{m+1}, \Psing^{m+1},
P_\Gamma^{m+1}, \vec X^{m+1}, \vec\kappa^{m+1},$ $\vec Y^{m+1},
\vec F_\Gamma^{m+1}, \mat W^{m+1}) 
\in \uspace^m(\vec g)\times\widehat\pspace^m \times\R\times \Wh \times [\Vh]^4
\times \matVh$ to {\rm (\ref{eq:GDa}--d)}, {\rm (\ref{eq:GDe}--d)}.
The proof for the reduced system is very similar. The homogeneous system to
consider is (\ref{eq:proofa},d--h) with $\uspace^m(\vec 0)$ replaced by
$\uspace^m_0(\vec 0)$. As before, we infer that (\ref{eq:proof2GD}) holds,
which yields that $\vec U = \vec 0 \in \uspace^m_0(\vec 0)$, 
$\vec \kappa = \vec 0$, and hence $\vec X = \vec F_\Gamma = \vec Y = \vec0$.

In order to prove the existence of a unique solution to (\ref{eq:FDCHa},b), 
we adapt the argument in \cite{BloweyE92} for the Cahn--Hilliard equation with
obstacle potential on a bounded fixed domain in $\R^d$.
We introduce the discrete inverse surface Laplacian 
$\mathcal{G}^{m+1} : \Whpint \to \Whpint$ defined by
\begin{equation} \label{eq:mathcalG}
\left\langle \nabs\,\mathcal{G}^{m+1}\, v, \nabs\,\xi
\right\rangle_{\Gamma^{m+1}} = 
\left\langle v, \xi \right\rangle^h_{\Gamma^{m+1}} 
\qquad\forall\ \xi \in \Whpint\,,
\end{equation}
where $\Whpint := 
\{ \xi \in \Whp : \langle \xi, 1 \rangle_{\Gamma^{m+1}} = 0\}$.
It immediately follows from 
$\langle \nabs\,v, \nabs\,v\rangle_{\Gamma^{m+1}} = 0 \Rightarrow
v = 0$ for all $v \in \Whpint$ that $\mathcal{G}^{m+1}$ is well-posed.
Next we rewrite (\ref{eq:FDCHa},b) as
\begin{subequations}
\begin{align}
& \frac\vartheta{\tau}
\left\langle \phaseC^{m+1}- \widehat \phaseC^m, 
\chi^{m+1}_k \right\rangle_{\Gamma^{m+1}}^h
+ \left\langle \nabs\, \Chempot^{m+1}, \nabs\, \chi^{m+1}_k
\right\rangle_{\Gamma^{m+1}}
= 0\quad\forall\ k \in \{1,\ldots,K_\Gamma\}\,.
\label{eq:proofCHa} \\
& \beta\,\gamma \left\langle \nabs\, \phaseC^{m+1}, \nabs\, [\chi - \phaseC^{m+1}]
\right\rangle_{\Gamma^{m+1}}
 \geq \left\langle \Chempot^{m+1} + g,\chi - \phaseC^{m+1}
\right\rangle_{\Gamma^{m+1}}^h
\quad\forall\ \chi \in \WhpVI\,, \label{eq:proofCHb} 
\end{align}
\end{subequations}
where $\widehat \phaseC^m \in \Whp$ is such that
$\langle \widehat \phaseC^m, 
\chi^{m+1}_k \rangle_{\Gamma^{m+1}}^h =
\langle \phaseC^m, 
\chi^m_k\rangle_{\Gamma^m}^h$ for $k\in\{1,\ldots,K_\Gamma\}$.
We note that
\begin{equation} \label{eq:Cint}
\left\langle \phaseC^{m+1}, 1 \right\rangle_{\Gamma^{m+1}}
= \left\langle \widehat \phaseC^m, 1 \right\rangle_{\Gamma^{m+1}}
= \left\langle \phaseC^m, 1 \right\rangle_{\Gamma^m}\,.
\end{equation}
It follows from (\ref{eq:Cint}), (\ref{eq:proofCHa}) and (\ref{eq:mathcalG}) 
that
\begin{equation} \label{eq:proofCP}
\Chempot^{m+1} = - \frac\vartheta\tau\,\mathcal{G}^{m+1}\,
(\phaseC^{m+1}- \widehat \phaseC^m) + \lambda^{m+1}\,,
\end{equation}
where $\lambda^{m+1} \in \R$ is a Lagrange multiplier associated with the
constraint (\ref{eq:Cint}). 
Hence $\phaseC^{m+1} \in \WhpVI$ is such that 
$\langle \phaseC^{m+1}, 1 \rangle_{\Gamma^{m+1}}
= \langle \phaseC^m, 1 \rangle_{\Gamma^m}$ and
\begin{align}
& \beta\,\gamma \left\langle \nabs\, \phaseC^{m+1}, \nabs\, [\chi - \phaseC^{m+1}]
\right\rangle_{\Gamma^{m+1}}
+ \frac\vartheta\tau
 \left\langle \mathcal{G}^{m+1}\,(\phaseC^{m+1}- \widehat \phaseC^m) 
-\lambda^{m+1} - g,\chi - \phaseC^{m+1} \right\rangle_{\Gamma^{m+1}}^h 
\nonumber \\ & \hspace{10cm}
\geq 0
\quad\forall\ \chi \in \WhpVI\,.
\label{eq:proofC}
\end{align}
Clearly, (\ref{eq:proofC}) is the Euler--Lagrange variational inequality 
for the strictly convex minimization problem
\begin{align}
\min_{\substack{\chi \in \WhpVI \\
\langle \chi, 1 \rangle_{\Gamma^{m+1}} = \langle \phaseC^m, 1 \rangle_{\Gamma^m}}}
& 
\left[
\frac{\beta\,\gamma}2 \left\langle \nabs\, \chi, \nabs\,\chi
\right\rangle_{\Gamma^{m+1}}
+ \frac{\vartheta}{2\,\tau} \left\langle 
\nabs\, \mathcal{G}^{m+1}\,(\chi - \widehat \phaseC^m) ,
\nabs\, \mathcal{G}^{m+1}\,(\chi - \widehat \phaseC^m)
\right\rangle_{\Gamma^{m+1}} \right. \nonumber \\ & \qquad\qquad \left.
- \frac\vartheta\tau
\left\langle g,\chi \right\rangle_{\Gamma^{m+1}}^h \right] .
\label{eq:proofmin}
\end{align}
Hence there exists a unique $\phaseC^{m+1} \in \WhpVI$ with
$\langle \phaseC^{m+1}, 1 \rangle_{\Gamma^{m+1}}
= \langle \phaseC^m, 1 \rangle_{\Gamma^m}$ 
and solving (\ref{eq:proofC}). 
Existence of the Lagrange multiplier $\lambda^{m+1}$ in (\ref{eq:proofCP}) 
then follows from a fixed point argument, see
\citet[p.\ 151]{BloweyE92}.
\end{proof}

\setcounter{equation}{0}
\section{Solution methods} \label{sec:6}
In this section we briefly describe possible solution methods for the linear
system \mbox{(\ref{eq:GDa}--d)}, (\ref{eq:GDe}--d), 
where we note that (\ref{eq:GDee}) decouples from the
remaining equations, and for the nonlinear system (\ref{eq:FDCHa},b).

In order to derive the linear system of equations for the coefficient vectors
of the finite element functions $(\vec U^{m+1},P^{m+1}, \Psing^{m+1}, 
P_\Gamma^{m+1}, \delta \vec X^{m+1}, \vec \kappa^{m+1}, \vec Y^{m+1},
\vec F^{m+1}_\Gamma)$ corresponding to (\ref{eq:GDa}--d), (\ref{eq:GDe},c,d),
where $\delta \vec X^{m+1} = \vec X^{m+1} - \vec\id\!\mid_{\Gamma^m}$,
we begin by introducing the following
matrices and vectors, where we closely follow our previous work in 
\cite{nsns}.
Let $i,\,j = 1 ,\ldots, K_\uspace^m$,
$n,q = 1 ,\ldots, K_\pspace^m$ and $k,l = 1 ,\ldots, K_\Gamma$. Then
\begin{align}
& [\vec B_\Omega]_{ij} := 
\left( \tfrac{\rho^m+I^m_0\,\rho^{m-1}}{2\,\tau}\, 
\phi_j^{\uspace^m} , \phi_i^{\uspace^m} 
\right)\,\mat\Id
+ 2\left(\left(\mu^m\,\mat D(\phi_j^{\uspace^m} \vec \ek_r), 
\mat D( \phi_i^{\uspace^m}\,\vec \ek_s) \right) \right)_{r,s=1}^d 
\nonumber \\ & \qquad\qquad
+ \tfrac12\left(\rho^m, [(\vec I^m_2\,\vec U^m\,.\,\nabla)\,\phi_j^{\uspace^m}
]\,\phi_i^{\uspace^m} - 
[(\vec I^m_2\,\vec U^m\,.\,\nabla)\,\phi_i^{\uspace^m}]
\,\phi_j^{\uspace^m} \right)\,\mat\Id \,,
\nonumber \\ & \qquad\qquad
+\frac{\rho_\Gamma}\tau
 \left\langle \varphi^{\uspace^m}_j, \varphi^{\uspace^m}_i
 \right\rangle_{\Gamma^m}^h\,\mat\Id
+ 2\mu_\Gamma \left(\left\langle \mat D^m_s (\pi^m\,\phi_j^{\uspace^m}\,
\vec \ek_r), \mat D^m_s(\pi^m\,\phi_i^{\uspace^m}\,\vec \ek_s) 
\right\rangle_{\Gamma^m} \right)_{r,s=1}^d \nonumber \\
& [\vec C_\Omega]_{iq} := - \left( \phi_q^{\pspace^m} ,
\left(\nabla\,.\,(\phi_i^{\uspace^m}\,\vec \ek_r)
\right) \right)_{r=1}^d,\qquad
[\Sbulk]_{il} :=  - \left(
\left\langle \chi^m_l, \nabs\,.\,(\pi^m\,\phi_i^{\uspace^m}\,\vec\ek_r)  
\right\rangle_{\Gamma^m} \right)_{r=1}^d \,, \nonumber \\
& \vec b_i = \left( \tfrac{I^m_0\,\rho^{m-1}}{\tau}\,\vec I^m_2\,\vec U^m + 
\rho^m\,\vec f^{m+1}, \phi_i^{\uspace^m}\right) 
+ \frac{\rho_\Gamma}\tau
 \left\langle \vec\Pi^m_{m-1}\,\vec U^m\!\mid_{\Gamma^{m-1}}, 
\varphi^{\uspace^m}_i  \right\rangle_{\Gamma^m}^h \nonumber \\ & \qquad
- \tfrac12\,\rho_+\left\langle (\vec U^m \,.\,\unitn)\,\vec U^m, 
 \phi_i^{\uspace^m} \right\rangle_{\partial_2\Omega}
\,;
\label{eq:mats}
\end{align}
where 
$\{\vec \ek_r\}_{r=1}^d$ denotes the standard basis in $\R^d$,
and where we have used the convention that the subscripts in the matrix
notations refer to the test and trial domains, respectively. 
A single subscript is used where the two domains are the same.
The entries of
$\vec D_\Omega$, for $i = 1 ,\ldots, K_\uspace^m$, are given by
$[\vec D_\Omega]_{i,1} := - \langle \phi_i^{\uspace^m}, 
\vec \omega^m \rangle_{\Gamma^m}^h$.

In order to provide a matrix-vector formulation for the full system 
(\ref{eq:GDa}--d), (\ref{eq:GDe},c,d), 
and in particular in view of (\ref{eq:GDf}), 
we recall from \citet[p.~64]{Dziuk08} that
\begin{align*} 
& 2\left\langle (\nabs\,\vec\xi)^T , \mat D_s^m(\vec\chi)\,
 (\nabs\,\vec\id)^T \right\rangle_{\Gamma^m}
\nonumber \\ & \
 = \sum_{i,j=1}^d \left\langle (\nabs)_j\,(\vec\xi)_i, (\nabs)_i\,(\vec\chi)_j 
 \right\rangle_{\Gamma^m}
- \sum_{i,j=1}^d \left\langle (\vec\nu^m)_i\,(\vec\nu^m)_j\, 
 \nabs\,(\vec\xi)_j, \nabs\,(\vec\chi)_i \right\rangle_{\Gamma^m}
\nonumber \\ & \hspace{2cm}
+ \left\langle \nabs\,\vec{\xi} , \nabs\,\vec\chi \right\rangle_{\Gamma^m}
\nonumber \\ & \
= \sum_{i,j=1}^d \left\langle (\nabs)_j\,(\vec\xi)_i, (\nabs)_i\,(\vec\chi)_j 
\right\rangle_{\Gamma^m}
+ \sum_{i,j=1}^d \left\langle (\delta_{ij} - (\vec\nu^m)_i\,(\vec\nu^m)_j)\, 
\nabs\,(\vec\xi)_j, \nabs\,(\vec\chi)_i \right\rangle_{\Gamma^m} .
\end{align*}
Moreover, we observe that
$\langle \nabs\,.\,\vec{\xi} , \nabs\,.\,\vec\chi \rangle_{\Gamma^m}
= \sum_{i,j=1}^d$ $\langle (\nabs)_j\,(\vec\xi)_j, (\nabs)_i\,(\vec\chi)_i 
\rangle_{\Gamma^m}$. Hence, in addition to (\ref{eq:mats}), we introduce the
following matrices and vectors, where $q = 1 ,\ldots, K_\uspace^m$,
and $k,l = 1 ,\ldots, K_\Gamma$
\begin{align*}
& [\vec{\mathcal{B}}_\Gamma]_{kl} := 
\left( \left\langle [\nabs]_j\,\chi^m_l, [\nabs]_i\,\chi^m_k 
\right\rangle_{\Gamma^m} \right)_{i,j=1}^d\,, \qquad
[\vec{\mathcal{R}}_\Gamma]_{kl} := \left\langle 
\nabs\,\chi_l^m\,.\,\nabs\,\chi_k^m , \mat\Id - \vec\nu^m \otimes \vec\nu^m
\right\rangle_{\Gamma^m} ,\nonumber \\
 &
[\Mbulk]_{ql} :=  
\left\langle \chi^m_l , \phi_q^{\uspace^m} \right\rangle_{\Gamma^m}\,\mat\Id 
\,, \qquad
[\vec{M}_\Gamma]_{kl} := 
\left\langle \chi^m_l, \chi^m_k \right\rangle_{\Gamma^m}^{h}
\,\mat\Id \,,  \nonumber \\ &
[\vec{M}_{\Gamma,\alpha}]_{kl} := 
\left\langle \alpha^m\,\chi^m_l, \chi^m_k \right\rangle_{\Gamma^m}^{h}
\,\mat\Id \,,  \qquad
[A_\Gamma]_{kl} := 
\left\langle \nabs\,\chi^m_l, \nabs\,\chi^m_k \right\rangle_{\Gamma^m}
\,, \qquad
[\vec{A}_\Gamma]_{kl} := [A_\Gamma]_{kl}\,\mat\Id \,, \nonumber \\ &
\vec c_k := - \left\langle \alpha^m\,\spont^m\,\vec\nu^m, 
\chi^m_k \right\rangle_{\Gamma^m}^{h}
+ \left\langle \alpha^{G,m}\,(\vec\Pi_{m-1}^m\,\vec\kappa^m 
+ \mat\Pi_{m-1}^m\,\mat W^m\,\vec\nu^m), 
\chi^m_k \right\rangle_{\Gamma^m}^{h}, \nonumber \\ &
[\vec d_\alpha]_k := \left\langle \alpha^m\,\spont^m, 
(\vec\Pi_{m-1}^m\,\vec\kappa^m\,.\, \nabs\,\chi_k^m)\, \vec\nu^m
\right\rangle_{\Gamma^m}^h , \nonumber \\ &
 [\vec d_\kappa]_{k} := \tfrac12\left\langle 
\alpha^m\,|\vec\Pi_{m-1}^m\,\vec\kappa^m - \spont^m\,\vec\nu^m|^2
- 2\,\vec\Pi_{m-1}^m\,\vec Y^m\,.\,\vec\Pi_{m-1}^m\,\vec\kappa^m 
, \nabs\,\chi_k^m \right\rangle_{\Gamma^m}^h \,,\nonumber \\ & 
[\vec d_\beta]_k := 
\beta \left\langle 
b_{CH}(\phaseC^m), \nabs\,\chi^m_k \right\rangle_{\Gamma^m}^h
- \beta\,\gamma \left( 
\left\langle (\nabs\,\phaseC^m)\otimes(\nabs\,\phaseC^m), 
\vec \ek_r \otimes \nabs\,\chi^m_k 
\right\rangle_{\Gamma^m} \right)_{r = 1}^d \nonumber \\ & \qquad = 
\beta \left\langle 
b_{CH}(\phaseC^m), \nabs\,\chi^m_k \right\rangle_{\Gamma^m}^h
- \beta\,\gamma \left\langle \nabs\,\phaseC^m\,.\,\nabs\,\chi^m_k, \nabs\,\phaseC^m 
\right\rangle_{\Gamma^m} , \nonumber \\ &
 [\vec d_G]_{k} := \tfrac12\left\langle 
\alpha^{G,m}\,(|\vec\Pi_{m-1}^m\,\vec\kappa^m|^2 
+ |\mat\Pi_{m-1}^m\,\mat W^m|^2), \nabs\,\chi_k^m \right\rangle_{\Gamma^m}^h 
, \nonumber \\ &
 [\vec d_Z]_{k} := 
\left\langle (\mat Z^m\,\vec\Pi_{m-1}^m\,\vec\kappa^m + \nabs\,.\,\mat Z^m)\,.\,
 \nabs\,\chi_k^m, \vec\nu^m \right\rangle_{\Gamma^m}^h
- \left\langle (\mat Z^m\,\vec\Pi_{m-1}^m\,\vec\kappa^m
 + \nabs\,.\,\mat Z^m)\,.\,\vec\nu^m,
\nabs\,\chi_k^m \right\rangle_{\Gamma^m}^h \nonumber \\ & \qquad\
- \sum_{i=1}^d \left(
\left\langle 
\nu^m_i\,\nabs\,\vec Z^m_i, \nu^m_r\,[\vec\nu^m \otimes \nabs\,\chi^m_k]
- \nabs\,\chi^m_k \otimes \vec\ek_r
 \right\rangle_{\Gamma^m} \right)_{r=1}^d . 
\end{align*}
Here we have made use of the facts that
\begin{align*}
[\vec{\mathcal{B}}_\Gamma]_{kl} & =
\left(
\left\langle \nabs\,.\,(\chi^m_l\,\vec\ek_j) , 
\nabs\,.\,(\chi^m_k\,\vec\ek_i) \right\rangle_{\Gamma^m}\right)_{i,j=1}^d
= \left(
\left\langle (\nabs\,\chi^m_l)\,.\,\vec\ek_j , 
(\nabs\,\chi^m_k)\,.\,\vec\ek_i \right\rangle_{\Gamma^m}\right)_{i,j=1}^d
\nonumber \\ & 
= \left( \left\langle [\nabs]_j\,\chi^m_l, [\nabs]_i\,\chi^m_k 
\right\rangle_{\Gamma^m} \right)_{i,j=1}^d 
\end{align*}
and that
\begin{align*}
& \left(
\left\langle 
\nu^m_i\,\nabs\,\vec Z^m_i, \vec\ek_r \otimes \nabs\,\chi^m_k
- \mat{\mathcal{P}}_{\Gamma^m}\,[\vec\ek_r \otimes \nabs\,\chi^m_k]
- [\nabs\,\chi^m_k \otimes \vec\ek_r]\, \mat{\mathcal{P}}_{\Gamma^m}
 \right\rangle_{\Gamma^m} \right)_{r=1}^d \nonumber \\ & \
= \left(
\left\langle 
\nu^m_i\,\nabs\,\vec Z^m_i, 
[\vec\nu^m\otimes\vec\nu^m]\,[\vec\ek_r \otimes \nabs\,\chi^m_k]
- \nabs\,\chi^m_k \otimes \vec\ek_r
+ [\nabs\,\chi^m_k \otimes \vec\ek_r]\, [\vec\nu^m\otimes\vec\nu^m]
 \right\rangle_{\Gamma^m} \right)_{r=1}^d \nonumber \\ & \
= \left(
\left\langle 
\nu^m_i\,\nabs\,\vec Z^m_i, \nu^m_r\,[\vec\nu^m \otimes \nabs\,\chi^m_k]
- \nabs\,\chi^m_k \otimes \vec\ek_r
+ \nu^m_r\,[\nabs\,\chi^m_k \otimes \vec\nu^m]
 \right\rangle_{\Gamma^m} \right)_{r=1}^d
\nonumber \\ & \
= \left(
\left\langle 
\nu^m_i\,\nabs\,\vec Z^m_i, \nu^m_r\,[\vec\nu^m \otimes \nabs\,\chi^m_k]
- \nabs\,\chi^m_k \otimes \vec\ek_r
 \right\rangle_{\Gamma^m} \right)_{r=1}^d
\end{align*}
for $i = 1,\ldots,d$, on noting that 
$\nabs\,\vec Z^m_i : [\nabs\,\chi^m_k \otimes \vec\nu^m] =
[(\nabs\,\vec Z^m_i)\,\vec\nu^m]\,.\,\nabs\,\chi^m_k = 
\vec 0\,.\,\nabs\,\chi^m_k = 0$.
Moreover, it clearly holds that $([\vec{\mathcal{B}}_\Gamma]_{kl})^T = 
[\vec{\mathcal{B}}_\Gamma]_{lk} =: [\vec{\mathcal{B}}^\star_\Gamma]_{kl}$.

Denoting the system matrix
\[
\begin{pmatrix}
 \vec B_\Omega & \vec C_\Omega & \vec D_\Omega & \Sbulk \\
 \vec C^T_\Omega & 0 & 0 & 0 \\
 \vec D^T_\Omega & 0 & 0 & 0 \\
 \SbulkT &  0 & 0 & 0 
\end{pmatrix} 
\]
as {\scriptsize$\begin{pmatrix}
\vec B_\Omega & \vec{\mathcal{C}} \\
\vec{\mathcal{C}}^T & 0 \end{pmatrix}$}, 
and letting $\widetilde P^{m+1} = (P^{m+1}, \Psing^{m+1},$ $P_\Gamma^{m+1})^T$, 
then the linear system (\ref{eq:GDa}--d), (\ref{eq:GDe},c,d) can be written as
\begin{equation}
\begin{pmatrix}
 \vec B_\Omega & \vec{\mathcal{C}} & 0 & 0 & 0 & -\Mbulk \\
 \vec {\mathcal{C}}^T & 0 & 0 & 0 & 0 & 0\\
 \MbulkT & 0 & 0 & \!\!\!-\frac1{\tau}\,\vec{M}_\Gamma\!\! & 0 & 0\\
0 & 0 & \vec{M}_\Gamma & \vec{A}_\Gamma & 0 & 0\\
0 & 0 & -\vec{M}_{\Gamma,\alpha} & 0 & \vec{M}_\Gamma & 0 \\
0 & 0 & 0 & 0 & -\vec A_\Gamma & \vec M_\Gamma 
\end{pmatrix} 
\begin{pmatrix} \vec U^{m+1} \\ \widetilde P^{m+1} \\ \vec\kappa^{m+1} \\ 
\delta\vec{X}^{m+1} \\ \vec Y^{m+1} \\ \vec F^{m+1}_\Gamma \end{pmatrix}
=
\begin{pmatrix} \vec b \\ 0 \\
0 \\ -\vec{A}_\Gamma\,\vec{X}^m \\ 
\vec c \\
 \vec{\mathcal{Z}}_\Gamma\,\vec Y^m - \vec d \end{pmatrix} ,
\label{eq:lin}
\end{equation}
where $\vec {\mathcal{Z}}_\Gamma := \vec{\mathcal{B}}_\Gamma - 
\vec{\mathcal{B}}_\Gamma^\star - \vec {\mathcal{R}}_\Gamma$
and
$\vec d =  \vec d_{\kappa} + \vec d_\alpha + \vec d_\beta + \vec d_G$.
For the solution of (\ref{eq:lin}) a Schur complement approach similar to
\cite{nsns} can be used. In particular, the Schur approach for eliminating 
$(\vec\kappa^{m+1},\delta \vec X^{m+1},\vec Y^{m+1},\vec F^{m+1}_\Gamma)$ from 
(\ref{eq:lin}) can be obtained as follows. Let 
\begin{equation*} 
\Theta_\Gamma:= \begin{pmatrix}
 0 & - \frac1{\tau}\,\vec{M}_\Gamma & 0 &  0\\
\vec{M}_\Gamma & \vec{A}_\Gamma & 0 & 0 \\
-\vec{M}_{\Gamma,\alpha} & 0 & \vec{M}_\Gamma & 0 \\
 0 & 0 & -\vec A_\Gamma & \vec M_\Gamma 
\end{pmatrix} \,.
\end{equation*}
Then (\ref{eq:lin}) can be reduced to
\begin{subequations}
\begin{align} \label{eq:schur}
&
\begin{pmatrix}
\vec B_\Omega + \alpha\,\vec T_\Omega
& \vec {\mathcal{C}} \\
\vec {\mathcal{C}}^T & 0 
\end{pmatrix}
\begin{pmatrix}
\vec U^{m+1} \\ \widetilde P^{m+1} 
\end{pmatrix}
= \begin{pmatrix}
\vec b + \alpha\,\vec g
 \\
0
\end{pmatrix}
\end{align}
and
\begin{equation} \label{eq:schurb}
\begin{pmatrix}
\vec\kappa^{m+1} \\ \delta\vec{X}^{m+1} \\ \vec Y^{m+1} \\ \vec F^{m+1}_\Gamma
\end{pmatrix}
 = \Theta_\Gamma^{-1}\,
\begin{pmatrix}
-\MbulkT\,\vec U^{m+1} \\ -\vec{A}_\Gamma\,\vec{X}^m \\ \vec c \\
\vec{\mathcal{Z}}_\Gamma\,\vec Y^m - \vec d
\end{pmatrix}
\,.
\end{equation}
\end{subequations}
In (\ref{eq:schur}) we have used the definitions
$$\vec T_\Omega = (0\ 0\ 0\ \Mbulk)\,\Theta_\Gamma^{-1}\, 
{\scriptsize \begin{pmatrix} \MbulkT \\ 0 \\ 0 \\ 0 \end{pmatrix}}
= \tau\,\Mbulk\,\vec M_\Gamma^{-1}\,\vec A_\Gamma\,\vec M_\Gamma^{-1}\,
\vec M_{\Gamma,\alpha}\,\vec M_\Gamma^{-1}\,
\vec A_\Gamma\,\vec M_\Gamma^{-1}\,\MbulkT$$
and 
\begin{align*}
\vec g & = (0\ 0\ 0\ \Mbulk)\,\Theta_\Gamma^{-1}\,
{\scriptsize \begin{pmatrix}
0 \\ - \vec{A}_\Gamma\,\vec{X}^m \\ \vec c \\
\vec{\mathcal{Z}}_\Gamma\,\vec Y^m - \vec d
\end{pmatrix}} .
\end{align*}
For the linear system (\ref{eq:schur})
well-known solution methods for finite element discretizations for the 
standard Navier--Stokes equations may be employed. We refer to 
\citet[\S5]{fluidfbp}, where we describe such solution methods in 
detail for a very similar situation.

The nonlinear system of algebraic equations arising from the discrete 
surface Cahn--Hilliard equation (\ref{eq:FDCHa},b) can be solved in the same
way that such variational inequalities for standard Cahn--Hilliard equations
are solved. In practice we employ the projection Gauss--Seidel method from
\cite{voids}, or the Uzawa-type iteration from \cite{vch}. 

\setcounter{equation}{0}
\section{Numerical results} \label{sec:7}

We implemented the scheme (\ref{eq:GDa}--d), (\ref{eq:GDe}--d), 
(\ref{eq:FDCHa},b) 
with the help of the finite element toolbox ALBERTA, see \cite{Alberta}. 
For the bulk mesh adaptation in our numerical computations 
we use the strategy from \cite{fluidfbp}, which results in a fine mesh
around $\Gamma^m$ and a coarse mesh further away from it. 

Given the initial triangulation $\Gamma^0$ and $\phaseC^0 \in \Whz$, 
with $\phaseC^0 \in [-1,1]$,
the initial data $\vec Y^0 \in \Vhz$, $\vec\kappa^0 \in \Vhz$ 
and $\mat W^0 \in \matVhz$ are always computed as 
\begin{equation*} 
 \left\langle \vec Y^{0} , \vec\eta \right\rangle_{\Gamma^0}^{h}
= \left\langle \alpha(\phaseC^0)\,(\vec\kappa^0 - \spont(\phaseC^0)\,\vec\nu^0)
- \alpha^G(\phaseC^0)\,(\vec\kappa^0 + \mat W^0\,\vec\nu^0), 
\vec\eta \right\rangle_{\Gamma^0}^h
\qquad\forall\ \vec\eta \in \Vhz\,,
\end{equation*}
where $\vec\kappa^0 \in \Vhz$ is the solution to
\begin{equation*}
 \left\langle \vec\kappa^{0} , \vec\eta \right\rangle_{\Gamma^0}^{h}
+ \left\langle \nabs\,\vec\id, \nabs\,\vec \eta \right\rangle_{\Gamma^0}
 = 0  \qquad\forall\ \vec\eta \in \Vhz\,,
\end{equation*}
and where $\mat W^0 \in \matVhz$ is the solution to
\[
 \left\langle \mat W^0 , \mat\zeta \right\rangle_{\Gamma^0}^h
+\tfrac12 \left\langle \vec\nu^0, 
[\mat\zeta + \mat\zeta^T] \,\vec\kappa^0
+ \nabs\,.\,[\mat\zeta + \mat\zeta^T] \right\rangle_{\Gamma^0}^h =0 
\qquad \forall\ \mat\zeta \in \matVhz \,.
\]

Throughout this section we set
\begin{subequations}
\begin{align} 
\alpha(s) & = \alpha_L(s) := \tfrac12\, (\alpha_+ + \alpha_-) 
+ \tfrac12\, (\alpha_+ - \alpha_-) \,s\,, \label{eq:alpha} \\
\spont(s) & = \tfrac12\, (\spont_+ + \spont_-) 
+ \tfrac12\, (\spont_+ - \spont_-) \,s\,, \label{eq:spont} \\
\alpha^G(s) & = \tfrac12\, (\alpha^G_+ + \alpha^G_-) 
+ \tfrac12\, (\alpha^G_+ - \alpha^G_-) \,s\,. \label{eq:alphaG}
\end{align}
\end{subequations}
We recall from the discussion around (\ref{eq:GB}) that it follows from
(\ref{eq:alphaG}), (\ref{eq:E}) and (\ref{eq:GB}) that 
only the difference
$(\alpha^G_+ - \alpha^G_-)$ plays a role in the evolutions with Gaussian
curvature. Moreover, for the choices (\ref{eq:alpha},c) the constraint
(\ref{eq:alphaGbound}) reduces to
\begin{equation} \label{eq:alphaGbound2}
\min\{\alpha_-,\alpha_+\} \geq \tfrac12\,|\alpha^G_+ - \alpha^G_-|\,.
\end{equation}
Unless otherwise stated, we use $\rho_\pm = 0$, $\mu_\pm = 1$, 
$\mu_\Gamma = 1$, $\rho_\Gamma = 0$, $\alpha_\pm = 1$, $\spont_\pm = 0$
and $\alpha^G_\pm = 0$.
Moreover, we normally use $\vartheta = \beta = 1$.

At times we will discuss the discrete energy of the numerical solutions. On
recalling Theorem~\ref{thm:stab} and (\ref{eq:alpham}), 
the discrete energy is defined by
$$
\mathcal{E}^h_{total} = 
\mathcal{E}_{kin}^h + \mathcal{E}^h_{\kappa} + \mathcal{E}^h_{CH}\,,
$$
where
\begin{align*}
\mathcal{E}_{kin}^h & 
= \tfrac12\, \|[\rho^m]^\frac12\,\vec U^{m+1} \|_0^2 
+ \tfrac12\, \rho_\Gamma \left\langle \vec U^{m+1}, \vec U^{m+1} 
\right\rangle_{\Gamma^m}^h, \\
\mathcal{E}_\kappa^h & 
= \tfrac12 \left\langle \alpha^m, |\vec\kappa^{m+1} - 
\spont^m\,\vec\nu^m|^2 \right\rangle_{\Gamma^m}^h
+ \tfrac12 \left\langle \alpha^{G,m}, |\vec\kappa^{m+1}|^2 - |\mat W^{m+1}|^2
\right\rangle_{\Gamma^m}^h, \nonumber \\ 
\mathcal{E}_{CH}^h & 
= \beta\,\left\langle b_{CH}(\phaseC^m), 1 \right\rangle_{\Gamma^m}^h, 
\end{align*}
represent the kinetic, curvature and Cahn--Hilliard parts of the 
discrete energy.

In plots where we show the concentration $\phaseC^m$ in grey scales, 
the colour scales linearly with $\phaseC^m$ 
ranging from -1 (white) to 1 (black).

\subsection{Numerical simulations in 2d}
We start with an initial shape in the form of a smooth letter ``C''.
The curve has length $2.823$ and we use $257$ elements on it.
For our choice of $\gamma = 0.02$ this 
yields on average about $6$ elements across the
interface, which asymptotically has thickness $\gamma\,\pi$. 
The time step size is $\tau=5\times10^{-4}$.
For the computational domain we choose $\Omega = (-1,1)^2$,
and we choose a random distribution for $\phaseC^0$ with mean value $-0.4$.
An experiment for $\spont_- = -\tfrac12$ and
$\spont_+ = -2$ is shown in Figures~\ref{fig:fig6}. We observe that due to the
choice of $\spont_\pm$, the phase $+1$ occupies the regions with smaller 
principal radius, while the phase $-1$ can be found where the membrane is
rather flat.
We show some more detail of the initial binodal decomposition in
Figure~\ref{fig:fig6_binodal}.
\begin{figure}
\center
\includegraphics[angle=-90,width=0.45\textwidth]{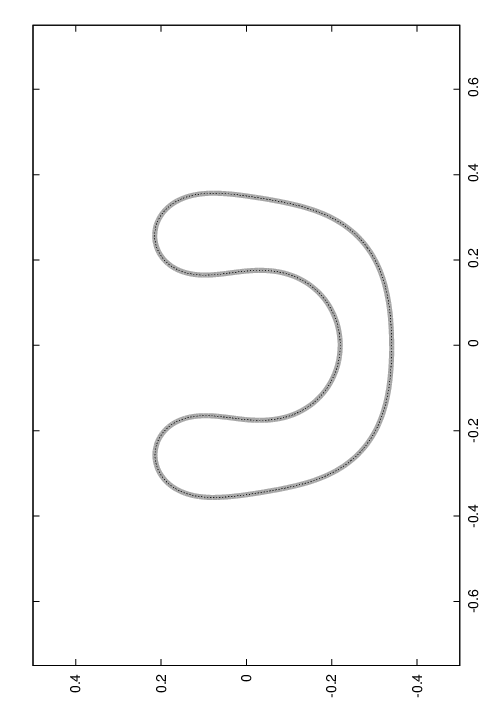} 
\includegraphics[angle=-90,width=0.45\textwidth]{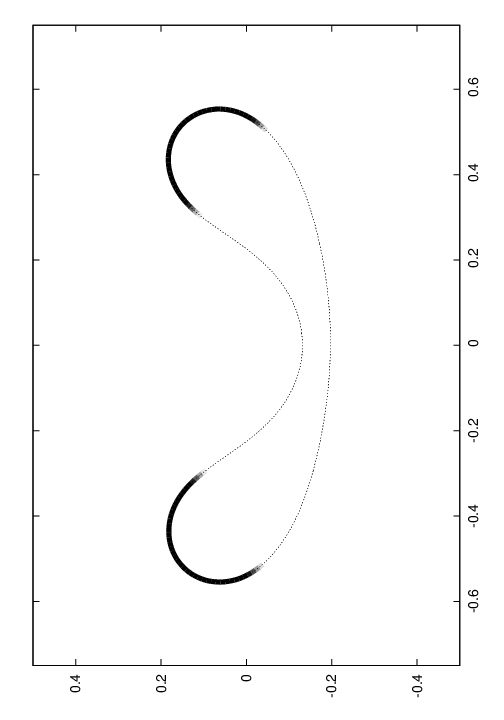} 
\includegraphics[angle=-90,width=0.45\textwidth]{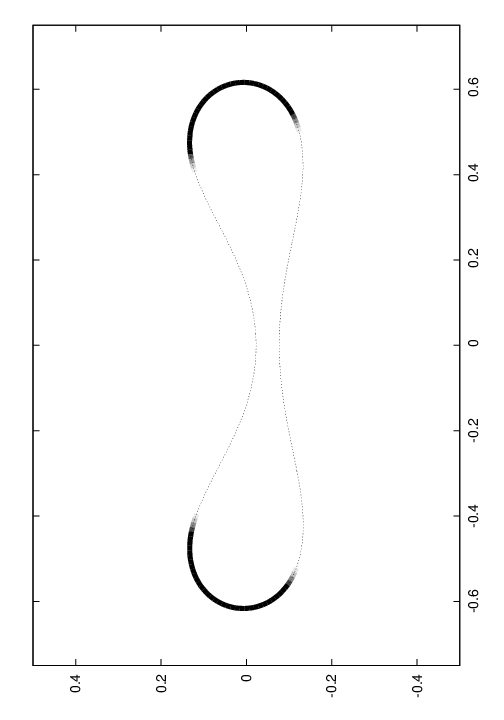} 
\includegraphics[angle=-90,width=0.45\textwidth]{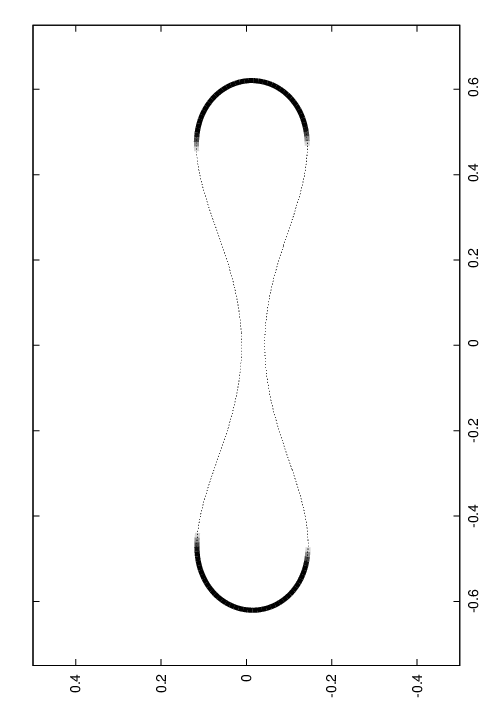} 
\includegraphics[angle=-90,width=0.45\textwidth]{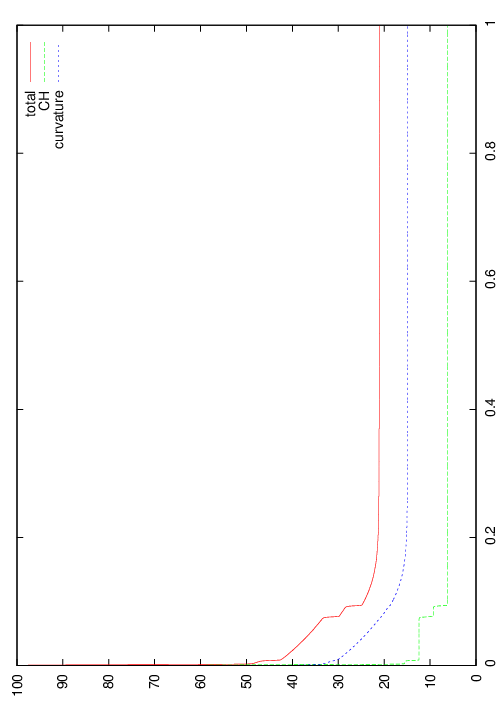} 
\caption{($\alpha_\pm = 1$, $\spont_- = -\tfrac12$, $\spont_+=-2$, $\beta = 1$)
Flow for a smooth letter ``C''.
We show $\phaseC^m$ on $\Gamma^m$ at times $t=0,\,0.1,\,0.2,\,1$.
Below a superimposed plot of the total discrete energy $\mathcal{E}^h_{total}$, 
the discrete Cahn--Hilliard energy, and the discrete curvature 
energy over $[0,1]$.
}
\label{fig:fig6}
\end{figure}%
\begin{figure}
\center
\center
\includegraphics[angle=-90,width=0.45\textwidth]{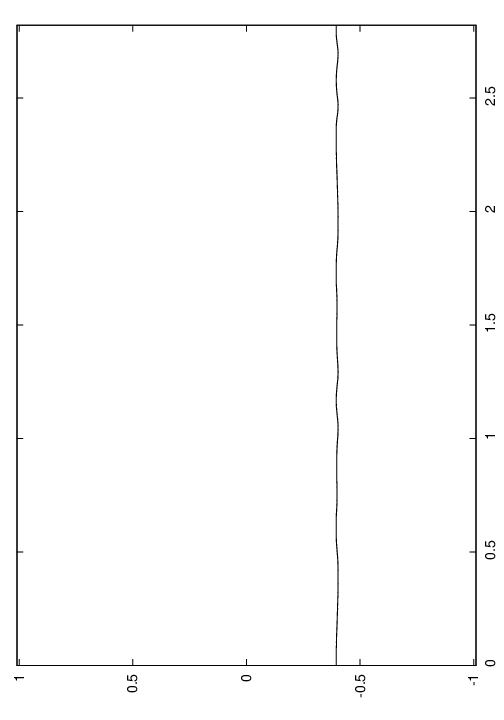} 
\includegraphics[angle=-90,width=0.45\textwidth]{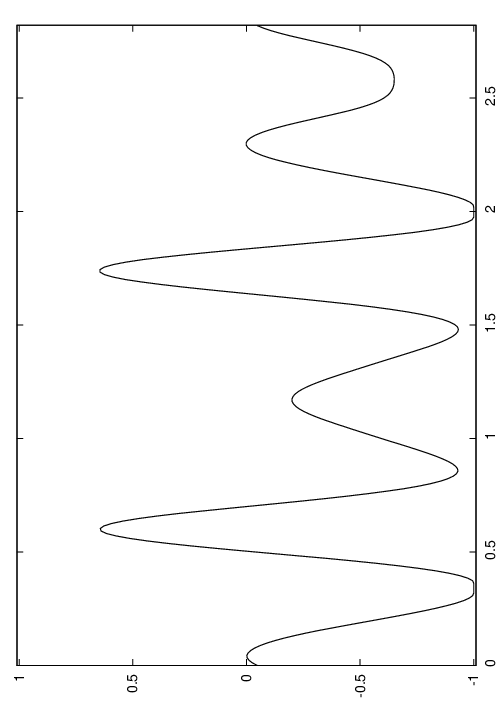} 
\includegraphics[angle=-90,width=0.45\textwidth]{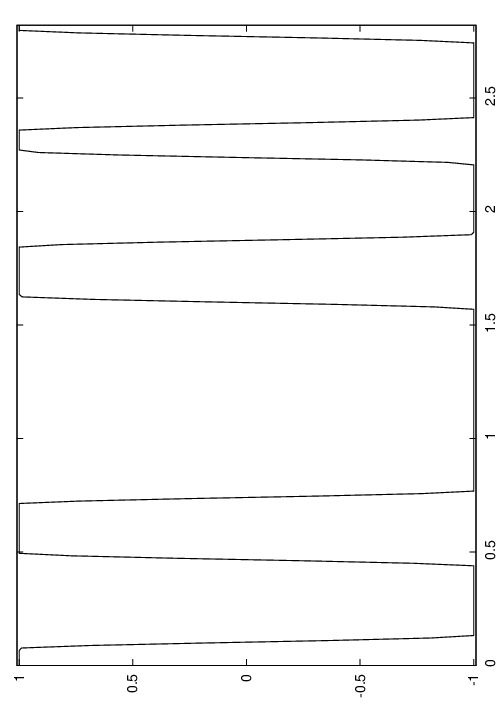} 
\includegraphics[angle=-90,width=0.45\textwidth]{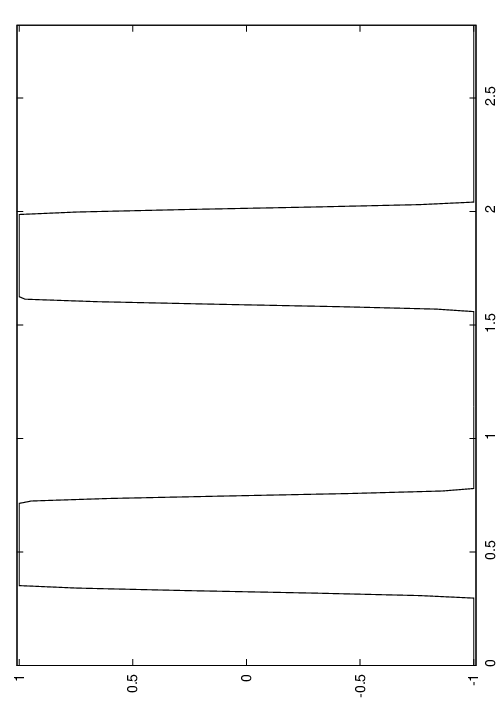}
\includegraphics[angle=-90,width=0.45\textwidth]{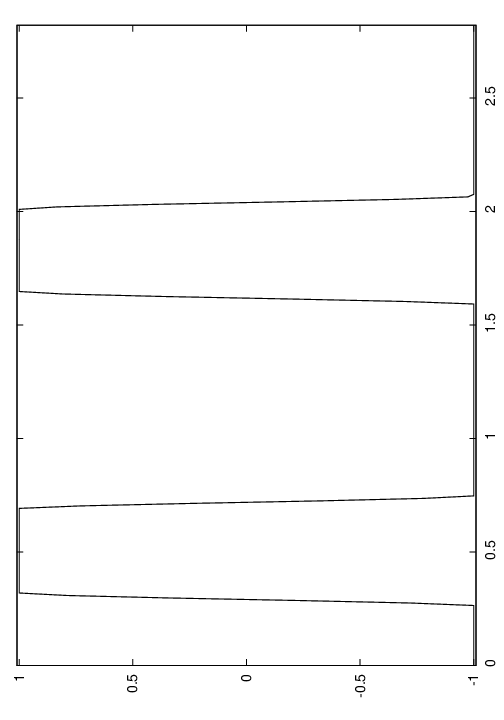}
\includegraphics[angle=-90,width=0.45\textwidth]{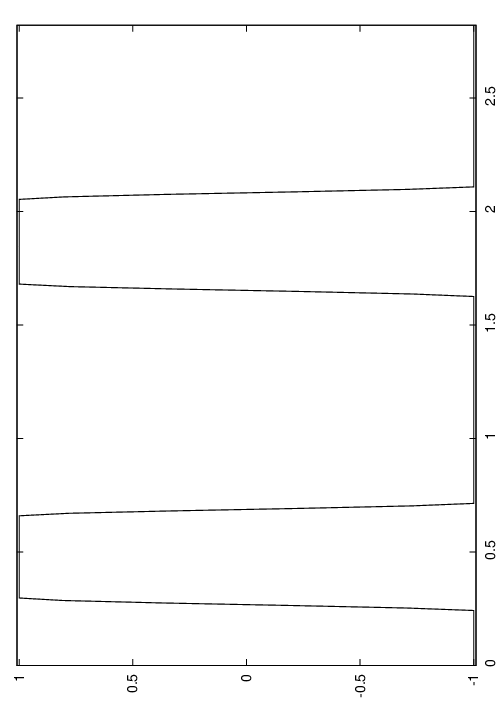} 
\caption{($\alpha_\pm = 1$, $\spont_- = -\tfrac12$, $\spont_+=-2$, $\beta = 1$)
Flow for a smooth letter ``C''.
We show arclength plots of $\phaseC^m$ at times $t=0,\,0.001,\,0.01,\,0.1,\,0.2,\,1$.
}
\label{fig:fig6_binodal}
\end{figure}%

We conducted the following shearing experiments on the domain $\Omega =
(-2,2)^2$ for an initial interface in the form of an ellipse, centred at the
origin, with axis lengths 1 and 2.5. 
The length of the polygonal interface is $5.75$, and it has $257$ elements.
For our choice of $\gamma = 0.05$ this 
yields on average about $7$ elements across the interface.
The time step size is $\tau=5\times10^{-4}$.
Once again we choose a random distribution for $\phaseC^0$ with mean value $-0.4$.
In particular, we
prescribe the inhomogeneous Dirichlet boundary condition
$\vec g(\vec z) = (z_2, 0)^T$ on $\partial_1\Omega = [-2,2] \times \{\pm2\}$.
The remaining parameters are given by $\rho=\rho_\Gamma=1$,
$\alpha_- = 0.05$, $\alpha_+ = 0.2$ and either
\begin{align} 
\text{(a)}\quad \mu_+ = 1,\quad \mu_- = 1\,,\quad \text{or}\quad
\text{(b)}\quad \mu_+ = 1,\quad \mu_- = 10\,. \label{eq:shear}
\end{align}
The results can be seen in Figures~\ref{fig:shear_rho0_mu1} 
and \ref{fig:shear_rho0_mu1_10}, and they should be compared to the
corresponding computations in the absence of any species effect, i.e.\
for $\phaseC^0 = -1$ constant, which can be seen in Figures~2 and 3 in \cite{nsns}.
As there, we observe tank treading when there is no viscosity contrast between
inner and outer phase, and we observe tumbling when there is a viscosity
contrast. The main difference to the computations in \cite{nsns}, though, 
is that here the regions occupied by the +1 phase on the vesicle remain
relatively straight throughout. This means that the tank treading motion in 
Figure~\ref{fig:shear_rho0_mu1} leads to concave shapes at times. 
Similarly, the phase distributions on the tumbling
vesicle in Figure~\ref{fig:shear_rho0_mu1_10} have a notable effect on the
vesicle shape, when compared with Figure~3 in \cite{nsns}.
\begin{figure}
\center
\includegraphics[angle=-90,width=0.45\textwidth]{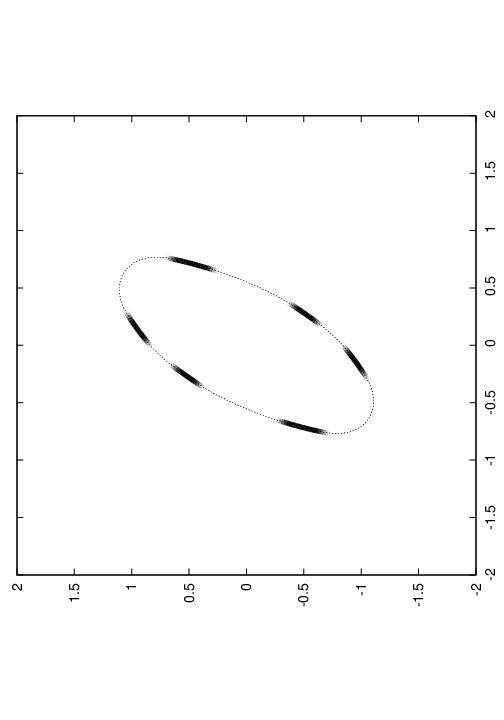}
\includegraphics[angle=-90,width=0.45\textwidth]{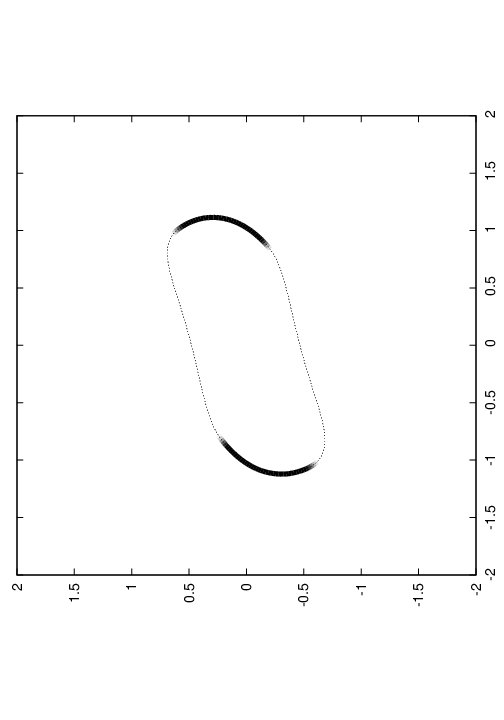}
\includegraphics[angle=-90,width=0.45\textwidth]{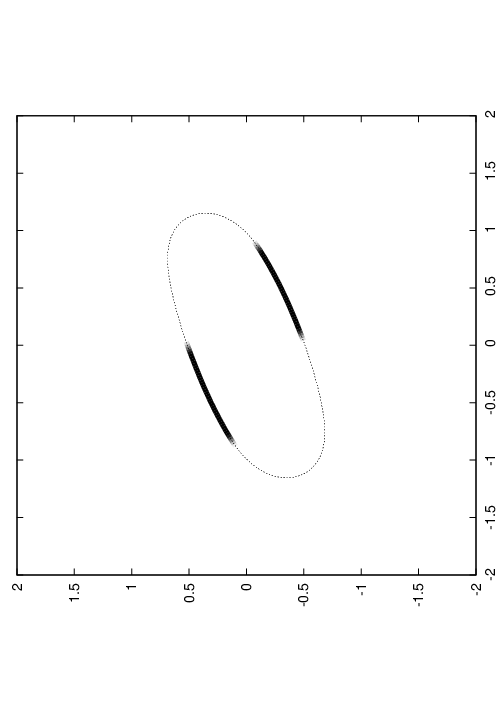}
\includegraphics[angle=-90,width=0.45\textwidth]{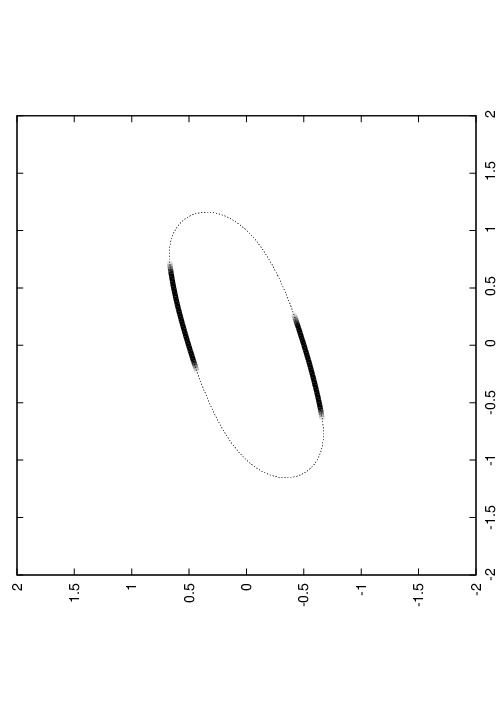}
\caption{($\alpha_- = 0.05$, $\alpha_+ = 0.2$, $\spont_\pm=0$, $\beta = 1$)
Shear flow with parameters as in (\ref{eq:shear}a), leading to tank treading.
The plots show the interface $\Gamma^m$, together with the concentration
$\phaseC^m$ at times $t=1,\ 11,\ 13,\ 15$ (top left to bottom right).
}
\label{fig:shear_rho0_mu1}
\end{figure}%
\begin{figure}
\center
\includegraphics[angle=-90,width=0.45\textwidth]{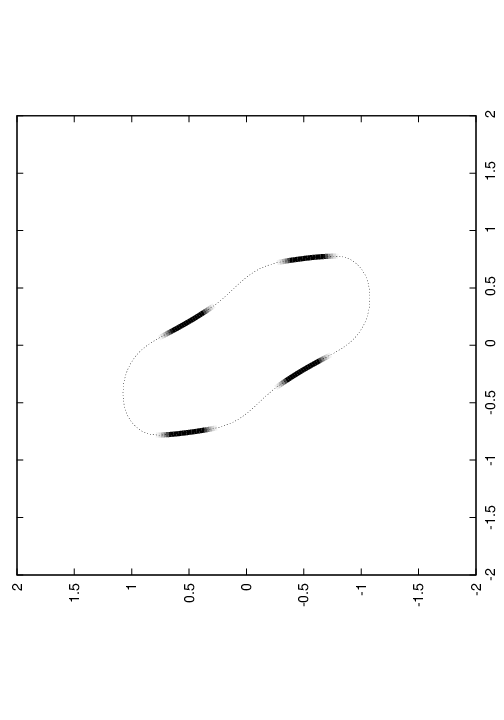}
\includegraphics[angle=-90,width=0.45\textwidth]{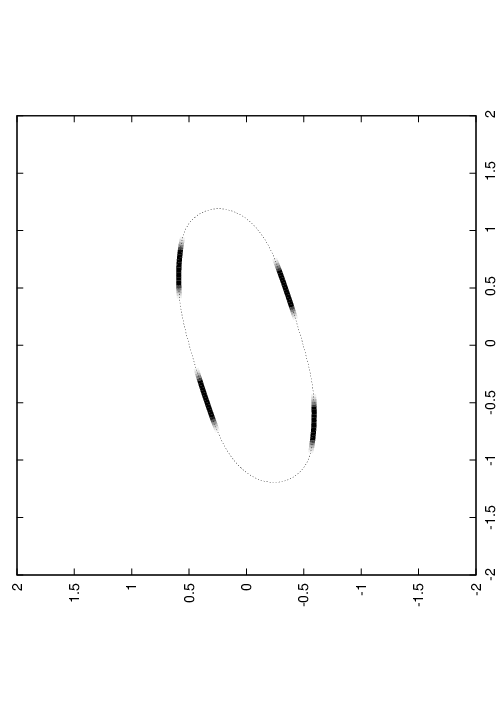}
\includegraphics[angle=-90,width=0.45\textwidth]{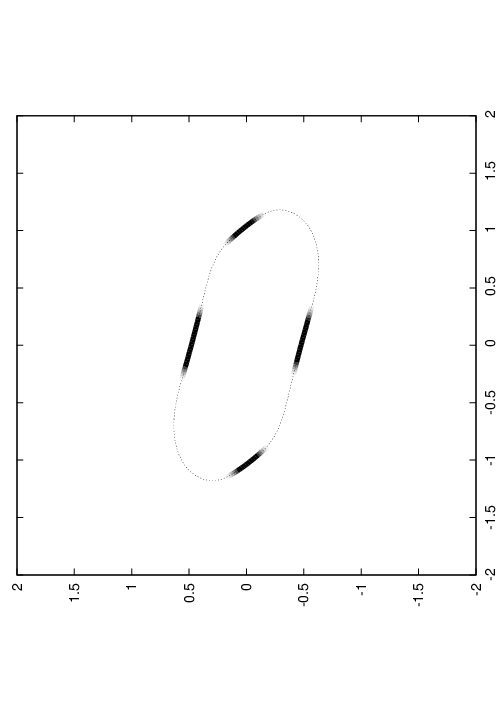}
\includegraphics[angle=-90,width=0.45\textwidth]{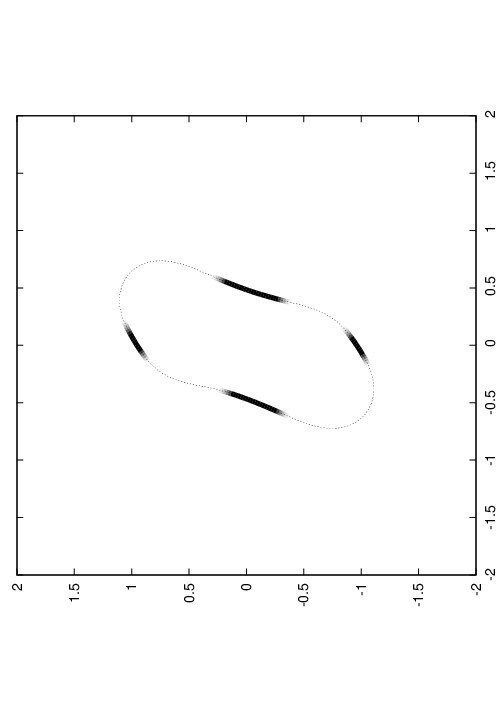}
\caption{($\alpha_- = 0.05$, $\alpha_+ = 0.2$, $\spont_\pm=0$, $\beta = 1$)
Shear flow with parameters as in (\ref{eq:shear}b), leading to tumbling.
The plots show the interface $\Gamma^m$, together with the concentration
$\phaseC^m$ at times $t=8,\ 11,\ 14,\ 17$ (top left to bottom right).}
\label{fig:shear_rho0_mu1_10}
\end{figure}%

Next we show a computation that highlights the Marangoni-type effects due to
the tangential terms in (\ref{eq:fGamma}). To this end, we start off with an
initial interface that has an elliptic shape, on which the two phases are
already well separated. The values of $\spont_\pm$ are then chosen such that a
tangential movement of the phases leads to a decrease in energy. 
In particular, we let $\spont_- = 0.5$, $\spont_+ = 2$ and
$\beta = 10$.
The length of the polygonal interface is $5.75$, and it has $257$ elements.
For our choice of $\gamma = 0.05$ this 
yields on average about $7$ elements across the interface.
The computational domain is $\Omega = (-2,2)^2$,
and the chosen time step size is $\tau=5\times10^{-4}$.
The results of the simulation are shown in Figure~\ref{fig:marangoni}.
It can be seen that due to the choice of $\spont_\pm$, the $+1$ phase moves
away from an area of large convex bending to an area that is at first almost
flat, and then settles on an area with a small concave bending.
In Figure~\ref{fig:marangoni_flow} we visualize the flow field for this
computation, and compare it with a computation when $\phaseC^0 = 1$ constant, 
so that there are no tangential forces in (\ref{eq:fGamma}).
One clearly sees the effect of the tangential force which induces
flow close to the interface also at later times.
\begin{figure}
\center
\mbox{
\includegraphics[angle=-90,width=0.33\textwidth]{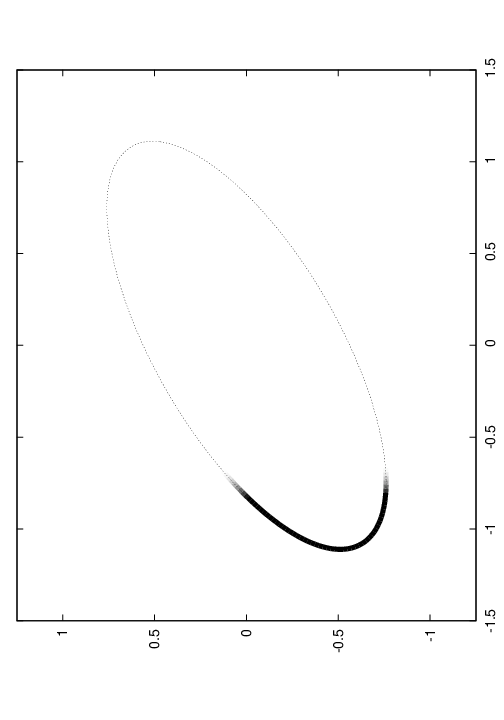} 
\includegraphics[angle=-90,width=0.33\textwidth]{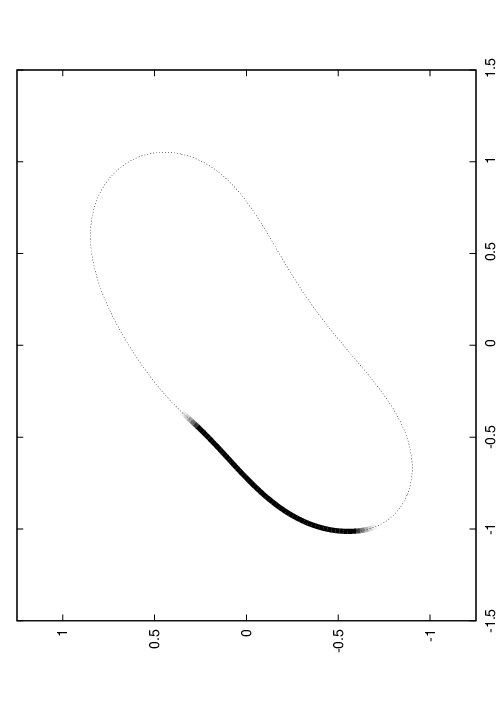} 
\includegraphics[angle=-90,width=0.33\textwidth]{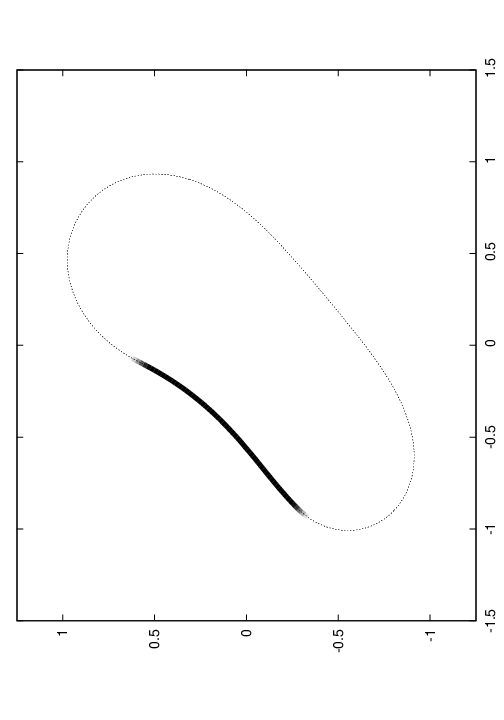}}
\mbox{
\includegraphics[angle=-90,width=0.33\textwidth]{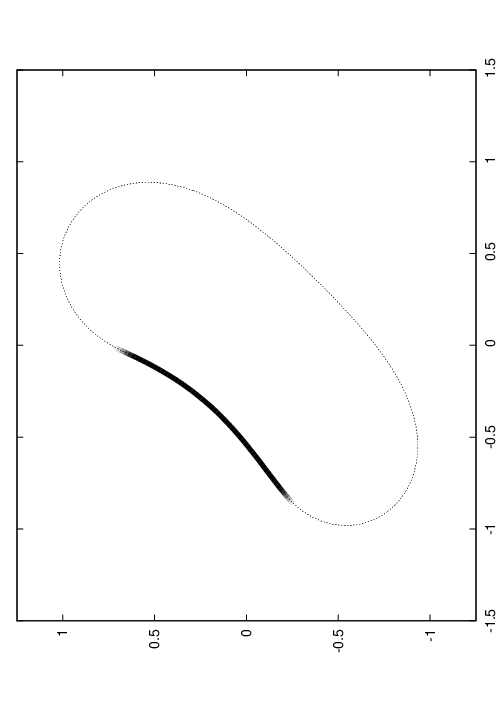} 
\includegraphics[angle=-90,width=0.33\textwidth]{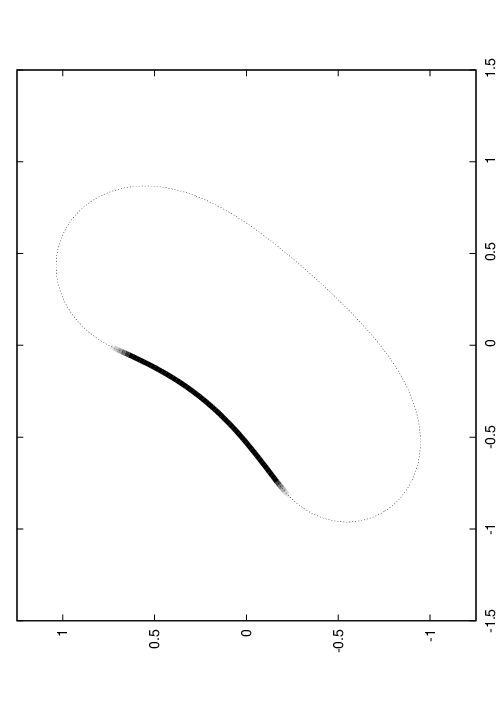} 
\includegraphics[angle=-90,width=0.33\textwidth]{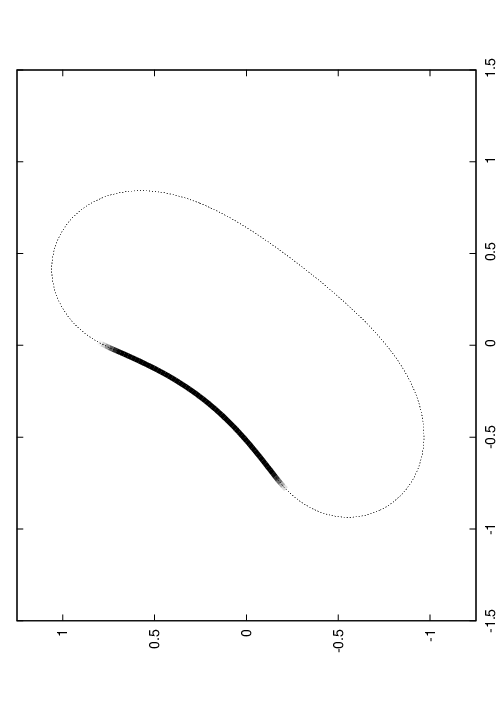}}
\caption{($\alpha_\pm = 1$, $\spont_- = 0.5$, $\spont_+ = 2$, $\beta = 10$)
Flow for an ellipse.
We show $\phaseC^m$ on $\Gamma^m$ at times $t=0,\,1,\,2,\,3,\,4,\,10$.
}
\label{fig:marangoni}
\end{figure}%
\begin{figure}
\center
\mbox{
\includegraphics[angle=-0,width=0.33\textwidth]{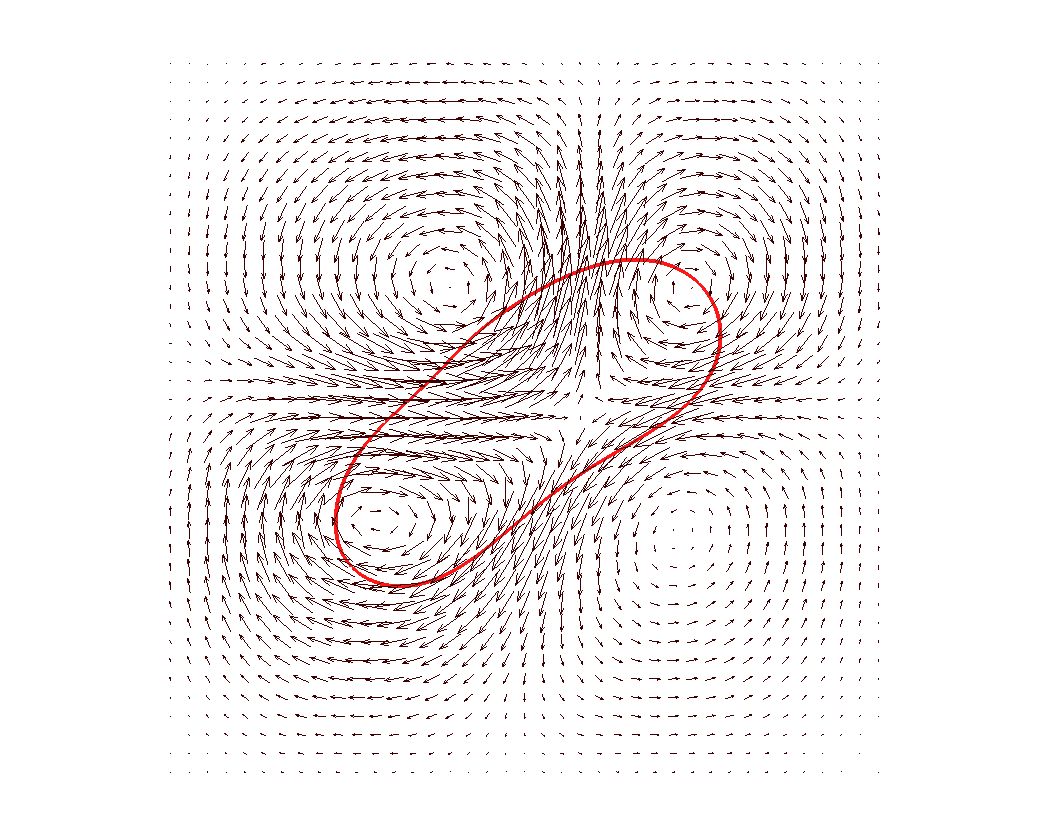} 
\includegraphics[angle=-0,width=0.33\textwidth]{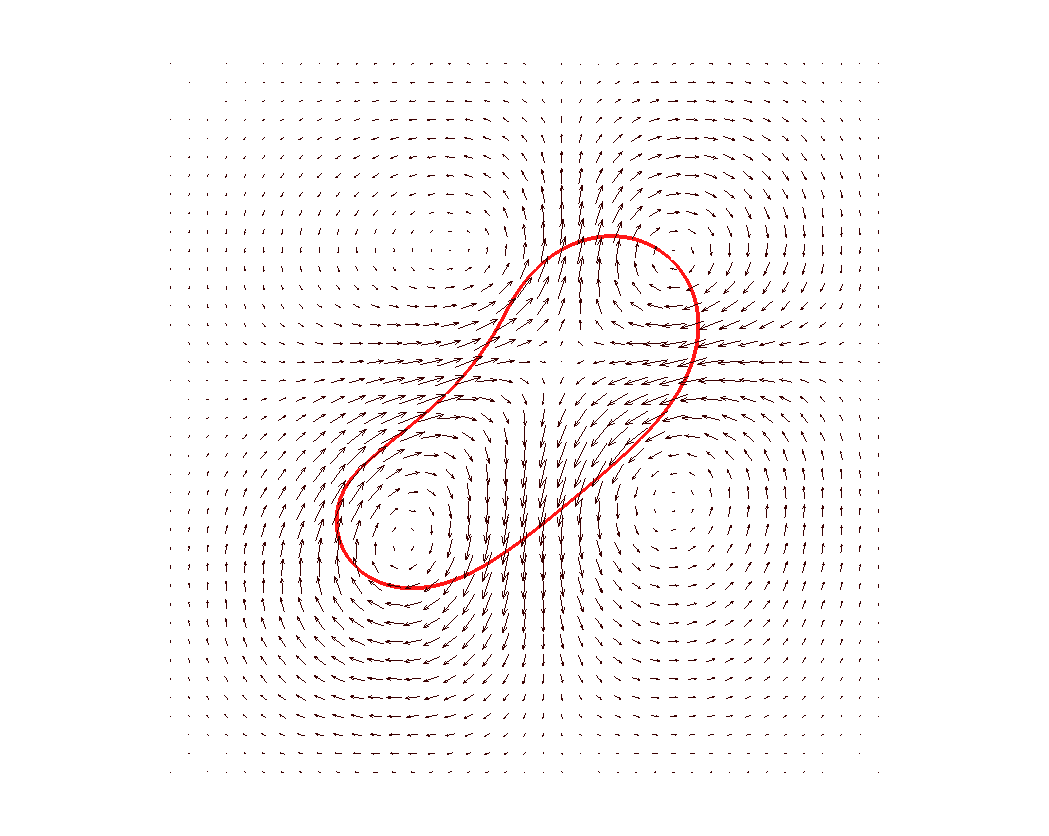} 
\includegraphics[angle=-0,width=0.33\textwidth]{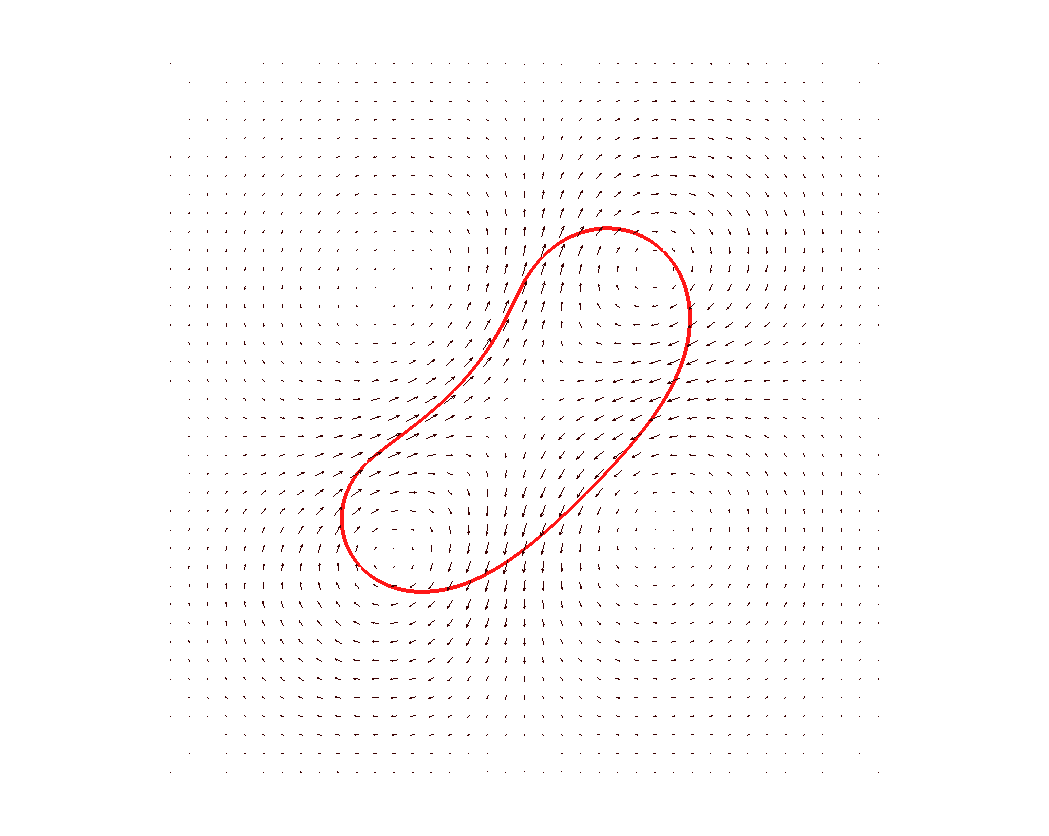}}
\mbox{
\includegraphics[angle=-0,width=0.33\textwidth]{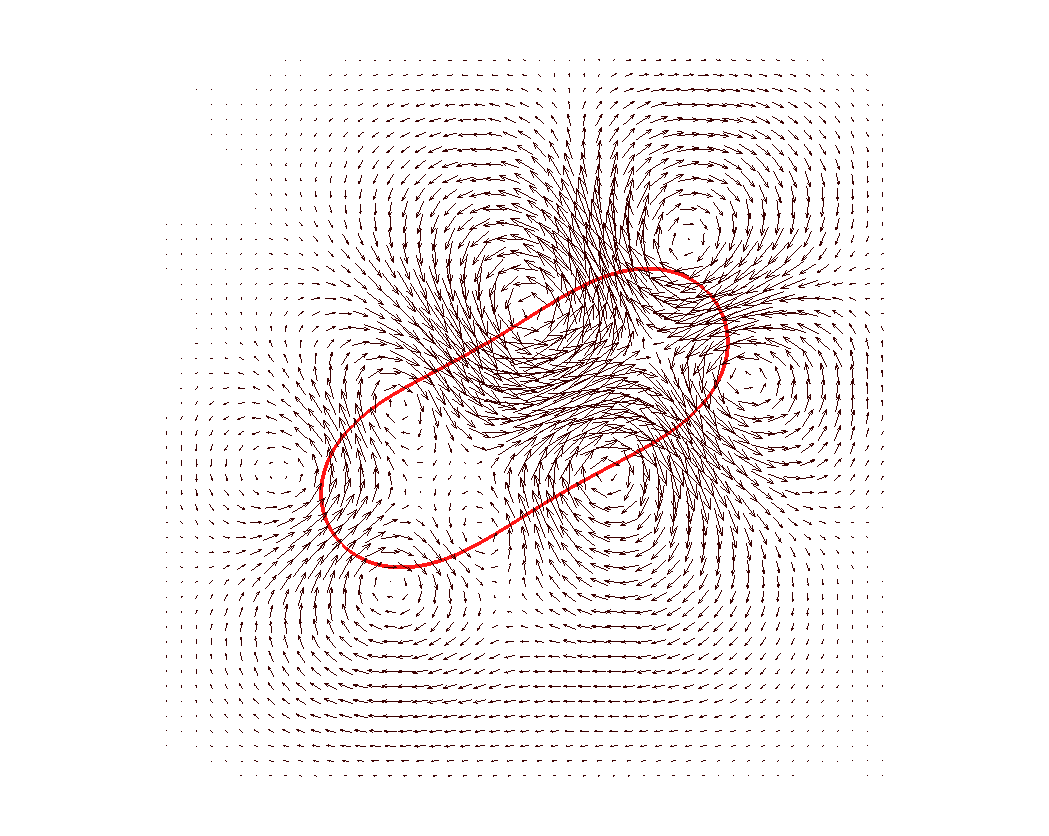} 
\includegraphics[angle=-0,width=0.33\textwidth]{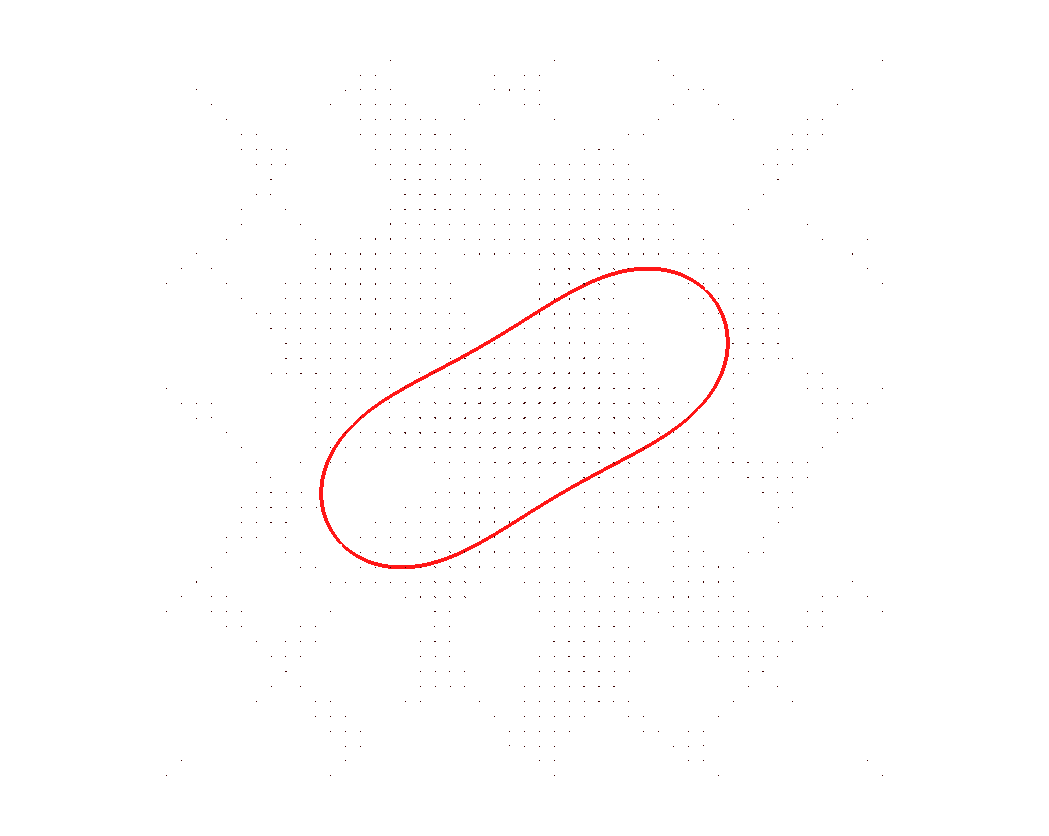}
\includegraphics[angle=-0,width=0.33\textwidth]{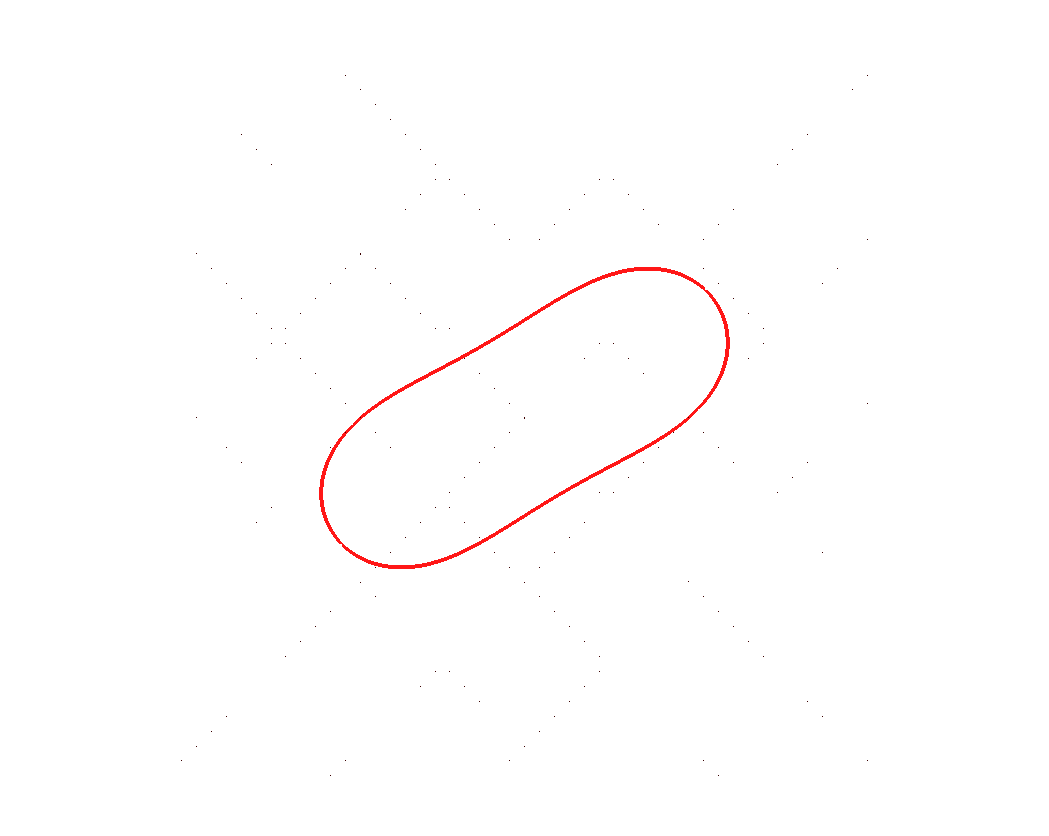}}
\caption{($\alpha_\pm = 1$, $\spont_- = 0.5$, $\spont_+ = 2$, $\beta = 10$)
Visualization of the flow field $\vec U^m$ at times $t=1,\,2, 3$
for the computation in Figure~\ref{fig:marangoni} (top), compared to the same
computation with $\phaseC^m = 1$ constant throughout (bottom).
}
\label{fig:marangoni_flow}
\end{figure}%

On replacing the definition in (\ref{eq:alpha}) with
\begin{subequations}
\begin{align} \label{eq:alphaC0}
\alpha(s) & = s^2\,\alpha_L(s) = 
\tfrac12\, (\alpha_+ + \alpha_-)\,s^2 
+ \tfrac12\, (\alpha_+ - \alpha_-) \,s^3\,, \\
\text{or}\quad
\alpha(s) & = (s^2 + \delta)\,\alpha_L(s) \,,\qquad \delta > 0\,,
\label{eq:alphaC0delta}
\end{align}
\end{subequations}
we can simulate $C^0$--junctions, see also \cite{Helmers13}, as long as
$\delta \to 0$ for $\gamma \to 0$.
We obtain interesting results starting from an ellipse, on which the two 
phases are already well separated, and using
$\spont_- = -0.2$, $\spont_+ = -2$ and $\beta=10$. 
The length of the polygonal interface is $5.75$, and it has $257$ elements.
For our choice of $\gamma = 0.05$ this 
yields on average about $7$ elements across the interface.
The computational domain is $\Omega = (-2,2)^2$,
and the chosen time step size is $\tau=5\times10^{-4}$.
In Figure~\ref{fig:2dC0C1new} we show the numerical steady states for the
two different evolutions. The nature of the $C^0$--junction can clearly be
seen, which allows for tangent discontinuities at the interface. This allows
the $+1$ phase to reduce its contribution to the overall
curvature energy. As a result, the total energy for the $C^0$--steady state
is $33.52$, which is smaller than the value $33.97$ for the $C^1$--case. 
For the curvature energy contributions the comparison is $2.32$ versus $2.83$, 
again in favour of the $C^0$--junction.
\begin{figure}
\center
\includegraphics[angle=-90,width=0.45\textwidth]{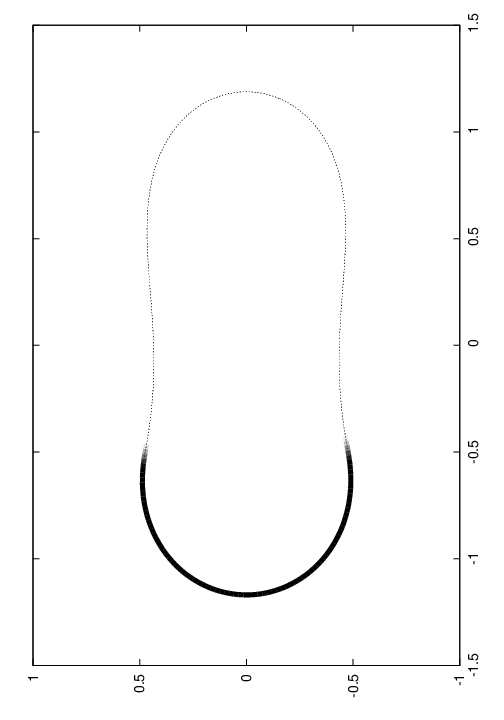} 
\includegraphics[angle=-90,width=0.45\textwidth]{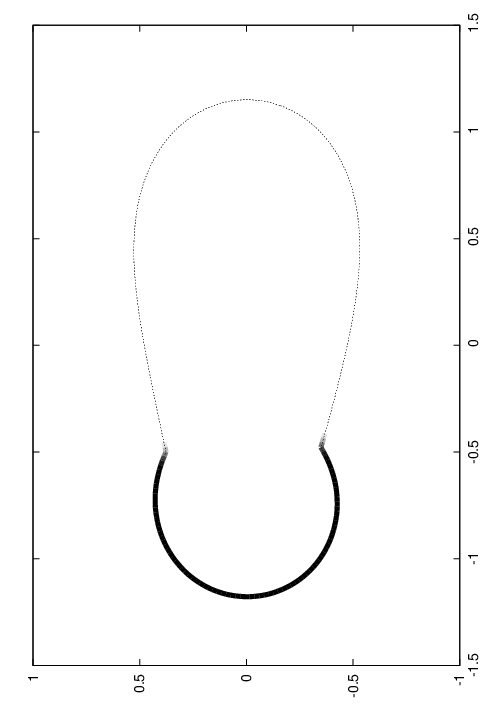} 
\includegraphics[angle=-90,width=0.45\textwidth]{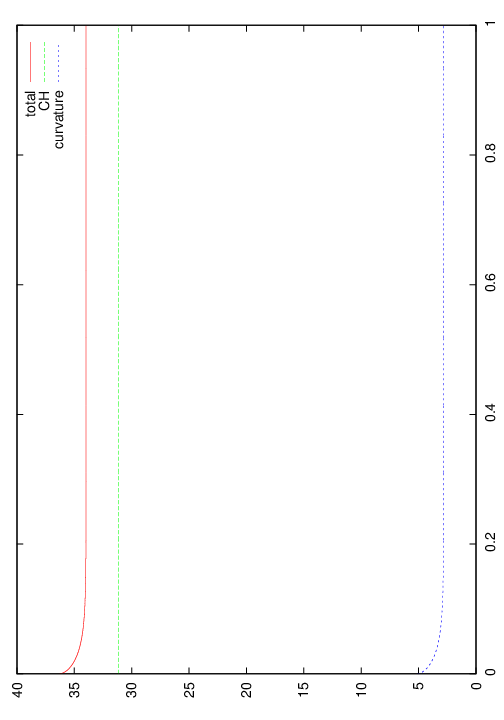} 
\includegraphics[angle=-90,width=0.45\textwidth]{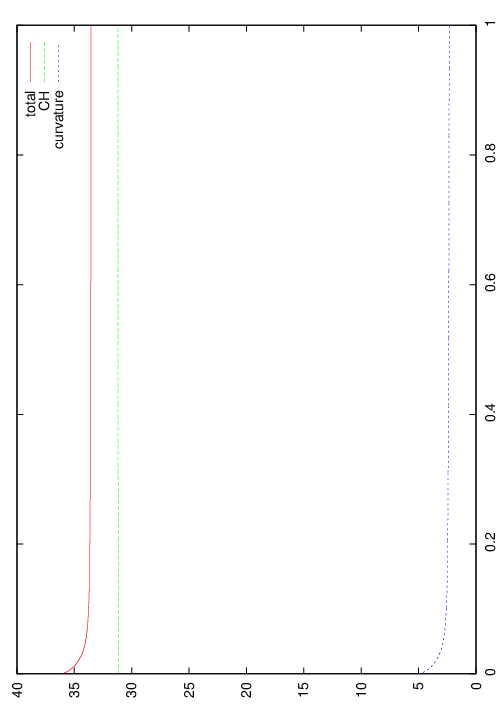} 
\caption{($\alpha_\pm = 1$, $\spont_- = -0.2$, $\spont_+ = -2$, 
$\beta = 10$)
Solution at time $t=1$ for the $C^1$--case (left) and the 
$C^0$--case (right).
Below a superimposed plot of the total discrete energy $\mathcal{E}^h_{total}$, 
the discrete Cahn--Hilliard energy, and the discrete curvature 
energy over $[0,1]$.
}
\label{fig:2dC0C1new}
\end{figure}%

\subsection{Numerical simulations in 3d}

As a first example for a three-dimensional simulation, 
we consider the evolution for an initially flat plate of total dimension 
$4\times4\times1$, similarly to \citet[Fig.\ 8]{nsns}. 
The triangulations $\Gamma^m$ satisfy
$(K_\Gamma,J_\Gamma) = (1538, 3072)$, and the polygonal surfaces have 
a surface area of $35.7$. 
This means that for our chosen value of $\gamma=0.2$, there are on
average about $5$ elements across the interfacial region on $\Gamma^m$.
As the computational domain we choose $\Omega = (-2.5,2.5)^3$,
and we use the time step size $\tau=10^{-3}$.
First we set $\alpha_\pm = 1$, $\spont_\pm = 0$ and $\beta = 1$,
so that the only effect of the two phase aspect is given by the
line energy contributions in the free energy.
The initial distribution for $\phaseC^0$ is random with mean value $-0.4$.
See Figure~\ref{fig:fig12} for the evolution in this case.
\begin{figure}
\center
\mbox{
\includegraphics[angle=-0,width=0.25\textwidth]{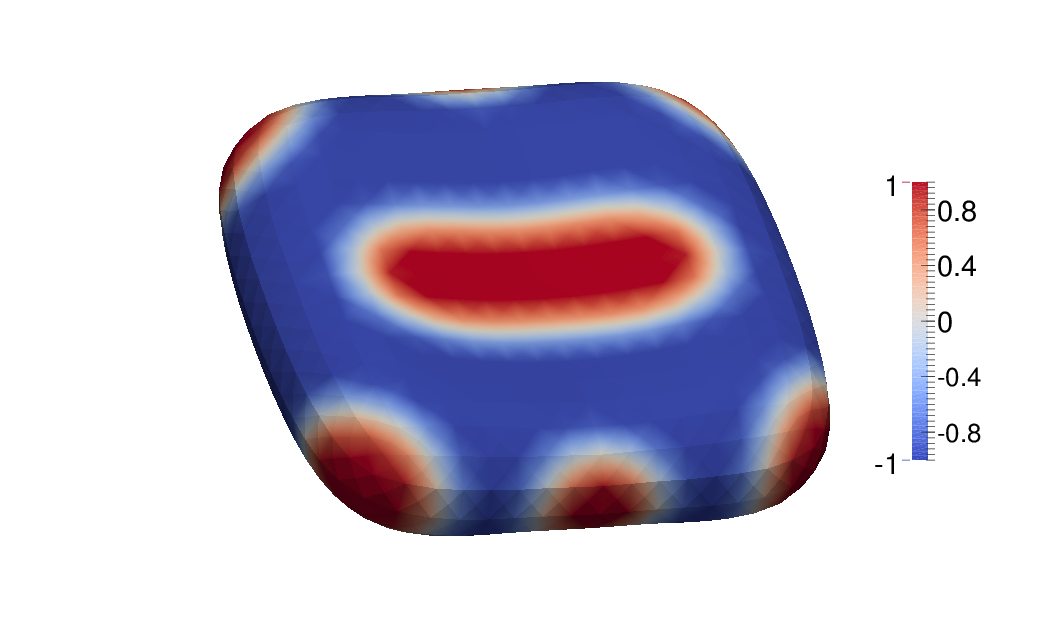}
\includegraphics[angle=-0,width=0.25\textwidth]{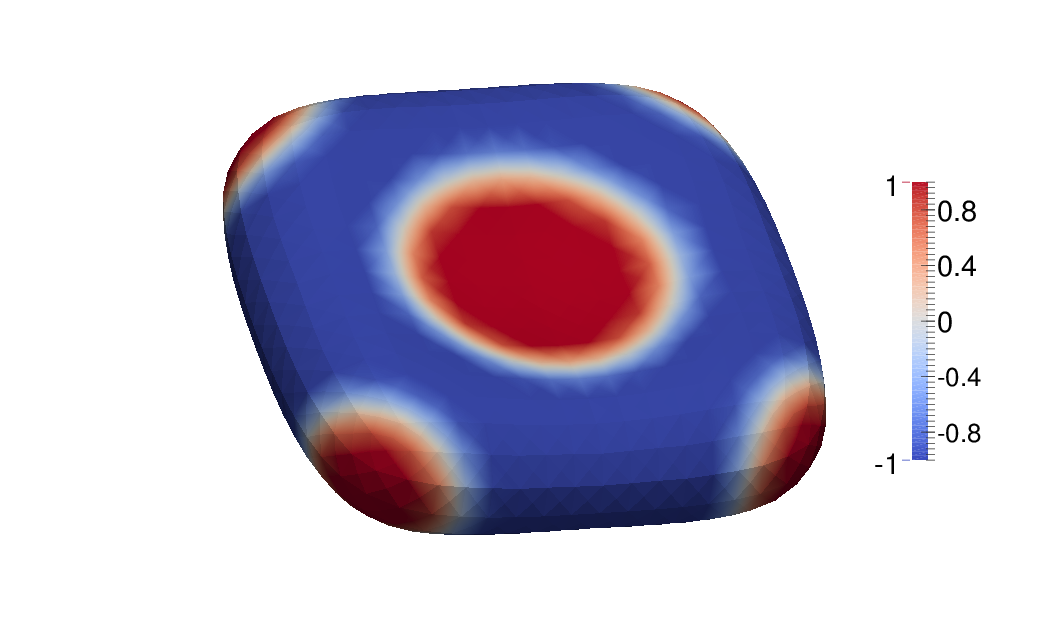}
\includegraphics[angle=-0,width=0.25\textwidth]{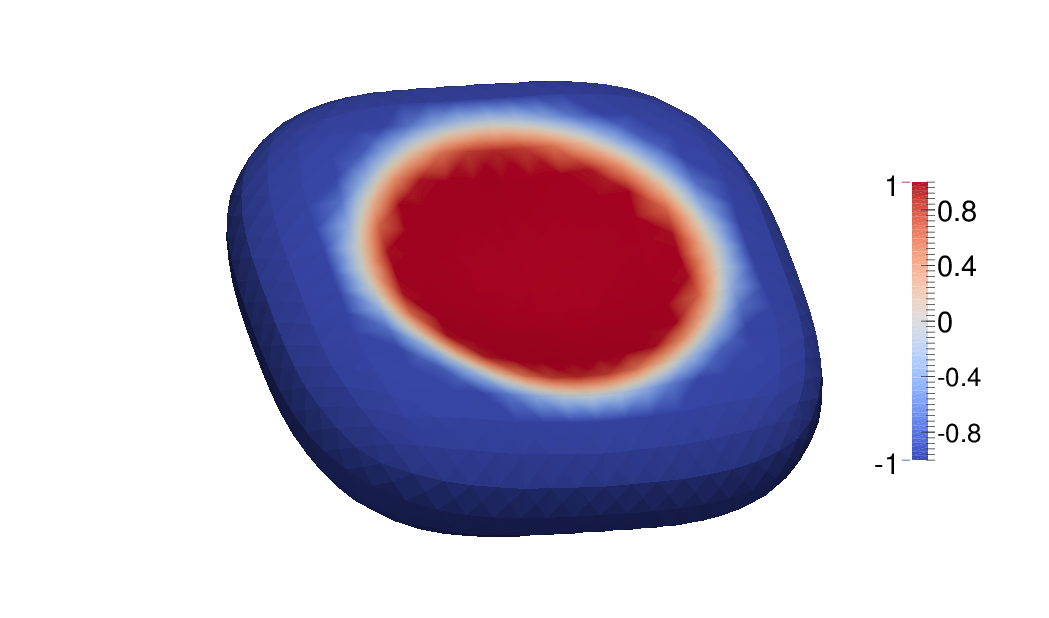}
\includegraphics[angle=-0,width=0.25\textwidth]{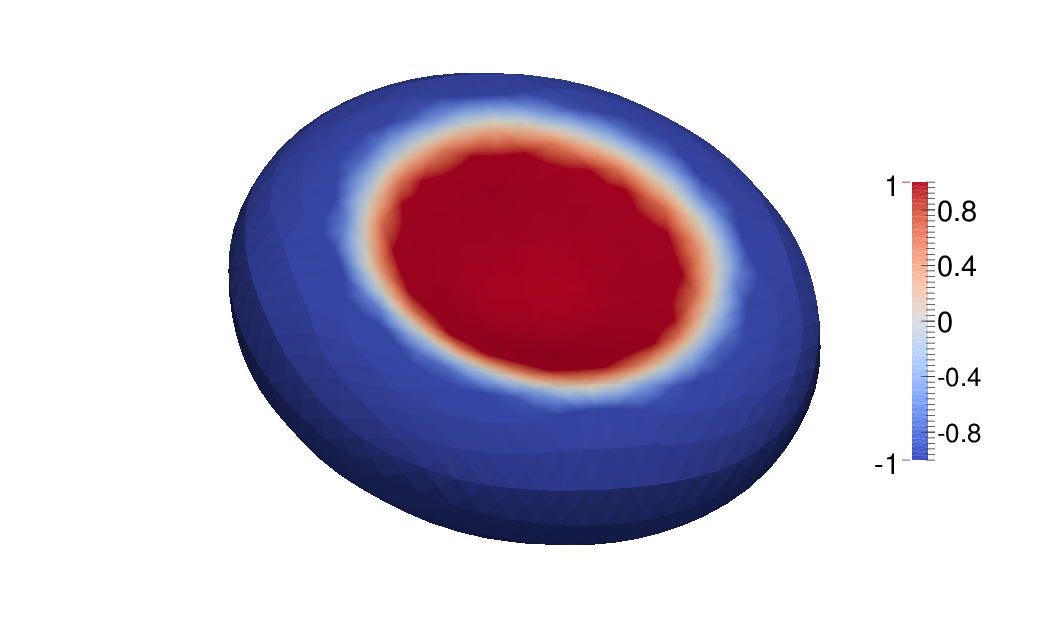}
}
\includegraphics[angle=-90,width=0.33\textwidth]{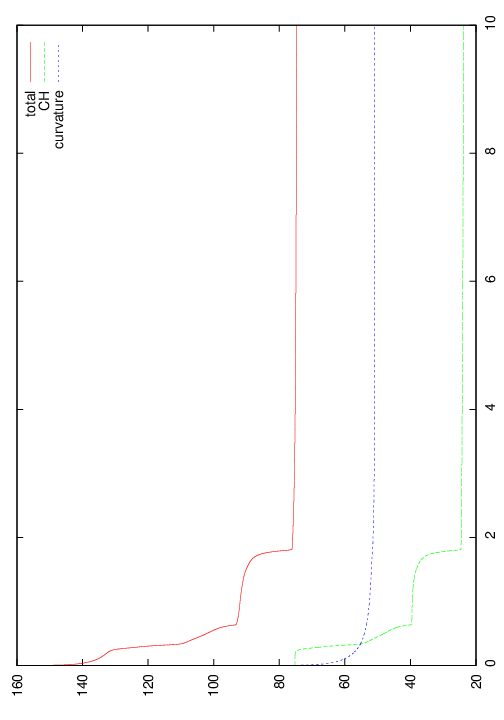}
\caption{($\alpha_\pm = 1$, $\spont_\pm = 0$, $\beta = 1$)
Plots of $\phaseC^m$ on $\Gamma^m$ at times $t=0.5,\ 1,\ 2,\ 10$.
Below a superimposed plot of the total discrete energy $\mathcal{E}^h_{total}$, 
the discrete Cahn--Hilliard energy, and the discrete curvature 
energy over $[0,10]$.
}
\label{fig:fig12}
\end{figure}%
Repeating the same experiment for $\alpha_- = \tfrac12$, $\alpha_+ = 2$ gives
the results in Figure~\ref{fig:fig13}. We note that the final shape is now a
bit flatter, since the $+1$ phase does not allow the inner part of the membrane
to become very concave.
\begin{figure}
\center
\mbox{
\includegraphics[angle=-0,width=0.25\textwidth]{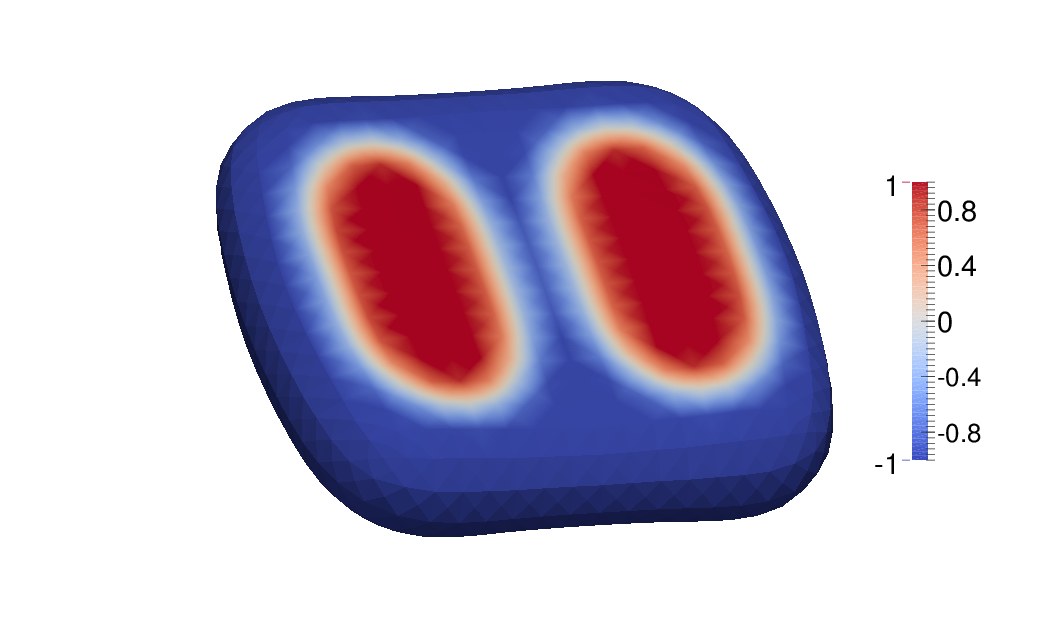}
\includegraphics[angle=-0,width=0.25\textwidth]{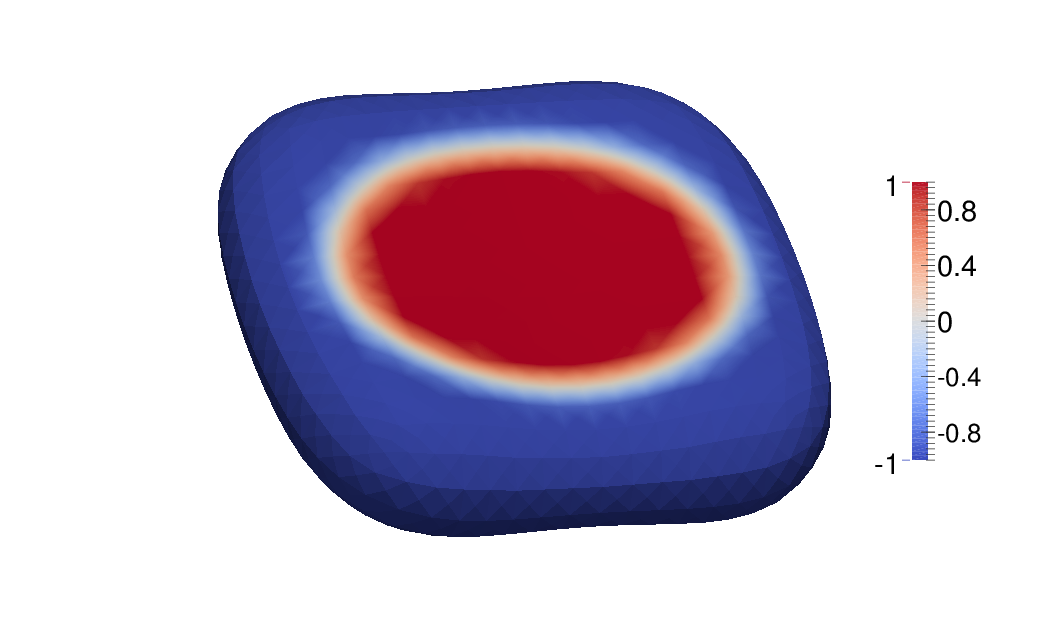}
\includegraphics[angle=-0,width=0.25\textwidth]{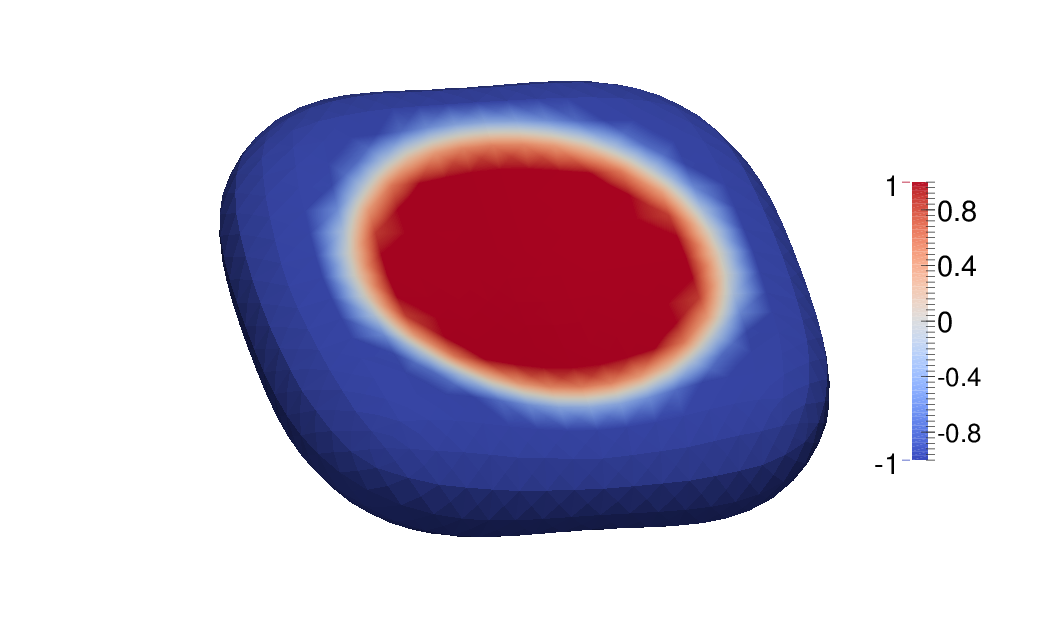}
\includegraphics[angle=-0,width=0.25\textwidth]{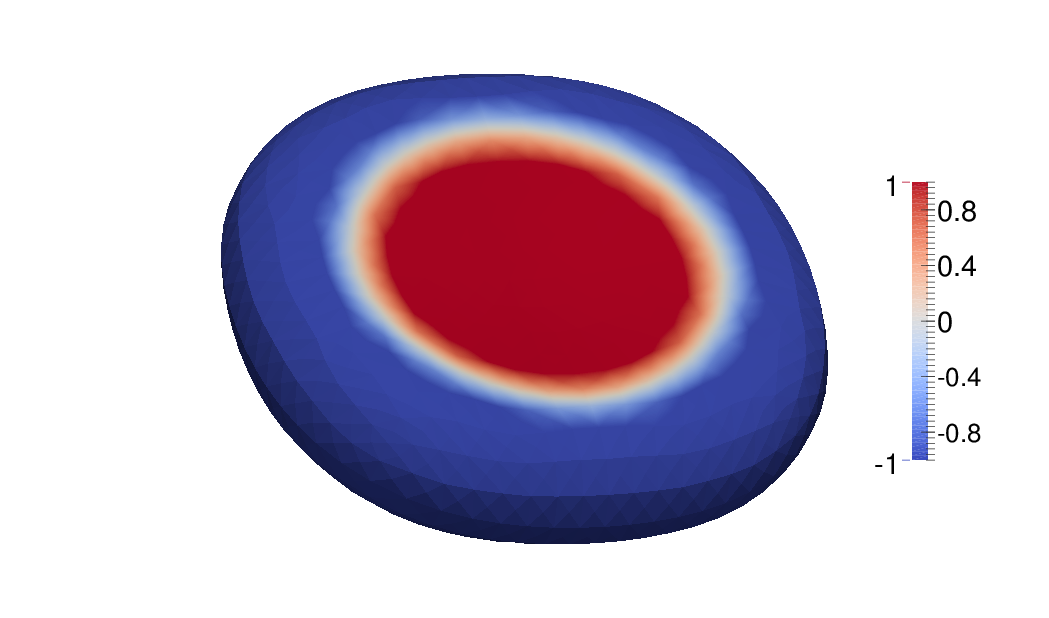}
}
\includegraphics[angle=-90,width=0.33\textwidth]{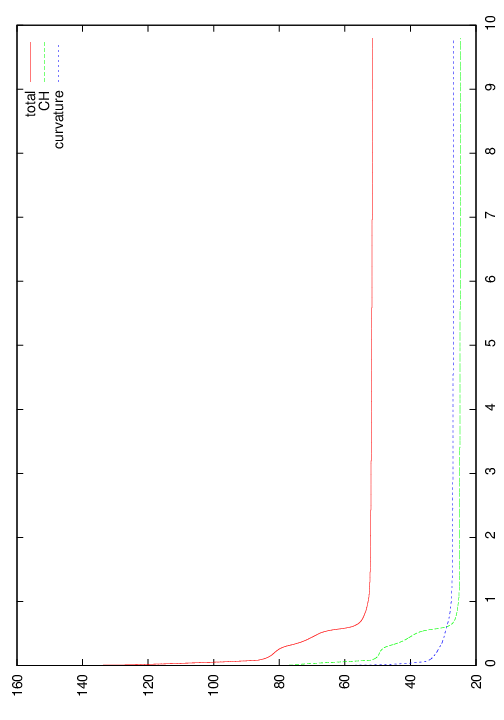}
\caption{($\alpha_- = \tfrac12$, $\alpha_+ = 2$, $\spont_\pm = 0$, $\beta = 1$)
Plots of $\phaseC^m$ on $\Gamma^m$ at times $t=0.5,\ 1,\ 2,\ 10$. Compared to
Figure~\ref{fig:fig12}, the final plot is less concave.
Below a superimposed plot of the total discrete energy $\mathcal{E}^h_{total}$, 
the discrete Cahn--Hilliard energy, and the discrete curvature 
energy over $[0,10]$.
}
\label{fig:fig13}
\end{figure}%

In order to investigate budding, we start from a four-armed shape with
well-developed interfaces between the two surface phases.
As we use a finer mesh with $(K_\Gamma,J_\Gamma) = (3074, 6144)$, we now choose
$\gamma = 0.1$. Moreover, we have set
$\alpha_\pm = 1$, $\spont_- = -\tfrac12$, $\spont_+ = -2$ to encourage the
forming of the buds. In the first experiment we set $\beta = 1$ and 
observe the results shown in Figure~\ref{fig:restart_beta1}.
\begin{figure}
\center
\mbox{
\includegraphics[angle=-0,width=0.33\textwidth]{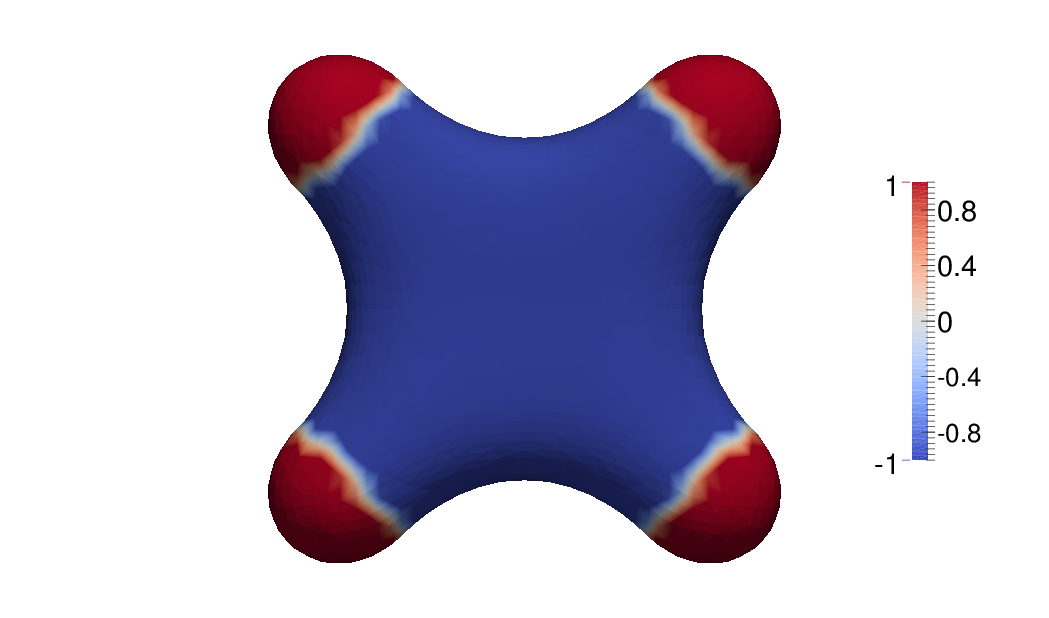}
\includegraphics[angle=-0,width=0.33\textwidth]{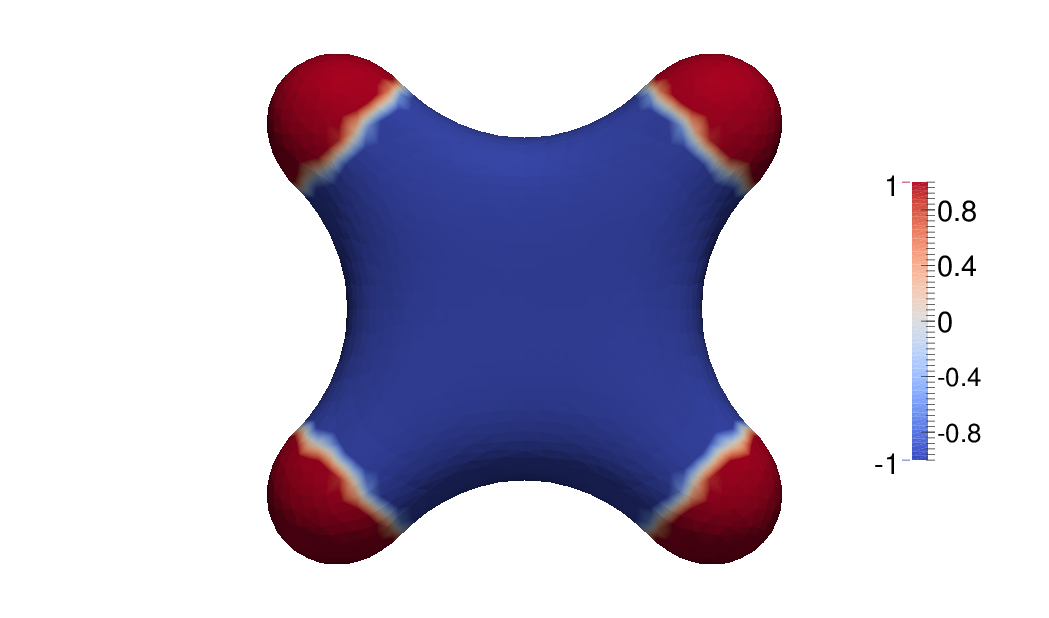}
\includegraphics[angle=-0,width=0.33\textwidth]{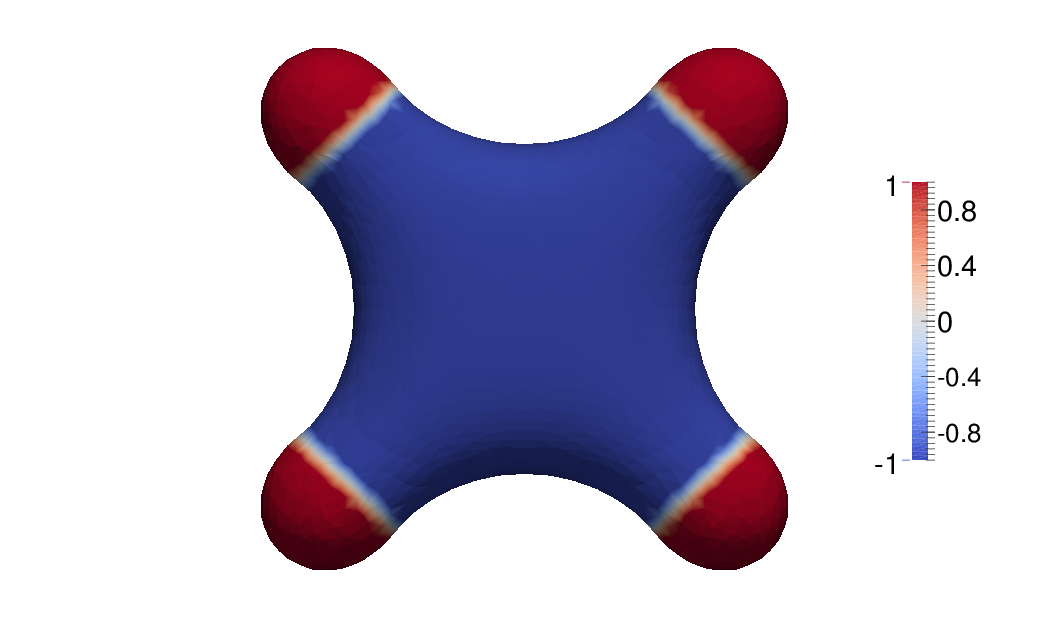}
}
\caption{($\alpha_\pm = 1$, $\spont_- = -\tfrac12$, $\spont_+ = -2$, 
$\beta = 1$)
Plots of $\phaseC^m$ on $\Gamma^m$ at times $t=0.5,\ 1,\ 5$.
}
\label{fig:restart_beta1}
\end{figure}%
The same experiment with $\beta = 5$ is shown in 
Figure~\ref{fig:restart_beta5}, where we observe budding behaviour now. In
particular, the $+1$ phase would like to pinch off the membrane at the four
corners.
\begin{figure}
\center
\mbox{
\includegraphics[angle=-0,width=0.33\textwidth]{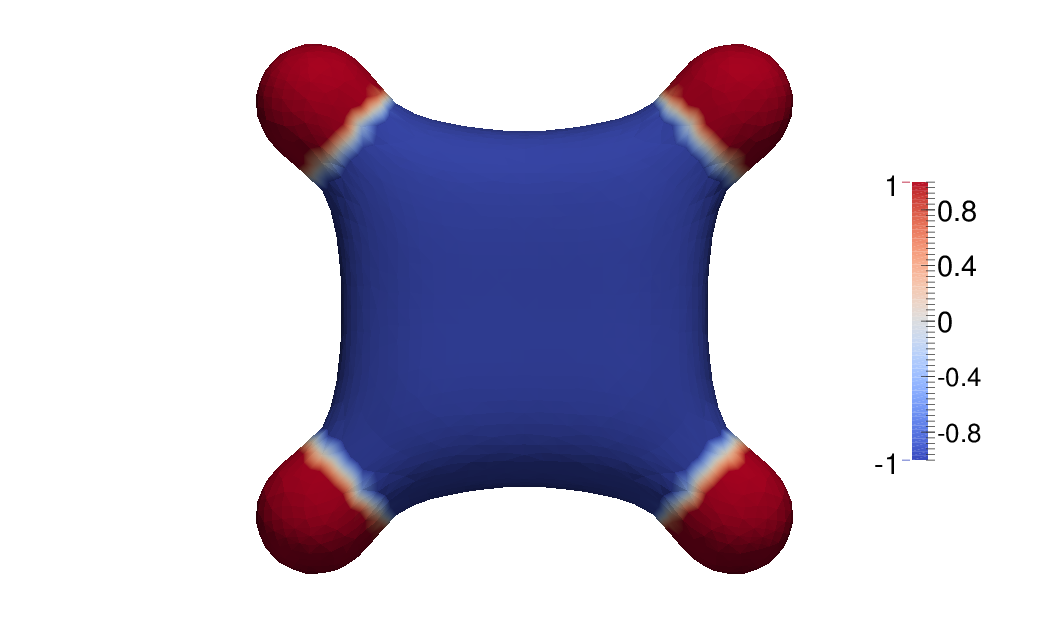}
\includegraphics[angle=-0,width=0.33\textwidth]{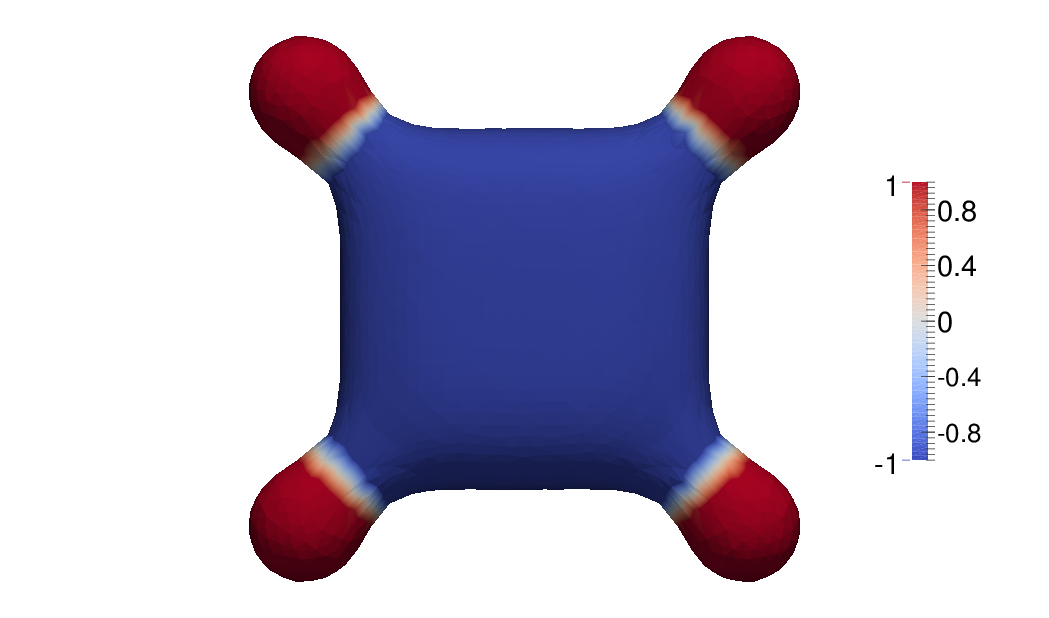}
\includegraphics[angle=-0,width=0.33\textwidth]{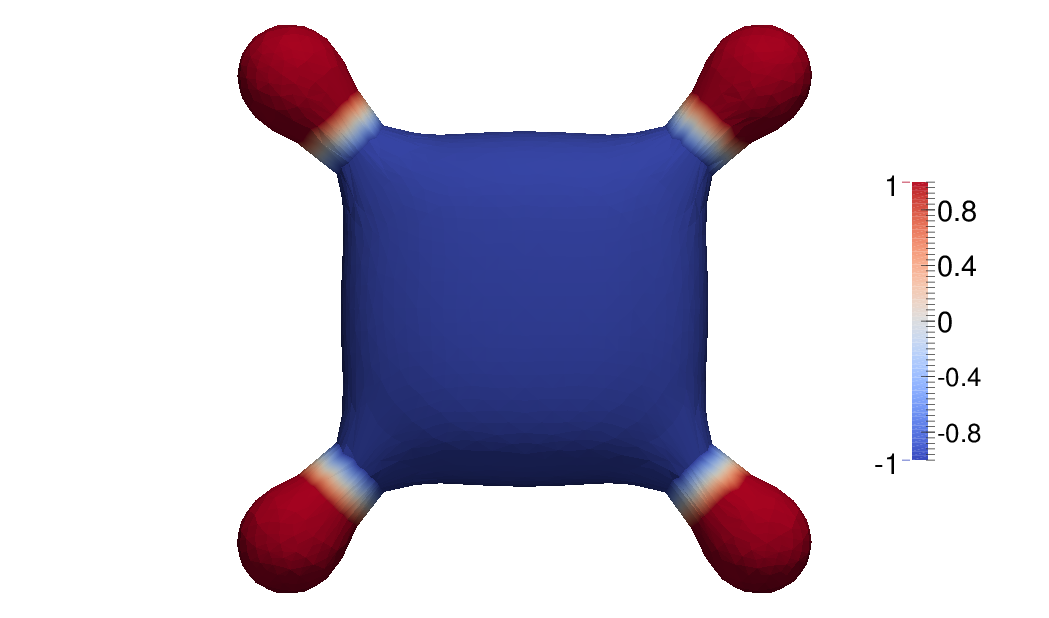}
}
\caption{($\alpha_\pm = 1$, $\spont_- = -\tfrac12$, $\spont_+ = -2$, 
$\beta = 5$)
Plots of $\phaseC^m$ on $\Gamma^m$ at times $t=0.5,\ 1,\ 5$.
}
\label{fig:restart_beta5}
\end{figure}%

The numerical simulation of a vesicle flowing through a constriction can be
seen in Figure~\ref{fig:vartheta_cons3d}. This is a two-phase analogue
of the simulation shown in \citet[Figure~9]{nsns}. 
Here we choose the initial shape of
the interface to be a biconcave surface resembling a human red blood cell.
The shape has surface area $2.23$, and the triangulations $\Gamma^m$ satisfy
$(K_\Gamma,J_\Gamma) = (3074, 6144)$.
This means that for our chosen value of $\gamma=0.05$, there are on
average about $6$ elements across the interfacial region on $\Gamma^m$.
As the computational domain we choose $\Omega = (-2,-1) \times (-1,1)^2 \cup
[-1,1]\times(-0.5,0.5)^2 \cup (1,2) \times (-1,1)^2$.
We define $\partial_2\Omega= \{2\} \times (-1,1)^2$ and on 
$\partial_1\Omega$ 
we set no-slip conditions, except on the left hand part
$\{-2\} \times [-1,1]^2$, where we prescribe the inhomogeneous boundary 
conditions $\vec g(\vec z) = ([1 - z_2^2 - z_3^2]_+, 0, 0)^T$ 
in order to model a Poiseuille-type flow. For the remaining parameters we set
$\alpha_- = 0.05$, $\alpha_+ = 0.1$ and $\vartheta = 100$.
We notice that during the evolution the membrane in
Figure~\ref{fig:vartheta_cons3d} deforms more than in the corresponding
simulation with only a single phase $\phaseC^0 = 1$, 
see \citet[Figure~9]{nsns}.
In particular, we observe that the $+1$ phase, which
prefers a relatively flat surface, forces the surface to remain deformed also
long after it has left the constriction.
\begin{figure}
\center
\mbox{
\includegraphics[angle=-0,width=0.33\textwidth]{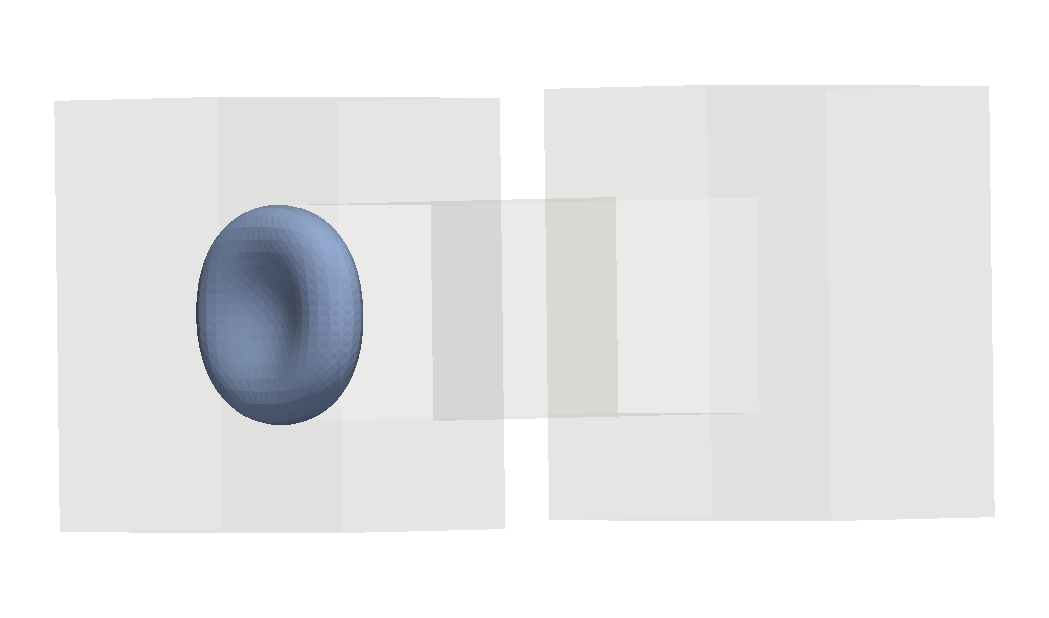}
\includegraphics[angle=-0,width=0.33\textwidth]{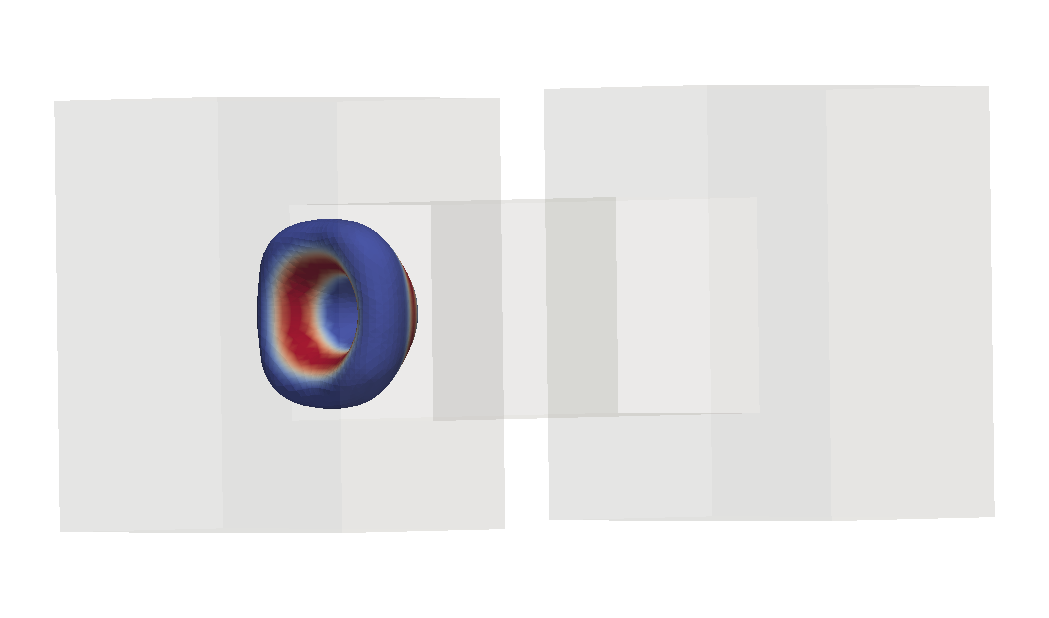}
\includegraphics[angle=-0,width=0.33\textwidth]{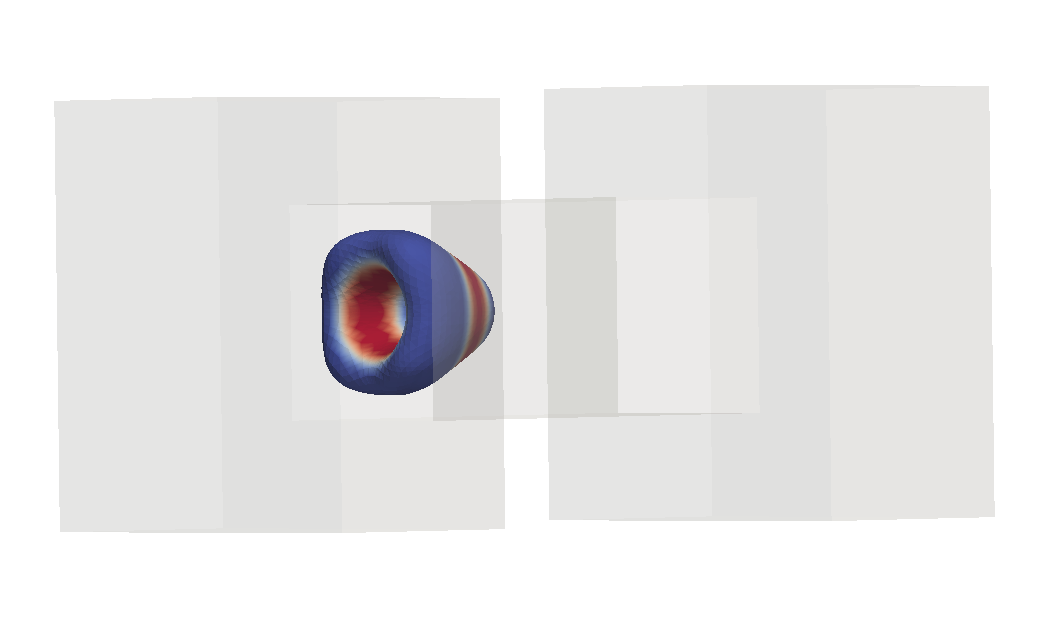}}
\mbox{
\includegraphics[angle=-0,width=0.33\textwidth]{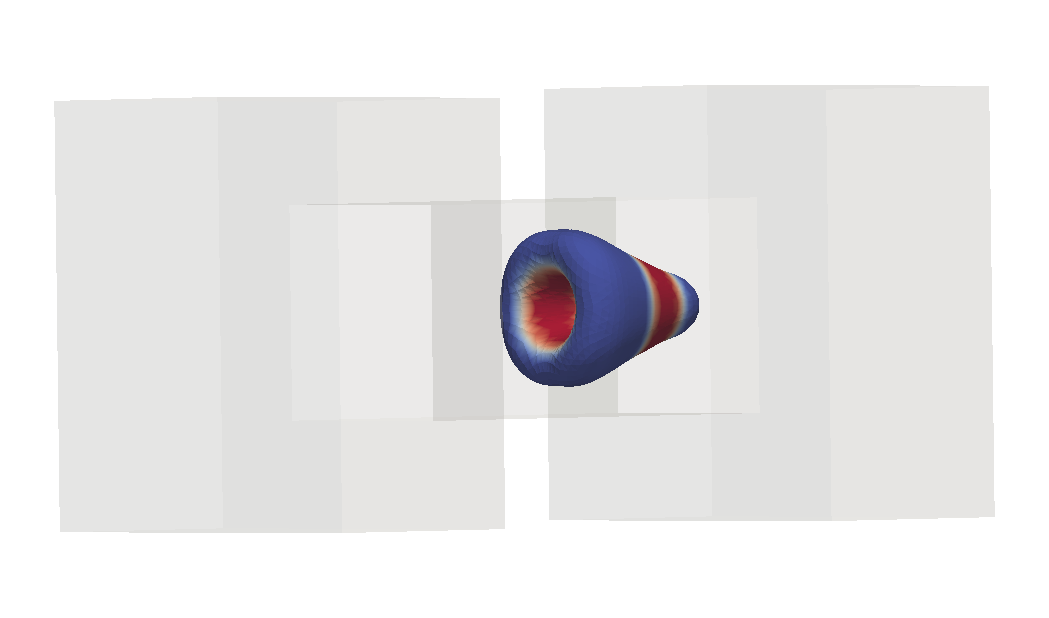}
\includegraphics[angle=-0,width=0.33\textwidth]{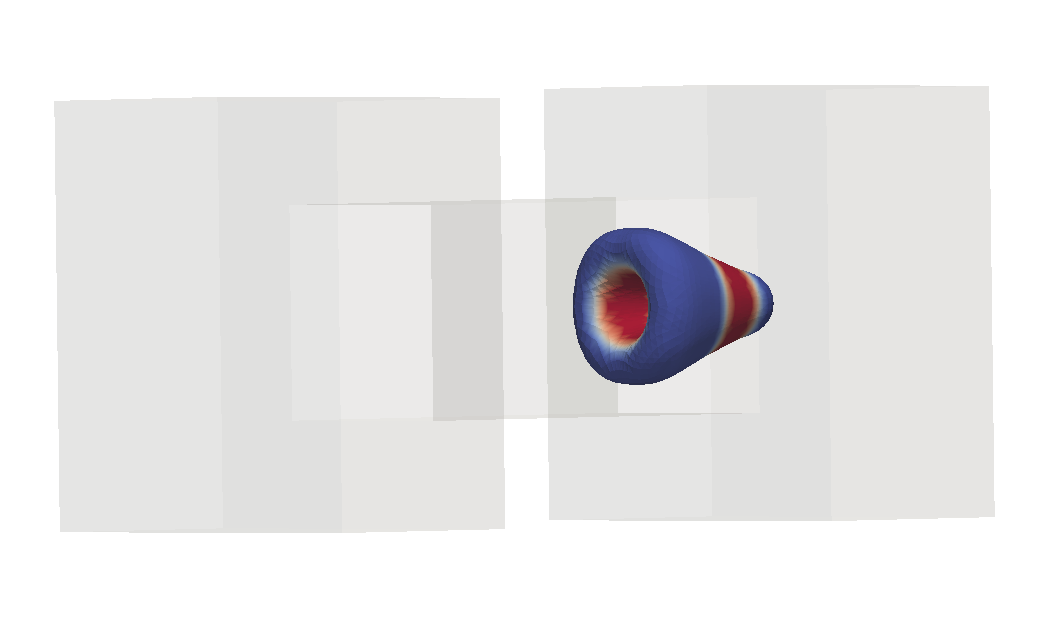}
\includegraphics[angle=-0,width=0.33\textwidth]{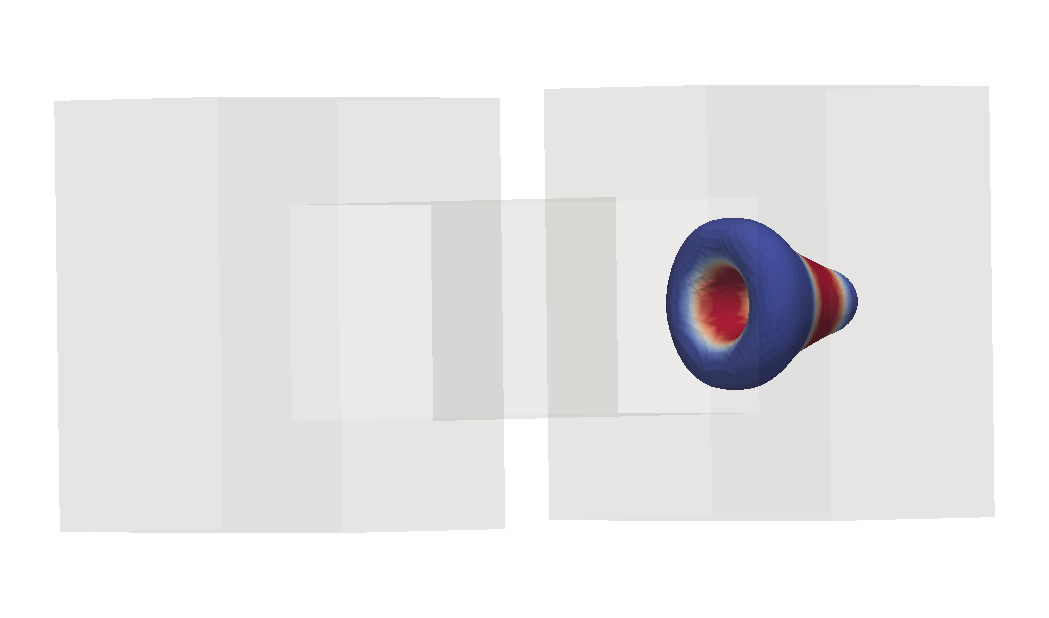}}
\mbox{
\includegraphics[angle=-0,width=0.33\textwidth]{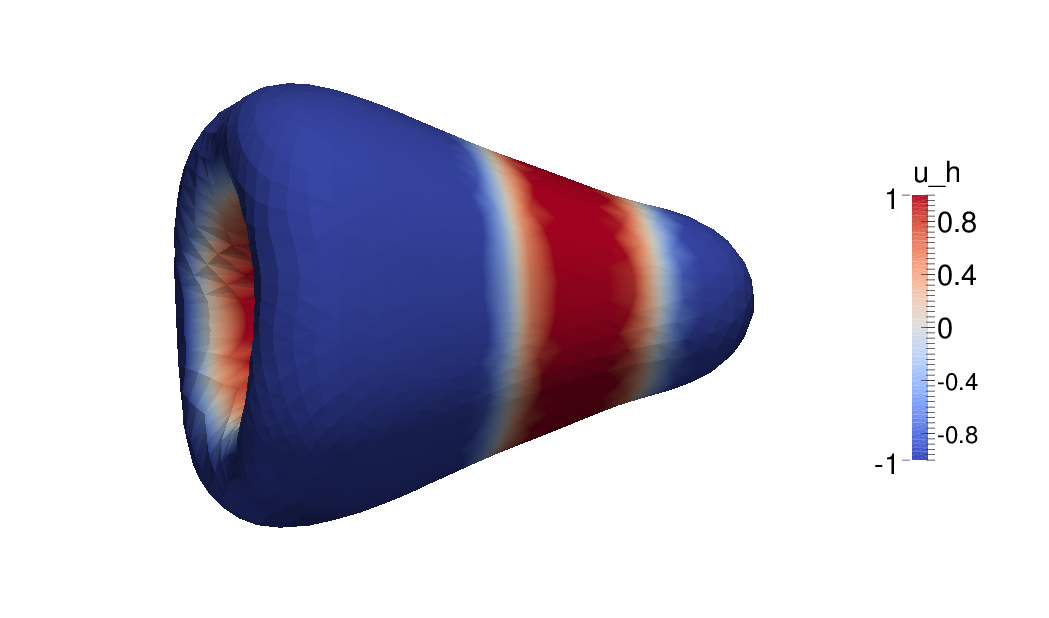}
\includegraphics[angle=-0,width=0.33\textwidth]{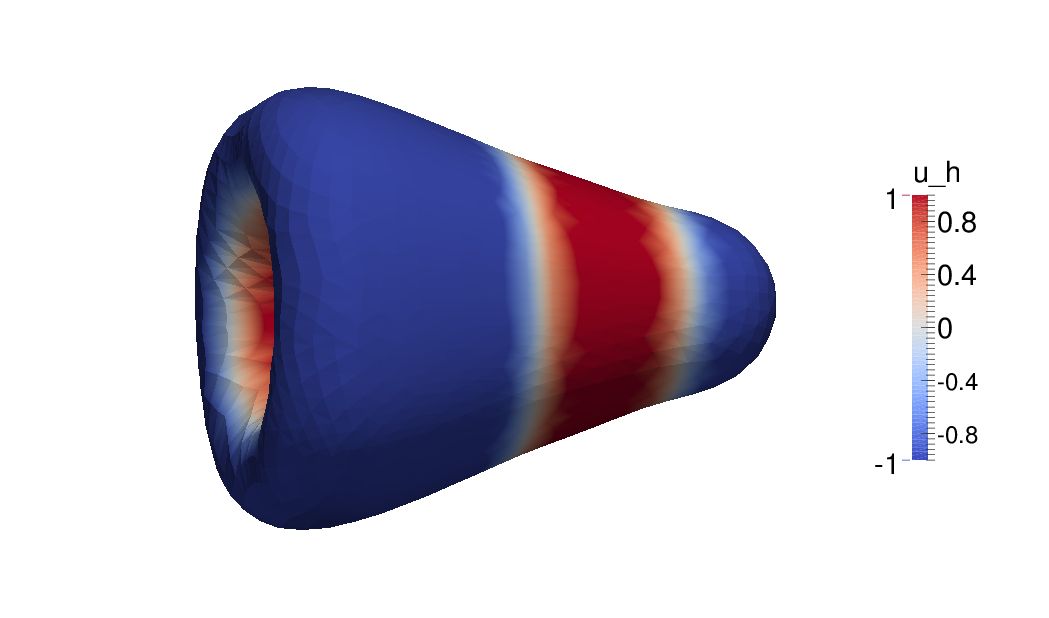}
\includegraphics[angle=-0,width=0.33\textwidth]{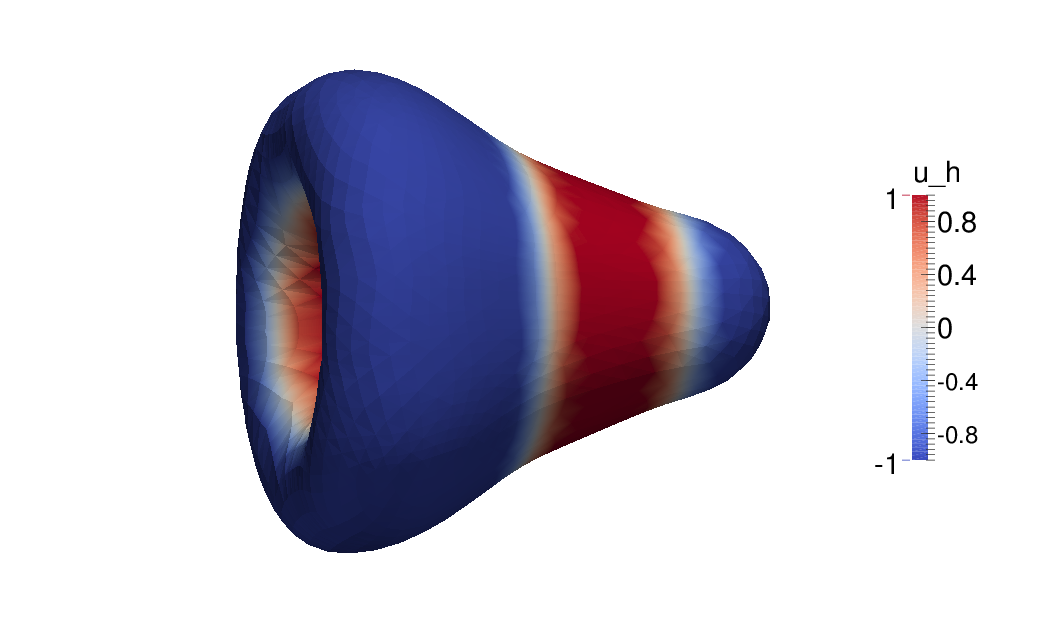}}
\center
\caption{
($\alpha_- = 0.05$, $\alpha_+ = 0.1$, $\spont_\pm = 0$, $\beta = 1$,
$\vartheta=100$)
Flow through a constriction. 
Plots of $\phaseC^m$ on $\Gamma^m$ at times $t=0,\ 0.3,\ 0.5,\ 1,\ 1.2,\ 1.5$.
Below we show enlarged plots of $\phaseC^m$ on $\Gamma^m$ at times $t=1,\ 1.2,\ 1.5$.
}
\label{fig:vartheta_cons3d}
\end{figure}%

In Figure~\ref{fig:spinodal2} we show a numerical experiment for 
spinodal decomposition on a membrane,
starting from a random distribution of phases with mean value $-0.4$.
The shape has surface area $35.7$, and the triangulations $\Gamma^m$ satisfy
$(K_\Gamma,J_\Gamma) = (6146, 12288)$.
This means that for our chosen value of $\gamma=0.1$, there are on
average about $6$ elements across the interfacial region on $\Gamma^m$.
\begin{figure}
\center
\mbox{
\includegraphics[angle=-0,width=0.33\textwidth]{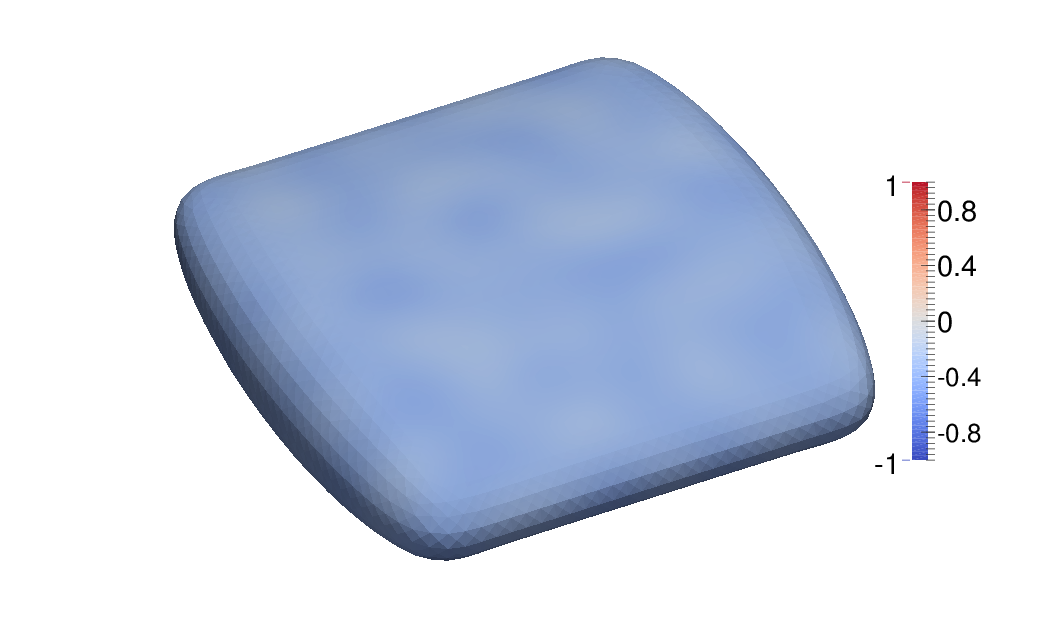}
\includegraphics[angle=-0,width=0.33\textwidth]{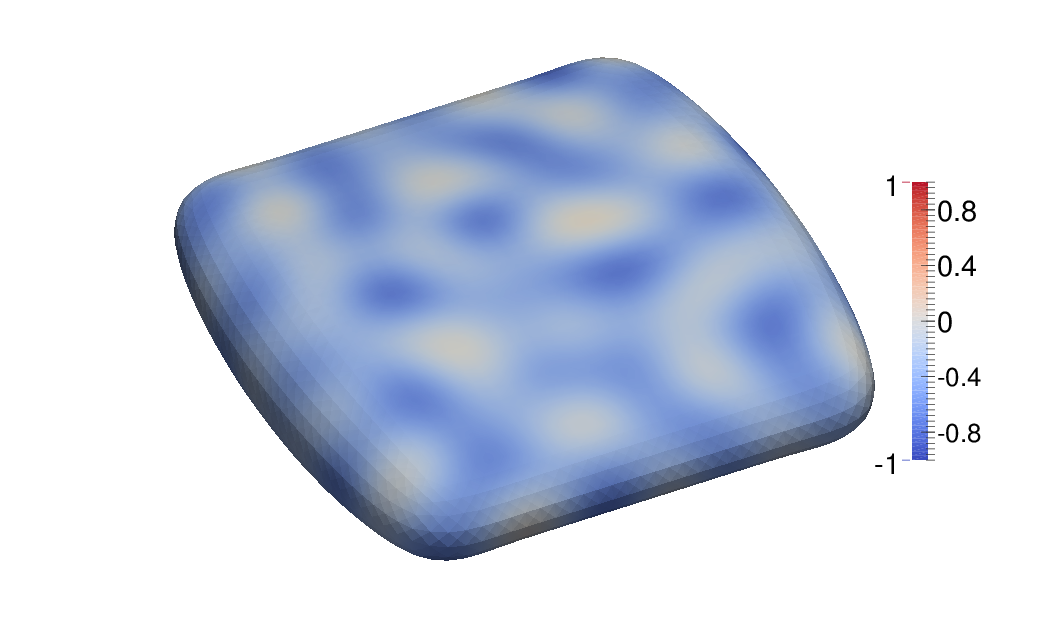}
\includegraphics[angle=-0,width=0.33\textwidth]{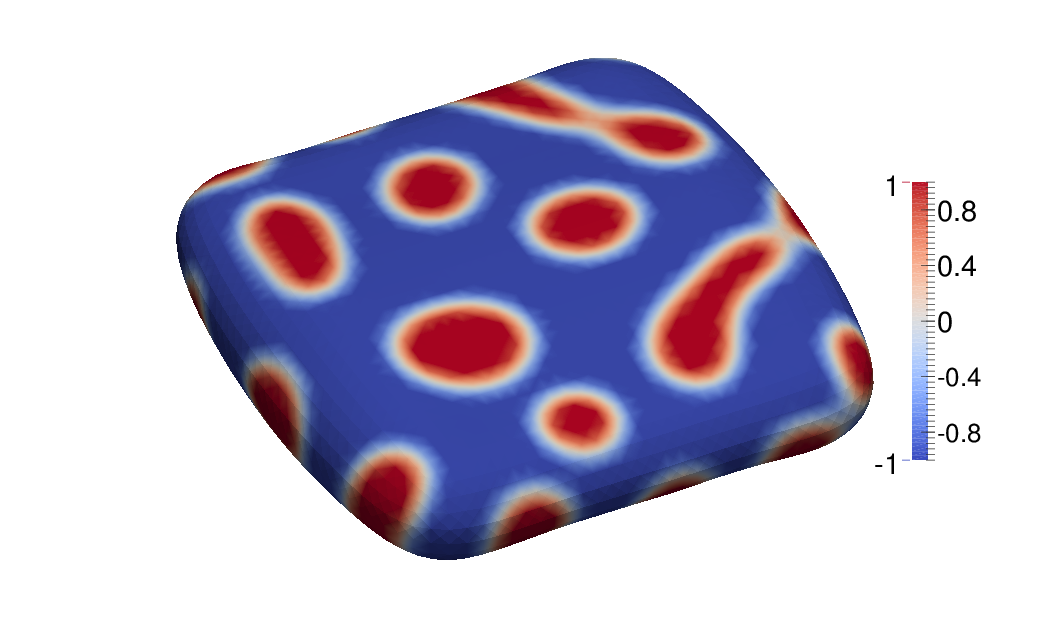}}
\mbox{
\includegraphics[angle=-0,width=0.33\textwidth]{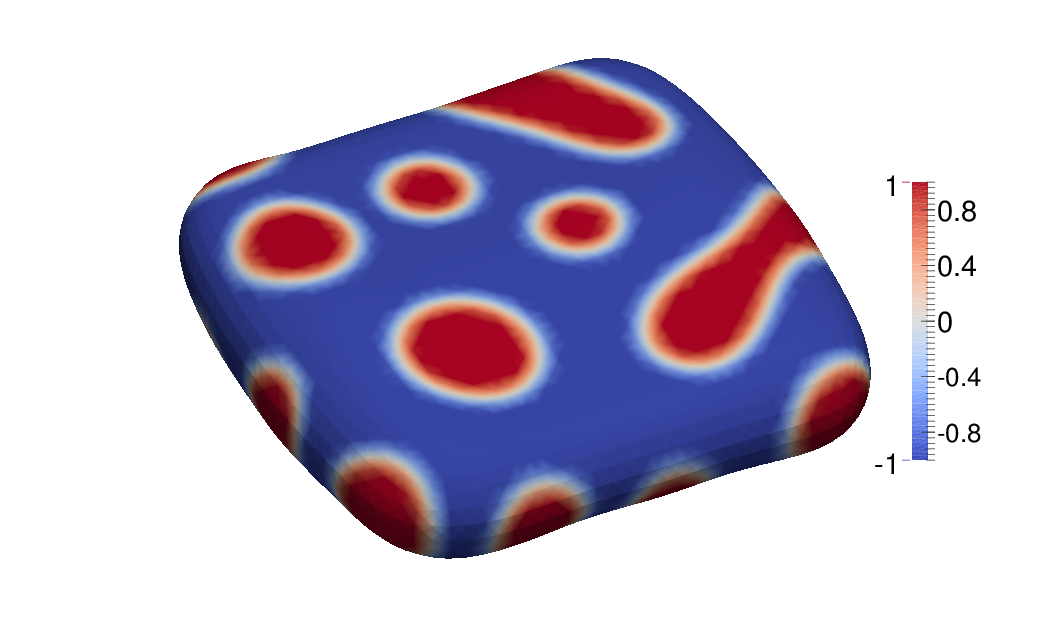}
\includegraphics[angle=-0,width=0.33\textwidth]{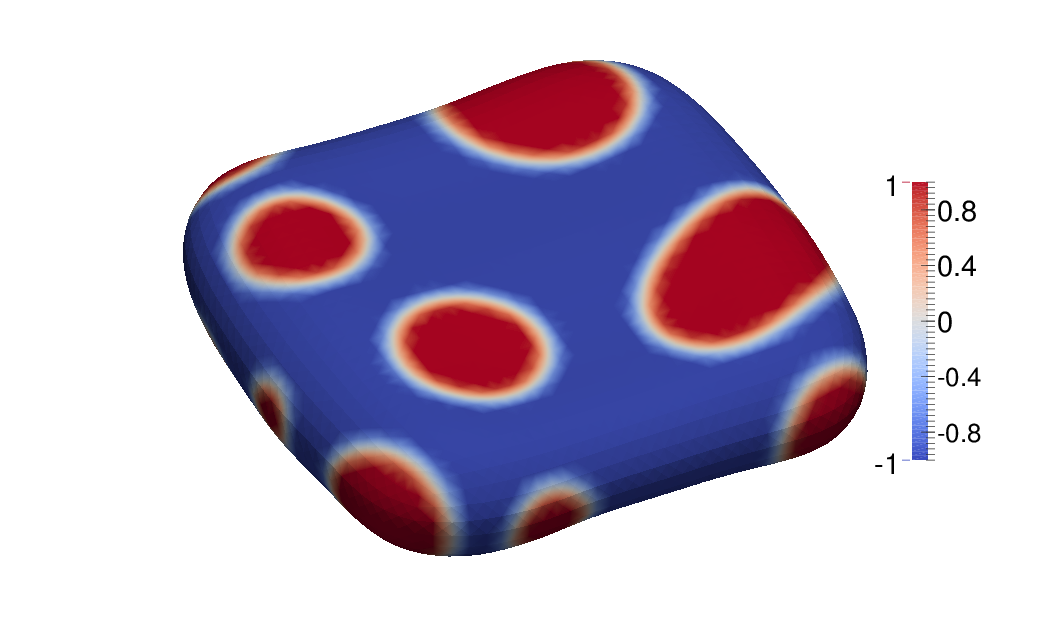}
\includegraphics[angle=-0,width=0.33\textwidth]{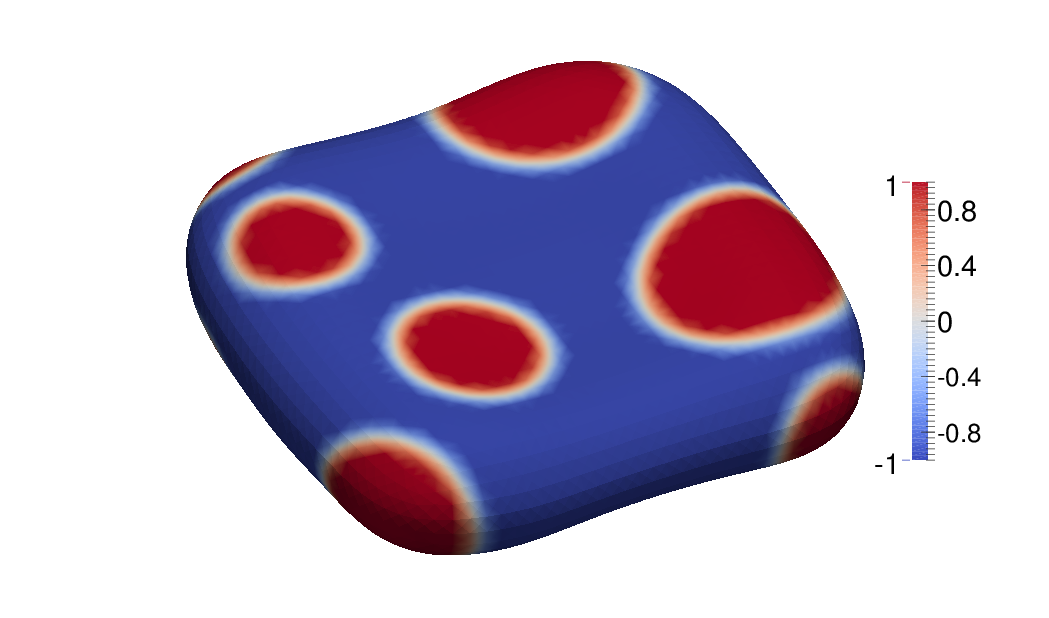}}
\includegraphics[angle=-90,width=0.45\textwidth]{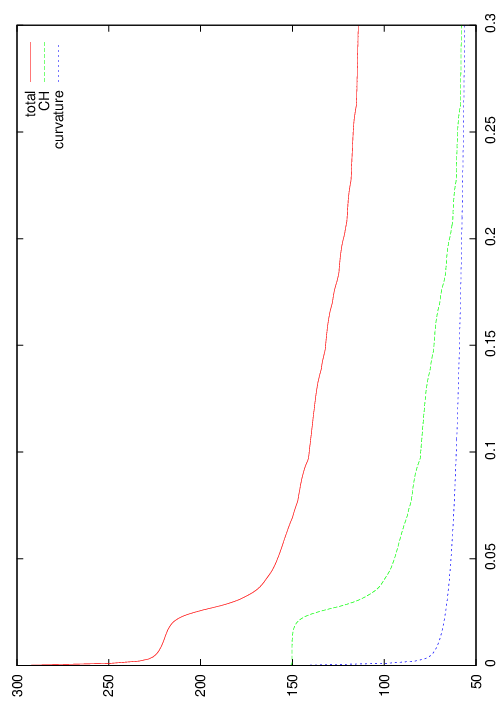}
\caption{($\alpha_\pm = 1$, $\spont_\pm = 0$, $\beta = 1$)
Spinodal decomposition on a membrane.
Plots of $\phaseC^m$ on $\Gamma^m$ at times $t=0.01,\ 0.02,\,0.05,\
0.1,\ 0.2,\ 0.3$.
Below a superimposed plot of the total discrete energy $\mathcal{E}^h_{total}$, 
the discrete Cahn--Hilliard energy, and the discrete curvature 
energy over $[0,0.3]$.
}
\label{fig:spinodal2}
\end{figure}%
Similarly, in Figure~\ref{fig:7spinodal}
we show the evolution for spinodal decomposition on a seven-arm
surface, where the initial phase variable is $\phaseC^0 = -0.4$ constant.
The shape has surface area $10.5$, and the triangulations $\Gamma^m$ satisfy
$(K_\Gamma,J_\Gamma) = (2314, 4624)$.
This means that for our chosen value of $\gamma=0.2$, there are on
average about $9$ elements across the interfacial region on $\Gamma^m$.
For the phase parameters we choose $\spont_- = -0.5$ and 
$\spont_+ = -2$.
The spontaneous curvature of the $+1$ phase leads to a preference of the $+1$
phase to be curved away from the outer normal. In accordance with this 
remark we observe that the $+1$ phase appears after the phase separation at 
the more highly curved tips of the fingers.
\begin{figure}
\center
\mbox{
\includegraphics[angle=-0,width=0.33\textwidth]{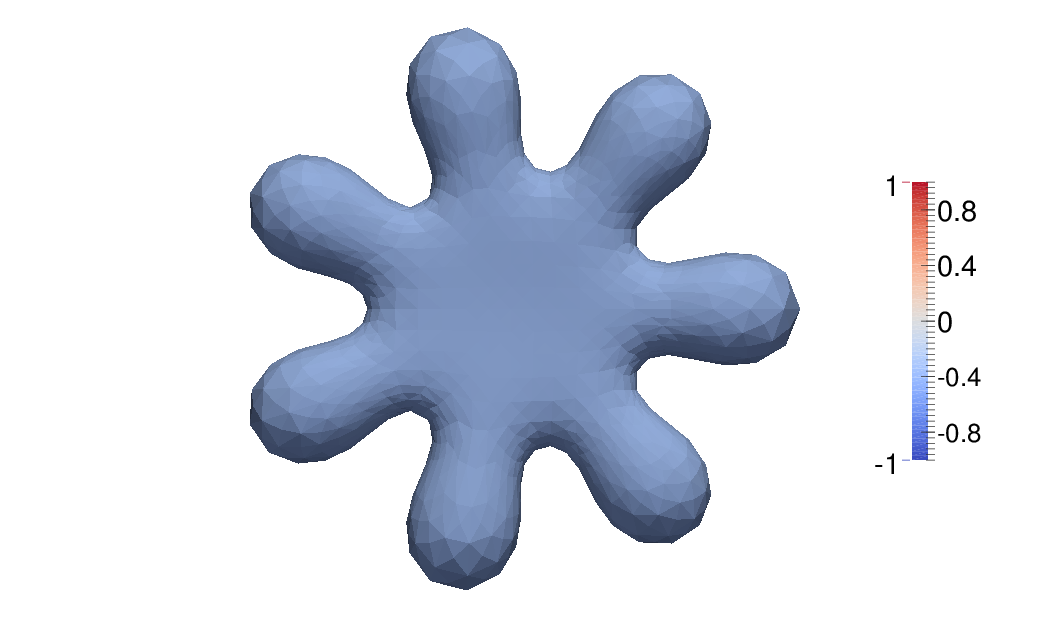}
\includegraphics[angle=-0,width=0.33\textwidth]{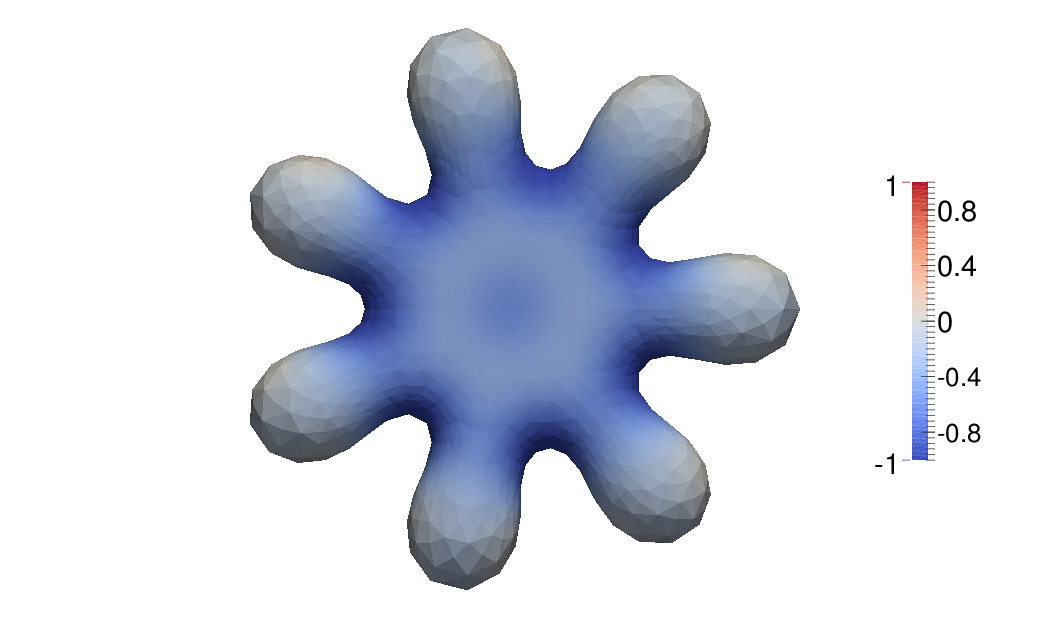}
\includegraphics[angle=-0,width=0.33\textwidth]{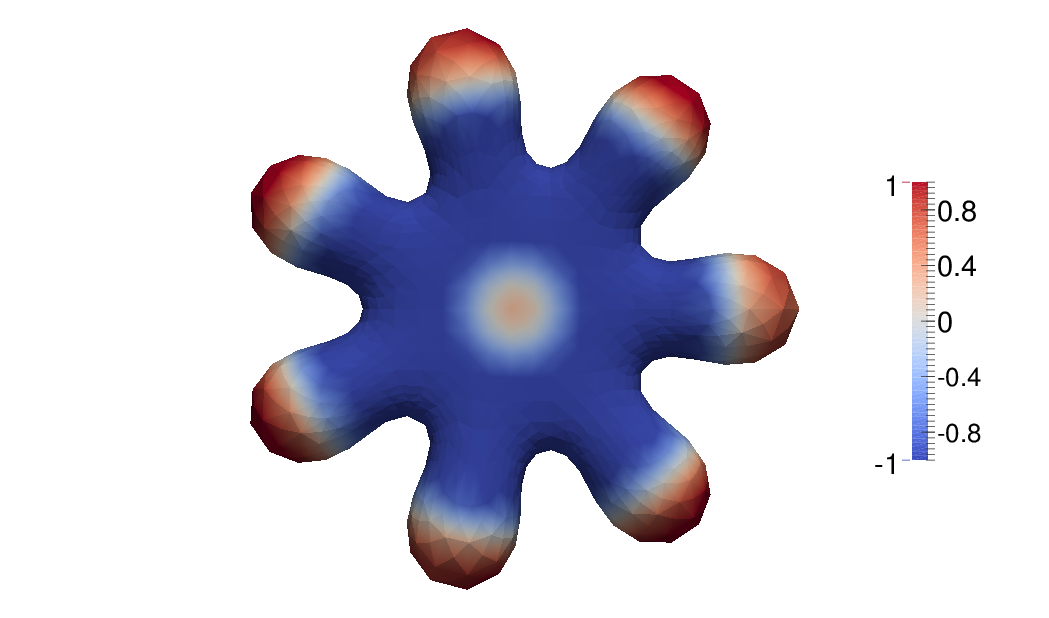}}
\mbox{
\includegraphics[angle=-0,width=0.33\textwidth]{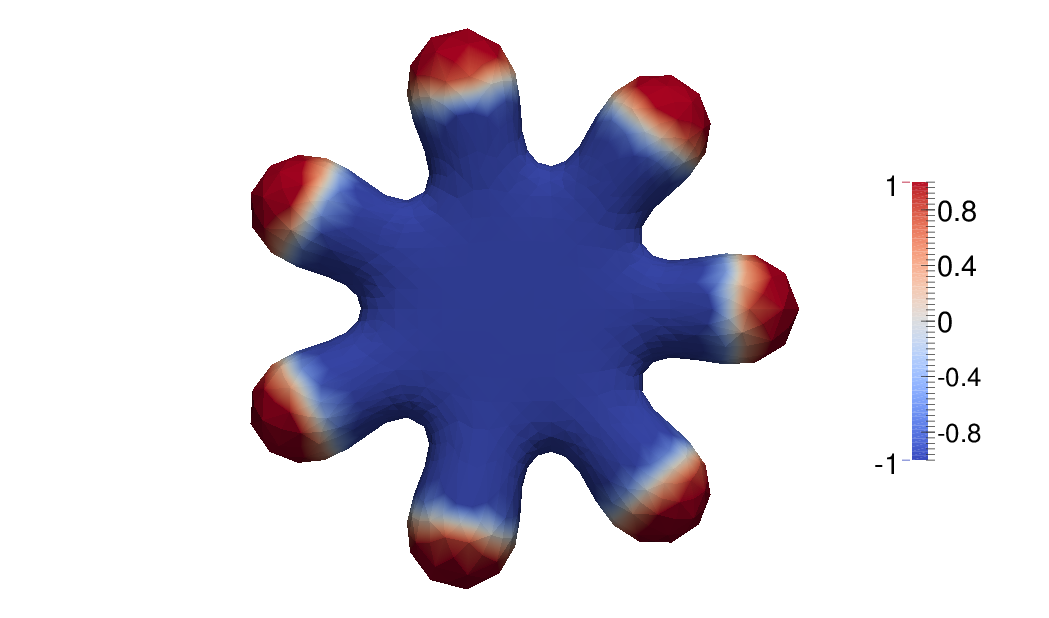}
\includegraphics[angle=-0,width=0.33\textwidth]{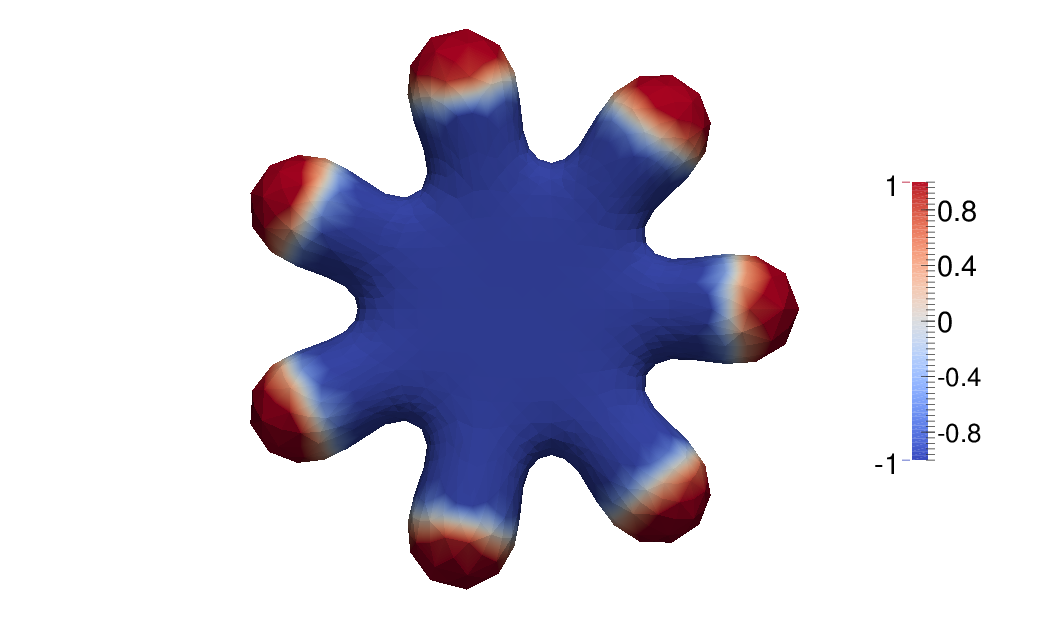}
\includegraphics[angle=-0,width=0.33\textwidth]{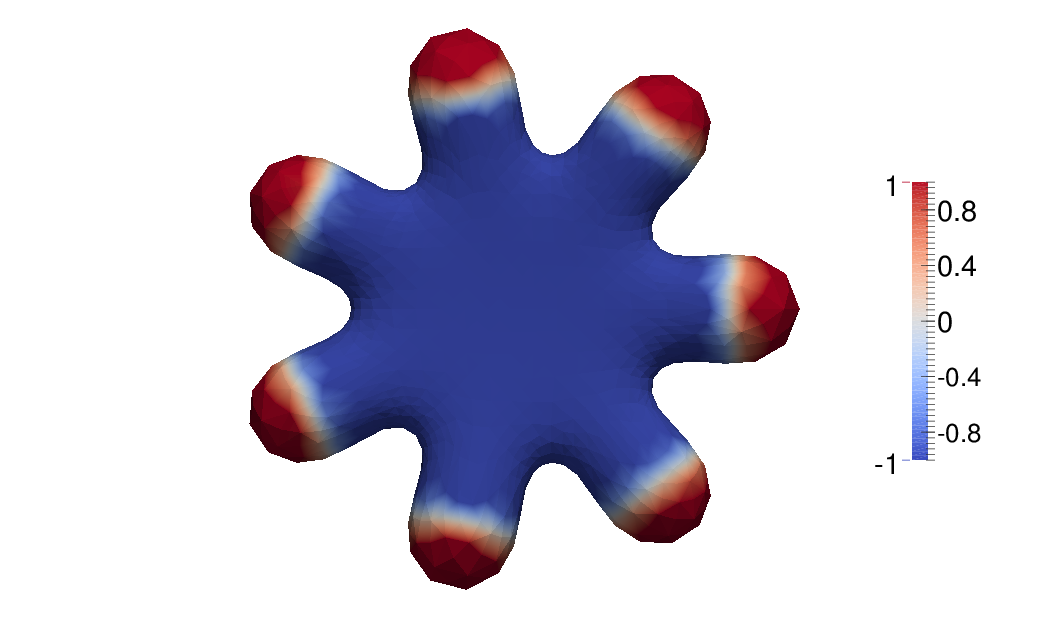}}
\includegraphics[angle=-90,width=0.45\textwidth]{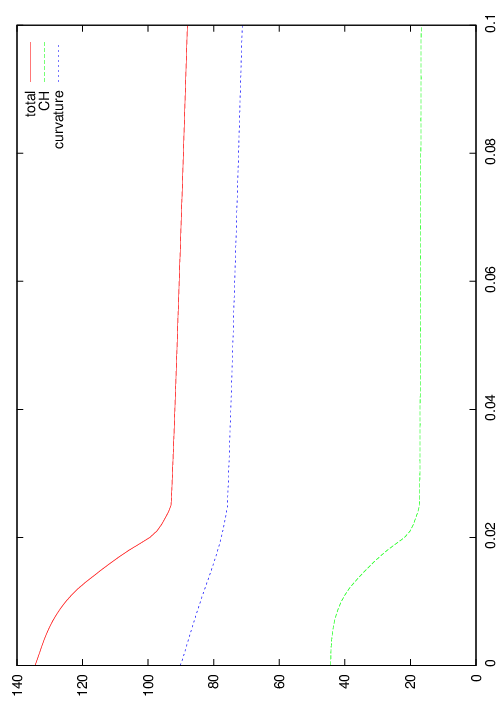}
\caption{($\alpha_\pm=1$, $\spont_- = -0.5$, $\spont_+ = -2$, $\beta=1$)
Spinodal decomposition on a seven-arm membrane.
Plots of $\phaseC^m$ on $\Gamma^m$ at times $t=0,\ 0.01,\ 0.02,\ 0.03,\
0.05,\ 0.1$.
Below a superimposed plot of the total discrete energy $\mathcal{E}^h_{total}$, 
the discrete Cahn--Hilliard energy, and the discrete curvature 
energy over $[0,0.1]$.
}
\label{fig:7spinodal}
\end{figure}%

In the following, we present some computations for $\alpha^G_\pm \not=0$. 
When we repeat the experiment in Figure~\ref{fig:fig12}
for the choices $\alpha^G_- = 0.5$, $\alpha^G_+ = 0$
and $\alpha^G_- = 0$, $\alpha^G_+ = 0.5$, we obtain the results in
Figures~\ref{fig:fig18} and \ref{fig:fig19}, respectively.
We note that for this choice of parameters, the bound (\ref{eq:alphaGbound2}) 
holds.
\begin{figure}
\center
\mbox{
\includegraphics[angle=-0,width=0.25\textwidth]{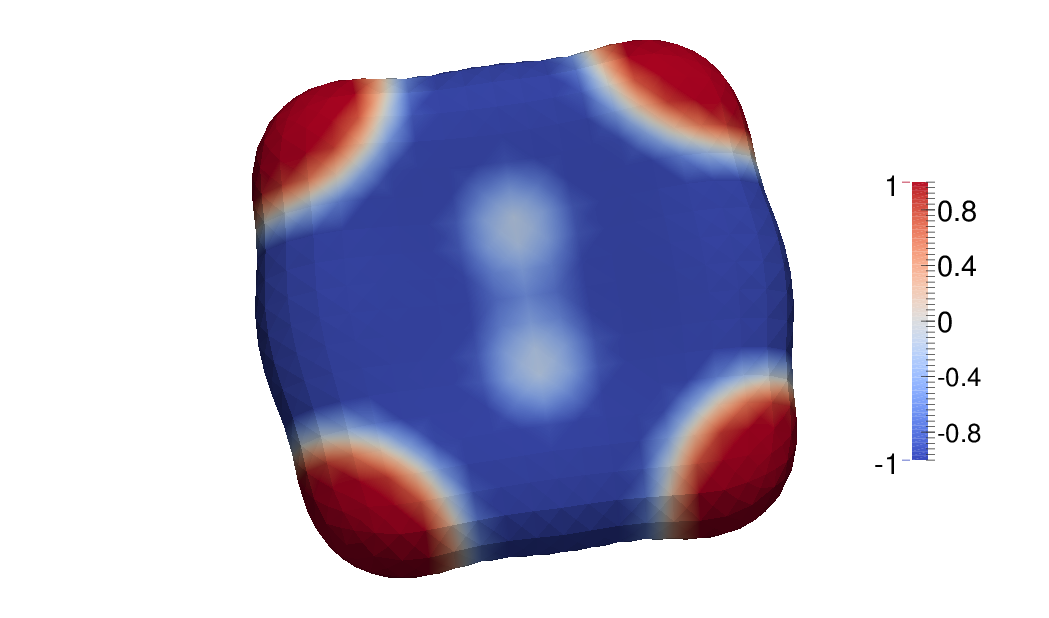}
\includegraphics[angle=-0,width=0.25\textwidth]{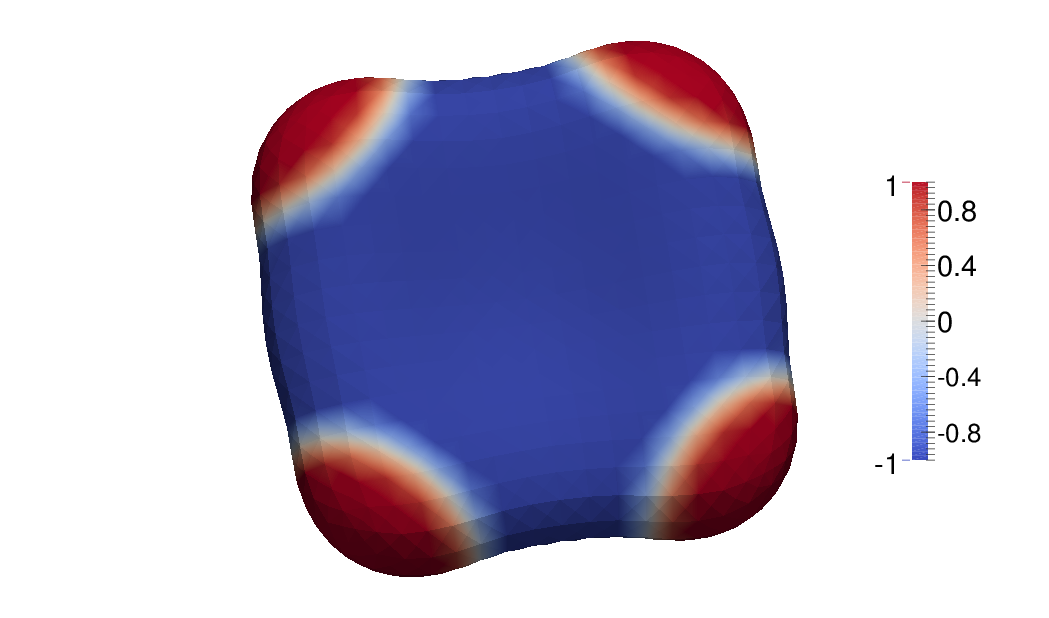}
\includegraphics[angle=-0,width=0.25\textwidth]{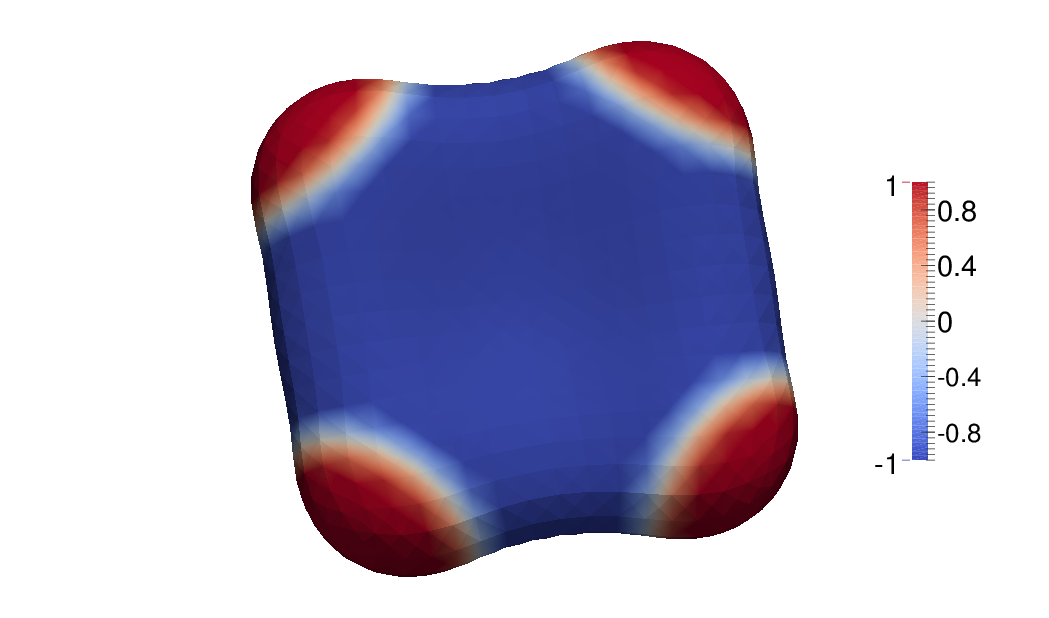}
\includegraphics[angle=-0,width=0.25\textwidth]{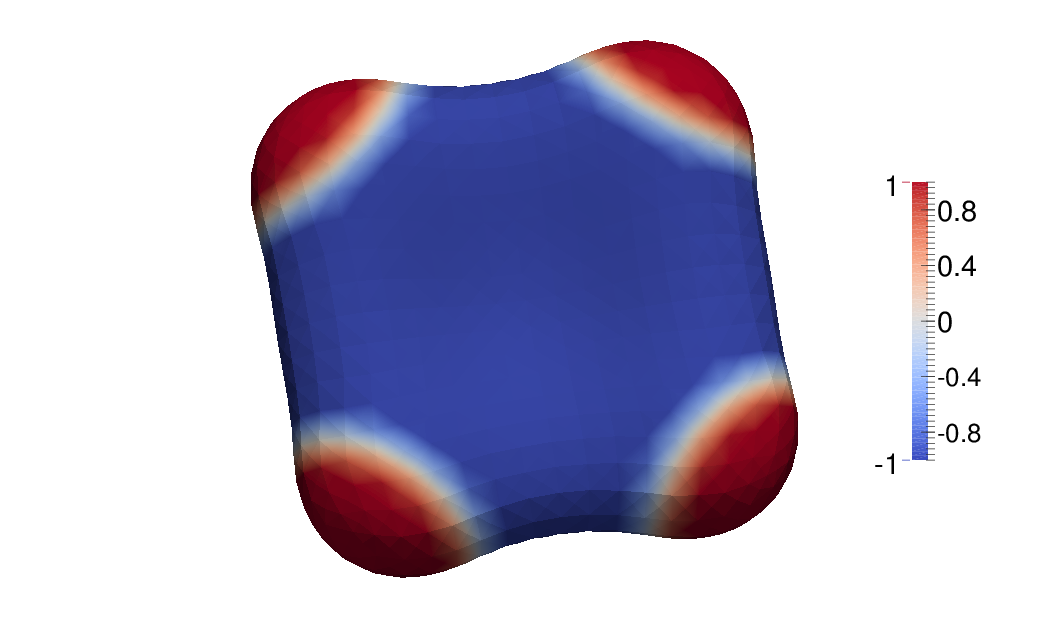}}
\includegraphics[angle=-90,width=0.33\textwidth]{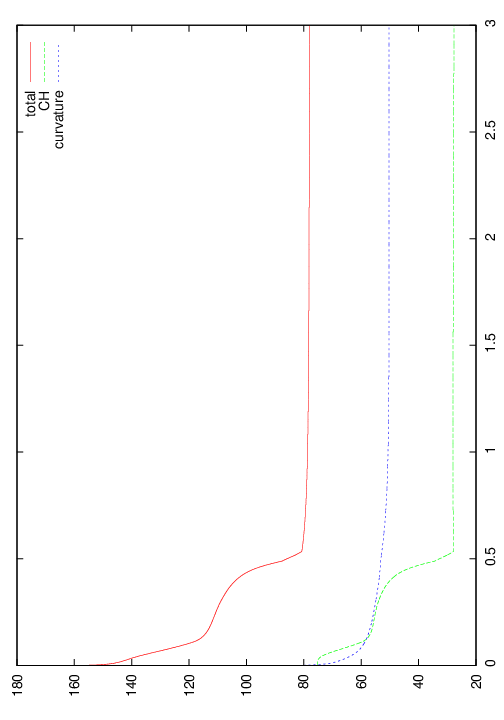}
\caption{($\alpha_\pm = 1$, $\spont_\pm = 0$, $\alpha^G_-=0.5$, $\alpha^G_+=0$,
$\beta = 1$)
Plots of $\phaseC^m$ on $\Gamma^m$ at times $t=0.5,\ 1,\ 2,\ 3$.
Below a superimposed plot of the total discrete energy $\mathcal{E}^h_{total}$, 
the discrete Cahn--Hilliard energy, and the discrete curvature 
energy over $[0,3]$.
}
\label{fig:fig18}
\end{figure}%
\begin{figure}
\center
\mbox{
\includegraphics[angle=-0,width=0.25\textwidth]{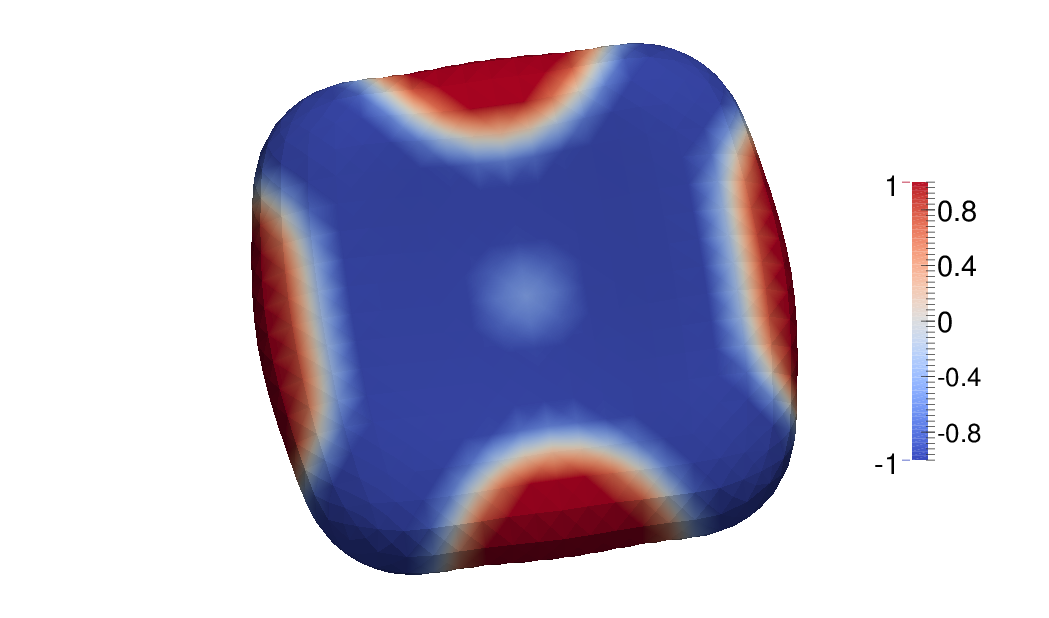}
\includegraphics[angle=-0,width=0.25\textwidth]{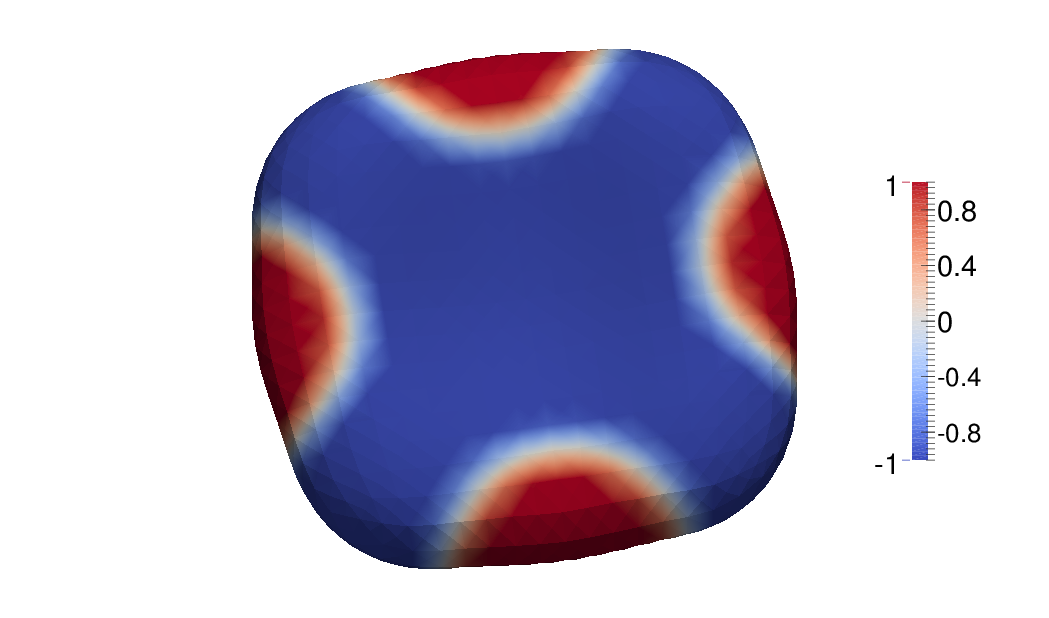}
\includegraphics[angle=-0,width=0.25\textwidth]{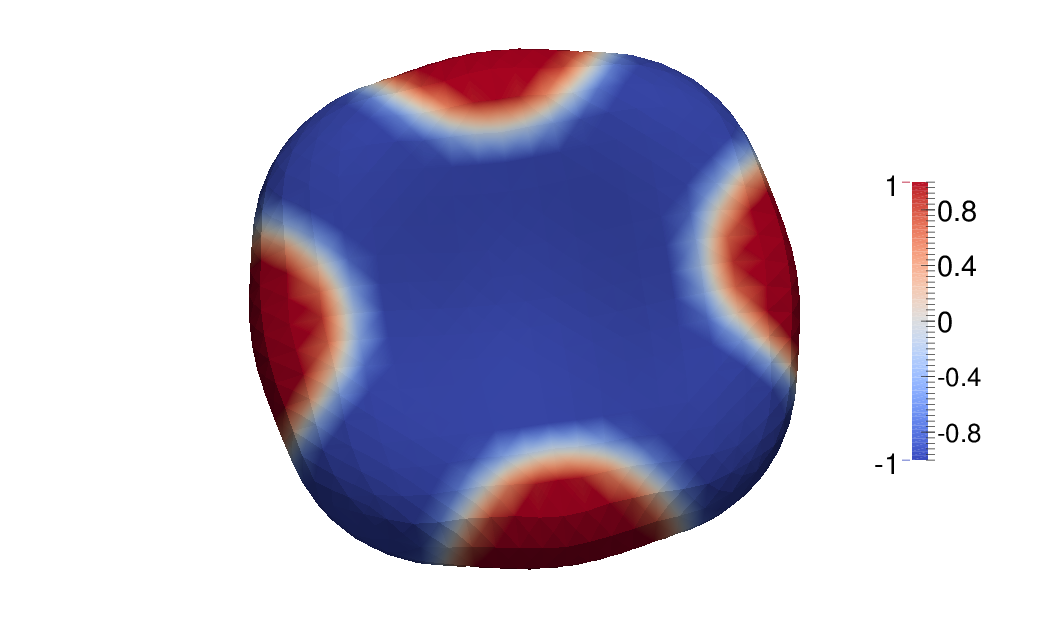}
\includegraphics[angle=-0,width=0.25\textwidth]{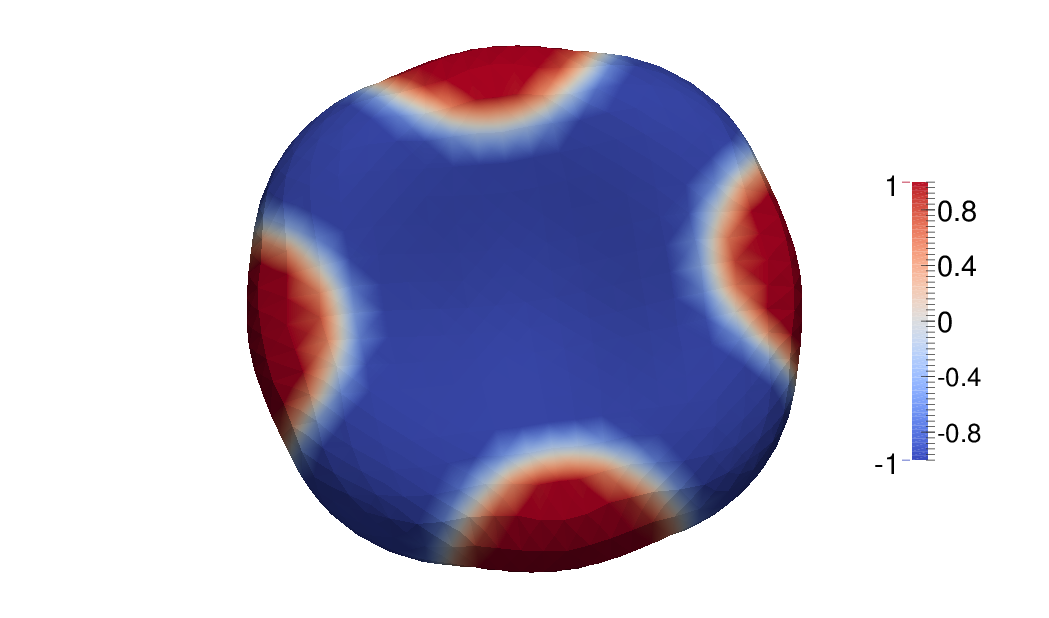}}
\includegraphics[angle=-90,width=0.33\textwidth]{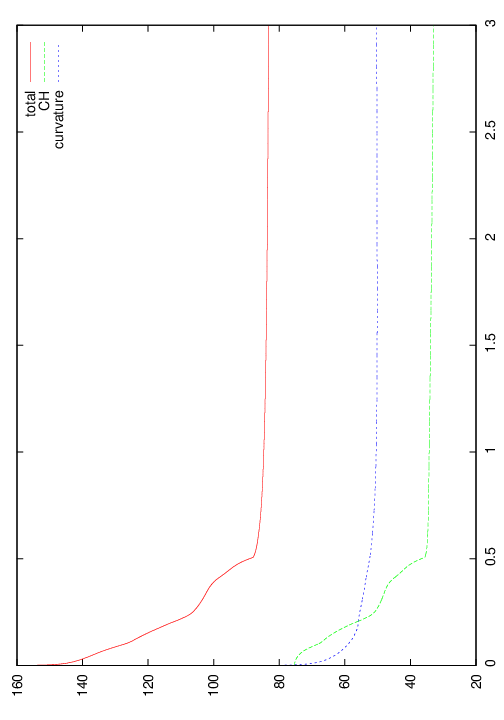}
\caption{($\alpha_\pm = 1$, $\spont_\pm = 0$, $\alpha^G_-=0$, $\alpha^G_+=0.5$,
$\beta = 1$)
Plots of $\phaseC^m$ on $\Gamma^m$ at times $t=0.5,\ 1,\ 2,\ 3$.
Below a superimposed plot of the total discrete energy $\mathcal{E}^h_{total}$, 
the discrete Cahn--Hilliard energy, and the discrete curvature 
energy over $[0,3]$.
}
\label{fig:fig19}
\end{figure}%
Comparing the results in Figure~\ref{fig:fig12} with the ones in
Figures~\ref{fig:fig18} and \ref{fig:fig19} clearly shows the
influence of the Gaussian energy terms.
In Figure~\ref{fig:fig18} the region of the largest Gaussian curvature is 
in the $+1$ phase and the region of the smallest Gaussian curvature is in 
the $-1$ phase. This is in accordance with the fact that the energy 
penalizes Gaussian curvature only in  the $-1$ phase. On the other hand, in 
Figure~\ref{fig:fig19} the region with the largest Gaussian curvature is the 
$-1$ phase and the $+1$ phase has a smaller Gaussian curvature when 
compared to Figure~\ref{fig:fig18}.

\cleardoublepage
\begin{appendix}
\setcounter{equation}{0}
\renewcommand{\theequation}{\Alph{section}.\arabic{equation}}
\section{Strong and weak formulations} \label{sec:A}
The goal of this Appendix is to relate the weak formulation,
(\ref{eq:LM3a}--e), (\ref{eq:CHb}), 
of the first variations with respect to the geometry and $\phasec$
of the energy in (\ref{eq:E}), 
to the strong formulations (\ref{eq:fGamma})
and (\ref{eq:strongCHb}), respectively. 
As we allow for
tangential motion, it is necessary to take into account variations which
are not necessarily  normal. This is in contrast to 
\cite{ElliottS10a}, where only normal variations were considered.

We recall that $\nabs = (\partial_{s_1},\ldots,\partial_{s_d})^T$, and note
from \citet[Lemma~2.6]{DziukE13} 
that for sufficiently smooth $\phi$ it holds that
\begin{equation} \label{eq:DE13}
\partial_{s_k}\,\partial_{s_i}\,\phi - 
\partial_{s_i}\,\partial_{s_k}\,\phi =
[(\nabs\,\vec\nu)\,\nabs\,\phi]_i\,\nu_k
- [(\nabs\,\vec\nu)\,\nabs\,\phi]_k\,\nu_i \qquad\forall\ i,k
  \in\{1,\ldots,d\}
\quad\text{on } \Gamma(t)\,.
\end{equation}
It follows from (\ref{eq:Weingarten}), (\ref{eq:DE13}) and 
(\ref{eq:secondform}) that
\begin{equation} \label{eq:new}
\Delta_s\,\vec\nu = \nabs\,(\nabs\,.\,\vec\nu) - |\nabs\,\vec\nu|^2\,\vec\nu=
- |\nabs\,\vec\nu|^2\,\vec\nu - \nabs\,\varkappa\,.
\end{equation}
Moreover, we have from (\ref{eq:secondform}),
(\ref{eq:secvarlocal}), (\ref{eq:normvar}), (\ref{eq:Weingarten}), 
(\ref{eq:PGamma}) and (\ref{eq:new}) that
\begin{align}
\matparteps\,\varkappa & = - \matparteps\,(\nabs\,.\,\vec\nu)
= -[ \nabs\,\vec\chi - 2\,\mat D_s(\vec\chi)] : \nabs\,\vec\nu - \nabs\,.\,
(\matparteps\,\vec\nu) \nonumber \\ &
= \nabs\,\vec\chi : \nabs\,\vec\nu + \nabs\,.\,( [\nabs\,\vec\chi]^T\,\vec\nu)
= 2\,\nabs\,\vec\chi : \nabs\,\vec\nu + (\Delta_s\,\vec\chi)\,\vec\nu
= \Delta_s\,(\vec\chi\,.\,\vec\nu) - \vec\chi\,.\, \Delta_s\,\vec\nu
\nonumber \\ &
= \Delta_s\,(\vec\chi\,.\vec\nu) +
|\nabs\,\vec\nu|^2\,(\vec\chi\,.\,\vec\nu) + \vec\chi\,.\,\nabs\,\varkappa\,.
\label{eq:matpartukappa}
\end{align}

\subsection{Derivation of the strong formulation}
We admit general variations $\vec\chi = \chi\,\vec\nu + \vec\chi_{\tan}$,
where $\vec\chi_{\tan}\,.\,\vec\nu = 0$, 
of (\ref{eq:E}) with respect to $\Gamma$,
whereas in 
\cite{ElliottS10a} only normal variations $\vec\chi = \chi\,\vec\nu$ of the
geometry are considered.

We consider first the bending energy in (\ref{eq:E}) 
and have from (\ref{eq:DElem5.2eps}),
on recalling (\ref{eq:bendingb}), that
\begin{align} &
\left[ \deldel\Gamma \left\langle b(\varkappa,\phasec) , 1
\right\rangle_{\Gamma(t)} \right] (\vec{\chi}) 
= \left\langle \alpha(\phasec)\,(\varkappa-\spont(\phasec)),\matparteps\,\varkappa
\right\rangle_{\Gamma(t)}
+ \left\langle b(\varkappa,\phasec) ,\nabs\,.\,\vec\chi
\right\rangle_{\Gamma(t)}.
\label{eq:appvar}
\end{align}
We obtain from (\ref{eq:appvar}) and (\ref{eq:matpartukappa}),
on recalling (\ref{eq:nabszeta}), that
\begin{subequations}
\begin{align} &
\left[ \deldel\Gamma \left\langle  b(\varkappa,\phasec), 1
\right\rangle_{\Gamma(t)} \right] (\vec{\chi}) \nonumber \\ &\quad
= \left\langle \Delta_s\,[\alpha(\phasec)\,(\varkappa-\spont(\phasec))] +
\alpha(\phasec)\,[(\varkappa-\spont(\phasec))\,|\nabs\,\vec\nu|^2 -
\tfrac12\,(\varkappa-\spont(\phasec))^2\,\varkappa], \vec\chi\,.\,\vec\nu
\right\rangle_{\Gamma(t)} \nonumber \\ &\qquad
- \left\langle \nabs\,b(\varkappa,\phasec), \vec\chi \right\rangle_{\Gamma(t)}
+ \left\langle \alpha(\phasec)\,(\varkappa-\spont(\phasec))\,\nabs\,\varkappa, \vec\chi
 \right\rangle_{\Gamma(t)} 
\nonumber \\ &\quad
= \left\langle \Delta_s\,[\alpha(\phasec)\,(\varkappa-\spont(\phasec))] +
\alpha(\phasec)\,[(\varkappa-\spont(\phasec))\,|\nabs\,\vec\nu|^2 -
\tfrac12\,(\varkappa-\spont(\phasec))^2\,\varkappa], \vec\chi\,.\,\vec\nu
\right\rangle_{\Gamma(t)} \nonumber \\ &\qquad
- \left\langle b_{,\phasec}(\varkappa,\phasec), 
\vec\chi \,.\,\nabs\,\phasec \right\rangle_{\Gamma(t)} .
\label{eq:appvar2}
\end{align}
In addition, it holds that
\begin{equation} \label{eq:appvar2c}
\left[ \deldel{\phasec} \left\langle b(\varkappa,\phasec) , 1
\right\rangle_{\Gamma(t)} \right] (\eta) =
\left\langle b_{,\phasec}(\varkappa,\phasec) ,\eta \right\rangle_{\Gamma(t)} .
\end{equation}
\end{subequations}
Choosing just a normal variation, $\vec\chi = \chi\,\vec\nu$, means that 
(\ref{eq:appvar2},b) collapses to the result in \citet[(4.5)]{ElliottS10a},
on noting (\ref{eq:bc}). 

Next, we consider the interfacial energy in (\ref{eq:E}). 
We have from (\ref{eq:DElem5.2eps}), (\ref{eq:nabsw2}) and (\ref{eq:nabszeta})
that
\begin{subequations}
\begin{align}
& \left[\deldel\Gamma\left\langle b_{CH}(\phasec), 1 \right\rangle_{\Gamma(t)}
\right](\vec\chi) =
- \gamma \left\langle \nabs\,\phasec, (\nabs\,\vec\chi)\,\nabs\,\phasec 
\right\rangle_{\Gamma(t)}
+ \left\langle \tfrac12\,\gamma\,|\nabs\,\phasec|^2 + \gamma^{-1}\,\Psi(\phasec),
\nabs\,.\,\vec\chi \right\rangle_{\Gamma(t)}
\nonumber \\ & =
- \left\langle (\tfrac12\,\gamma\,|\nabs\,\phasec|^2 + \gamma^{-1}\, \Psi(\phasec))
\,\varkappa, \vec\chi\,.\,\vec\nu \right\rangle_{\Gamma(t)} 
- \left\langle \nabs\,(\tfrac12\,\gamma\,|\nabs\,\phasec|^2 + \gamma^{-1}\, \Psi(\phasec))
, \vec\chi \right\rangle_{\Gamma(t)}
\nonumber \\ &\quad
+ \gamma \left\langle
\nabs\,.\,[(\nabs\,\phasec)\otimes(\nabs\,\phasec)], \vec\chi\right\rangle_{\Gamma(t)},
\label{eq:ECHvar}
\end{align}
where we have noted from (\ref{eq:nabszeta}) that
\begin{align*}
\left\langle \nabs\,\phasec, (\nabs\,\vec\chi)\,\nabs\,\phasec 
\right\rangle_{\Gamma(t)} = - \left\langle
\nabs\,.\,[(\nabs\,\phasec)\otimes(\nabs\,\phasec)], \vec\chi\right\rangle_{\Gamma(t)}.
\end{align*}
In addition, it holds that
\begin{equation} \label{eq:ECHvarc}
\left[ \deldel{\phasec} \left\langle  b_{CH}(\phasec), 1
\right\rangle_{\Gamma(t)} \right] (\eta) =
\left\langle - \gamma\,\Delta_s\,\phasec + \gamma^{-1}\, \Psi'(\phasec), 
\eta \right\rangle_{\Gamma(t)} .
\end{equation}
\end{subequations}
Once again, choosing a normal variation, $\vec\chi = \chi\,\vec\nu$, means that 
(\ref{eq:ECHvar},b) collapses to \citet[(4.8)]{ElliottS10a}, on noting that
\[
\vec\nu\,.\left(\nabs\,.\,[(\nabs\,\phasec)\otimes(\nabs\,\phasec)]\right) = 
- \nabs\,\vec\nu : [(\nabs\,\phasec)\otimes(\nabs\,\phasec)]\,.
\]

For $d=3$ only, we compute the first variation of the Gaussian curvature 
bending energy in (\ref{eq:E}). We start by deriving an expression for 
$\matparteps\,\Gauss$. On recalling (\ref{eq:Gauss3d}), we first compute
\begin{equation} \label{eq:noteA1}
\tfrac12\,\matparteps\,|\nabs\,\vec\nu|^2 = \nabs\,\vec\nu : 
\matparteps\,(\nabs\,\vec\nu)\,.
\end{equation}
{From} (\ref{eq:secvarwlocal}) we have that
\begin{equation*} 
\matparteps\,(\nabs\,\nu_i) = 
[ \nabs\,\vec\chi - 2\,\mat D_s(\vec \chi)]\,\nabs\,\nu_i +
\nabs\,(\matparteps\,\nu_i) \qquad i \in \{1,2,3\}\,,
\end{equation*}
yielding, on noting (\ref{eq:Weingarten}) and (\ref{eq:normvar}), that
\begin{align} \label{eq:noteA2}
\matparteps\,(\nabs\,\vec\nu) &
= \matparteps\,(\nabs\,\vec\nu)^T 
= [ \nabs\,\vec\chi - 2\,\mat D_s(\vec \chi)]\,(\nabs\,\vec\nu)^T +
[\nabs\,(\matparteps\,\vec\nu)]^T \nonumber \\ &
= [ \nabs\,\vec\chi - 2\,\mat D_s(\vec \chi)]\,\nabs\,\vec\nu +
[\nabs\,( [\nabs\,\vec\chi]^T\,\vec\nu)]^T\,.
\end{align}
We deduce from (\ref{eq:noteA1}), (\ref{eq:noteA2}), (\ref{eq:Weingarten}),
(\ref{eq:Ds}) and (\ref{eq:PGamma}) that
\begin{equation} \label{eq:noteA3}
\tfrac12\,\matparteps\,|\nabs\,\vec\nu|^2 = 
- \nabs\,\vec\nu : (\nabs\,\vec\chi)^T\,\nabs\,\vec\nu 
- \nabs\,\vec\nu : \nabs\,([\nabs\vec\chi]^T\,\vec\nu)
= T_1 + T_2\,.
\end{equation}
Adopting the standard summation convention, we have that
\begin{align}
T_1 & 
= - (\partial_{s_j}\,\nu_i)\,(\partial_{s_k}\,\chi_i)\,\partial_{s_j}\,\nu_k 
= - (\partial_{s_i}\,\nu_k)\,(\partial_{s_j}\,\chi_k)\,\partial_{s_i}\,\nu_j
= - (\partial_{s_k}\,\nu_i)\,(\partial_{s_j}\,\chi_k)\,\partial_{s_j}\,\nu_i
\label{eq:noteA4}
\end{align}
and, on noting (\ref{eq:Weingarten}), that 
\begin{align}
T_2 & 
= - (\partial_{s_j}\,\nu_i)\,\partial_{s_j}\,((\partial_{s_i}\,\chi_k)\,\nu_k)
= - (\partial_{s_j}\,\nu_i)\,\partial_{s_j}\,(
 \partial_{s_i}\,(\chi_k\,\nu_k) - \chi_k\,\partial_{s_i}\,\nu_k)
\nonumber \\ & 
= - \partial_{s_j}\,((\partial_{s_j}\,\nu_i)\,\partial_{s_i}\,(\chi_k\,\nu_k))
+ (\partial_{s_j}\,\partial_{s_j}\,\nu_i)\,\partial_{s_i}\,(\chi_k\,\nu_k)
+ (\partial_{s_j}\,\nu_i)\,\partial_{s_j}\,(\chi_k\,\partial_{s_k}\,\nu_i)
\nonumber \\ & 
= - \partial_{s_j}\,((\partial_{s_j}\,\nu_i)\,\partial_{s_i}\,(\chi_k\,\nu_k))
+ (\partial_{s_j}\,\partial_{s_j}\,\nu_i)\,\partial_{s_i}\,(\chi_k\,\nu_k)
\nonumber \\ & \qquad
+ (\partial_{s_j}\,\nu_i)\left[
(\partial_{s_j}\,\partial_{s_k}\,\nu_i)\,\chi_k + 
 (\partial_{s_j}\,\chi_k)\,\partial_{s_k}\,\nu_i \right]
\nonumber \\ & 
= -\nabs\,.\,( (\nabs\,\vec\nu)\,\nabs\,(\vec\chi\,.\,\vec\nu) ) 
+ (\Delta_s\,\vec\nu) \,.\, \nabs\,(\vec\chi\,.\,\vec\nu)
\nonumber \\ & \qquad
+ (\partial_{s_j}\,\nu_i)\left[
(\partial_{s_j}\,\partial_{s_k}\,\nu_i)\,\chi_k + 
 (\partial_{s_j}\,\chi_k)\,\partial_{s_k}\,\nu_i \right] .
\label{eq:noteA5}
\end{align}
Next, we note from (\ref{eq:noteA1}) and (\ref{eq:Weingarten}) that
\begin{align}
\chi_k\,(\partial_{s_j}\,\nu_i)\,
(\partial_{s_j}\,\partial_{s_k}\,\nu_i) & 
= \chi_k\,(\partial_{s_j}\,\nu_i)\left[ \partial_{s_k}\,\partial_{s_j}\,\nu_i
- [(\nabs\,\vec\nu)\,\nabs\,\nu_i]_j\,\nu_k \right] \nonumber \\ & 
= \tfrac12\,\vec\chi\,.\,\nabs\,|\nabs\,\vec\nu|^2 - 
( (\nabs\,\vec\nu)^2 : \nabs\,\vec\nu)\,\vec\chi\,.\,\vec\nu\,.
\label{eq:noteA6}
\end{align}
Combining (\ref{eq:noteA3})--(\ref{eq:noteA6}) and noting (\ref{eq:Weingarten}) 
yields that
\begin{align}
\tfrac12\,\matparteps\,|\nabs\,\vec\nu|^2 & = 
-\nabs\,.\,( (\nabs\,\vec\nu)\,\nabs\,(\vec\chi\,.\,\vec\nu) ) 
+ (\Delta_s\,\vec\nu) \,.\, \nabs\,(\vec\chi\,.\,\vec\nu)
+ \tfrac12\,\vec\chi\,.\,\nabs\,|\nabs\,\vec\nu|^2  \nonumber \\ & \qquad
- \tr( (\nabs\,\vec\nu)^3 )\,\vec\chi\,.\,\vec\nu\,.
\label{eq:noteA7}
\end{align}
As the eigenvalues of $-\nabs\vec\nu$ are $0$, $\varkappa_1$ and $\varkappa_2$,
we have from (\ref{eq:Weingarten}) and (\ref{eq:secondform}) that
\begin{align} \label{eq:noteA8}
(\nabs\,\vec\nu)^2 : (\nabs\,\vec\nu) =
\tr( (\nabs\,\vec\nu)^3 ) & = - (\varkappa_1^3 + \varkappa_2^3)
= -(\varkappa_1^2 + \varkappa_2^2 - \varkappa_1\,\varkappa_2)\,
(\varkappa_1 + \varkappa_2) \nonumber \\ & 
= (\Gauss - |\nabs\,\vec\nu|^2)\,\varkappa\,.
\end{align}
Combining (\ref{eq:noteA7}) and (\ref{eq:noteA8}), on noting (\ref{eq:new}),
yields that 
\begin{align}
\tfrac12\,\matparteps\,|\nabs\,\vec\nu|^2 & = 
-\nabs\,.\,( (\nabs\,\vec\nu)\,\nabs\,(\vec\chi\,.\,\vec\nu) ) 
- (\nabs\,\varkappa) \,.\, \nabs\,(\vec\chi\,.\,\vec\nu)
+ \tfrac12\,(\nabs\,|\nabs\,\vec\nu|^2)\,.\,\vec\chi  \nonumber \\ & \qquad
+ (|\nabs\,\vec\nu|^2 - \Gauss)\,\varkappa\,\vec\chi\,.\,\vec\nu\,.
\label{eq:noteA9}
\end{align}
On recalling (\ref{eq:Gauss3d}) and (\ref{eq:matpartukappa}), and combining
with (\ref{eq:noteA9}), we finally have that
\begin{align}
\matparteps\,\Gauss & 
= \tfrac12\,\matparteps\,(\varkappa^2 - |\nabs\,\vec\nu|^2) =
\varkappa\,\matparteps\,\varkappa - \tfrac12\,\matparteps\,|\nabs\,\vec\nu|^2
\nonumber \\ & 
= \varkappa\left[ \Delta_s\,(\vec\chi\,.\,\vec\nu) 
+ |\nabs\,\vec\nu|^2\,\vec\chi\,.\,\vec\nu + \vec\chi\,.\,\nabs\,\varkappa
\right]
+ \nabs\,.\,( (\nabs\,\vec\nu)\,\nabs\,(\vec\chi\,.\,\vec\nu) ) 
\nonumber \\ & \qquad
+ (\nabs\,\varkappa) \,.\, \nabs\,(\vec\chi\,.\,\vec\nu)
- \tfrac12\,(\nabs\,|\nabs\,\vec\nu|^2)\,.\,\vec\chi
- (|\nabs\,\vec\nu|^2 - \Gauss)\,\varkappa\,\vec\chi\,.\,\vec\nu
\nonumber \\ & 
= \varkappa\,\Delta_s\,(\vec\chi\,.\,\vec\nu) 
+ \tfrac12\,(\nabs\,\varkappa^2)\,.\,\vec\chi 
- \tfrac12\,(\nabs\,|\nabs\,\vec\nu|^2)\,.\,\vec\chi
+ \nabs\,.\,( (\nabs\,\vec\nu)\,\nabs\,(\vec\chi\,.\,\vec\nu) ) 
\nonumber \\ & \qquad
+ (\nabs\,\varkappa) \,.\, \nabs\,(\vec\chi\,.\,\vec\nu)
+ \Gauss\,\varkappa\,\vec\chi\,.\,\vec\nu
\nonumber \\ & 
= \nabs\,.\,[(\varkappa\,\mat\Id +
\nabs\,\vec\nu)\,\nabs\,(\vec\chi\,.\,\vec\nu)] + \nabs\,\Gauss\,.\,\vec\chi
+ \Gauss\,\varkappa\,\vec\chi\,.\,\vec\nu \,.
\label{eq:noteA10}
\end{align}

On noting (\ref{eq:DElem5.2eps}), (\ref{eq:nabszeta}) 
and (\ref{eq:noteA10}), we have that
\begin{subequations}
\begin{align} &
\left[ \deldel\Gamma \left\langle \alpha^G(\phasec) , \Gauss 
\right\rangle_{\Gamma(t)} \right] (\vec{\chi}) \nonumber \\ &\quad
= 
\left\langle \alpha^G(\phasec) , \matparteps\,\Gauss \right\rangle_{\Gamma(t)}
+ \left\langle \alpha^G(\phasec)\,\Gauss,\nabs\,.\,\vec\chi \right\rangle_{\Gamma(t)}
\nonumber \\ &\quad
= - \left\langle \Gauss, 
\vec\chi\,.\,\nabs\,\alpha^G(\phasec) \right\rangle_{\Gamma(t)}
+ \left\langle \alpha^G(\phasec), \matparteps\,\Gauss - \vec\chi\,.\,
\nabs\,\Gauss - \varkappa\,\Gauss\,\vec\chi\,.\vec\nu
\right\rangle_{\Gamma(t)}
\nonumber \\ &\quad
= - \left\langle \Gauss, 
\vec\chi\,.\,\nabs\,\alpha^G(\phasec) \right\rangle_{\Gamma(t)}
+ \left\langle \alpha^G(\phasec), \nabs\,.\,
[ ( \varkappa\,\mat\Id + \nabs\,\vec\nu)\,\nabs\,(\vec\chi\,.\,\vec\nu)]
\right\rangle_{\Gamma(t)}
\nonumber \\ &\quad
= - \left\langle \Gauss, 
\vec\chi\,.\,\nabs\,\alpha^G(\phasec) \right\rangle_{\Gamma(t)}
+ \left\langle \nabs\,.\, [ ( \varkappa\,\mat\Id + \nabs\,\vec\nu)\,
\nabs\,\alpha^G(\phasec) ], \vec\chi\,.\,\vec\nu]
\right\rangle_{\Gamma(t)} .
\label{eq:Gaussvar}
\end{align}
In addition, it holds that
\begin{equation} \label{eq:Gaussvarc}
\left[ \deldel{\phasec} \left\langle  \alpha^G(\phasec), \Gauss
\right\rangle_{\Gamma(t)} \right] (\eta) =
\left\langle (\alpha^G)'(\phasec)\,\eta, \Gauss \right\rangle_{\Gamma(t)} .
\end{equation}
\end{subequations}
Once again, (\ref{eq:Gaussvar},b) collapses to \citet[(4.6)]{ElliottS10a} 
if $\vec\chi = \chi\,\vec\nu$.
Finally, the Cayley--Hamilton theorem applied to 
$-\nabs\,\vec\nu$ yields, on recalling (\ref{eq:Weingarten}), that
\begin{align*}
(\nabs\,\vec\nu)^3 + \varkappa\,(\nabs\,\vec\nu)^2 +
\Gauss\,\nabs\,\vec\nu & = \mat 0  
\quad \Rightarrow \quad 
(\nabs\,\vec\nu + \varkappa\,\mat\Id)\,\mat{\mathcal{P}_\Gamma} 
= \Gauss\,(-\nabs\,\vec\nu)^{-1}\,\mat{\mathcal{P}_\Gamma} \,,
\end{align*}
where we note that $(\nabs\,\vec\nu)^{-1}$ is well-defined on the tangent 
space.
With this identity it is possible to show that $\nabs\,.\,(
[\varkappa\,\mat\Id + \nabs\,\vec\nu ]\,\nabs\,\alpha^G(\phasec))=
\Hat\Delta_s\,\alpha^G (\phasec)$, where $\Hat\Delta_s$ is the second surface
Laplacian used in the paper of
\cite{MerckerM-CRH13} to derive the first variation of the Gaussian curvature
bending energy.
However, comparing our (\ref{eq:fGamma}) and e.g.\ Lemma~5.1 in 
\cite{MerckerM-CRH13}, there appears to be a sign discrepancy in the latter.

It follows from (\ref{eq:appvar2},b), (\ref{eq:ECHvar},b),
(\ref{eq:Gaussvar},b) and (\ref{eq:E},b) that
\begin{subequations}
\begin{equation} \label{eq:A2final}
\left[\deldel{\Gamma}\,E(\Gamma(t),\phasec(t))\right](\vec\chi) 
= - \left\langle \vec f_\Gamma, \vec\chi \right\rangle_{\Gamma(t)}
\end{equation}
and
\begin{equation} \label{eq:A2finalc}
\left[\deldel{\phasec}\,E(\Gamma(t),\phasec(t))\right](\eta) 
= \left\langle \chempot, \eta \right\rangle_{\Gamma(t)},
\end{equation}
\end{subequations}
where $\vec f_\Gamma$ and $\chempot$ are defined in
(\ref{eq:fGamma}) and (\ref{eq:strongCHb}), respectively.

\subsection{Weak formulation equals strong formulation}
Here we show that the weak formulation 
(\ref{eq:LM3a}--e), (\ref{eq:CHb}) equals the strong formulation
(\ref{eq:fGamma}), (\ref{eq:strongCHb}). 

Recall from (\ref{eq:ddGLa}) and (\ref{eq:LM3a}) 
that minus the first variation of the 
Lagrangian (\ref{eq:Lag3}) with respect to the geometry is given by 
\begin{align} 
\vec f_\Gamma & = - \left[\deldel{\Gamma}\,L\right](\vec\chi) 
\nonumber \\ & =
\left\langle \nabs\,\vec y, \nabs\,\vec\chi \right\rangle_{\Gamma(t)}
+ \left\langle \nabs\,.\,\vec y, \nabs\,.\,\vec\chi \right\rangle_{\Gamma(t)}
-2\left\langle (\nabs\,\vec y)^T , \mat D_s(\vec \chi)\,(\nabs\,\vec\id)^T
\right\rangle_{\Gamma(t)}
\nonumber \\ & \qquad 
-\tfrac12\left\langle [ \alpha(\phasec)\,|\vec\varkappa - \spont(\phasec)\,\vec\nu|^2
- 2\,(\vec y\,.\,\vec\varkappa)]\,\nabs\,\vec\id,\nabs\,\vec\chi
\right\rangle_{\Gamma(t)}
\nonumber \\ & \qquad 
- \left\langle \alpha(\phasec)\,(\vec\varkappa -\spont(\phasec)\,\vec\nu)\,\spont(\phasec),
[\nabs\,\vec\chi]^T\,\vec \nu \right\rangle_{\Gamma(t)}
\nonumber \\ & \qquad 
- \beta 
\left\langle \tfrac12\,\gamma\,|\nabs\,\phasec|^2 + \gamma^{-1}\,\Psi(\phasec),
\nabs\,.\,\vec\chi \right\rangle_{\Gamma(t)}
+ \beta\,\gamma \left\langle (\nabs\,\phasec)\otimes(\nabs\,\phasec) , \nabs\,\vec\chi
\right\rangle_{\Gamma(t)} 
\nonumber \\ & \qquad 
- \tfrac12 \left\langle \alpha^G(\phasec) \,( |\vec\varkappa|^2 - |\mat w|^2 ),
\nabs\,.\,\vec\chi \right\rangle_{\Gamma(t)} 
+ \left\langle \mat w: \mat z,\nabs\,.\,\vec\chi
\right\rangle_{\Gamma(t)} 
+ \left\langle \vec\nu\,.\,(\nabs\,.\,\mat z),\nabs\,.\,\vec\chi
\right\rangle_{\Gamma(t)} 
\nonumber \\ & \qquad 
+ \left\langle \vec\nu\,.\,(\mat z\,\vec\varkappa),\nabs\,.\,\vec\chi
\right\rangle_{\Gamma(t)} 
\nonumber \\ & \qquad 
+ \sum_{i=1}^d \left[
\left\langle \nu_i\,\nabs\,\vec z_i,\nabs\,\vec\chi
\right\rangle_{\Gamma(t)} 
- 2 \left\langle \nu_i\,(\nabs\,\vec z_i)^T,\mat D_s(\vec\chi)\,
(\nabs\,\vec\id)^T \right\rangle_{\Gamma(t)} \right]
\nonumber \\ & \qquad 
- \left\langle \mat z\,\vec\varkappa, [\nabs\,\vec\chi]^T\,\vec\nu
\right\rangle_{\Gamma(t)} 
- \left\langle \nabs\,.\,\mat z, [\nabs\,\vec\chi]^T\,\vec\nu
\right\rangle_{\Gamma(t)} 
 = \sum_{\ell=1}^{14} T_\ell
\label{eq:app1}
\end{align}
for $\vec\chi \in \Vt$.

On recalling Remark~\ref{rem:simplifications2} and
$\vec\varkappa = \varkappa\,\vec\nu$, we have that
\begin{equation} \label{eq:matz}
\mat z = -\alpha^G(\phasec)\,\mat w = -\alpha^G(\phasec)\,\nabs\,\vec\nu
\quad\Rightarrow\quad
\vec z_i = \mat z\, \vec\ek_i = - \alpha^G(\phasec)\, \nabs\,\nu_i = - \alpha^G(\phasec)\,
\partial_{s_i}\,\vec\nu\,,
\end{equation}
and so it follows that $\mat z\,\vec\varkappa = \vec 0$. 
Hence $T_{11}= T_{13} = 0$. Moreover, 
$\vec\nu\,.\,[\nabs\,\vec\chi]^T\,\vec\nu = 0$, which implies that $T_5=0$.
In addition, we recall from (\ref{eq:LM3b}) 
and $\mat w^T\,\vec\nu = \mat w\,\vec\nu = \vec 0$
that 
\begin{equation} \label{eq:vecy}
\vec y = y\,\vec\nu \quad\text{with}\quad
y = \alpha(\phasec)\,(\varkappa - \spont(\phasec)) + \alpha^G(\phasec)\,\varkappa\,,
\end{equation}
and so as $\nabs\,.\,\vec\nu = - \varkappa$ it holds, on recalling
(\ref{eq:bendingb}), (\ref{eq:matw}) and (\ref{eq:nabszeta}), that
\begin{align} 
\sum_{\ell\in\{2,4,8\}} T_\ell & 
= - \left\langle y\,\varkappa, 
\nabs\,.\,\vec\chi \right\rangle_{\Gamma(t)} +T_4 + T_8
= - \left\langle b(\varkappa,\phasec) + \alpha^G(\phasec)\,\Gauss, 
\nabs\,.\,\vec\chi \right\rangle_{\Gamma(t)}
 \nonumber \\ & 
= \left\langle \nabs\left[b(\varkappa,\phasec) + \alpha^G(\phasec)\,\Gauss\right], 
\vec\chi \right\rangle_{\Gamma(t)}
+ \left\langle \left[b(\varkappa,\phasec) + \alpha^G(\phasec)\,\Gauss\right]\varkappa, 
\vec\chi\,.\,\vec\nu \right\rangle_{\Gamma(t)}
.
\label{eq:T248}
\end{align}
It also holds, on noting \citet[(A.22), (A.19)]{pwfopen} and (\ref{eq:DD}), 
where we stress that the notation $\mat D(\vec\chi)$ there differs from 
$\mat D_s(\vec\chi)$ here by a factor 2 and by the absence of the projections
$\mat{\mathcal{P}}_\Gamma$, that
\begin{align}
\sum_{\ell\in\{1,3\}} T_\ell & = 
\left\langle \nabs\,(y\,\vec\nu), \nabs\,\vec\chi \right\rangle_{\Gamma(t)}
-2\left\langle [\nabs\,(y\,\vec\nu)]^T, \mat D_s(\vec \chi)\,(\nabs\,\vec\id)^T
\right\rangle_{\Gamma(t)}
\nonumber \\ & =
\left\langle \nabs\,(y\,\vec\nu) , \nabs\,\vec\chi
\right\rangle_{\Gamma(t)}
- \left\langle [\nabs\,(y\,\vec\nu)]^T ,
( \nabs\,\vec\chi + (\nabs\,\vec\chi)^T )\,\mat{\mathcal P}_\Gamma 
\right\rangle_{\Gamma(t)}
\nonumber \\ & =
\left\langle \nabs\,(y\,\vec\nu), (\vec\nu\otimes\vec\nu)\,\nabs\,\vec\chi 
\right\rangle_{\Gamma(t)}
- \left\langle y\,\nabs\,\vec\nu, \nabs\,\vec\chi \right\rangle_{\Gamma(t)}
\nonumber \\ & =
\left\langle \nabs\,y, \nabs\,(\vec\chi \,.\,\vec\nu) \right\rangle_{\Gamma(t)}
-\left\langle \nabs\,.\,[y\,(\nabs\,\vec\nu)^T\,\vec\chi], 1 
\right\rangle_{\Gamma(t)}
- \left\langle y\, (|\nabs\,\vec\nu|^2\,\vec\nu + \nabs\,\varkappa), \vec\chi
\right\rangle_{\Gamma(t)}
\nonumber \\ & =
\left\langle \nabs\,y, \nabs\,(\vec\chi\,.\,\vec\nu)\right\rangle_{\Gamma(t)}
- \left\langle y\, (|\nabs\,\vec\nu|^2\,\vec\nu + \nabs\,\varkappa), \vec\chi
\right\rangle_{\Gamma(t)}
, \label{eq:T1and3}
\end{align}
where in the last equality we have noted that $\Gamma(t)$ is a closed surface.
Moreover, we note from (\ref{eq:vecy}) and (\ref{eq:bendingb}) that
\begin{align}
y\,\nabs\,\varkappa & = [\alpha(\phasec)\,(\varkappa - \spont(\phasec))
+ \alpha^G(\phasec)\,\varkappa]\,\nabs\,\varkappa  \nonumber\\ &
= \nabs\,(b(\varkappa,\phasec) + \tfrac12\,\alpha^G(\phasec)\,\varkappa^2)
- [b_{,\phasec}(\varkappa,\phasec)
+ \tfrac12\,(\alpha^G)'(\phasec)\,\varkappa^2]\,\nabs\,\phasec\,.
\label{eq:ynabsvarkappa}
\end{align}
Combining (\ref{eq:T248}), (\ref{eq:T1and3}) and (\ref{eq:ynabsvarkappa}) 
yields, on noting (\ref{eq:nabszeta}), (\ref{eq:vecy}) and
(\ref{eq:matw}), that
\begin{align}
\sum_{\ell\in\{1,\ldots,4,8\}} T_\ell & =
- \left\langle \Delta_s\,[\alpha(\phasec)\,(\varkappa - \spont(\phasec))
+ \alpha^G(\phasec)\,\varkappa], \vec\chi\,.\,\vec\nu \right\rangle_{\Gamma(t)}
\nonumber\\ & \qquad
+  \left\langle [b(\varkappa,\phasec) + \alpha^G(\phasec)\,\Gauss]\,\varkappa, 
\vec\chi\,.\,\vec\nu \right\rangle_{\Gamma(t)}
+ \left\langle [b_{,\phasec}(\varkappa,\phasec) + \tfrac12\,(\alpha^G)'(\phasec)\,\varkappa^2]
\,\nabs\,\phasec,\vec\chi\right\rangle_{\Gamma(t)}
\nonumber\\ & \qquad
-\left\langle (\alpha(\phasec)\,(\varkappa - \spont(\phasec)) + \alpha^G(\phasec)\,\varkappa)
\, |\nabs\,\vec\nu|^2, \vec\chi\,.\,\vec\nu
\right\rangle_{\Gamma(t)} 
\nonumber\\ & \qquad
- \tfrac12 \left\langle \nabs\,(\alpha^G(\phasec)\,|\nabs\,\vec\nu|^2), \vec\chi
\right\rangle_{\Gamma(t)} .
\label{eq:T1to48b}
\end{align}

It holds on noting (\ref{eq:nabszeta}) that
\begin{align}
\sum_{\ell\in\{6,7\}} T_\ell &=
\beta \left\langle \nabs \,[\tfrac12\,\gamma\,|\nabs\,\phasec|^2 
+ \gamma^{-1}\,\Psi(\phasec)] ,
\vec\chi \right\rangle_{\Gamma(t)} 
+ \beta \left\langle [\tfrac12\,\gamma\,|\nabs\,\phasec|^2 +
\gamma^{-1}\,\Psi(\phasec)]\,\varkappa, 
\vec\chi\,.\,\vec\nu \right\rangle_{\Gamma(t)} 
\nonumber\\ & \qquad
- \beta\,\gamma \left\langle  
\nabs\,.\,((\nabs\,\phasec) \otimes (\nabs\,\phasec)),
\vec\chi \right\rangle_{\Gamma(t)} .
\label{eq:T67}
\end{align}
In addition, we have from (\ref{eq:matz}) that
\begin{align}
\sum_{\ell\in\{9,10\}} T_\ell &=
\left\langle \mat w : \mat z + \vec\nu\,.\,(\nabs\,.\,\mat z),
\nabs\,.\,\vec\chi \right\rangle_{\Gamma(t)} \nonumber\\ & 
= - \left\langle \alpha^G(\phasec)\,|\nabs\,\vec\nu|^2 + \vec\nu\,.\,[
\nabs\,.\,(\alpha^G(\phasec)\,\nabs\,\vec\nu)], \nabs\,.\,\vec\chi 
\right\rangle_{\Gamma(t)} = 0 \,,
\label{eq:T910}
\end{align}
where we have observed from (\ref{eq:new}) that
\begin{align}
\vec\nu\,.\,[\nabs\,.\,(\alpha^G(\phasec)\,\nabs\,\vec\nu)] & =
\vec\nu\,.\,[\alpha^G(\phasec)\,\Delta_s\,\vec\nu + (\nabs\,\vec\nu)\,
\nabs\,\alpha^G(\phasec)] =
\alpha^G(\phasec)\,\vec\nu\,.\,\Delta_s\,\vec\nu \nonumber \\ & = 
\alpha^G(\phasec)\,\vec\nu\,.\,[ - |\nabs\,\vec\nu|^2\,\vec\nu - \nabs\,\varkappa]
= -\alpha^G(\phasec)\,|\nabs\,\vec\nu|^2\,.
\label{eq:alphaGweingarten}
\end{align}

It follows from (\ref{eq:DD}) and $\mat{\mathcal P}_\Gamma = \nabs\,\vec\id$ 
that
\begin{align}
T_{12} & = \sum_{i=1}^d \left[
\left\langle \nu_i\,\nabs\,\vec z_i,\nabs\,\vec\chi
\right\rangle_{\Gamma(t)} 
- 2 \left\langle \nu_i\,(\nabs\,\vec z_i)^T,\mat D_s(\vec\chi)\,
(\nabs\,\vec\id)^T \right\rangle_{\Gamma(t)} \right]
\nonumber \\ & 
= - \sum_{i=1}^d
\left\langle \nu_i\,(\nabs\,\vec z_i)^T,\nabs\,\vec\chi
\right\rangle_{\Gamma(t)} ,
\label{eq:T12a}
\end{align}
provided that we can show that
\begin{align}
& \sum_{i=1}^d \left[
\left\langle \nu_i\,\nabs\,\vec z_i,\nabs\,\vec\chi
\right\rangle_{\Gamma(t)} 
-  \left\langle \nu_i\,(\nabs\,\vec z_i)^T,(\nabs\,\vec\chi)^T\, 
\mat{\mathcal P}_\Gamma
\right\rangle_{\Gamma(t)} 
\right]
= 0\,.
\label{eq:T12proof}
\end{align}
In order to
establish (\ref{eq:T12proof}), we note, on recalling (\ref{eq:matz}) and
(\ref{eq:DE13}), that
\begin{align}
& 
\nu_i\,\nabs\,\vec z_i : \nabs\,\vec\chi - 
\nu_i\,(\nabs\,\vec z_i)^T : [(\nabs\,\vec\chi)^T\, \mat{\mathcal P}_\Gamma]
 = 
\nu_i\,(\nabs\,\vec z_i)_{kj} \left[
\partial_{s_j}\,\chi_k - (\partial_{s_j}\,\chi_l)\,(\delta_{lk} -
\nu_l\,\nu_k)\right]
\nonumber \\ & \qquad
= \nu_i\,(\nabs\,\vec z_i)_{kj}\,(\partial_{s_j}\,\chi_l)\,\nu_l\,\nu_k
= -\nu_i\,[\partial_{s_j}\,(\alpha^G(\phasec)\,\partial_{s_k}\nu_i)]
\,(\partial_{s_j}\,\chi_l)\,\nu_l\,\nu_k
\nonumber \\ & \qquad
= -\nu_i\,\nu_l\,\nu_k\,\alpha^G(\phasec)\,(\partial_{s_j}\,\partial_{s_k}\,\nu_i)\,
\,\partial_{s_j}\,\chi_l
= \nu_i\,\nu_l\,\nu_k\,\alpha^G(\phasec)\,
[(\nabs\,\vec\nu)\,\nabs\,\nu_i]_j\,\nu_k\,\partial_{s_j}\,\chi_l
\nonumber \\ & \qquad
= \nu_i\,\nu_l\,\alpha^G(\phasec)\,
[(\nabs\,\vec\nu)\,\nabs\,\nu_i]_j\,\partial_{s_j}\,\chi_l
= \nu_i\,\nu_l\,\alpha^G(\phasec)\,
(\partial_{s_k}\,\nu_j)\,(\partial_{s_k}\,\nu_i)\,\partial_{s_j}\,\chi_l
= 0\,,
\label{eq:T12proofa}
\end{align}
since $\nu_i\,\partial_{s_k}\,\nu_i = \frac12\,\partial_{s_k}\,|\vec\nu|^2=0$.

Returning to (\ref{eq:T12a}), we have on noting (\ref{eq:matz}),
(\ref{eq:nabszeta}), (\ref{eq:DE13}) and (\ref{eq:new}) that
\begin{align} &
T_{12} = - \sum_{i=1}^d 
\left\langle \nu_i\,\nabs\,\vec z_i,(\nabs\,\vec\chi)^T
\right\rangle_{\Gamma(t)} 
= \left\langle \nu_i\,\partial_{s_l}\,(\alpha^G(\phasec)\,\partial_{s_k}\,\nu_i),
\partial_{s_k}\,\chi_l \right\rangle_{\Gamma(t)} 
\nonumber \\ & \quad =
\left\langle \alpha^G(\phasec)\,\nu_i\,\partial_{s_l}\,\partial_{s_k}\,\nu_i,
\partial_{s_k}\,\chi_l \right\rangle_{\Gamma(t)} 
= - \left\langle \alpha^G(\phasec)\,(\partial_{s_l}\,\nu_i)\,\partial_{s_k}\,\nu_i,
\partial_{s_k}\,\chi_l \right\rangle_{\Gamma(t)} 
\nonumber \\ & \quad =
\left\langle \alpha^G(\phasec)\left[
(\partial_{s_k}\,\partial_{s_l}\,\nu_i)\,\partial_{s_k}\,\nu_i 
+ (\partial_{s_l}\,\nu_i)\,\partial_{s_k}\,\partial_{s_k}\,\nu_i \right]
+ (\partial_{s_k}\,\alpha^G(\phasec))\,(\partial_{s_l}\,\nu_i)\,\partial_{s_k}\,\nu_i
, \chi_l \right\rangle_{\Gamma(t)} 
\nonumber \\ & \quad =
\left\langle \alpha^G(\phasec)\,\partial_{s_k}\,\nu_i\left[
\partial_{s_l}\,\partial_{s_k}\,\nu_i - [(\nabs\,\vec\nu)\,\nabs\,\nu_i]_k\,
\nu_l\right], \chi_l \right\rangle_{\Gamma(t)} 
- \left\langle \alpha^G(\phasec)\,(\partial_{s_l}\,\nu_i)\,
\partial_{s_i}\,\varkappa, \chi_l \right\rangle_{\Gamma(t)} 
\nonumber \\ & \qquad 
+ \left\langle (\partial_{s_k}\,\alpha^G(\phasec))\,(\partial_{s_l}\,\nu_i)\,
\partial_{s_k}\,\nu_i, \chi_l \right\rangle_{\Gamma(t)} 
\nonumber \\ & \quad =
\tfrac12 \left\langle \alpha^G(\phasec)\,\nabs\,|\nabs\,\vec\nu|^2, \vec\chi
 \right\rangle_{\Gamma(t)} 
- \left\langle \alpha^G(\phasec)\,(\nabs\,\vec\nu)^2, \nabs\,\vec\nu\,
(\vec\chi\,.\,\vec\nu) \right\rangle_{\Gamma(t)} 
\nonumber \\ & \qquad 
- \left\langle \alpha^G(\phasec)\,(\nabs\,\vec\nu)\,\nabs\,\varkappa, 
\vec\chi \right\rangle_{\Gamma(t)} 
+ \left\langle (\nabs\,\vec\nu)^2\,\nabs\,\alpha^G(\phasec),
\vec\chi \right\rangle_{\Gamma(t)} . \label{eq:T12b}
\end{align}
The remaining term from (\ref{eq:app1}) can be rewritten, on noting 
(\ref{eq:matz}), (\ref{eq:nabszeta}), (\ref{eq:new}) and 
(\ref{eq:Weingarten}), as 
\begin{align} &
T_{14} =
 - \sum_{i=1}^d 
\left\langle (\partial_{s_i}\,\vec z_i) \otimes \vec\nu, (\nabs\,\vec\chi)^T
\right\rangle_{\Gamma(t)} 
\nonumber \\ & \quad =
 \left\langle \partial_{s_i}\,(\alpha^G(\phasec)\,\partial_{s_i}\,\nu_k),\nu_l\,
\partial_{s_k}\,\chi_l \right\rangle_{\Gamma(t)} 
\nonumber \\ & \quad =
 \left\langle (\partial_{s_i}\,\alpha^G(\phasec))\,\partial_{s_i}\,\nu_k 
+ \alpha^G(\phasec)\,\partial_{s_i}\,\partial_{s_i}\,\nu_k , 
\nu_l\,\partial_{s_k}\,\chi_l \right\rangle_{\Gamma(t)} 
\nonumber \\ & \quad =
 \left\langle (\partial_{s_i}\,\alpha^G(\phasec))\,\partial_{s_i}\,\nu_k 
- \alpha^G(\phasec)\,\partial_{s_k}\,\varkappa , 
\nu_l\,\partial_{s_k}\,\chi_l \right\rangle_{\Gamma(t)} 
\nonumber \\ & \quad =
- \left\langle (\partial_{s_k}\,[(\partial_{s_i}\,\alpha^G(\phasec))\,
\partial_{s_i}\,\nu_k])\,\nu_l +
(\partial_{s_i}\,\alpha^G(\phasec))\,(\partial_{s_i}\,\nu_k)\,\partial_{s_k}\,\nu_l
, \chi_l \right\rangle_{\Gamma(t)} 
\nonumber \\ & \qquad
- \left\langle (\partial_{s_i}\,\alpha^G(\phasec))\,
(\partial_{s_k}\,\nu_i)\,\varkappa\,\nu_k , \chi_l\,\nu_l
\right\rangle_{\Gamma(t)} 
\nonumber \\ & \qquad
+ \left\langle (\partial_{s_k}\,[\alpha^G(\phasec)\,\partial_{s_k}\,\varkappa])\,
\nu_l + \alpha^G(\phasec)\,(\partial_{s_k}\,\varkappa)\,\partial_{s_k}\,\nu_l, \chi_l 
\right\rangle_{\Gamma(t)} 
\nonumber \\ & \quad =
- \left\langle \nabs\,.\,[(\nabs\,\vec\nu)\,\nabs\,\alpha^G(\phasec)],
  \vec\chi\,.\,\vec\nu  \right\rangle_{\Gamma(t)} 
- \left\langle (\nabs\,\vec\nu)^2\,\nabs\,\alpha^G(\phasec), \vec\chi 
 \right\rangle_{\Gamma(t)} 
\nonumber \\ & \qquad
+ \left\langle \nabs\,.\,(\alpha^G(\phasec)\,\nabs\,\varkappa), \vec\chi\,.\,\vec\nu
 \right\rangle_{\Gamma(t)} 
+ \left\langle \alpha^G(\phasec)\,(\nabs\,\vec\nu)\,\nabs\,\varkappa, \vec\chi
 \right\rangle_{\Gamma(t)} . \label{eq:T14a}
\end{align}
Hence we have from (\ref{eq:T12b}) and (\ref{eq:T14a}), on noting
(\ref{eq:noteA8}) for $d=3$ and on recalling that $\alpha^G=0$ for $d=2$, that
\begin{align}
\sum_{\ell\in\{12,14\}} T_\ell & =
\tfrac12 
\left\langle \alpha^G(\phasec)\,\nabs\,|\nabs\,\vec\nu|^2, \vec\chi
 \right\rangle_{\Gamma(t)} 
+ \left\langle \alpha^G(\phasec)\,[ |\nabs\,\vec\nu|^2 - \Gauss]\,\varkappa, 
\vec\chi\,.\,\vec\nu \right\rangle_{\Gamma(t)} 
\nonumber \\ & \qquad 
-\left\langle \nabs\,.\,[(\nabs\,\vec\nu)\,\nabs\,\alpha^G(\phasec)], 
\vec\chi\,.\,\vec\nu \right\rangle_{\Gamma(t)} 
+ \left\langle \nabs\,.\,[\alpha^G(\phasec)\,\nabs\,\varkappa], 
\vec\chi\,.\,\vec\nu \right\rangle_{\Gamma(t)} 
.
\label{eq:T1214}
\end{align}

Combining (\ref{eq:T1214})  with (\ref{eq:T1to48b}) yields, on recalling
(\ref{eq:matw}), that
\begin{align}
\sum_{\ell\in\{1,\ldots4,8,12,14\}} T_\ell & =
- \left\langle \Delta_s\,[\alpha(\phasec)\,(\varkappa - \spont(\phasec))]
+ \alpha(\phasec)\,(\varkappa - \spont(\phasec))\,|\nabs\vec\nu|^2 
- b(\varkappa,\phasec)\,\varkappa, \vec\chi\,.\,\vec\nu \right\rangle_{\Gamma(t)}
\nonumber\\ & \qquad
- \left\langle \nabs\,.\,([\varkappa\,\mat\Id + \nabs\,\vec\nu]\,
\nabs\,\alpha^G(\phasec)), \vec\chi\,.\,\vec\nu \right\rangle_{\Gamma(t)}
\nonumber\\ & \qquad
+ \left\langle [b_{,\phasec}(\varkappa,\phasec) + (\alpha^G)'(\phasec)\,\Gauss]
\,\nabs\,\phasec,\vec\chi\right\rangle_{\Gamma(t)}
 .
\label{eq:T1to481214}
\end{align}
Summing (\ref{eq:T1to481214}) and (\ref{eq:T67}) yields the strong form 
(\ref{eq:fGamma}). 

Finally, (\ref{eq:CHb}), (\ref{eq:bc}), (\ref{eq:matz}), (\ref{eq:Gauss3d}) and
(\ref{eq:nabszeta}) immediately yield (\ref{eq:strongCHb}). 
\end{appendix}

\noindent{\bf Acknowledgements}

\noindent The authors gratefully acknowledge the support 
of the Regensburger Universit\"atsstiftung Hans Vielberth.

\providecommand\noopsort[1]{}\def\soft#1{\leavevmode\setbox0=\hbox{h}\dimen7=\ht0\advance
  \dimen7 by-1ex\relax\if t#1\relax\rlap{\raise.6\dimen7
  \hbox{\kern.3ex\char'47}}#1\relax\else\if T#1\relax
  \rlap{\raise.5\dimen7\hbox{\kern1.3ex\char'47}}#1\relax \else\if
  d#1\relax\rlap{\raise.5\dimen7\hbox{\kern.9ex \char'47}}#1\relax\else\if
  D#1\relax\rlap{\raise.5\dimen7 \hbox{\kern1.4ex\char'47}}#1\relax\else\if
  l#1\relax \rlap{\raise.5\dimen7\hbox{\kern.4ex\char'47}}#1\relax \else\if
  L#1\relax\rlap{\raise.5\dimen7\hbox{\kern.7ex
  \char'47}}#1\relax\else\message{accent \string\soft \space #1 not
  defined!}#1\relax\fi\fi\fi\fi\fi\fi}

\end{document}